\documentclass[a4paper]{article}
\usepackage{amsmath,amsthm,amsfonts,amssymb}
\usepackage{thmtools}
\usepackage[T1]{fontenc}
\usepackage{placeins}
\usepackage{graphicx}
\usepackage{subfigure}
\usepackage{euscript}
\usepackage{units}
\usepackage{enumitem}
\usepackage[bookmarks=true,hidelinks]{hyperref}
\usepackage{bbm}

\renewcommand{\thefootnote}{\fnsymbol{footnote}}	

\newtheorem{theorem}{Theorem}[section]

\newtheorem{corollary}[theorem]{Corollary}
\newtheorem{proposition}[theorem]{Proposition}

\theoremstyle{definition}
\newtheorem{definition}[theorem]{Definition}
\newtheorem{terminology}[theorem]{Terminology}
\newtheorem{remark}[theorem]{Remark}
\newtheorem*{remark*}{Remark}

\newtheorem{example}[theorem]{Example}
\newtheorem{algorithm}[theorem]{Algorithm}
\numberwithin{equation}{section}
\numberwithin{figure}{section}
\numberwithin{table}{section}
\renewcommand{\thefootnote}{\fnsymbol{footnote}}	

\def\Xint#1{\mathchoice
{\XXint\displaystyle\textstyle{#1}}
{\XXint\textstyle\scriptstyle{#1}}
{\XXint\scriptstyle\scriptscriptstyle{#1}}
{\XXint\scriptscriptstyle\scriptscriptstyle{#1}}
\!\int}
\def\XXint#1#2#3{{\setbox0=\hbox{$#1{#2#3}{\int}$ }
\vcenter{\hbox{$#2#3$ }}\kern-.57\wd0}}
\def\intavg{\Xint-}

\DeclareMathOperator*{\esssup}{ess\ sup}
\DeclareMathOperator{\supp}{supp}

\DeclareMathOperator{\TV}{TV}

\newcommand{\wto}{\rightharpoonup}
\newcommand{\wsto}{\overset{*}{\wto}}

\renewcommand{\geq}{\geqslant} 
\renewcommand{\leq}{\leqslant}  

\newcommand{\eps} {\varepsilon}

\renewcommand{\epsilon}{\varepsilon}
\renewcommand{\phi}{\varphi}
\renewcommand{\i}{\ifmmode\mathit{\mathchar"7010 }\else\char"10 \fi}
\renewcommand{\j}{\ifmmode\mathit{\mathchar"7011 }\else\char"11 \fi}

\newcommand{\Soln}{\EuScript{S}}
\newcommand{\mathbbx}{\mathbb}

\newcommand{\E}{\mathbbx{E}}
\newcommand{\R}{\mathbbx{R}}
\newcommand{\N}{\mathbbx{N}}
\newcommand{\D}{\EuScript{D}}
\newcommand{\Z}{\mathbbx{Z}}
\newcommand{\M}{\EuScript{M}}
\newcommand{\Young}{\mathbf{Y}}
\newcommand{\Sigmaalg}{\EuScript{F}}		
\newcommand{\Borel}{\EuScript{B}}

\newcommand{\Prob}{\EuScript{P}}
\newcommand{\Emv}{\EuScript{E}}
\newcommand{\imv}{\sigma}		
\newcommand{\rand}{X}	
\newcommand{\holder}{\gamma}	
\newcommand{\tve}{r}		
\newcommand{\ballfunc}{\kappa}	
\newcommand{\vel}{w}
\newcommand{\velx}{\vel^1}		
\newcommand{\vely}{\vel^2}		
\newcommand{\ev}{{v}}
\newcommand{\amp}{\eps}		
\DeclareMathOperator{\id}{id}
\newcommand{\ind}{\mathbbm{1}}

\newcommand{\cell}{\EuScript{C}}
\newcommand{\xt}{z}	
\newcommand{\Dx}{{\Delta x}}
\newcommand{\Dy}{\Delta y}

\newcommand{\Dl}{{\Delta x}}

\newcommand{\ip}[2]{\langle #1, #2\rangle}
\newcommand{\avg}[1]{\overline{#1}}
\newcommand{\sgn}{{\rm sgn}}

\newcommand{\hf}{{\unitfrac{1}{2}}}
\newcommand{\iphf}{{i+\hf}}
\newcommand{\imhf}{{i-\hf}}

\newcommand{\jphf}{{j+\hf}}
\newcommand{\jmhf}{{j-\hf}}

\newcommand{\nablax}{{\nabla_x}}

\begin{document}

\title{
Construction of approximate entropy measure valued solutions for hyperbolic systems of conservation laws}
\date{June 23, 2015} 
\author{Ulrik S. Fjordholm\thanks{Department of Mathematical Sciences, Norwegian University of Science and Technology, Trondheim, N-7491, Norway}, \,  Roger K\"appeli\thanks{Seminar for Applied Mathematics, ETH Z\"urich, R\"amistrasse 101, Z\"urich, Switzerland}, \, Siddhartha Mishra\footnotemark[2], \, Eitan Tadmor\thanks{Department of Mathematics, Center of Scientific Computation and Mathematical Modeling (CSCAMM), Institute for Physical sciences and Technology (IPST), University of Maryland MD 20742-4015, USA}\, {}\thanks{S.M. and R.K were supported in part by ERC STG. N 306279, SPARCCLE. E.T. was supported in part by NSF grants DMS10-08397, RNMS11-07444 (KI-Net) and ONR grant N00014-1210318. Many of the computations were performed at CSCS Lugano through Project s345. SM thanks Prof. Christoph Schwab (ETH Zurich) for several helpful comments and suggestions.}}

\maketitle

\begin{verbatim}
    "There is no theory for the initial value problem for compressible 
    flows in two space dimensions once shocks show up, much less in three 
    space dimensions. This is a scientific scandal and a challenge."
\end{verbatim}

\hspace{7.3cm}{P. D. Lax, 2007 Gibbs Lecture \cite{Lax07}}

\begin{abstract}
Entropy solutions have been widely accepted as the suitable solution framework for systems of conservation laws in several space dimensions. However, recent results in \cite{CDL1,CDL2} have demonstrated that entropy solutions may not be unique. In this paper, we present numerical evidence that demonstrates that state of the art numerical schemes \emph{may not} necessarily converge to an entropy solution of systems of conservation laws as the mesh is refined. Combining these two facts, we argue that entropy solutions may not be suitable as a solution framework for systems of conservation laws, particularly in several space dimensions.

Furthermore, we propose a more general notion, that of \emph{entropy measure valued solutions}, as an appropriate solution paradigm for systems of conservation laws. To this end, we present a detailed numerical procedure, which constructs stable approximations to entropy measure valued solutions and provide sufficient conditions that guarantee that these approximations converge to an entropy measure valued solution as the mesh is refined, thus providing a viable numerical framework for systems of conservation laws in several space dimensions. A large number of numerical experiments that illustrate the proposed schemes are presented and are utilized to examine several interesting properties of the computed entropy measure valued solutions.

\end{abstract}

\noindent
{\sl 1991 Mathematics Subject Classification}. 65M06, 35L65, 35R06.

\noindent
{\bf Keywords}. {Hyperbolic conservation laws, uniqueness, stability, entropy condition, measure-valued solutions,
atomic initial data, random field, weak BV estimate, weak* convergence.}

\newpage
\setcounter{tocdepth}{2}
\tableofcontents
\newpage 

\section{Introduction}
A large number of problems in physics and engineering are modeled by \emph{systems of conservation laws}
\begin{subequations}\label{eq:cauchy}
\begin{align}
\partial_t u + \nablax\cdot f(u) &= 0  \label{eq:cl} \\
u(x,0) &= u_0(x).
\end{align}
\end{subequations}
Here, the unknown $u = u(x,t) : \R^d\times\R_+ \to \R^N$ is the vector of \emph{conserved variables} and $f = (f^1, \dots, f^d) : \R^N \to \R^{N\times d}$ is the \emph{flux function}. We denote $\R_+ := [0,\infty)$.

The system \eqref{eq:cl} is \emph{hyperbolic} if the flux Jacobian $\partial_u(f \cdot n)$ has real eigenvalues for all $n \in \R^d$ with $|n| = 1$. Examples for hyperbolic systems of conservation laws include the shallow water equations of oceanography, the Euler equations of gas dynamics, the magnetohydrodynamics (MHD) equations of plasma physics, the equations of nonlinear elastodynamics and the Einstein equations of general relativity. We refer to \cite{DAF1,GR1} for more theory on hyperbolic conservation laws.

\subsection{Mathematical framework}
It is well known that solutions of the Cauchy problem \eqref{eq:cauchy} can develop discontinuities such as shock waves in finite time, even when the initial data is smooth. Hence, solutions of hyperbolic systems of conservation laws \eqref{eq:cauchy} are sought in the weak (distributional) sense.
\begin{definition}
A function $u\in L^\infty(\R^d\times\R_+,\R^N)$ is a \emph{weak solution} of \eqref{eq:cauchy} if it satisfies \eqref{eq:cauchy} in the sense of distributions:
\begin{equation}\label{eq:wsoln}
\int_{\R_+}\int_{\R^d} \partial_t \phi(x,t) u(x,t) + \nablax\phi(x,t) \cdot f(u(x,t))\ dxdt + \int_{\R^d}\phi(x,0) u_0(x) \ dx = 0
\end{equation}
for all test functions $\phi\in C^1_c(\R^d\times\R_+)$.
\end{definition}

Weak solutions are in general not unique: infinitely many weak solutions may exist after the formation of discontinuities. Thus, to obtain uniqueness, additional admissibility criteria have to be imposed. These admissibility criteria take the form of entropy conditions, which are formulated in terms of entropy pairs.

\begin{definition}
A pair of functions $(\eta,q)$ with $\eta:\R^N\to\R$, $q:\R^N\to\R^d$ is called an \emph{entropy pair} if $\eta$ is convex and $q$ satisfies the compatibility condition $q' = \eta' \cdot f'$.
\end{definition}

\begin{definition}
A weak solution $u$ of \eqref{eq:cauchy} is an \emph{entropy solution} if the entropy inequality
\[
\partial_t\eta(u) + \nablax\cdot q(u) \leq 0 \qquad \text{in } \D'(\R^d\times\R_+)
\]
is satisfied for all entropy pairs $(\eta,q)$, that is, if
\begin{equation}\label{eq:entrcond}
\int_{\R_+}\int_{\R^d} \partial_t \phi(x,t) \eta(u(x,t)) + \nablax\phi(x,t) \cdot q(u(x,t))\ dxdt + \int_{\R^d}\phi(x,0) \eta(u_0(x)) \ dx \geq 0
\end{equation}
for all nonnegative test functions $0\leq\phi\in C^1_c(\R^d\times\R_+)$.
\end{definition}

For the special case of scalar conservation laws ($N=1$), every convex function $\eta$ gives rise to an entropy pair by letting $q(u) := \int^u \eta'(\xi) f'(\xi) d\xi$. This rich family of entropy pairs was used by Kruzkhov \cite{Kruz1} to obtain existence, uniqueness and stability of solutions for scalar conservation laws.

Corresponding (global) well-posedness results for systems of conservation laws are much harder to obtain. Lax \cite{Lax} showed existence and stability of entropy solutions for one-dimensional systems of conservation laws for the special case of Riemann initial data. Existence results for the Cauchy problem for one-dimensional systems were obtained by Glimm in \cite{Glimm1} using the random choice method and by Bianchini and Bressan \cite{BB1} with the vanishing viscosity method. Uniqueness and stability results for one-dimensional systems were shown by Bressan and co-workers \cite{Bress1}. All of these results rely on an assumption that the initial data is ``sufficiently small'', i.e., lies sufficiently close to some constant.

On the other hand, \emph{no global existence and uniqueness (stability) results are currently available} for a generic system of conservation laws in several space dimensions. In fact, recent results (see \cite{CDL1,CDL2,CDL3} and references therein) provide counterexamples which illustrate that entropy solutions for multi-dimensional systems of conservation laws \emph{are not necessarily unique}. These results raise serious questions about the appropriateness of entropy solutions as the standard solution framework for systems of conservation laws. It can be argued that one needs to impose even further admissibility criteria, in addition to the entropy inequality \eqref{eq:entrcond}, to single out a solution among the infinitely many solutions constructed in \cite{CDL1,CDL2,CDL3}. One possible approach in determining these selection criteria is to employ suitable numerical schemes and observe which, if any, of the entropy solutions are approximated by these schemes.

\subsection{Numerical schemes}
Numerical schemes have played a leading role in the study of systems of conservation laws, and a wide variety of numerical methods for approximating \eqref{eq:cauchy} are currently available. These include the very popular and highly successful numerical framework of finite volume and finite difference schemes,  based on approximate Riemann solvers or on Riemann-solver-free centered differencing (see \cite{GR1,CJST98,LEV1,centpack}) which utilize TVD \cite{Har83}, ENO \cite{HEOC1} or WENO \cite{JS96} non-oscillatory reconstruction techniques and strong stability preserving (SSP) Runge-Kutta time integrators \cite{GST1}. Another popular alternative is the Discontinuous Galerkin finite element method \cite{CS1}.

The primary goal in the analysis of numerical schemes approximating \eqref{eq:cauchy} is proving convergence to an entropy solution as the mesh is refined. This issue has been addressed in the special case of (first-order) monotone schemes for \emph{scalar} conservation laws (see \cite{CM1} for the one-dimensional case and \cite{KRO} for multiple dimensions) using the TVD property. Corresponding convergence results for (formally) arbitrarily high-order accurate finite difference schemes for scalar conservation laws were obtained recently in \cite{FJO1}, see also \cite{FTSID2}. Convergence results for (arbitrarily high order) space time DG discretization for scalar conservation laws were obtained in \cite{JJS1} and for the spectral viscosity method in \cite{Tad5}.

The question of convergence of numerical schemes for systems of conservation laws is significantly harder. Currently, there are \emph{no} rigorous proofs of convergence for any kind of finite volume (difference) and finite element methods to the entropy solutions of a generic system of conservation laws, even in one space dimension. Convergence aside, even the stability of numerical approximations to systems of conservation laws is mostly open. The only notion of numerical stability for systems of conservation laws that has been analyzed rigorously so far is that of entropy stability -- the design of numerical approximations that satisfy a discrete version of the entropy inequality. Such schemes have been devised in \cite{TAD1,TAD2, FTSID2, HSID1}. However, entropy stability may not suffice to ensure the convergence of approximate solutions.

\subsection{Two numerical experiments}
Given the lack of rigorous stability and convergence results for systems of conservation laws, it has become customary in the field to rely on numerical benchmark tests to demonstrate the convergence of the scheme empirically. One such benchmark test case is the radial Sod shock tube \cite{LEV1}.

\subsubsection{Sod shock tube}
In this test, we consider the compressible Euler equations of gas dynamics in two space dimensions (see Section \ref{sec:6}) as a prototypical hyperbolic system of conservation laws. The initial data for the two-dimensional version of the well-known Sod shock tube problem is given by
\begin{equation}
\label{eq:sodinit}
u_0(x) = \begin{cases}
u_L & \text{if } |x| \leq r_0 \\
u_R & \text{if } |x| > r_0,
\end{cases}
\end{equation}
with $\rho_L = p_L = 3$, $\rho_R=p_R=1$, $\velx = \vely = 0$. The computational domain is $[-0.5,0.5]^2$, with $r_0 = 0.15$, and we use periodic boundary conditions.

To begin with, we consider a perturbed version of the Sod shock tube by setting the initial data
\begin{equation}\label{eq:sodpinit}
u^{\amp}_0(x) = u_0(x) + \amp X(x),
\end{equation}
where $\amp>0$ is a small amplitude of the  perturbation $X(\cdot)$ associated with the following state variables ---
$\rho,p$ and $\vel=(\velx,\vely)^\top$,
\begin{equation}
\label{eq:sodper}
X_\rho = X_p = 0, \quad X_{\velx}(x) = \sin(2\pi x_1), \quad X_{\vely}(x) = \sin(2\pi x_2).
\end{equation}

\begin{figure}[ht]
\includegraphics[width=0.3\linewidth]{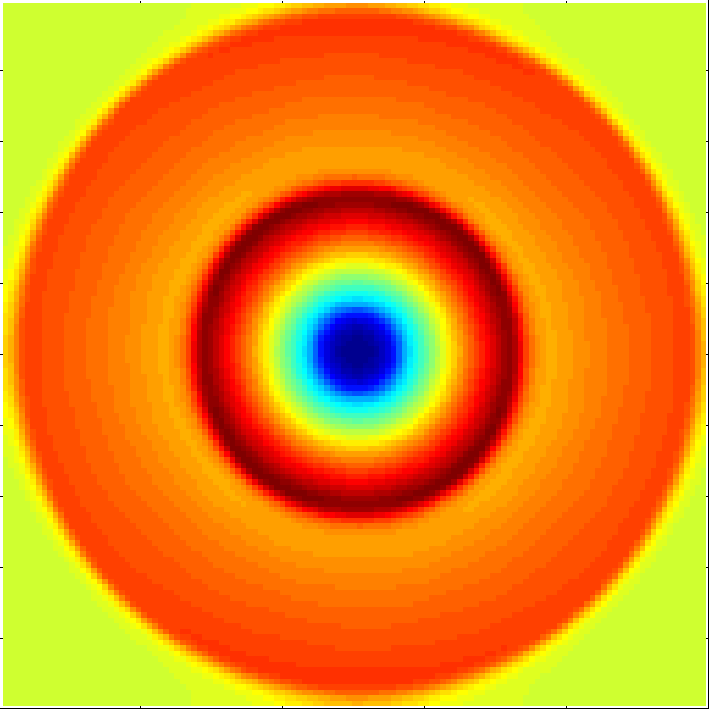}\hfill
\includegraphics[width=0.3\linewidth]{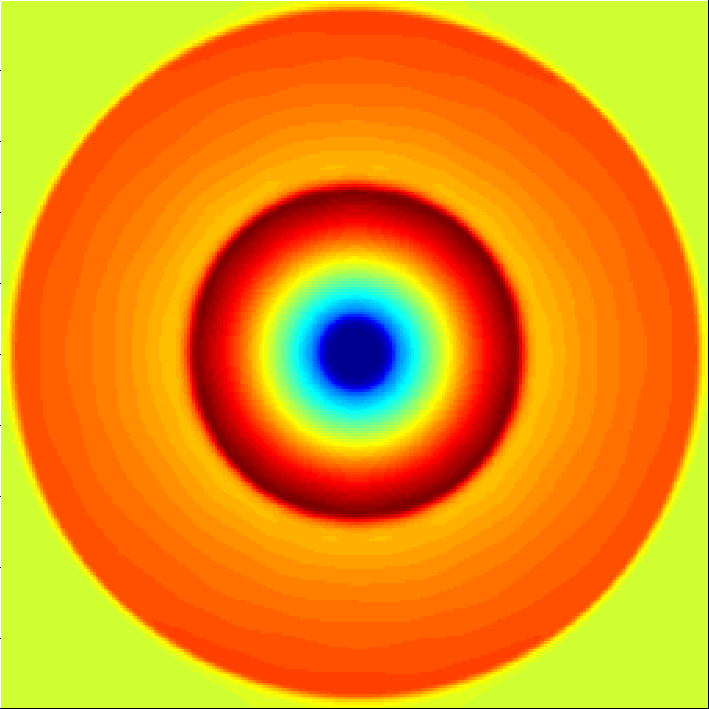}\hfill
\includegraphics[width=0.3\linewidth]{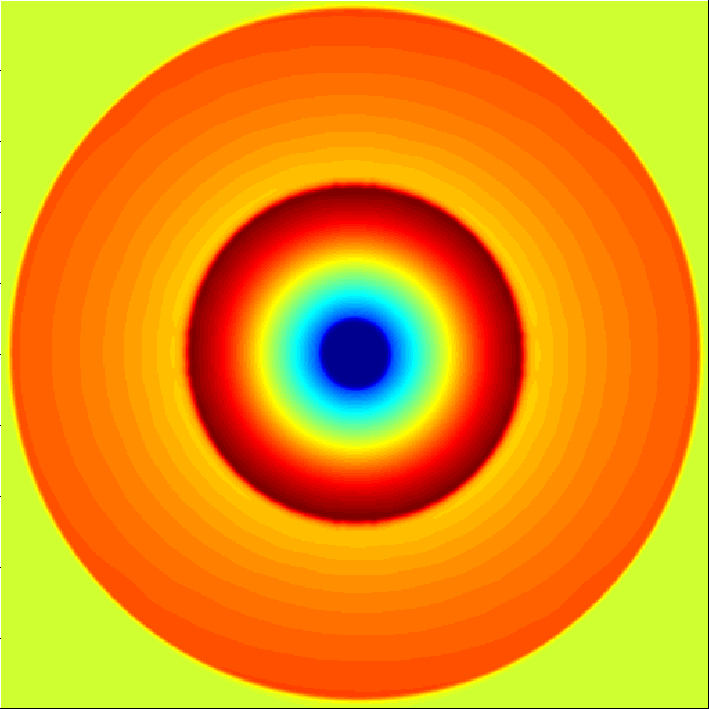}
\caption{Density for the Sod shock tube problem, computed with TECNO2 finite difference scheme of \cite{FTSID2}, with initial data \eqref{eq:sodpinit} at time $t=0.24$. Left to right: $\Dx = 1/128, 1/256, 1/512$.}
\label{fig:sod1}
\end{figure}
First we set $\amp = 0.01$ and compute the approximate solutions of the two-dimensional Euler equations \eqref{eq:2deuler} with the second-order TeCNO2 finite difference scheme of \cite{FTSID2}. In Figure \ref{fig:sod1}, we present the computed densities at time $t=0.24$ for three different mesh resolutions. The figure clearly indicates convergence as the mesh is refined. To further quantify this convergence, we compute the difference in the approximate solutions on two successive mesh resolutions:
\begin{equation}
\label{eq:chyrates}
E^{\Dx} = \|u^{\Dx} - u^{\Dx/2}\|_{L^1([-0.5,0.5]^2)},
\end{equation}
and plot the results for density in Figure \ref{fig:sod23}(a). The results clearly indicate that the numerical approximations form a Cauchy sequence in $L^1$, and hence converge. The same numerical experiment was performed with a different scheme: a second-order high-resolution scheme based on an HLLC solver using the MC limiter, implemented in the FISH code \cite{Kap1}. Similar convergence results were obtained (omitted here for brevity).
\begin{figure}[t]
\centering
\subfigure[$L^1$ Cauchy rates \eqref{eq:chyrates} ($y$-axis) in the density at time $t=0.24$ vs.\ number of gridpoints ($x$-axis)]{\includegraphics[width=6cm]{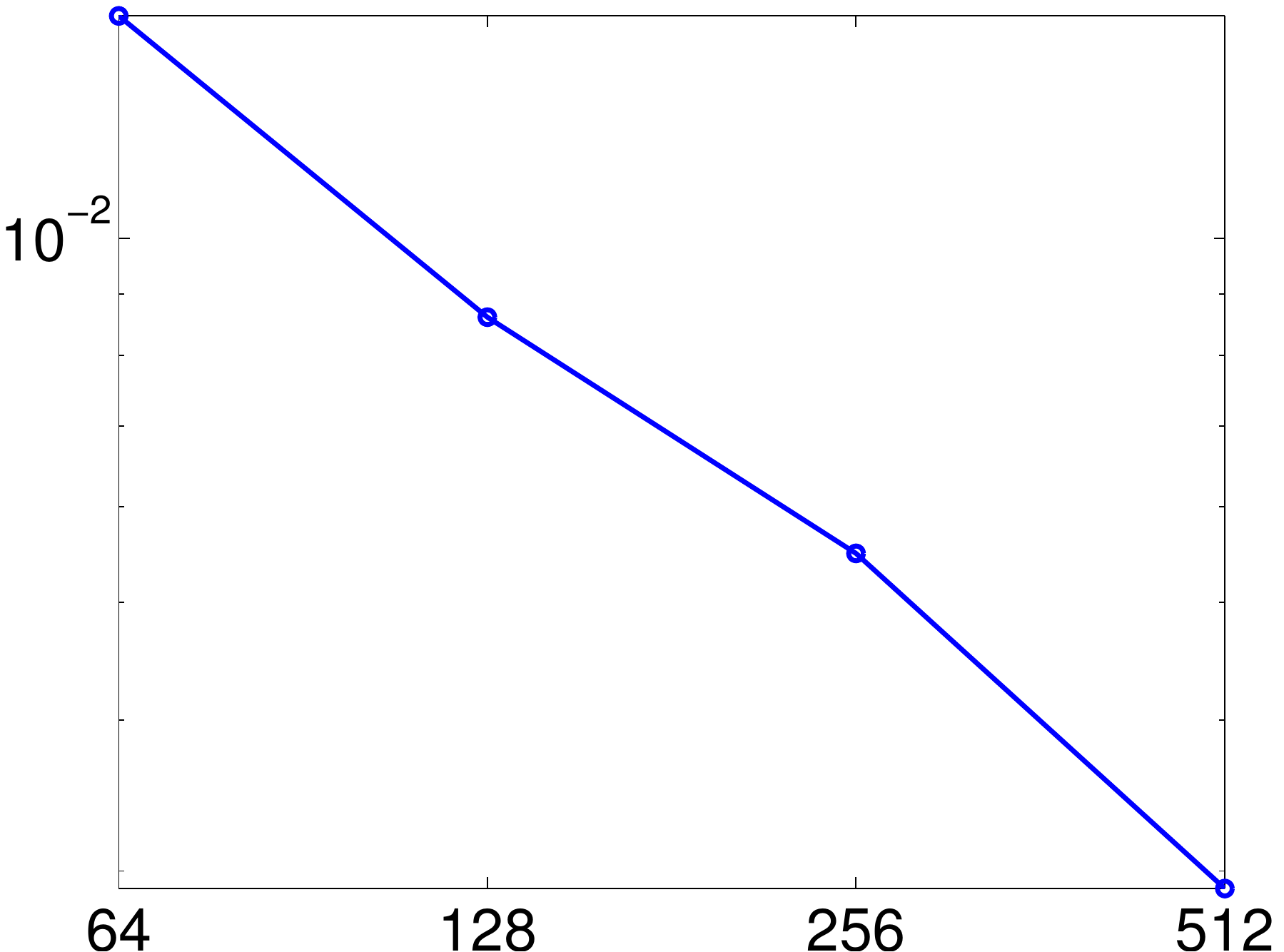}}
\hspace{1cm}
\subfigure[$L^1$ error with respect to the unperturbed solution \eqref{eq:sodinit} ($y$-axis) vs.\ the perturbation parameter $\amp$ ($x$-axis)]{\includegraphics[width=6cm]{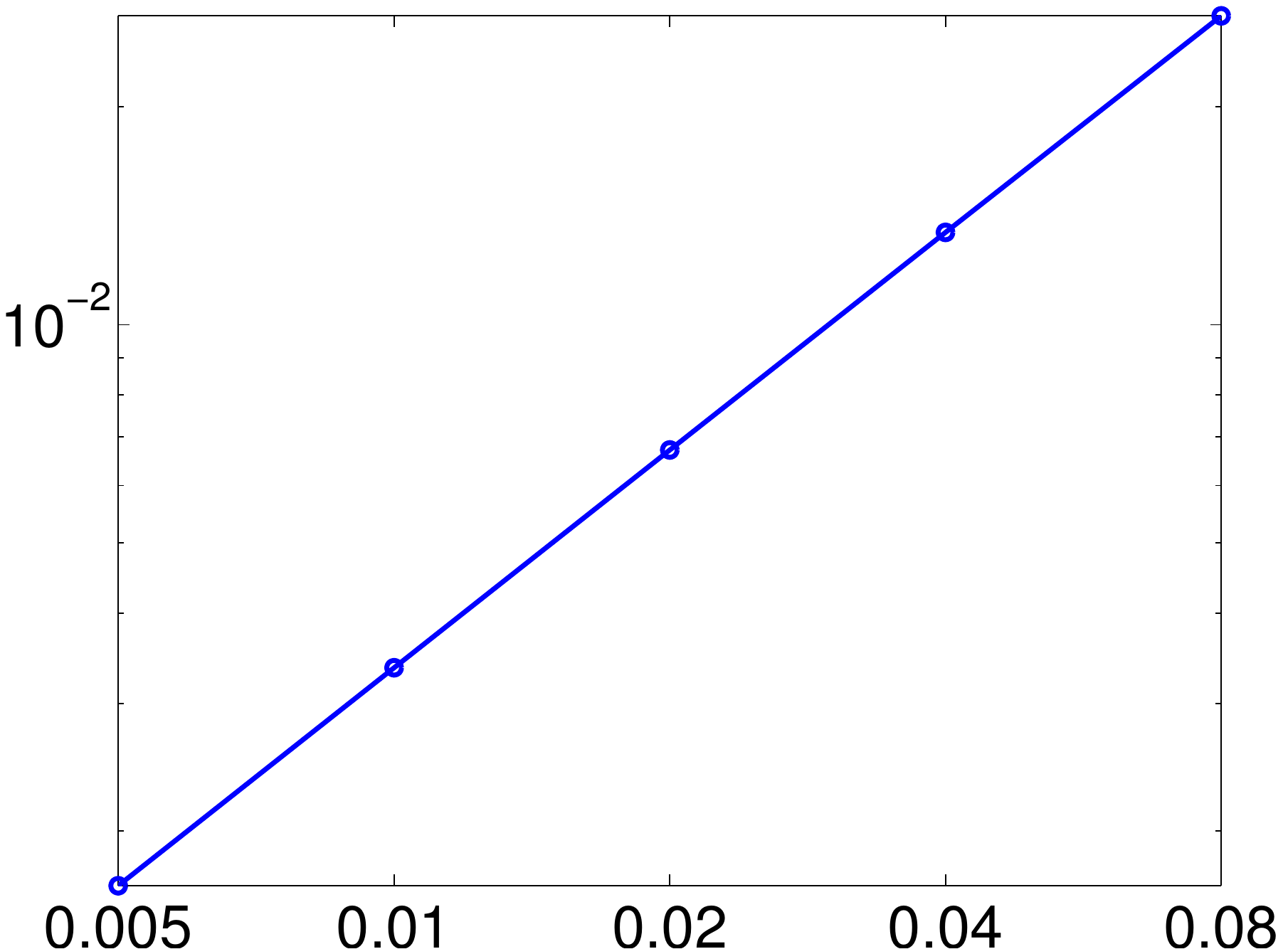}}
\caption{$L^1$ differences in density $\rho$ at time $t=0.24$ for the Sod shock tube problem with initial data \eqref{eq:sodpinit}.}
\label{fig:sod23}
\end{figure}

Next, we investigate numerically the issue of stability of this system with respect to perturbations in the initial data. To this end, we use exactly the same set up as the previous numerical experiment but let the perturbation amplitude $\amp \rightarrow 0$ in \eqref{eq:sodpinit} and plot in Figure \ref{fig:sod23}(b) the error in computed density (at a fixed mesh resolution of $1024^2$ points) for successively lower values of $\amp$. The reference solution is computed with the finest mesh resolution of $1024^2$ using the unperturbed initial data \eqref{eq:sodinit}. The results clearly show convergence to the unperturbed solution in the $\amp\to0$ limit.

The above numerical example suggests convergence of the approximate numerical solutions to an entropy solution, at least for some benchmark test cases. The computed solutions were observed to be stable with respect to perturbations of the initial data. In the literature it is not uncommon to extrapolate from benchmark test cases like the Sod shock tube and expect that the numerical approximations converge as the mesh is refined for all possible sets of flow configurations.

\subsubsection{Kelvin-Helmholtz problem}
We question the universality of the above observed empirical convergence and stability results by considering the following set of initial data for the two-dimensional Euler equations (see Section \ref{sec:6}):
\begin{equation}
\label{eq:khi}
u_0(x) = \begin{cases}
u_L & \text{if } 0.25 < x_2 < 0.75 \\
u_R & \text{if } x_2 \leq 0.25 \text{ or } x_2 \geq 0.75,
\end{cases}
\end{equation}
with $\rho_L = 2$, $\rho_R = 1$, $\velx_L = -0.5$, $\velx_R = 0.5$, $\vely_L=\vely_R=0$ and $p_L=p_R=2.5$. It is readily seen that this is a steady state, i.e., that $u(x,t) \equiv u_0(x)$ is an entropy solution.

Next, we  add the same perturbation \eqref{eq:sodpinit} to the initial data \eqref{eq:khi} and compute approximate solutions in the computational domain $[0,1]^2$ with periodic boundary conditions, for different $\Dx>0$. A series of approximate solutions using the TeCNO2 scheme of \cite{FTSID2} and perturbation amplitude $\amp=0.01$ are shown in Figure \ref{fig:kh-density1}. The results show that there is no sign of any convergence as the mesh is refined. As a matter of fact, structures at smaller and smaller scales are formed with mesh refinement. This lack of convergence is quantified by plotting the differences between successive mesh levels \eqref{eq:chyrates} for the density in Figure \ref{fig:kh-errors}(a). The results show that as the mesh is refined, the approximate solutions do \emph{not} seem to form a Cauchy sequence in $L^1$ (at least for the mesh resolutions that have been tested), and hence may not converge. The results presented in Figures \ref{fig:kh-density1} and \eqref{fig:kh-errors}(a) are computed with the TeCNO scheme of \cite{FTSID2}. Very similar results were also obtained with the FISH code \cite{Kap1} and the ALSVID finite volume code \cite{fmwrsid2}. Furthermore, convergence in even weaker $W^{-1,p}$, $1<p\leq \infty$, norms was also not observed. Thus, one cannot deduce convergence of even bulk properties of the flow, such as the average domain temperature, in this particular case.

Finally, we check stability of the numerical solutions as the perturbation parameter $\amp \rightarrow 0$. We compute numerical approximations at a fixed fine grid resolution of $1024^2$ points with successively lower values of $\amp$. These results are compared with the steady state solution \eqref{eq:khi} and are presented in Figure \ref{fig:kh-errors}(b). The $L^1$ difference results clearly show that there is no convergence to the steady state solution \eqref{eq:khi} as $\amp \rightarrow 0$.

\begin{figure}[h]
\centering
\includegraphics[width=0.3\linewidth]{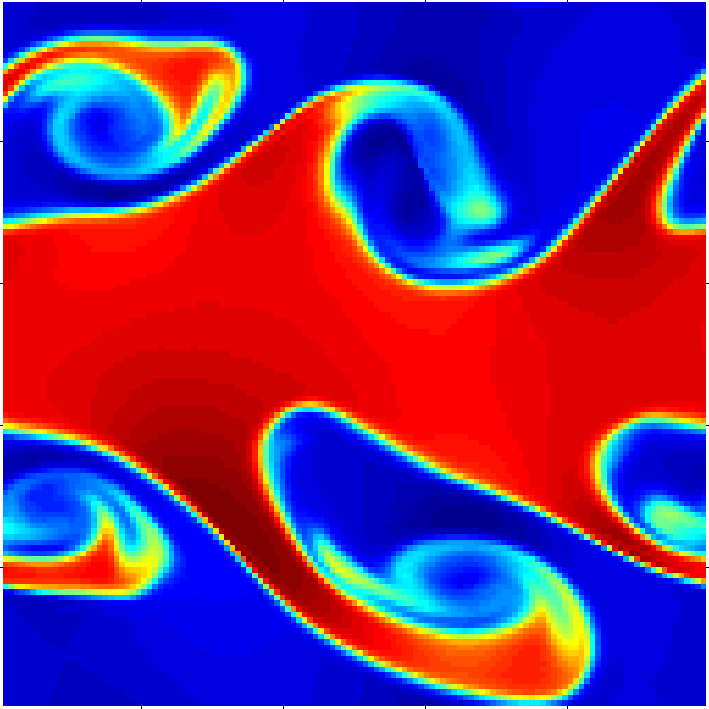}\hfill
\includegraphics[width=0.3\linewidth]{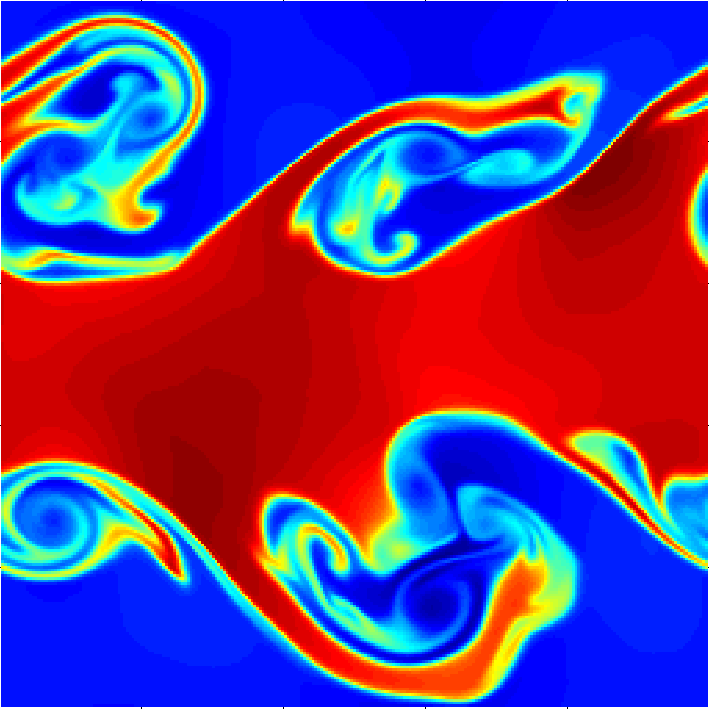}\hfill
\includegraphics[width=0.3\linewidth]{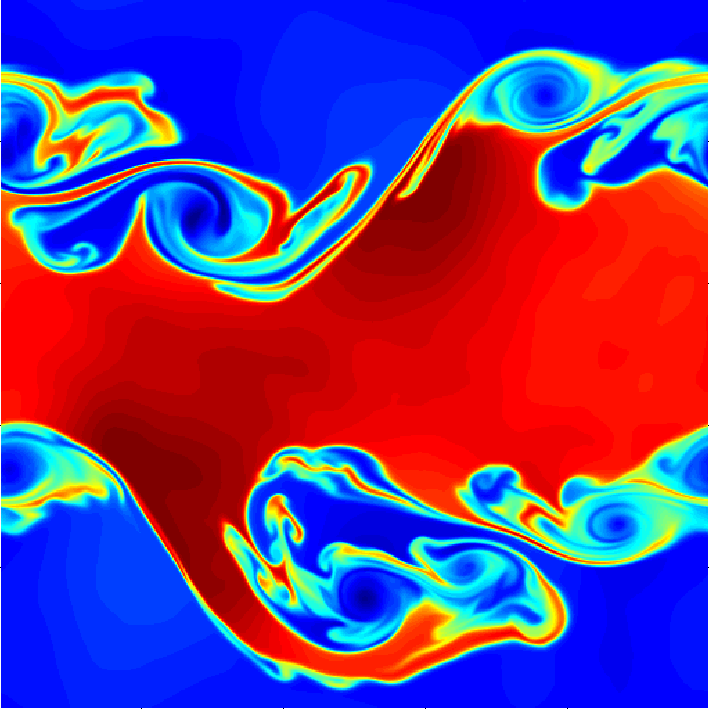}
\caption{Density for the Kelvin-Helmholtz problem \eqref{eq:khi} with perturbation \eqref{eq:sodpinit} and perturbation parameter $\amp=0.01$. Left to right: $\Dx = 1/128, 1/256, 1/512$, at time $t=1$}
\label{fig:kh-density1}
\end{figure}

\begin{figure}
\centering
\subfigure[$L^1$ Cauchy rates \eqref{eq:chyrates} ($y$-axis) vs.\ number of gridpoints ($x$-axis) for the perturbed  problem \eqref{eq:sodpinit}, \eqref{eq:khi} with $\amp = 0.01$.]{\includegraphics[width=6cm]{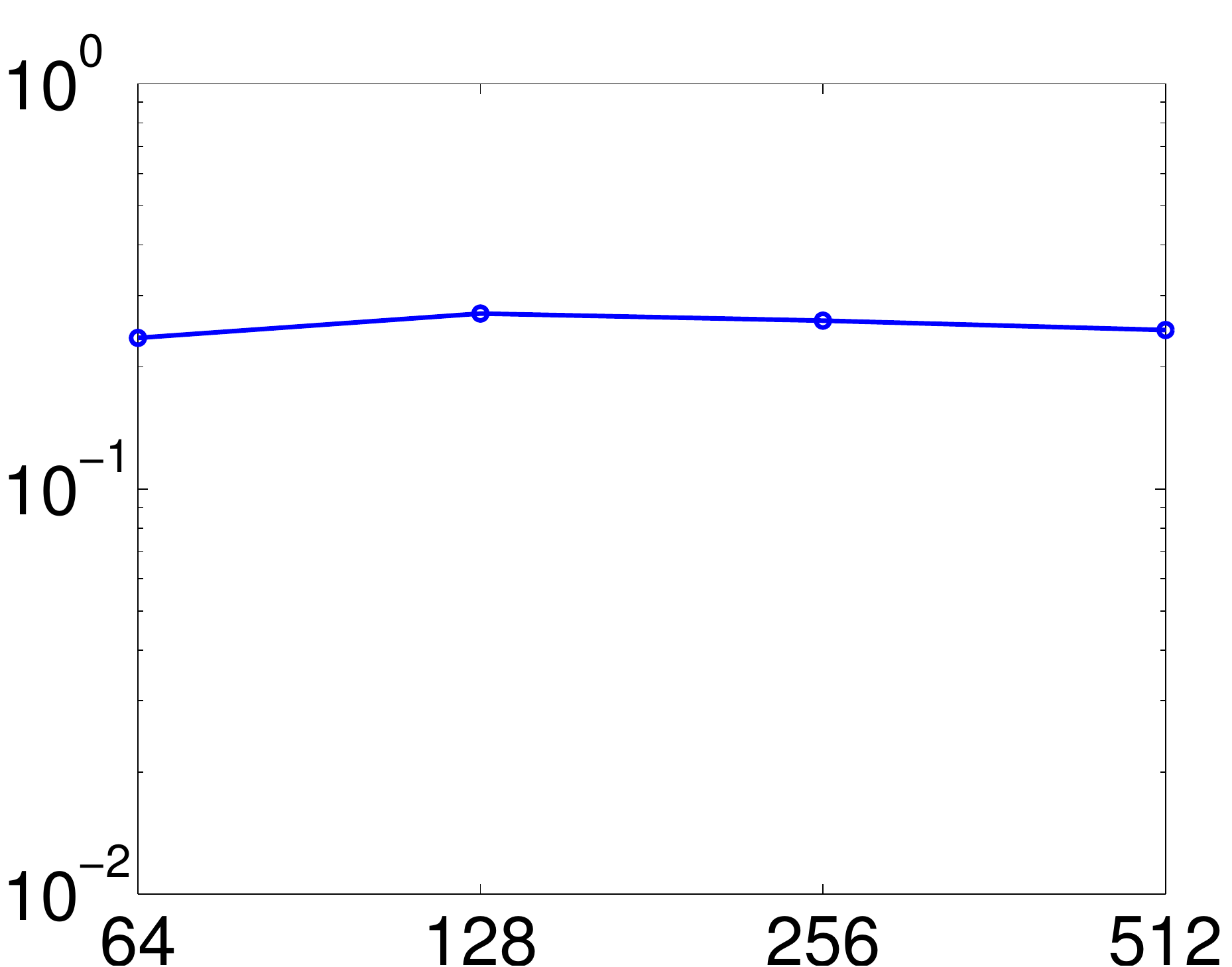}}
\hspace{1cm}
\subfigure[$L^1$ error with respect to the steady state solution \eqref{eq:khi} of the unperturbed Kelvin-Helmholtz problem ($y$-axis) vs.\ perturbation parameter $\amp$, at a fixed mesh with $1024^2$ points.]{\includegraphics[width=6cm]{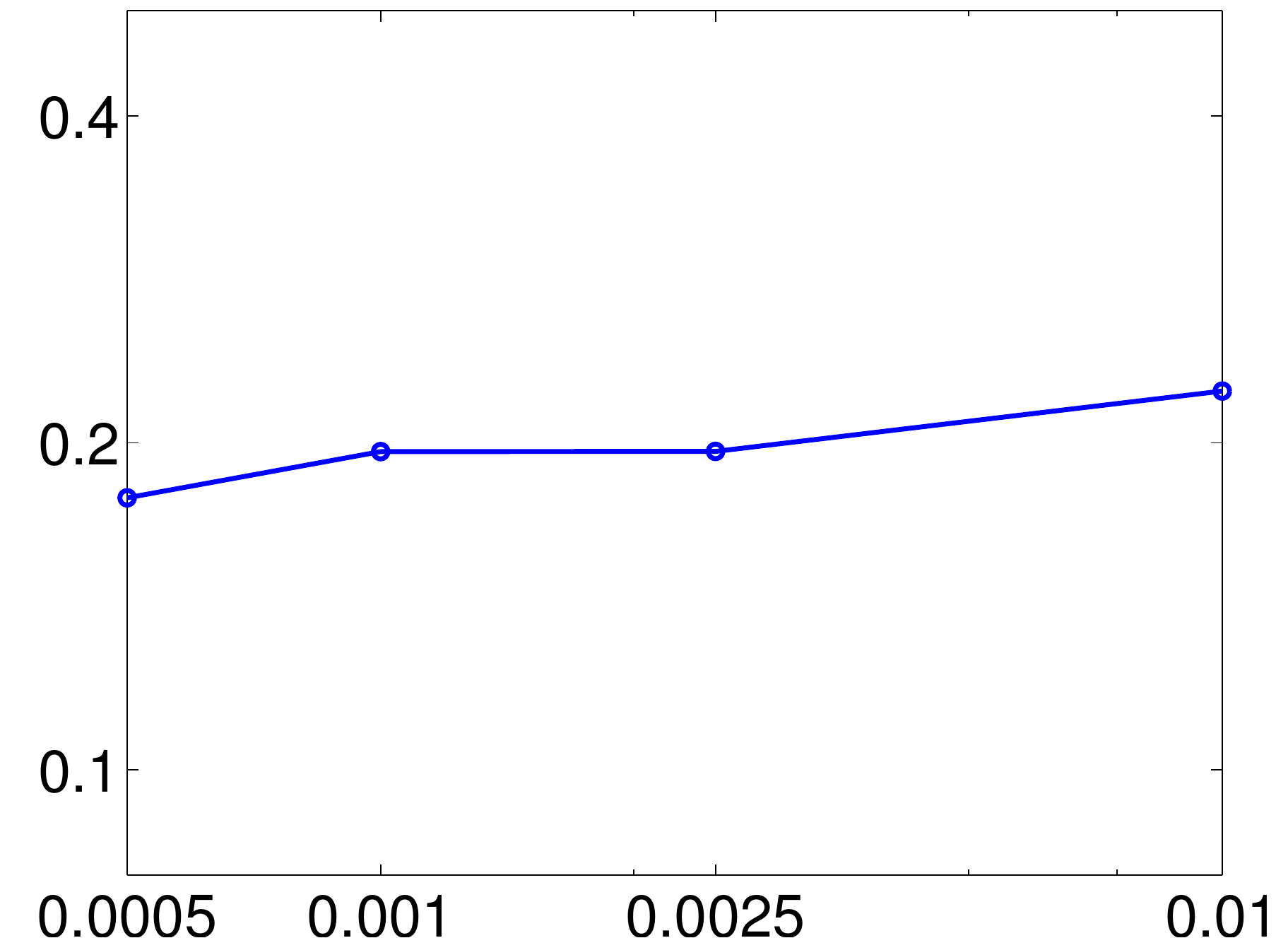}}
\caption{$L^1$ differences in density $\rho$ at time $t=2$ for the Kelvin-Helmholtz problem \eqref{eq:khi}.}
\label{fig:kh-errors}
\end{figure}

\subsection{A different notion of solutions}
The above experiment clearly demonstrates that in general, a whole host\footnote{We have tested at least three types of schemes, TeCNO scheme of \cite{FTSID2}, the high-resolution HLLC scheme of \cite{Kap1} and the finite volume scheme of \cite{fmwrsid2}, and obtained similar non-convergence and instability results as presented above. We strongly suspect that any numerical method will not converge or be stable with respect to perturbations in the initial data for this particular example.} of state of the art numerical schemes do not seem to converge (even for very fine mesh resolutions) to an entropy solution for multi-dimensional systems of conservation laws. In fact, structures at smaller and smaller scales are formed as the mesh is refined. This fact does not imply that the numerical approximations are at fault (given that all the tested schemes, based on different design philosophies, behaved in the same manner), but rather the notion of entropy solutions does not adequately describe the complex flow phenomena that are modeled by systems of conservation laws such as the compressible Euler equations. 

When combined with the recent results on the non-uniqueness of entropy solutions of systems of conservation laws \cite{CDL1,CDL2} and references therein, our numerical evidence strongly suggests that entropy solutions may not be an appropriate solution framework for systems of conservation laws, in many respects. In particular, entropy solutions may not suffice to characterize the limits of numerical approximations to conservation laws in a stable manner. 

Based on the fact that oscillations persist on finer and finer scales (see Figure \ref{fig:kh-density1}) for numerical approximations of \eqref{eq:cauchy}, we focus on the concept of \emph{entropy measure valued solutions}, introduced by DiPerna in \cite{DIP1}, see also \cite{DM1}. In this framework, solutions of the system of conservation laws \eqref{eq:cauchy} are no longer integrable functions, but parameterized probability measures, or \emph{Young measures}, which are able to represent the limit behavior of sequences of oscillatory functions. This solution concept was further based on the work of Tartar \cite{T1} on characterizing the weak limits of bounded sequences of functions. More recently, Glimm and co-workers (\cite{CG1,GL1} and references therein) have also hypothesized  that entropy measure valued solutions are the appropriate notion of solutions for hyperbolic conservation laws, particularly in several space dimensions.

\subsection{Aims and scope of the current paper}
In the current paper:
\begin{itemize}
\item  We consider entropy measure valued solutions for the Cauchy problem \eqref{eq:cauchy}, in the sense of DiPerna \cite{DIP1}. We study the existence and stability of the entropy measure valued solutions.

\item The main aim of the current paper is to approximate entropy measure valued solutions numerically. To this end, we propose an algorithm based on the realization of Young measures as the law of random fields and approximate the solution random fields with suitable finite difference (volume) numerical schemes. We propose a set of sufficient conditions that a scheme has to satisfy in order to converge to an entropy measure valued solution as the mesh is refined. Examples of such convergent schemes are also provided. This provides a viable and rigorous numerical framework for multi-dimensional systems of conservation laws, within the framework of entropy measure valued solutions. 

\item We present a large number of numerical experiments to validate the proposed theory. The numerical approximations are also employed to study the stability as well as other interesting properties of entropy measure valued solutions.
\end{itemize}

The rest of this paper is organized as follows: in Section \ref{sec:2}, we provide a short but self-contained description of Young measures (see also Appendix \ref{app:young}) and then define entropy measure valued solutions for a generalized Cauchy problem, corresponding to the system of conservation law \eqref{eq:cauchy}. The well-posedness of the entropy measure valued solutions is discussed in Section \ref{sec:3}. In Section \ref{sec:4}, we discuss finite difference schemes approximating \eqref{eq:cauchy} and propose abstract criteria that these schemes have to satisfy in order to converge to entropy measure valued solutions. Two schemes satisfying the abstract convergence framework are presented in Section \ref{sec:5}. In Section \ref{sec:6}, we present numerical experiments that illustrate the convergence properties of the schemes and discuss the stability and related properties of entropy measure valued solutions.

\section{Young measures and entropy measure valued solutions}\label{sec:2}
A \emph{Young measure} on a set $D\subset\R^k$ (in our setting, $D=\R^d\times\R_+$ will represent space-time) is a function $\nu$ which  assigns to every point $y\in D$  a probability measure $\nu_y \in \Prob(\R^N)$ on the phase space $\R^N$. The set of all Young measures from $D$ to $\R^N$ is denoted by $\Young(D,\R^N)$. We can compose a Young measure with a continuous function $g$ by defining $\ip{\nu_y}{g} := \int_{\R^N}g(\xi)d\nu_y(\xi)$, the expectation of $g$ with respect to the probability measure $\nu_y$. Note that this defines a real-valued function of $y\in D$.

Every measurable function $u : D \to \R^N$ gives rise to a Young measure by letting
\[
\nu_y := \delta_{u(y)},
\]
where $\delta_\xi$ is the Dirac measure centered at $\xi\in\R^N$. Such Young measures are called \emph{atomic}.

If $\nu^1,\nu^2,\dots$ is a sequence of Young measures then there are two notions of convergence. We say that $\nu^n$ converge \emph{weak*} to a Young measure $\nu$ (written $\nu^n \wto \nu$) if $\ip{\nu^n}{g} \wsto \ip{\nu}{g}$ in $L^\infty(D)$ for all $g\in C_0(\R^N)$, that is, if
\begin{equation}
\label{eq:weak-conv}
\int_D \phi(z) \ip{\nu^n_z}{g}\ dz \to \int_D \phi(z) \ip{\nu_z}{g}\ dz \qquad \forall\ \phi\in L^1(D).
\end{equation}
By the \emph{fundamental theorem of Young measures} (see Theorem \ref{thm:young}), \emph{any} suitably bounded sequence of Young measures has a weak* convergent subsequence.

We say that the sequence $\{\nu^n\}$ converges \emph{strongly} to $\nu$ (written $\nu^n \to \nu$) if
\begin{equation}\label{eq:sconv}
\bigl\|W_p(\nu^n,\nu)\bigr\|_{L^p(D)} \to 0
\end{equation}
for some $p\in[1,\infty)$, where $W_p$ is the $p$-\emph{Wasserstein distance}
\[
W_p(\mu,\rho) := \inf\left\{\int_{\R^N\times\R^N}|\xi-\zeta|^p\ d\pi(\xi,\zeta)\ :\ \pi\in\Pi(\mu,\rho)\right\}^{1/p}
\]
which metrises the topology of weak convergence on the set $\Prob^p(\R^N):=\left\{\mu\in \Prob(\R^N) \ :\ \langle \mu, |\xi|^p\rangle <\infty\right\}$. Here, $\Pi(\mu,\rho)$ is the set of probability measures on $\R^N\times\R^N$ with marginals $\mu,\rho \in \Prob^p(\R^N)$ (see also Appendix \ref{app:wasserstein}).

We refer to Appendix \ref{app:young} for a more rigorous and detailed introduction to Young measures.

\subsection{The measure valued (MV) Cauchy problem}
As mentioned in the introduction, we will seek a more general (weaker) notion of solutions to the Cauchy problem for a system of conservation laws \eqref{eq:cauchy} by requiring that the solutions be Young measures, instead of integrable functions. Equipped with the notation of the previous section, we propose the following generalized Cauchy problem (corresponding to the system \eqref{eq:cauchy}): find a $\nu\in\Young(\R^d\times\R_+,\R^N)$ such that
\begin{equation}\label{eq:mvcauchy}
\begin{split}
\partial_t \ip{\nu}{\id} + \nablax\cdot \ip{\nu}{f} &= 0 \\
\nu_{(x,0)} &= \imv_x,
\end{split}
\end{equation}
where $\imv \in \Young(\R^d,\R^N)$ is the \emph{initial measure-valued data} and $\id(\xi)=\xi$ is the identity function on $\R^N$. The above MV Cauchy problem is interpreted as follows.

\begin{definition}[DiPerna \cite{DIP1}]
A Young measure $\nu\in\Young(\R^d\times\R_+,\R^N)$ is a \emph{measure-valued (MV) solution} of \eqref{eq:mvcauchy} if \eqref{eq:mvcauchy} holds in the sense of distributions, i.e.,
\begin{equation}\label{eq:mvsoln}
\int_{\R_+}\int_{\R^d} \partial_t \phi(x,t) \ip{\nu_{(x,t)}}{\id} + \nablax\phi(x,t) \cdot \ip{\nu_{(x,t)}}{f}\ dxdt + \int_{\R^d}\phi(x,0) \ip{\imv_x}{\id} \ dx = 0
\end{equation}
for all test functions $\phi\in C^1_c(\R^d\times\R_+)$.
\end{definition}

\begin{definition}[DiPerna \cite{DIP1}]
A Young measure $\nu\in\Young(\R^d\times\R_+,\R^N)$ is an \emph{entropy measure-valued (EMV) solution} of \eqref{eq:mvcauchy} if in addition to being a measure valued solution (satisfying \eqref{eq:mvsoln}), it also satisfies
\begin{equation}\label{eq:mventrineq}
\partial_t \ip{\nu}{\eta} + \nablax\cdot \ip{\nu}{q} \leq 0 \qquad \text{in } \D'(\R^d\times\R_+)
\end{equation}
for every entropy pair $(\eta,q)$, that is, if
\begin{equation}\label{eq:emvsoln}
\int_{\R_+}\int_{\R^d} \partial_t \phi(x,t) \ip{\nu_{(x,t)}}{\eta} + \nablax\phi(x,t) \cdot \ip{\nu_{(x,t)}}{q}\ dxdt + \int_{\R^d}\phi(x,0) \ip{\imv_x}{\eta} \ dx \geq 0
\end{equation}
for all nonnegative test functions $0\leq\phi\in C^1_c(\R^d\times\R_+)$.
\end{definition}

\begin{remark}\label{rem:initialdata}
The formulations \eqref{eq:mvsoln} and \eqref{eq:emvsoln} impose the initial data $\sigma$ in a \emph{very weak} manner. The weak formulation \eqref{eq:mvsoln} requires, roughly speaking, that $\lim_{t\to0} \ip{\nu_{(x,t)}}{\id} = \ip{\sigma_x}{\id}$, i.e.\ that the \emph{barycenters} (or mean) of $\nu_{(x,0)}$ and $\sigma_x$ should coincide. The inequality \eqref{eq:emvsoln} implies that $\limsup_{t\to0} \ip{\nu_{(x,t)}}{\eta} \leq \ip{\sigma_x}{\eta}$ (in Theorem \ref{thm:scalarstable2} we require a slightly stronger form of this inequality). The requirement that barycenters of two measures should coincide will imply that the measures themselves coincide \emph{only} if one of the two is a Dirac mass. Correspondingly, our condition for initial data implies equality at $t=0$ \emph{only} when the initial data is atomic. This is precisely the setting which we choose to focus on in the present paper. In a forthcoming paper \cite{FLM15} we consider the question of uniqueness when the initial data is non-atomic, and how to interpret the initial condition in this more complex setting.
\end{remark}

We denote by $\Emv(\imv)$ the set of all entropy MV solutions of the MV Cauchy problem \eqref{eq:mvcauchy} with initial MV data $\imv$. It is readily seen that every entropy solution $u$ of \eqref{eq:cauchy} gives rise to an EMV solution of \eqref{eq:mvcauchy} with $\imv = \delta_{u_0}$, by defining $\nu_{(x,t)} := \delta_{u(x,t)}$, the atomic Young measure concentrated at $u$. Thus, the set $\Emv(\imv)$ is at least as large as the set of entropy solutions of \eqref{eq:cauchy} whenever $\imv$ is atomic, $\imv = \delta_{u_0}$.

\begin{remark}
Although our focus in the current paper will be on the specific case of atomic initial data, we still consider the more general setting of the MV Cauchy problem \eqref{eq:mvcauchy} as it enables us to formulate numerical approximations in a unified manner.

In practice, the initial data $u_0$ in \eqref{eq:cl} is obtained from a measurement or observation process. Since measurements (observations) are intrinsically uncertain, it is customary to model this initial uncertainty statistically by considering the initial data $u_0$ as a random field. Given the fact that the law of a random field is a Young measure, we can also model this initial uncertainty with non-atomic initial measures in the measure valued (MV) Cauchy problem \eqref{eq:mvcauchy}. Thus, our formulation suffices to include various formalisms for uncertainty quantification of conservation laws, i.e., the determination of solution uncertainty given uncertain initial data. See \cite{schwabsid1,sssid1,sssid3} and references therein for an extensive discussion on uncertainty quantification for conservation laws.
\end{remark}

\section{Well-posedness of EMV solutions}
\label{sec:3}
The questions of existence, uniqueness and stability of EMV solutions of \eqref{eq:mvcauchy} are of fundamental significance. We start with a discussion of the scalar case.

\subsection{Scalar conservation laws}
The question of existence of EMV solutions for scalar conservation laws was considered by DiPerna in \cite{DIP1}. We slightly generalize his result for a non-atomic initial data as follows.

\begin{theorem}\label{thm:scalarstable}
Consider the MV Cauchy problem \eqref{eq:mvcauchy} for a scalar conservation law. If the initial data $\sigma$ is uniformly bounded (see Appendix \ref{app:youngunifbounded}), then there exists an EMV solution of \eqref{eq:mvcauchy}.
\end{theorem}
\begin{proof}
By Proposition \ref{prop:lawexists}, there exists a probability space $(\Omega,\Sigmaalg,P)$ and a random field $u_0:\Omega\times\R^d\to\R$ with law $\sigma$. By the uniform boundedness of $\sigma$, we have $\|u_0\|_{L^\infty(\Omega\times\R^d)} < \infty$.

For each $\omega\in\Omega$, let $u(\omega;x,t)$ be the entropy solution of \eqref{eq:cauchy} with initial data $u_0(\omega)$, and define $\nu$ as the law of $u$. Then for every entropy pair $(\eta,q)$ and every test function $0\leq\phi\in C^1_c(\R^d\times\R_+)$, we have
\begin{align*}
\int_{\R_+}\int_{\R^d}& \partial_t \phi(x,t) \ip{\nu_{(x,t)}}{\eta} + \nablax\phi(x,t) \cdot \ip{\nu_{(x,t)}}{q}\ dxdt \\
&= \int_{\R_+}\int_{\R^d}\partial_t \phi(x,t) \int_{\Omega} \eta(u(\omega;x,t))\ dP(\omega) + \nablax\phi(x,t) \cdot \int_{\Omega}q(u(\omega;x,t))\ dP(\omega)dxdt \\
&= \int_{\Omega}\int_{\R_+}\int_{\R^d}\partial_t \phi(x,t)\eta(u(\omega;x,t)) + \nablax\phi(x,t) \cdot q(u(\omega;x,t))\ dxdtdP(\omega) \\
&\geq -\int_{\Omega}\int_{\R^d}\phi(x,0)\eta(u_0(\omega;x))\ dxdP(\omega) \\
&= -\int_{\R^d}\phi(x,0) \ip{\imv_x}{\eta} \ dx,
\end{align*}
by Fubini's theorem and the entropy stability of $u(\omega)$ for each $\omega$. This proves the entropy inequality \eqref{eq:emvsoln}.
\end{proof}

Although EMV solutions exist for scalar conservation laws with non-atomic measure valued initial data, they may not be unique. Here is a simple counter-example (see also Schochet \cite{Sch89}).
\begin{example}\label{ex:nonunique}
Consider Burgers' equation
\[
\partial_t u + \partial_x \left(\frac{u^2}{2}\right) = 0.
\]
Denote by $\lambda$ the Lebesgue measure on $\R$, and by $\lambda_A$ is the restriction of $\lambda$ to a subset $A\subset\R$, i.e.\ $\lambda_A(B) = \lambda(A\cap B)$. We define $\Omega=[0,1]$, $\Sigmaalg=\Borel([0,1])$ (the Borel $\sigma$-algebra on $[0,1]$) and $P=\lambda_{[0,1]}$. Let $u_0$ and $\tilde{u}_0$ be the random fields
\[
u_0(\omega;x) := \begin{cases}
1 + \omega & \text{for } x < 0 \\
\omega & \text{for } x > 0,
\end{cases}
\qquad
\tilde{u}_0(\omega;x) := \begin{cases}
1 + \omega & \text{for } x < 0 \\
1-\omega & \text{for } x > 0,
\end{cases}
\qquad
\omega\in[0,1],\ x\in\R.
\]
It is readily checked that the law of both $u_0$ and $\tilde{u}_0$ in $(\Omega,\Sigmaalg,P)$ equals
\[
\imv_x = \begin{cases}
\lambda_{[1,2]} & \text{for } x < 0 \\
\lambda_{[0,1]} & \text{for } x > 0.
\end{cases}
\]

Note that both of the random fields $u_0$ and $\tilde{u}_0$ model the Burgers' equation with \emph{uncertainty in initial shock location} which is widely considered in the UQ literature, see \cite{schwabsid1} and references therein. Although their laws are the same --- i.e., that the initial Young measure is the same in both cases --- the resulting two-point correlations (particularly for points left and right of the origin) differ.   

The entropy solutions $u(\omega)$ and $\tilde{u}(\omega)$ of the Riemann problems with initial data $u_0(\omega)$ and $\tilde{u}_0(\omega)$ are given by
\[
u(\omega;x,t) = \begin{cases}
1+\omega & \text{if } x/t < \hf+\omega \\
\omega & \text{if } x/t > \hf+\omega,
\end{cases}
\qquad
\tilde{u}(\omega;x,t) = \begin{cases}
1+\omega & \text{if } x/t < 1 \\
1-\omega & \text{if } x/t > 1.
\end{cases}
\]
To compute the law $\nu$ of $u$ we rewrite $u$ as
\[
u(\omega;x,t) = \begin{cases}
1+\omega & \text{if } x/t-\hf < \omega \\
\omega & \text{if } x/t-\hf > \omega.
\end{cases}
\]
Hence, if $x/t-\hf < 0$ then $\nu_{(x,t)} = \lambda_{[1,2]}$, whereas if $x/t-\hf > 1$ then $\nu_{(x,t)} = \lambda_{[0,1]}$. When $0\leq x/t-\hf \leq 1$ we have for every $g\in C_0(\R^N)$
\begin{align*}
\ip{\nu_{(x,t)}}{g} &= \int_0^1 g(u(\omega;x,t))\ d\omega = \int_{x/t-\hf}^1 g(1+\omega)\ d\omega + \int_0^{x/t-\hf} g(\omega)\ d\omega \\
&= \int_{x/t+\hf}^2 g(\omega)\ d\omega + \int_0^{x/t-\hf} g(\omega)\ d\omega \\
&= \int_\R g(\omega)\ d\lambda_{[\nicefrac{x}{t}+\nicefrac{1}{2},2]}(\omega) + \int_\R g(\omega)\ d\lambda_{[0,\nicefrac{x}{t}-\hf]}(\omega).
\end{align*}
After a similar calculation for $\tilde{\nu}$ we find that
\[
\nu_{(x,t)} = \begin{cases}
\lambda_{[1,2]} & \text{if } x/t < \hf \\
\lambda_{[\nicefrac{x}{t}+\nicefrac{1}{2},2]} + \lambda_{[0,\nicefrac{x}{t}-\hf]} & \text{if } \hf < x/t < \unitfrac{3}{2} \\
\lambda_{[0,1]} & \text{if } \unitfrac{3}{2} < x/t,
\end{cases}
\qquad
\tilde{\nu}_{(x,t)} = \begin{cases}
\lambda_{[1,2]} & \text{if } x/t < 1 \\
\lambda_{[0,1]} & \text{if } x/t > 1.
\end{cases}
\]
Note that in fact $\nu_{(x,t)}$ and $\tilde{\nu}_{(x,t)}$ converges to $\sigma_x$ \emph{strongly} as $t\to0$ for all $x\neq0$. Thus, $\nu$ and $\tilde{\nu}$ are EMV solutions with the same initial MV data $\imv$, but do not coincide. \qedhere
\end{example}

The above example clearly illustrates that the MV Cauchy problem \eqref{eq:mvcauchy} may not have unique solutions, even for the scalar case, when the initial data is a non-atomic Young measure. Hence, it raises serious questions whether the notion of an entropy measure-valued solution is useful. However, the following  result shows that when restricting attention to the relevant class of \emph{atomic initial data}, then EMV solutions of the \emph{scalar} MV Cauchy problem \eqref{eq:mvcauchy} are stable.

\begin{theorem}\label{thm:scalarstable2}
Consider the scalar case $N=1$. Let $u_0\in L^1\cap L^\infty(\R^d)$ and let $\sigma\in\Young(\R^d)$ be uniformly bounded. Let $u\in L^1\cap L^\infty(\R^d\times\R_+)$ be the entropy solution of the scalar conservation law \eqref{eq:cauchy} with initial data $u_0$. Let $\nu$ be \textbf{any} EMV solution of \eqref{eq:mvcauchy} which satisfies
\begin{equation}\label{eq:initialdata}
\limsup_{T\to 0} \frac{1}{T}\int_0^T \int_{\R^d} \ip{\nu_{(x,t)}}{|u(x,t)-\xi|}\ dxdt \leq \int_{\R^d} \ip{\sigma_x}{|u_0(x)-\xi|}\ dx.
\end{equation}
Then for all $t>0$,
\[
\int_{\R^d} \ip{\nu_{(x,t)}}{|u(x,t) - \xi|}\ dx \leq \int_{\R^d} \ip{\sigma_x}{|u_0(x)-\xi|}\ dx,
\]
or equivalently,
\[
\Bigl\|W_1\bigl(\nu_{(\cdot,t)}, \delta_{u(\cdot,t)}\bigr)\Bigr\|_{L^1(\R^d)} \leq \Bigl\|W_1\bigl(\sigma, \delta_{u_0}\bigr)\Bigr\|_{L^1(\R^d)}.
\]
In particular, if $\sigma = \delta_{u_0}$ then $\nu = \delta_{u}$.
\end{theorem}
\begin{proof}
We follow DiPerna \cite{DIP1} who proved the uniqueness of scalar MV solutions subject to atomic initial data.
Here, we quantify stability in terms of the $W_1$-metric, which is related to the $L^1(x,\ev)$-stability of \emph{kinetic} solutions associated with \eqref{eq:cauchy}; see \cite{PT91}.

For $\xi\in\R$, let $(\eta(\xi,u), q(\xi,u))$ be the Kruzkov entropy pair, defined as
\[
\eta(\xi,u) := |\xi-u|, \qquad q(\xi,u) := \sgn(\xi-u)(f(\xi)-f(u)), \qquad u,\xi\in\R.
\]
By \cite[Theorem 4.1]{DIP1} we know that for any entropy solution $u$ of \eqref{eq:cauchy} and any entropy MV solution $\nu$ of \eqref{eq:mvcauchy}, we have
\[
\partial_t \ip{\nu_{(x,t)}}{\eta\bigl(\xi,u(x,t)\bigr)} + \nablax\cdot\ip{\nu_{(x,t)}}{q\bigl(\xi,u(x,t)\bigr)} \leq 0 \qquad \text{in } \D'(\R^d\times(0,\infty)),
\]
that is,
\[
\int_{\R_+}\int_{\R^d}\left( \partial_t \phi(x,t) \int_{\R^N}\eta\bigl(\xi,u(x,t)\bigr)\ d\nu_{(x,t)}(\xi) + \nablax\phi(x,t) \cdot \int_{\R^N}q\bigl(\xi,u(x,t)\bigr)\ d\nu_{(x,t)}(\xi)\right)dxdt \geq 0
\]
for all test functions $0\leq\phi\in C^1_c\bigl(\R^d\times(0,\infty)\bigr)$. Setting $\phi(x,t) = \theta(t)$ for a $\theta\in C_c^\infty\bigl((0,\infty)\bigr)$, we get
\[
\int_{\R_+} \theta'(t) V(t)\ dt \geq 0, \qquad V(t) := \int_{\R^d}\ip{\nu_{(x,t)}}{|\xi - u(x,t)|}\ dx.
\]
Letting $\theta$ be a smooth approximation of the indicator function on an interval $[0,t_0]$, we find in light of \eqref{eq:initialdata} that $V(t_0) \leq \int_{\R^d} \ip{\sigma_x}{|u_0(x)-\xi|}\ dx$ for almost every $t_0>0$.
\end{proof}

\subsection{Systems of conservation laws}
It is clear from the above discussion that non-atomic initial data might lead to multiple EMV solutions, see also the discussions in remark $2.3$. However, the scalar results also suggest some possible stability with respect to perturbations of \emph{atomic} initial data. Based on these considerations, we propose the following (weaker) notion of stability.

\begin{terminology}\label{def:mvstab}
The MV Cauchy problem \eqref{eq:mvcauchy} is \emph{MV stable} if the following property holds.
\begin{quote}
For every $u_0 \in L^\infty(\R^d,\R^N)$ and $\sigma \in \Young(\R^d,\R^N)$ such that
$$
{\mathcal D}\left(\delta_{u_0}, \sigma\right) \ll 1,
$$
there exists an EMV solution $\nu \in \Emv(\delta_{u_0})$ such that
$$
{\mathcal D}\left(\nu, \nu^{\sigma}\right) \ll 1
$$
for every EMV solution $\nu^{\sigma} \in \Emv(\imv)$ (or a subset thereof).
\end{quote}
\end{terminology}

(Recall that $\Emv(\imv)$ denotes the set of all entropy MV solutions to the MV Cauchy problem \eqref{eq:mvcauchy}.) We have intentionally left out several details in the above definition: the admissible set of initial data; the subset of $\Emv(\cdot)$ for which the MV Cauchy problem is stable; and the distance ${\mathcal D}$ on the set of Young measures. Still, the concept of MV stability carries one of the main messages in this paper: despite the well-documented instability of entropic weak solutions, as shown for example in the introduction and in Section \ref{sec:6}, one could still hope for a stable solution of systems of conservation laws, when it is interpreted as a measure-valued solution, subject to  atomic initial data.

Carrying out the full scope of this paradigm for general systems of conservation laws is currently beyond reach. Instead, we examine the question of whether EMV solutions of selected systems of conservation laws are  stable or not with the aid of numerical experiments reported in Section \ref{sec:6}. As for the analytical aspects, we recall that in the scalar case, measure-valued perturbations of atomic initial data are stable (Theorem \ref{thm:scalarstable2}).
In the following theorem we prove the MV stability in the case of systems, provided we further limit ourselves to   MV perturbations of  \emph{classical solutions} of \eqref{eq:mvcauchy}. The proof,  along the lines of \cite[Theorem 2.2]{DMT12},  implies weak-strong uniqueness, as in \cite{BDS1}. In particular, the theorem provides consistency of EMV solutions with classical solutions of \eqref{eq:cauchy}, as long as the latter exists.

\begin{theorem}\label{thm:classicalsoln}
Assume that there exists a classical solution $u\in W^{1,\infty}(\R^d\times\R_+,\R^N)$ of \eqref{eq:cauchy} with initial data $u_0$, both taking values in a compact set $K\subset\R^N$. Let $\nu$ be an EMV solution of \eqref{eq:mvcauchy} such that the support of both $\nu$ and its initial MV data $\sigma$ are contained in $K$. Assume that $\eta$ is uniformly convex on $K$. Then for all $t>0$, there exists a constant $C$ depending on $u$, such that 
\[
\int_{\R^d} \ip{\nu_{(x,t)}}{|u(x,t) - \xi|^2}\ dx \leq e^{Ct}\int_{\R^d} \ip{\sigma_x}{|u_0(x)-\xi|^2}\ dx,
\]
or equivalently,
\[
\Bigl\|W_2\bigl(\nu_{(\cdot,t)}, \delta_{u(\cdot,t)}\bigr)\Bigr\|_{L^2(\R^d)} \leq e^{Ct} \Bigl\|W_2\bigl(\sigma, \delta_{u_0}\bigr)\Bigr\|_{L^2(\R^d)}.
\]
In particular, if $\sigma = \delta_{u_0}$ then $\nu = \delta_{u}$, and so any (classical, weak or measure-valued) solution must coincide with $u$.
\end{theorem}
\begin{proof}
\newcommand{\au}{\overline{u}}
Denote $\au := \ip{\nu}{\id}$ and $\au_0 := \ip{\sigma}{\id}$. Define the \emph{entropy variables} $\ev = \ev(x,t) := \eta'(u(x,t))$ and denote $\ev_0 := \ev(x,0) = \eta'(u_0)$. It is readily verified that $\ev_t = -(f^i)'(u) \partial_{i}\ev$ (where $\partial_i = \frac{\partial}{\partial_{x_i}}$). Here and in the remainder of the proof we use the Einstein summation convention.

Subtracting \eqref{eq:mvsoln} from \eqref{eq:wsoln} and putting $\phi(x,t) = \ev(x,t)\theta(t)$ for some $\theta \in C_c^1(\R_+)$ gives
\begin{align*}
0 &= \int_{\R_+}\int_{\R^d} (\au-u)\cdot\bigl(\ev_t\theta + \ev\theta'\bigr) + \bigl(\ip{\nu}{f^i} - f^i(u)\bigr)\cdot \partial_i\ev\theta\ dxdt + \int_{\R^d} (\au_0-u_0)\cdot \ev_0\theta(0)\ dx \\*
&= \int_{\R_+}\int_{\R^d} (\au-u)\cdot \ev\theta' + \bigl(\underbrace{\ip{\nu}{f^i} - f^i(u) - (f^i)'(u)(\au-u)}_{=: Z^i}\bigr)\cdot \partial_i\ev\theta\ dxdt + \int_{\R^d} (\au_0-u_0)\cdot \ev_0\theta(0)\ dx
\end{align*}
Next, note that since $u$ is a classical solution, the entropy inequality \eqref{eq:entrcond} is in fact an equality. Hence, subtracting \eqref{eq:emvsoln} from \eqref{eq:entrcond} and putting $\phi(x,t) = \theta(t)$ gives
\begin{align*}
0 &\leq \int_{\R_+}\int_{\R^d} \bigl(\ip{\nu}{\eta}-\eta(u)\bigr)\theta'\ dxdt + \int_{\R^d} \bigl(\ip{\sigma}{\eta}-\eta(u_0)\bigr)\theta(0)\ dx.
\end{align*}
Subtracting these two expressions thus gives
\begin{equation}\label{eq:relentrineq}
0 \leq \int_{\R_+}\int_{\R^d} \hat{\eta}\theta' - Z^i\cdot \partial_i\ev\theta\ dxdt + \int_{\R^d} \hat{\eta}_0\theta(0)\ dx.
\end{equation}
where
\[
\hat{\eta} := \ip{\nu}{\eta}-\eta(u) -(\au-u)\cdot \ev, \qquad \hat{\eta}_0 := \ip{\sigma}{\eta}-\eta(u_0) - (\au_0-u_0)\cdot \ev_0.
\]

Let $\delta>0$, and let $t>0$ be a Lebesgue point for the function $s \mapsto \int_\R \hat{\eta}(x,s)\ dx$. We define
\[
\theta(s) := \begin{cases}
1 & s<t \\
1 - \frac{s-t}{\delta} & t \leq s < t+\delta \\
0 & t+\delta\leq s.
\end{cases}
\]
Taking the limit $\delta \to 0$ in \eqref{eq:relentrineq} then gives
\[
\int_{\R^d} \hat{\eta}(t,x)\ dx \leq - \int_0^{t}\int_{\R^d} Z^i\cdot \partial_i\ev\ dxds + \int_{\R^d} \hat{\eta}_0\ dx.
\]
Since $\nu_{(x,s)}$ is a probability distribution, it follows from the uniform convexity of $\eta$ that
\[
\hat{\eta} = \int_K \eta(\xi) - \eta(u) - \eta'(u)\cdot (\xi-u)\ d\nu_{(x,s)} \geq c\int_K |u-\xi|^2\ d\nu_{(x,s)} = c\ip{\nu_{(x,s)}}{|u-\xi|^2}.
\]
Similarly, by the $L^\infty$ bound on both $u$ and $\partial_i\ev$, we have
\[
\hat{\eta}_0 \leq C\ip{\sigma}{|u_0-\xi|^2} \qquad \text{and} \qquad |Z^i\cdot \partial_i\ev| \leq C\ip{\nu}{|u-\xi|^2}.
\]
Hence,
\[
\int_{\R^d}\ip{\nu_{(x,t)}}{|u-\xi|^2}\ dx \leq C\int_0^{t}\int_\R \ip{\nu_{(x,s)}}{|u-\xi|^2}\ dxds + C\int_{\R^d} \ip{\sigma_x}{|u_0-\xi|^2}\ dx.
\]
By the integral form of Gr\"onwall's lemma, we obtain the desired result.
\end{proof}
\begin{remark}
In addition to proving consistency of entropy measure valued solutions with classical solutions (when they exist), the above theorem also provides local (in time) uniqueness of MV solutions in the following sense. Let $u_0 \in W^{1,\infty}(\R^d,\R^N)$ be the initial data in \eqref{eq:cauchy}, then by standard results \cite{DAF1}, we have local (in time) existence of a unique classical solution $u \in W^{1,\infty}(\R^d \times \R_+,\R^N)$. By the above theorem, $\delta_u$ is also the unique EMV solution of the MV Cauchy problem \eqref{eq:mvcauchy} with initial data $\delta_{u_0}$. However, uniqueness can break down once this MV solution develops singularities.
\end{remark}

\section{Construction of approximate EMV solutions}\label{sec:4}
Although existence results for specific systems of conservation laws such as polyconvex elastodynamics \cite{DMT12}, two-phase flows \cite{Frid1,Frid2} and transport equations \cite{CG1} are available, there exists no  global existence result for a generic system of conservation laws. We pursue a different approach by constructing approximate EMV solutions and proving their convergence. A procedure for constructing approximate EMVs is outlined in the present section. It provides a constructive proof of existence of EMV solutions for a generic system of conservation laws,
and it is implemented in the numerical simulations reported in Section \ref{sec:6}.

\subsection{Numerical approximation of EMV solutions}
The construction of approximate EMV solutions consists of several ingredients.  It begins with a proper choice of a numerical scheme for approximating the system of conservation laws \eqref{eq:cauchy}.
\subsubsection{Numerical schemes for one- and multi-dimensional conservation laws}
For simplicity, we begin with the description of a numerical scheme for a one-dimensional  system of conservation laws,  \eqref{eq:cauchy} with $d=1$.
We  discretize our computational domain into cells $\cell_i := [x_\imhf,x_\iphf)$ with mesh size $\Dx = x_\iphf - x_\imhf$ and midpoints
$$
x_i := \frac{x_\imhf + x_\iphf}{2}.
$$
Note that we consider a uniform mesh size $\Dx$ only for the sake of simplicity of the exposition.
Next, we discretize the one-dimensional system, $\partial_t u+\partial_x f(u)=0$, with the following semi-discrete finite difference scheme for $u^\Dx_i(t)\equiv u^\Dx(x_i,t)$ (cf.\ \cite{GR1,LEV1}):
\begin{subequations}\label{eqs:fvm}
\begin{equation}\label{eq:fvm}
\begin{split}
\frac{d}{dt}u^\Dx_i(t) + \frac{1}{\Dx}\left(F^\Dx_\iphf(t) - F^\Dx_\imhf(t)\right) &= 0 \qquad t>0,\ i\in\Z \\
u^\Dx_i(0) &= u^\Dx_0(x_i) \qquad i\in\Z.
\end{split}
\end{equation}
Here, $u_0^{\Dx}$ is an approximation to the initial data $u_0$. Henceforth, the dependence of $u$ and $F$ on $\Dx$ will be suppressed for notational convenience. The \emph{numerical flux function} $F_\iphf(t)$ is a function depending on $u(x_j,t)$ for $j=i-p+1,\dots,i+p$ for some $p\in\N$. It is assumed to be consistent with $f$ and locally Lipschitz continuous, i.e., for every compact $K\subset\R^N$ there is a $C>0$ such that
\[
|F_\iphf(t) - f(u_i(t))| \leq C\sum_{j=i-p+1}^{i+p} |u_j-u_i|
\]
whenever $u(x_j,t)\in K$ for $j=i-p+1,\dots,i+p$.

The semi-discrete scheme \eqref{eq:fvm} needs to be integrated in time to define a fully discrete numerical approximation. Again for simplicity, we will use an exact time integration, resulting in
\begin{equation}\label{eq:exactint}
u^\Dx_i(t+\Delta t) = u^\Dx_i(t) - \frac{1}{\Dx}\int_t^{t+\Delta t}\left(F_\iphf(\tau) - F_\imhf(\tau)\right)\ d\tau.
\end{equation}
\end{subequations}
The function $t \mapsto u(x_i,t)$ is then Lipschitz, that is,
\[
|u^\Dx(x_i,t) - u^\Dx(x_i,s)| \leq \frac{C}{\Dx}|t-s| \qquad \forall\ i\in\Z,\ t,s\in[0,T].
\]
In particular, for all $\Dx>0$ and $i\in\N$, the function $t \mapsto u(x_i,t)$ is differentiable almost everywhere.
We denote the  evolution operator associated with the one-dimensional scheme \eqref{eqs:fvm} with mesh size $\Dx$ by $\Soln^\Dl$, so that $u^\Dx = \Soln^\Dl u_0$.

A similar framework applies to systems of conservation laws in several space dimensions.
To simplify the notation we restrict ourselves to  the two-dimensional case (with the usual relabeling $(x_1,x_2) \mapsto (x,y)$), $\partial_t u+\partial_x f^x(u)+\partial_y f^y(u)=0$. 
\par We discretize our two-dimensional computational domain with into cells with mesh size $\Delta:=(\Dx_1,\Dx_2)$: with the usual relabeling $(\Dx_1,\Dx_2) \mapsto (\Dx,\Dy)$), these two-dimensional cells, $\cell_{i,j} := [x_\imhf,x_\iphf) \times [y_{j-1/2},y_{j+1/2})$ are assumed to a have a fixed mesh ratio,  $\Dx = x_\iphf - x_\imhf$ and $\Dy = y_\jphf - y_\jmhf$ such  that $\Dy = c \Dx$ for some constant $c$. Let
$$
\left(x_i,y_j\right) = \left(\frac{x_\imhf + x_\iphf}{2},\frac{y_\jmhf + y_\jphf}{2} \right)
$$
denote the mid-cells. We end up with the following semi-discrete finite difference scheme for $u^{\Delta}_{ij}= u^{\Delta}(x_i,y_j,t)$  (cf.\ \cite{LEV1,GR1}):
\begin{subequations}\label{eqs:fvm2}
\begin{equation}\label{eq:fvm2}
\begin{split}
\frac{d}{dt}u^{\Delta}_{ij}(t) + \frac{1}{\Dx}\left(F^{x,\Dx}_{\iphf,j}(t) - F^{x,\Dx}_{\imhf.j}(t)\right) &+ \frac{1}{\Dy}\left(F^{y,\Dy}_{i,\jphf}(t) - F^{y,\Dy}_{i,\jmhf}(t)\right) = 0, \qquad t>0,  \\
u^{\Delta}_{ij}(0) &= u^{\Delta}_0(x_i,y_j) \qquad i\in\Z.
\end{split}
\end{equation}
Here, $u_0^{\Delta}\approx u_0$ is the approximate initial data and  $F^{x,\Dx}_{\iphf,j}, F^{y,\Dy}_{i,\jphf}$ are the locally Lipschitz {numerical flux functions} which are  assumed to be  consistent with the flux function $f = \left(f^x,f^y\right)$.
We integrate the semi-discrete scheme \eqref{eq:fvm2} exactly in time to obtain
\begin{equation}\label{eq:exactint2D}
\begin{split}
u^{\Delta}_{ij}(t+\Delta t) =   u^{\Delta}_{ij}(t) & -  \frac{1}{\Dx}\int_t^{t+\Delta t}\left(F^{x,\Dx}_{\iphf,j}(\tau) - F^{x,\Dx}_{\imhf,j}(\tau)\right)\ d\tau   \\
 & -    \frac{1}{\Dy}\int_t^{t+\Delta t}\left(F^{y,\Dy}_{i,\jphf}(\tau) - F^{y,\Dy}_{i,\jmhf}(\tau)\right)\ d\tau.
\end{split}
\end{equation}
\end{subequations}
We denote the evolution operator corresponding to  \eqref{eqs:fvm2} and associated with the two dimensional  mesh size $\Delta:=(\Dx,\Dy)$ by $\Soln^\Delta$.

\subsubsection{Weak-$\ast$ convergent schemes}
The next ingredient in the construction of approximate EMV solutions  for \eqref{eq:mvcauchy} is to  employ the above numerical schemes in the following three step algorithm.

\begin{algorithm}\label{alg:approxmv}~
\begin{description}
\item[\textbf{Step 1:}] Let $u_0: \Omega \mapsto L^{\infty} (\R^d)$ be a random field on a probability space $(\Omega,\Sigmaalg,P)$ such that the initial Young measure $\sigma$ in \eqref{eq:mvcauchy} is the law of the random field $u_0$ (see Proposition \ref{prop:lawexists}).
\item[\textbf{Step 2:}] We evolve the initial random field by applying the numerical scheme \eqref{eq:fvm} for every $\omega \in \Omega$ to obtain an approximation $u^\Dl(\omega) := \Soln^{\Dl}u_0(\omega)$ to the solution random field $u(\omega)$, corresponding to the initial random field $u_0(\omega)$.
\item[\textbf{Step 3:}] Define the approximate measure-valued solution $\nu^\Dl$ as the law of $u^\Dl$ with respect to $P$ (see Appendix \ref{app:younglaw}).
\end{description}
\end{algorithm}
By Proposition \ref{prop:lawwelldef} (Appendix \ref{app:younglaw}), $\nu^{\Dl}$ is a Young measure. This sequence of Young measures $\nu^{\Dl}$ serve as approximations to the EMV solutions of \eqref{eq:mvcauchy}.

Next, we show that if the numerical scheme \eqref{eq:fvm} satisfies a set of criteria, then the approximate Young measures $\nu^{\Dl}$ generated by Algorithm 4.1 will converge weak* to an EMV solution of \eqref{eq:mvcauchy}. Specific examples for such weak* convergent schemes is provided in Section \ref{sec:5}.  To simplify the presentation, we restrict attention to the one-dimensional case; the argument is readily extended to the general multi-dimensional case, and the details can be found in
\cite{FJO1} (see section $3.2$, in particular Lemmas $3.4$ and $3.5$).

\begin{theorem}
\label{thm:convmv}
Assume that the approximate solutions $u^\Dl$ generated by the one-dimensional numerical scheme \eqref{eqs:fvm} satisfy the following:
\begin{subequations}\label{eqs:EMVcond}
\begin{itemize}
\item {\bf Uniform boundedness:}
\begin{equation}
\label{eq:linf}
\|u^{\Dx}(\omega)\|_{L^{\infty}(\R \times \R_+)} \leq C, \quad \forall\ \omega \in \Omega,\ \Dx > 0.
\end{equation}
\item {\bf Weak BV:}  There exists $1\leq \tve < \infty$ such that
\begin{equation}\label{eq:tvbound}
\lim_{\Dx\to 0}\int_0^T \sum_i \left|u^\Dx_{i+1}(\omega,t) - u_i^\Dx(\omega,t)\right|^\tve\Dx dt = 0 \qquad \forall\ \omega\in\Omega
\end{equation}
\item {\bf Entropy consistency:} The numerical scheme \eqref{eq:fvm} is entropy stable with respect to an entropy pair $(\eta,q)$ i.e, there exists a numerical entropy flux $Q = Q_\iphf(t)$, consistent with the entropy flux $q$ and locally Lipschitz, such that computed solutions satisfy the discrete entropy inequality
\begin{equation}\label{eq:dentrineq}
\frac{d}{dt}\eta(u^\Dx) + \frac{1}{\Dx}\left(Q^\Dx_\iphf - Q^\Dx_\imhf \right) \leq 0 \qquad \forall\ t>0,\ i\in\Z,\ \omega \in \Omega.
\end{equation}
\item {\bf Consistency with initial data:} If $\sigma^\Dx$ is the law of $u_0^\Dx$, then
\begin{equation}\label{eq:initialcond}
\lim_{\Dx\to 0}\int_{\R}\psi(x)\ip{\sigma^\Dx_x}{\id}\ dx = \int_{\R}\psi(x)\ip{\sigma_x}{\id}\ dx \qquad \forall\ \psi\in C_c^1(\R).
\end{equation}
and
\begin{equation}\label{eq:entrinitialcond}
\limsup_{\Dx\to 0}\int_{\R}\psi(x) \ip{\imv^\Dx_x}{\eta} \ dx \leq \int_{\R}\psi(x) \ip{\imv_x}{\eta} \ dx \qquad \forall\ 0\leq \psi\in C_c^1(\R)
\end{equation}
\end{itemize}
\end{subequations}
Then the approximate Young measures $\nu^{\Dx}$  converge weak* (up to a subsequence) as $\Delta x \rightarrow 0$, to an EMV solution $\nu \in \Young(\R\times \R_+,\R^N)$ of \eqref{eq:mvcauchy}.
\end{theorem}

\begin{proof}
From the assumption \eqref{eq:linf} that $u^\Dx$ is $L^\infty$-bounded, it follows that $\nu^\Dx$ is compactly supported, in the sense that its support $\supp\nu^\Dx_{(x,t)}$ lies in a fixed compact subset of $\R^N$ for every $(x,t)$; see Appendix \ref{app:youngunifbounded}. The fundamental theorem of Young measures (see Appendix \ref{app:younglpbound}) gives the existence of a $\nu\in\Young(\R^d\times\R_+,\R^N)$ and a subsequence of $\nu^\Dx$ such that $\nu^\Dx \wto \nu$ weak*

First, we show that the limit Young measure $\nu$ satisfies the entropy inequality \eqref{eq:emvsoln}. To this end,  let $\phi \in C_c^1(\R\times[0,T))$. Then
\begin{eqnarray*}
\begin{split}
\lefteqn{\int_0^T\int_{\R^d} \ip{\nu_{(x,t)}}{\eta}\partial_t\phi(x,t) + \ip{\nu_{(x,t)}}{q}\partial_x\phi(x,t)\ dxdt}\\
& &  \qquad = \lim_{\Dx\to 0} \int_0^T\int_{\R^d} \ip{\nu^\Dx_{(x,t)}}{\eta}\partial_t\phi(x,t) + \ip{\nu^\Dx_{(x,t)}}{q}\partial_x\phi(x,t)\ dxdt
\end{split}
\end{eqnarray*}
by the weak* convergence $\nu^\Dx \wto \nu$. Denote $\eta^\Dx(\omega,x,t) := \eta(u^\Dx(\omega,x,t))$. Then for every $\Dx>0$ we have
\begin{align*}
\int_0^T\int_{\R^d} &\ip{\nu^\Dx_{(x,t)}}{\eta}\partial_t\phi(x,t)\ dxdt + \int_{\R^d}\phi(x,0) \ip{\imv^\Dx_x}{\eta} \ dx = \int_\R\int_0^T -\partial_t\ip{\nu^\Dx_{(x,t)}}{\eta}\phi(x,t)\ dtdx \\
=&\ \int_\Omega\int_\R\int_0^T -\partial_t \eta^\Dx(\omega,x,t)\phi(x,t)\ dtdxdP(\omega) \\
\geq&\ \int_\Omega\int_\R\int_0^T \sum_i \ind_{\cell_i}(x)\frac{Q_\iphf(\omega,t) - Q_\imhf(\omega,t)}{\Dx}\ \phi(x,t)dtdxdP(\omega) \\
=&\ \int_\Omega\int_0^T\sum_i \frac{Q_\iphf(\omega,t) - Q_\imhf(\omega,t)}{\Dx}\int_{\cell_i}\phi(x,t)\ dxdtdP(\omega) \\
=&\ \int_\Omega\int_0^T\sum_i \left(Q_\iphf(\omega,t) - Q_\imhf(\omega,t)\right)\avg{\phi}_i^\Dx(t)\ dtdP(\omega) \\
=&\ -\int_\Omega\int_0^T\sum_i Q_\iphf(\omega,t)\frac{\avg{\phi}_{i+1}^\Dx(t) - \avg{\phi}_i^\Dx(t)}{\Dx}\ \Dx dtdP(\omega) \\
=&\ - \int_\Omega\int_0^T \sum_i q(u_i^\Dx(\omega,t))\frac{\avg{\phi}_{i+1}^\Dx(t) - \avg{\phi}_i^\Dx(t)}{\Dx}\ \Dx dtdP(\omega)\\
&\ -\int_\Omega\int_0^T \sum_i \left(Q_\iphf(\omega,t) - q(u_i^\Dx(\omega,t))\right)\frac{\avg{\phi}_{i+1}^\Dx(t) - \avg{\phi}_i^\Dx(t)}{\Dx}\ \Dx dtdP(\omega).
\end{align*}
(We have written $\avg{\phi}^\Dx_i(t) := \frac{1}{\Dx}\int_{\cell_i}\phi(x,t)\ dx$.) The first term can be written as
\begin{align*}
- \int_\Omega\int_0^T \sum_i q(u_i^\Dx(\omega,t))\frac{\avg{\phi}_{i+1}^\Dx(t) - \avg{\phi}_i^\Dx(t)}{\Dx}\ \Dx dt
&= -\int_0^T \sum_i \ip{\nu^\Dx_{(x_i,t)}}{q}\frac{\avg{\phi}_{i+1}^\Dx(t) - \avg{\phi}_i^\Dx(t)}{\Dx}\ \Dx dtdP(\omega) \\
&\to -\int_0^T\int_\R \ip{\nu_{(x,t)}}{q} \partial_x \phi(x,t)\ dxdt.
\end{align*}
The second term goes to zero:
\begin{align*}
\Big|\int_\Omega&\int_0^T \sum_i \left(Q_\iphf(\omega,t) - q(u_i^\Dx(\omega,t))\right)\frac{\avg{\phi}_{i+1}^\Dx(t) - \avg{\phi}_i^\Dx(t)}{\Dx}\ \Dx dtdP(\omega) \Big| \\
&\leq C\int_\Omega\int_0^T \sum_i \left|u^\Dx_{i+1}(\omega,t) - u_i^\Dx(\omega,t)\right| \left|\frac{\avg{\phi}_{i+1}^\Dx(t) - \avg{\phi}_i^\Dx(t)}{\Dx}\right|\ \Dx dtdP(\omega) \\
&\leq C\sup_\omega\left(\int_0^T \sum_i \left|u^\Dx_{i+1}(\omega,t) - u^\Dx_i(\omega,t)\right|^\tve\ \Dx dt\right)^{1/\tve} \|\partial_x\phi\|_{L^{\tve'}(\R\times(0,T))} \\
&\to 0
\end{align*}
by \eqref{eq:tvbound}, where $\tve'$ is the conjugate exponent of $\tve$. In conclusion, the limit $\nu$ satisfies \eqref{eq:emvsoln}.

The proof that the limit measure $\nu$ satisfies \eqref{eq:mvsoln} follows from the above by setting $\eta = \pm\id$ and $q = \pm f$.
\end{proof}

A similar construction can be readily performed in several space dimensions. To this end, we replace $\Soln^\Dx$ in Step 2 of Algorithm \ref{alg:approxmv} with the two-dimensional solution operator $\Soln^{\Delta}$, and the corresponding approximate solution $u^\Dx$ with $u^{\Delta}$. The weak* convergence of the resulting approximate Young measure $\nu^{\Delta}$ is described below.
\begin{theorem}
\label{thm:convmv2}
Assume that the approximate solutions $u^{\Delta}$ generated by scheme \eqref{eq:fvm2} satisfy the following:
\begin{itemize}
\item {\bf Uniform boundedness}:
\begin{equation}
\label{eq:linf2}
\|u^{\Delta}(\omega)\|_{L^{\infty}(\R^2 \times \R_+)} \leq C, \quad \forall \omega \in \Omega, \Dx,\Dy > 0.
\end{equation}
\item {\bf Weak BV}:  There exist  $1\leq \tve < \infty$ such that
\begin{equation}\label{eq:tvbound2}
\lim_{\Dx,\Dy\to 0}\int_0^T \sum_{i,j} \left(\left|u^\Delta_{i+1,j}(\omega,t) - u_{i,j}^\Delta(\omega,t)\right|^\tve + \left|u^\Delta_{i,j+1}(\omega,t) - u_{i,j}^\Delta(\omega,t)\right|^\tve \right) \Dx \Dy dt = 0 \qquad \forall\ \omega\in\Omega
\end{equation}
\item {\bf Entropy consistency}: The numerical scheme \eqref{eq:fvm2} is entropy stable with respect to an entropy pair $(\eta,q)$, in the sense that there exist locally Lipschitz numerical entropy fluxes $(Q^{x,\Dx},Q^{y,\Dy}) = (Q^{x,\Dx}_{\iphf,j}(t),Q^{y,\Dy}_{i,\jphf}(t))$, consistent with the entropy flux $q = (q^x,q^y)$, such that computed solutions satisfy the discrete entropy inequality
\begin{equation}\label{eq:dentrineq2}
\frac{d}{dt}\eta(u^\Delta) + \frac{1}{\Dx}\left(Q^{x,\Dx}_{\iphf,j} - Q^{x,\Dx}_{\imhf,j} \right) + \frac{1}{\Dy}\left(Q^{y,\Dy}_{i,\jphf} - Q^{y,\Dy}_{i,\jmhf} \right) \leq 0 \qquad \forall\ t>0,\ i,j\in\Z,\ \omega \in \Omega.
\end{equation}
\item {\bf Consistency with initial data}: Let $\sigma^{\Delta}$ be the law of the random field $u_0^{\Delta}$ that approximates the initial random field $u_0$. Then, the consistency conditions \eqref{eq:initialcond} and \eqref{eq:entrinitialcond} hold.
\end{itemize}
Then, the approximate Young measures $\nu^{\Delta}$ converge weak* (up to a subsequence) to a Young measure $\nu \in \Young(\R^2 \times \R_+, \R^N)$ as $\Dx,\Dy \rightarrow 0$ and $\nu$ is an EMV solution of \eqref{eq:mvcauchy} i.e,
\end{theorem}
The proof of the above theorem is a simple generalization of the proof of convergence theorem \ref{thm:convmv}, see Section $3.2$ of \cite{FJO1} (in particular lemmas $3.4$ and $3.5$) for details. The above construction can also be readily extended to three spatial dimensions.

\begin{remark}
The uniform $L^{\infty}$ bound \eqref{eq:linf}, \eqref{eq:linf2}  is a technical assumption that we require in this article. This assumption can be relaxed to only an $L^p$ bound. This extension is described in a forthcoming paper \cite{FSID5}.
\end{remark}
\begin{remark}
The conditions \eqref{eq:initialcond} and \eqref{eq:entrinitialcond}, which say that $\sigma^\Dx\to\sigma$ in a certain sense, are weaker than weak* convergence. It is readily checked that a sufficient condition for this is that $u_0 \in L^1(\R;\R^N)\cap L^\infty(\R;\R^N)$ and $u_0^\Dx(\omega,\cdot) \to u_0(\omega,\cdot)$ in $L^1(\R^d;\R^N)$ for all $\omega\in\Omega$ (which in fact implies that $\sigma^\Dx \to \sigma$ strongly).
\end{remark}

\subsubsection{Weak-$\ast$ convergence with atomic initial data}
In view of the nonuniqueness example \ref{ex:nonunique}, one can not expect an unique construction of EMV solutions for general MV initial data.  Instead, as argued before, we focus our attention on perturbations of atomic initial data $\sigma = \delta_{u_0}$ for some $u_0 \in L^1(\R^d,\R^N) \cap L^{\infty}(\R^d,\R^N)$. We construct approximate EMV solutions of \eqref{eq:mvcauchy} in this case using  the following specialization of Algorithm 4.1.

\begin{algorithm}\label{alg:atomic}
Let $(\Omega, \Sigmaalg, P)$ be a probability space and let $\rand : \Omega\to L^1(\R^d)\cap L^\infty(\R^d)$ be a random variable satisfying $\|X\|_{L^1(\R^d)} \leq 1$ $P$-almost surely.
\begin{description}
\item[\textbf{Step 1:}] Fix a small number $\amp>0$.  Perturb $u_0$ by defining $u_0^{\amp}(\omega,x) := u_0(x) + \amp X(\omega,x)$. Let $\imv^{\amp}$ be the law of $u_0^{\amp}$.
\item[\textbf{Step 2:}] For each $\omega\in\Omega$, let $u^{\Dl,\amp}(\omega) := \Soln^\Dl u_0^{\amp}(\omega)$, with $\Soln^{\Dl}$ being the solution operator corresponding to the numerical scheme \eqref{eqs:fvm}.
\item[\textbf{Step 3:}] Let $\nu^{\Dl,\amp}$ be the law of $u^{\Dl,\amp}$ with respect to $P$.
\qed\qedhere
\end{description}
\end{algorithm}

\begin{theorem}
\label{thm:alphaconv}
Let $\{\nu^{\Dx,\amp}\}$ be the family approximate EMV solutions constructed by Algorithm \ref{alg:atomic}. Then there exists a subsequence $(\Dx_n,\amp_n) \to 0$  such that
\[
\nu^{\Dx_n,\amp_n}\wto\nu\in \Emv(\delta_{u_0}),
\]
that is, $\nu^{\Dx_n,\amp_n}$ converges weak* to an EMV solution $\nu$ with atomic initial data $u_0$.
\end{theorem}
\begin{proof}
By Theorem \ref{thm:convmv} we know  that for every $\amp > 0$ there exists a subsequence $\nu^{\Dx_n,\amp}$ which converges weak* to an EMV solution $\nu^{\amp}$ of \eqref{eq:mvcauchy} with initial data $\imv^{\amp}$. Thus, \eqref{eq:emvsoln} holds with $(\nu,\imv)$ replaced by  $(\nu^{\amp},\imv^\amp)$;
we abbreviate the corresponding entropy statement  as \eqref{eq:emvsoln}${}_\amp$. The convergence of the sequence $\nu^{\amp_n}$ as $\amp_n \rightarrow 0$ is a consequence of the fundamental theorem of Young measures: by Theorem \ref{thm:young}, there exists a weak* convergent subsequence $\nu^{\amp_n} \rightharpoonup \nu$. The fact that $\nu$ is  an EMV solution follows at once by  taking the limit $\amp_n \to 0$ in \eqref{eq:emvsoln}${}_{\amp_n}$.
\end{proof}


\subsection{What are we computing? Weak* convergence of space-time averages}\label{sec:what}
We begin by quoting \cite[p. 143]{Lax07}: ``Just because we cannot prove that compressible flows with prescribed initial values exist doesn't mean that we cannot compute them" . The question is what are the computed quantities encoded in the EMV solutions.

According to Theorems \ref{thm:convmv}, \ref{thm:alphaconv}, the approximations generated by Algorithm $4.1$ and $4.5$ converge to an EMV solution in the following sense: for all $g \in C_0(\R^N)$ and $\psi \in L^1(\R^d \times \R_+)$,
\begin{equation}\label{eq:func1}
\lim\limits_{\Dx \rightarrow 0} \int_{\R_+}\int_{\R^d} \psi(x,t) \ip{\nu^{\Dl}_{(x,t)}}{g}\ dxdt = \int_{\R_+}\int_{\R^d} \psi(x,t) \ip{\nu_{(x,t)}}{g}\ dxdt.
\end{equation}
As we assume that the approximate solutions are $L^\infty$-bounded (property \eqref{eq:linf}), any $g \in C(\R^N)$ can serve as a test function in \eqref{eq:func1}; see Appendix \ref{app:younglpbound}. In particular, we can choose $g(\xi) = \xi$ to obtain the \emph{mean} of the measure valued solution. Similarly, the variance can be computed by choosing the test function $g(\xi) = \xi \otimes \xi$. Higher statistical moments can be computed analogously.

In practice, the goal of any numerical simulation is to accurately compute \emph{statistics of space-time averages} or \emph{statistics of functionals of interest} of solution variables and to compare them to experimental or observational data. Thus, the weak* convergence of approximate Young measures, computed by Algorithms $4.1$ and $4.5$ provides an approximation of exactly these \emph{observable} quantities of interest.

\subsubsection{Monte Carlo approximation}
In order to compute statistics of space-time averages in \eqref{eq:func1}, we need to compute phase space integrals with respect to the measure $\nu^{\Dx}$:
$$
\ip{\nu^{\Dx}_{(x,t)}}{g} := \int\limits_{\R^N} g(\xi)\ d\nu^{\Dx}_{(x,t)}(\xi).
$$
 The last ingredient in our  construction of EMV solutions, therefore, is  numerical approximation which is necessary to compute these phase space integrals. To this end, we utilize the equivalent representation of the measure $\nu^{\Dx}$ as the law of the \emph{random field} $u^{\Dx}$:
\begin{equation}\label{eq:func2}
\ip{\nu^{\Dx}_{(x,t)}}{g} := \int\limits_{\R^N} g(\xi)\ d\nu^{\Dx}_{(x,t)}(\xi) = \int_{\Omega} g(u^{\Dx}(\omega;x,t))\ dP(\omega).
\end{equation}
We will approximate this integral by a Monte Carlo sampling procedure:
\begin{algorithm}\label{alg:montecarlo}
Let $\Dx > 0$ and let $M$ be a positive integer. Let $\sigma^{\Dx}$ be the initial Young measure in \eqref{eq:mvcauchy} and let $u_0^\Dx$ be a random field $u_0^\Dx:\Omega\times\R^d \to \R^N$ such that $\sigma^{\Dx}$ is the law of $u_0^\Dx$.
\begin{description}
\item[{\bf Step 1:}] Draw $M$ independent and identically distributed random fields $u_0^{\Dx,k}$ for $k=1,\dots, M$.
\item[{\bf Step 2:}] For each $k$ and \emph{for a fixed} ${\omega}\in\Omega$, use the finite difference scheme \eqref{eq:fvm} to numerically approximate the conservation law \eqref{eq:cauchy} with initial data $u_0^{\Dx,k}({\omega})$. Denote $u^{\Dx,k}({\omega}) = \Soln^{\Dx}u_0^{\Dx,k}({\omega}).$
\item[{\bf Step 3:}] Define the approximate measure-valued solution
\[
\nu^{\Dx,M} := \frac{1}{M}\sum_{k=1}^M \delta_{u^{\Dx,k}({\omega})}.
\]
\qed\qedhere
\end{description}
\end{algorithm}
For every $g\in C(\R^N)$ we have
\[
\ip{\nu^{\Dx,M}}{g} = \frac{1}{M} \sum_{k=1}^{M} g\bigl(u^{\Dx,k}(\omega)\bigr).
\]
Thus, the space-time average \eqref{eq:func1} is approximated by
\begin{equation}\label{eq:mc1}
\int_{\R_+}\int_{\R^d} \psi(x,t) \ip{\nu^\Dx_{(x,t)}}{g}\ dxdt \approx \frac{1}{M} \sum_{k=1}^{M} \int_{\R_+}\int_{\R^d} \psi(x,t) g\bigl(u^{\Dx,k}(\omega;x,t)\bigr)\ dxdt.
\end{equation}
Note that, as in any Monte Carlo method, the approximation $\nu^{\Dx,M}$ depends on the choice of $\omega\in\Omega$, i.e., the choice of seed in the random number generator. However, we can prove that the quality of approximation is independent of this choice, $P$-almost surely:

\begin{theorem}[Convergence for large samples]\label{thm:mcconv}
Algorithm \ref{alg:montecarlo} converges, that is,
\[
\nu^{\Dx,M} \wto \nu^\Dx \quad \text{weak*},
\]
and, for a subsequence $M\to\infty$, $P$-almost surely. Equivalently, for every $\psi \in L^1(\R^d \times \R_+)$ and $g \in C(\R^N)$,
\begin{equation}\label{eq:mc2}
\lim_{M \to \infty}  \frac{1}{M} \sum_{k=1}^{M} \int_{\R_+}\int_{\R^d} \psi(x,t) g\bigl(u^{\Dx,k}(x,t)\bigr) dx dt = \int_{\R_+}\int_{\R^d} \psi(x,t) \ip{\nu^{\Dx}_{(x,t)}}{g}\ dxdt.
\end{equation}
The limits are uniform in $\Dx$.
\end{theorem}

The proof involves an adaptation of the law of large numbers for the present setup and is provided in Appendix \ref{app:MCproof}.
Combining \eqref{eq:mc2} with the convergence established in Theorem \ref{thm:convmv},
we conclude with the following.
\begin{corollary}[Convergence with mesh refinement]
There are subsequences $\Dx \to 0$ and $M\rightarrow \infty$ such that
\[
\nu^{\Dx,M} \wto \nu \quad \text{weak*,}
\]
or equivalently, for every $\psi \in L^1(\R^d \times \R_+)$ and $g \in C(\R^N)$,
\begin{equation}\label{eq:mc3}
\lim_{\Dx \to 0} \lim_{M \to \infty}  \frac{1}{M} \sum_{k=1}^{M} \int_{\R_+}\int_{\R^d} \psi(x,t) g\bigl(u^{\Dx,k}(x,t)\bigr)\ dx dt =
\int_{\R_+}\int_{\R^d} \psi(x,t) \ip{\nu_{(x,t)}}{g}\ dxdt
\end{equation}
The limits in $\Dx$ and $M$ are interchangeable.
\end{corollary}

\section{Examples of weak* convergent numerical schemes}
\label{sec:5}
In this section, we provide concrete examples of numerical schemes that satisfy the criteria \eqref{eqs:EMVcond}  of Theorem \ref{thm:convmv},
for weak* convergence to EMV solutions of \eqref{eq:mvcauchy}.

\subsection{Scalar conservation laws}
We begin by considering scalar conservation laws. Monotone finite difference (volume) schemes (see \cite{CM1,GR1} for a precise definition) for scalar equations are uniformly bounded in $L^{\infty}$ (as they satisfy a discrete maximum principle), satisfy a discrete entropy inequality  (using the Crandall-Majda numerical entropy fluxes \cite{CM1}) and are TVD -- the total variation of the approximate solutions is non-increasing over time. Consequently, the approximate solutions satisfy the weak BV estimate \eqref{eq:tvbound} (resp, \eqref{eq:tvbound2} in the multi-dimensional case) with $\tve = 1$. Thus, monotone schemes, approximating scalar conservation laws, satisfy all the abstract criteria of Theorem \ref{thm:convmv}.

In fact, one can obtain a precise convergence rate for monotone schemes \cite{Kutz1}:
\begin{equation}\label{eq:kutz}
\left\|u^\Dx(\omega,\cdot,t) - u(\omega,\cdot,t)\right\|_{L^1(\R^d)} \leq Ct\TV(u_0(\omega))\sqrt{|\Dx|} \qquad \forall \ \omega,
\end{equation}
where $u(\omega) = \lim_{\Dx\to 0} u^\Dx(\omega)$ denotes the entropy solution of the Cauchy problem for a scalar conservation law with initial data $u_0(\omega)$. Using this error estimate, we obtain the following strong convergence results for monotone schemes.
\begin{theorem}\label{thm:convscalar}
Let $\nu^\Dx$ be generated by Algorithm \ref{alg:approxmv}, and let $\nu$ be the law of the entropy solution $u(\omega)$. If $\TV(u_0(\omega)) \leq C$ for all $\omega\in\Omega$, then $\nu^\Dx \to \nu$ strongly as $\Dx\to 0$.
\end{theorem}
\begin{proof}
Define $\pi^\Dx_{\xt}\in\Prob(\R^N\times\R^N)$ as the law of the random variable $\left(u^\Dx(\xt), u(\xt)\right)$,
\[
\pi^\Dx_{\xt}(A) := P\left(\bigl(u^\Dx(\xt), u(\xt)\bigr) \in A\right), \qquad A\subset \R\times\R \text{ Borel measurable}.
\]
Then $\pi^\Dx_{\xt}$ is a Young measure for all $\xt$ and $\Dx>0$. Clearly, $\pi^\Dx_{\xt} \in \Pi\bigl(\nu^\Dx_{\xt}, \nu_{\xt}\bigr)$, and hence
\[
W_1\Bigl(\nu^\Dx_{\xt}, \nu_{\xt}\Bigr) \leq \int_{\R^N\times\R^N} |\xi - \zeta|\ d\pi(\xi, \zeta)
= \int_{\Omega} |u^\Dx(\omega,x,t) - u(\omega,x,t)|\ dP(\omega).
\]
Hence, by Kutznetsov's error estimate \eqref{eq:kutz},
\[
\int_0^T \int_\R W_1\Bigl(\nu^\Dx_{\xt}, \nu_{\xt}\Bigr)\ dxdt \leq C\sqrt{|\Dx|} \to 0 \qquad \text{as } \Dx\to 0.
\]
\end{proof}

\begin{remark}
We can relax the uniform boundedness of $\TV(u_0(\omega))$ to just integrability of the function $\omega \mapsto \TV(u_0(\omega))$.
\end{remark}
\begin{remark}
Note that, in light of Theorem \ref{thm:scalarstable} and Example \ref{ex:nonunique}, the limit entropy measure-valued solution $\nu$ is unique only if the initial measure-valued data $\sigma$ is atomic.
\end{remark}

\subsection{Systems of conservation laws}
We present two classes of schemes, approximating systems of conservation laws, that satisfy the convergence criteria
\eqref{eqs:EMVcond} of Theorem \ref{thm:convmv}, respectively the convergence criteria of Theorem \ref{thm:convmv2}. 

\subsubsection{TeCNO finite difference schemes}
The \textbf{TeCNO} schemes, introduced in \cite{FTSID2,FJO1}, are finite difference schemes of the form \eqref{eq:fvm} with flux function
\begin{equation}\label{eq:esf1}
F_\iphf := \tilde{F}^p_\iphf - \frac{1}{2}D_\iphf \bigl(\ev_{i+1}^- - \ev_i^+\bigr).
\end{equation}
Here, $\tilde{F}^p_\iphf$ is a $p$-th order accurate ($p\in\N$) entropy conservative numerical flux (see \cite{TAD2,LMR1}), $D_\iphf$ is a positive definite matrix, and $\ev_j^\pm$ are the cell interface values of a $p$-th order accurate ENO reconstruction of the entropy variable $\ev := \eta'(u)$ (see \cite{HEOC1,FTSID1}). The multi-dimensional version (on a Cartesian grid) was also designed in \cite{FTSID2}, see also \cite{FJO1}. It was shown in \cite{FTSID2,FJO1} that the TeCNO schemes
\begin{itemize}
\item are (formally) $p$-th order accurate
\item are entropy stable -- they satisfy a discrete entropy inequality of the form \eqref{eq:dentrineq} (see Theorem $4.1$ of \cite{FTSID2} for the one-dimensional case and Theorem $6.1$ of \cite{FTSID2} for the multi-dimensional case) 
\item have weakly bounded variation, i.e., they satisfy a bound of the form \eqref{eq:tvbound} in the one-dimensional case and \eqref{eq:tvbound2} (see theorem 6.6 of \cite{FJO1} and in general section 3.2 of \cite{FJO1} for the multi-dimensional case).  
\end{itemize}
Hence, under the assumption \eqref{eq:linf} that the scheme is bounded in $L^{\infty}$, the approximate Young measures, generated by the TeCNO scheme, converge to an EMV solution of \eqref{eq:mvcauchy}.

\subsubsection{Shock capturing space time Discontinuous Galerkin (DG) schemes}
Although suitable for Cartesian grids, finite difference schemes of the type \eqref{eq:fvm} are difficult to extend to unstructured grids in several space dimensions. For problems with complex domain geometry that necessitates the use of unstructured grids (triangles, tetrahedra), an alternative discretization procedure is the space-time discontinuous finite element procedure of \cite{JS1,JJS1,Bar1,HSID1}. In this procedure, the entropy variables serve as degrees of freedom and entropy stable numerical fluxes like \eqref{eq:esf1} need to be used at cell interfaces. Further stabilization terms like streamline diffusion and shock capturing terms are also necessary. In a recent paper \cite{HSID1}, it was shown that a shock capturing streamline diffusion space-time DG method satisfied a discrete entropy inequality and a suitable version of the weak BV bound \eqref{eq:tvbound}, see Theorem $3.1$ of \cite{HSID1} for the precise statements and results. Hence, this method was also shown to converge to an EMV solution in \cite{HSID1} (see Theorems $4.1$ and $4.2$ of \cite{HSID1}). We remark that the space-time DG methods are fully discrete, in contrast to semi-discrete finite difference schemes such as \eqref{eq:fvm}.

\section{Numerical Results}\label{sec:6}
Our overall goal in this section will be to compute approximate EMV solutions of \eqref{eq:mvcauchy} with atomic initial data using Algorithm \ref{alg:atomic}, as well as to investigate the stability of these solutions with respect to initial data. In Sections \ref{sec:khnonatomic} and \ref{sec:khatomic} we consider the Kelvin-Helmholtz problem \eqref{eq:khi}. In Section \ref{sec:richtmesh} we consider the  Richtmeyer-Meshkov problem; see e.g.\ \cite{GGZ99} and the references therein.

For the rest of the section, we will present numerical experiments for the two-dimensional compressible Euler equations
\begin{equation}
\label{eq:2deuler}
\frac{\partial}{\partial t}
\begin{pmatrix}
\rho \\ \rho \velx \\ \rho \vely \\ E
\end{pmatrix}
+ \frac{\partial}{\partial x_1}
\begin{pmatrix}
\rho \velx \\ \rho (\velx)^2 + p \\ \rho \velx\vely \\ (E+p)\velx
\end{pmatrix}
+ \frac{\partial}{\partial x_2}
\begin{pmatrix}
\rho \vely \\ \rho \velx \vely \\ \rho (\vely)^2 + p \\ (E+p)\vely
\end{pmatrix}
= 0.
\end{equation}
Here, the density $\rho$, velocity field $(\velx,\vely)$, pressure $p$ and total energy $E$ are related by the equation of state
$$
E = \frac{p}{\gamma -1} + \frac{\rho ((\velx)^2 + (\vely)^2)}{2}.
$$
The relevant entropy pair is given by
\[
\eta(u)=\frac{-\rho s}{\gamma - 1}, \qquad q^1(u) = \velx\eta(u), \qquad q^2(u) = \vely\eta(u).
\]
with $s = \log(p) - \gamma \log(\rho)$ being the thermodynamic entropy. The adiabatic constant $\gamma$ is set to $1.4$.

\subsection{Kelvin-Helmholtz problem: mesh refinement ($\Dx \downarrow 0$)}\label{sec:khnonatomic}
As our first numerical experiment, we consider the two-dimensional compressible Euler equations of gas dynamics \eqref{eq:2deuler} with the initial data
\begin{equation}
\label{eq:kh}
u_0(x, \omega) = \begin{cases}
u_L & \text{if } I_1 < x_2 < I_2 \\
u_R & \text{if } x_2 \leq I_1 \text{ or } x_2 \geq I_2,
\end{cases}
\qquad x\in[0,1]^2
\end{equation}
with $\rho_L = 2$, $\rho_R = 1$, $\velx_L = -0.5$, $\velx_R = 0.5$, $\vely_L=\vely_R=0$ and $p_L=p_R=2.5$. 

The computational domain is $[0,1]^2$ and we consider periodic boundary conditions. Furthermore, the interface profiles
\[
I_j = I_j(x,\omega) := J_j + \amp Y_j(x,\omega), \qquad j=1,2
\]
are chosen to be small perturbations around $J_1:=0.25$ and $J_2:=0.75$, respectively, with
\[
Y_j(x,\omega) = \sum_{n=1}^m a_j^n(\omega) \cos\left(b_j^n(\omega) + 2n\pi x_1\right), \qquad j=1,2.
\]
Here, $a_j^n = a_j^n(\omega) \in [0,1]$ and $b_j^n = b_j^n(\omega)\in[-\pi,\pi]$, $i=1,2$, $n=1,\dots,m$ are randomly chosen numbers. The coefficients $a_j^n$ have been normalized such that $\sum_{n=1}^m a_j^n = 1$ to guarantee that $|I_j(x,\omega) - J_j| \leq \amp$ for $j=1,2$. We set $m=10$.

Observe that by making $\epsilon$ small, this $\omega$-ensemble of initial data lies inside an arbitrarily small ball centered at $u_0$. Indeed, it is readily checked that measured in, say, the \ $L^p([0,1]^2)$-norm, every sample $u_0(\cdot,\omega)$ is $O(\eps^{1/p})$ away from the unperturbed steady state in \eqref{eq:khi}. 

A representative (single realization with fixed $\omega$) initial datum for the density in shown in Figure \ref{fig:kh2d_ini} (left). We observe that the resulting measure valued Cauchy problem involves a random perturbation of the interfaces between the two streams (jets). This should be contrasted with initial value problem \eqref{eq:khi},  \eqref{eq:sodpinit}, where the amplitude was randomly perturbed. We note that the law of the above initial datum can readily be written down and serves as the initial Young measure in the measure valued Cauchy problem \eqref{eq:mvcauchy}. Observe that this Young measure is not atomic for some points in the domain.

\begin{figure}[ht]
\includegraphics[width=0.5\linewidth]{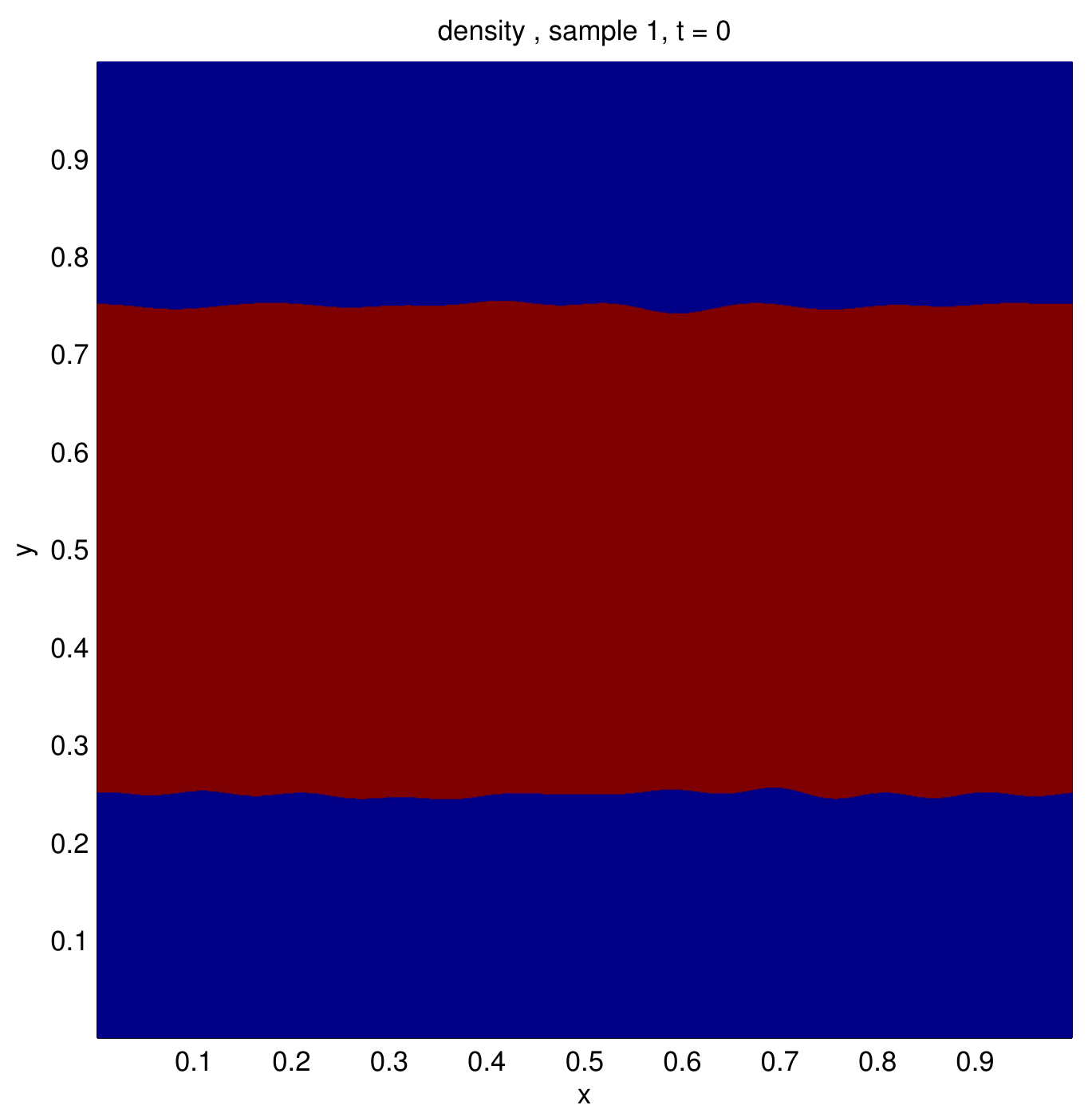}\hfill
\includegraphics[width=0.5\linewidth]{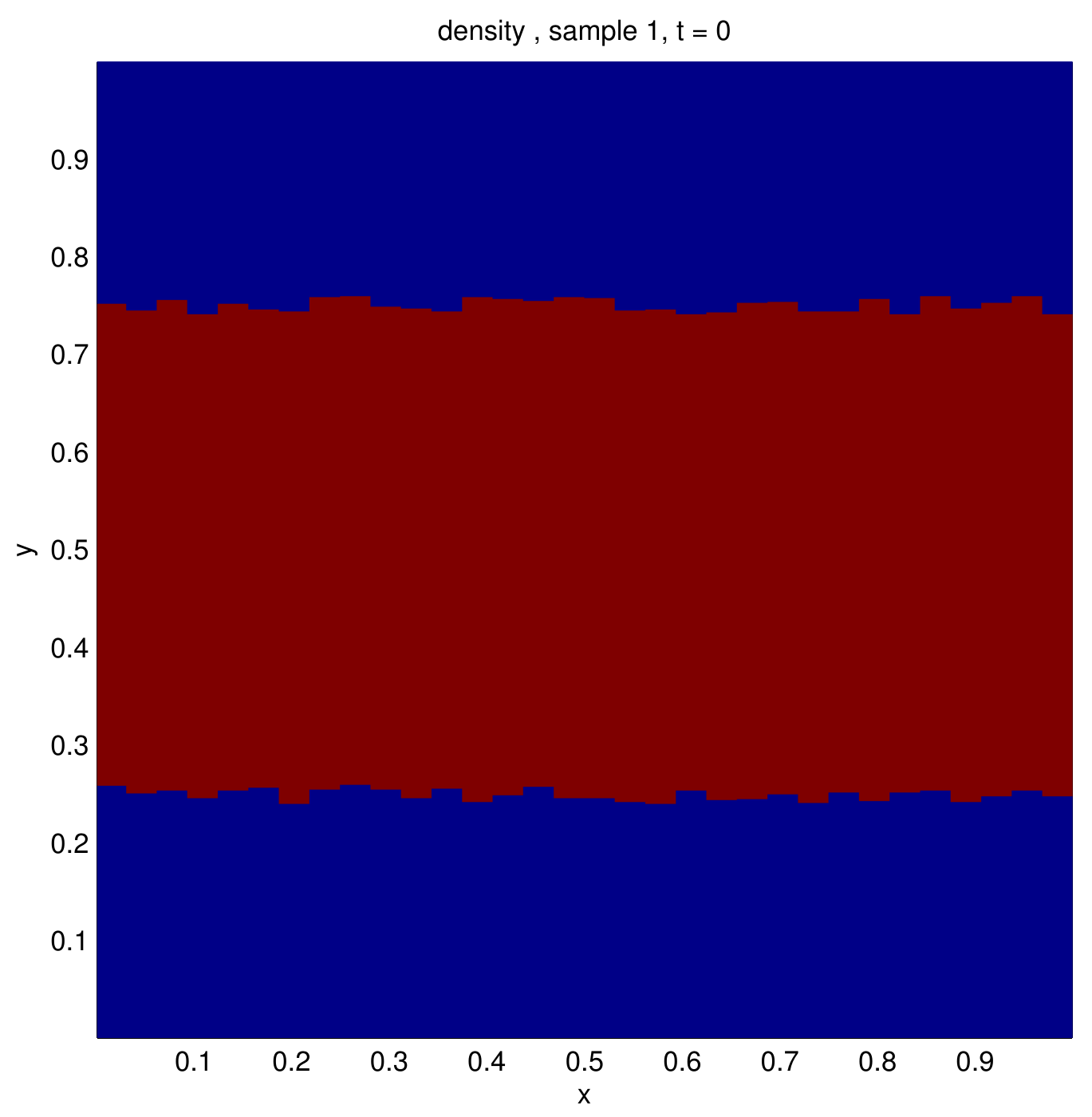}
\caption{Representative initial data for the Kelvin-Helmholtz problem: amplitude perturbation (left) and phase perturbation (right) with $\amp = 0.01$.}
\label{fig:kh2d_ini}
\end{figure}

\subsubsection{Lack of sample convergence}
We approximate the above MV Cauchy problem with the second-order entropy stable TeCNO2 scheme of \cite{FTSID2}. In Figure \ref{fig:3} we show the density at time $t = 2$ for a single sample, i.e, for a fixed $\omega \in \Omega$, at different grid resolutions, ranging from $128^2$ points to $1024^2$ points. The figure suggests that the approximate solutions do not seem to converge as the mesh is refined. In particular, finer and finer scale structures are formed as the mesh is refined, as already seen in Figure \ref{fig:kh-density1}. To further verify this lack of convergence, we compute the $L^1$ difference of the approximate solutions at successive mesh levels \eqref{eq:chyrates} and present the results in Figure \ref{fig:4}. We observe that this difference does not go to zero, suggesting that the approximate solutions may not converge as the mesh is refined.
\begin{figure}
\centering
\subfigure[$128^2$]{\includegraphics[width=0.45\linewidth]{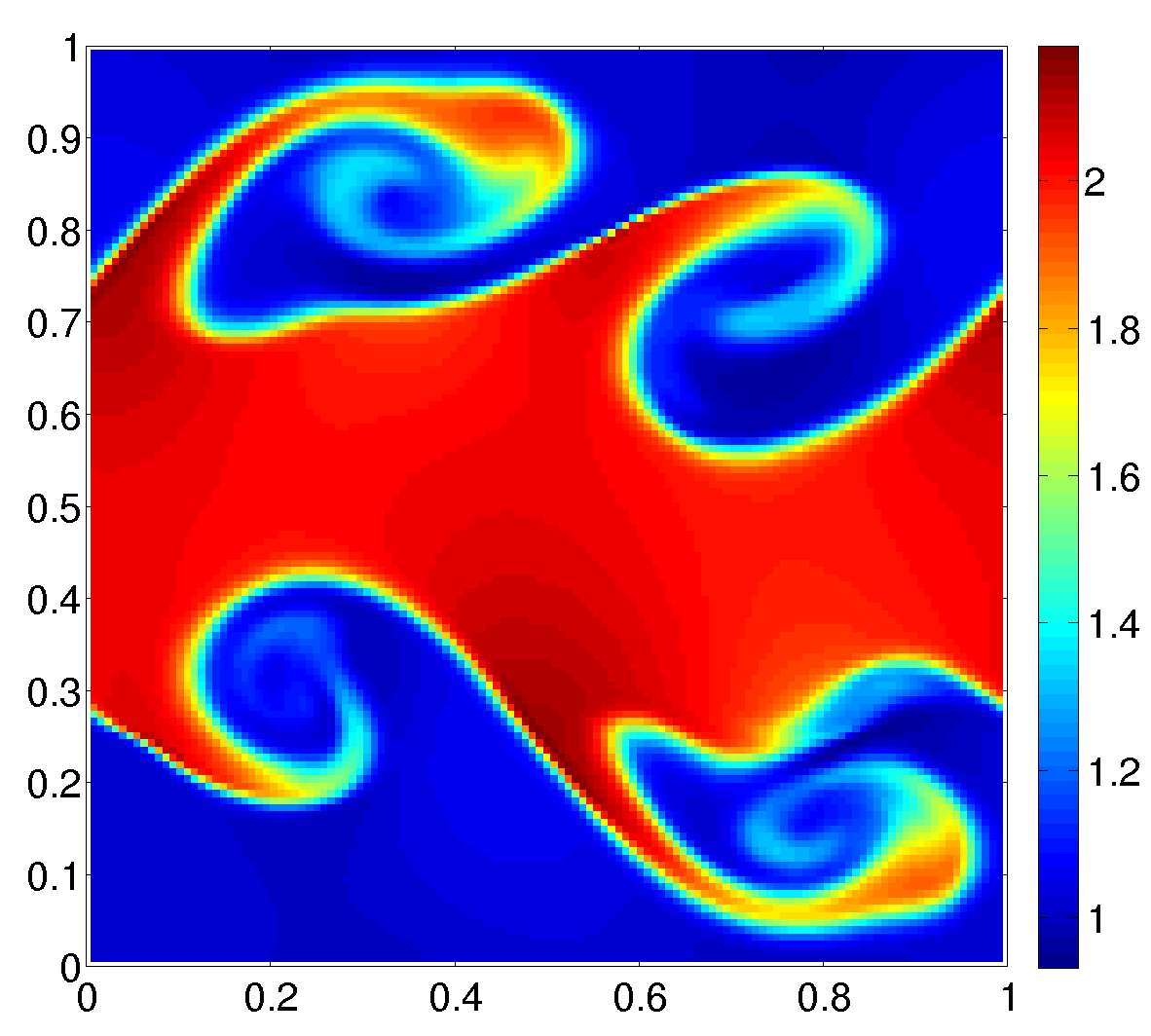}}
\subfigure[$256^2$]{\includegraphics[width=0.45\linewidth]{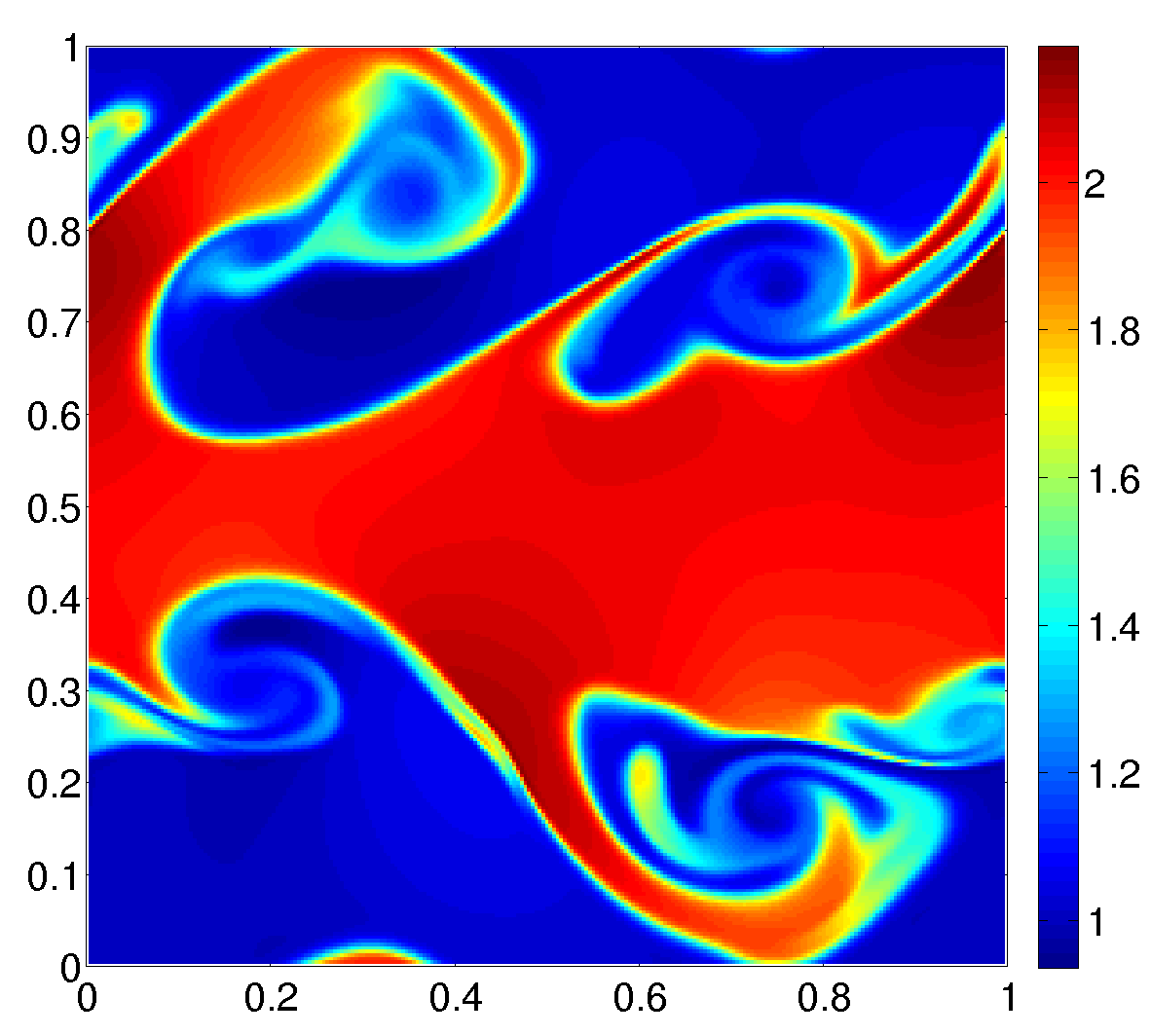}} \\
\subfigure[$512^2$]{\includegraphics[width=0.45\linewidth]{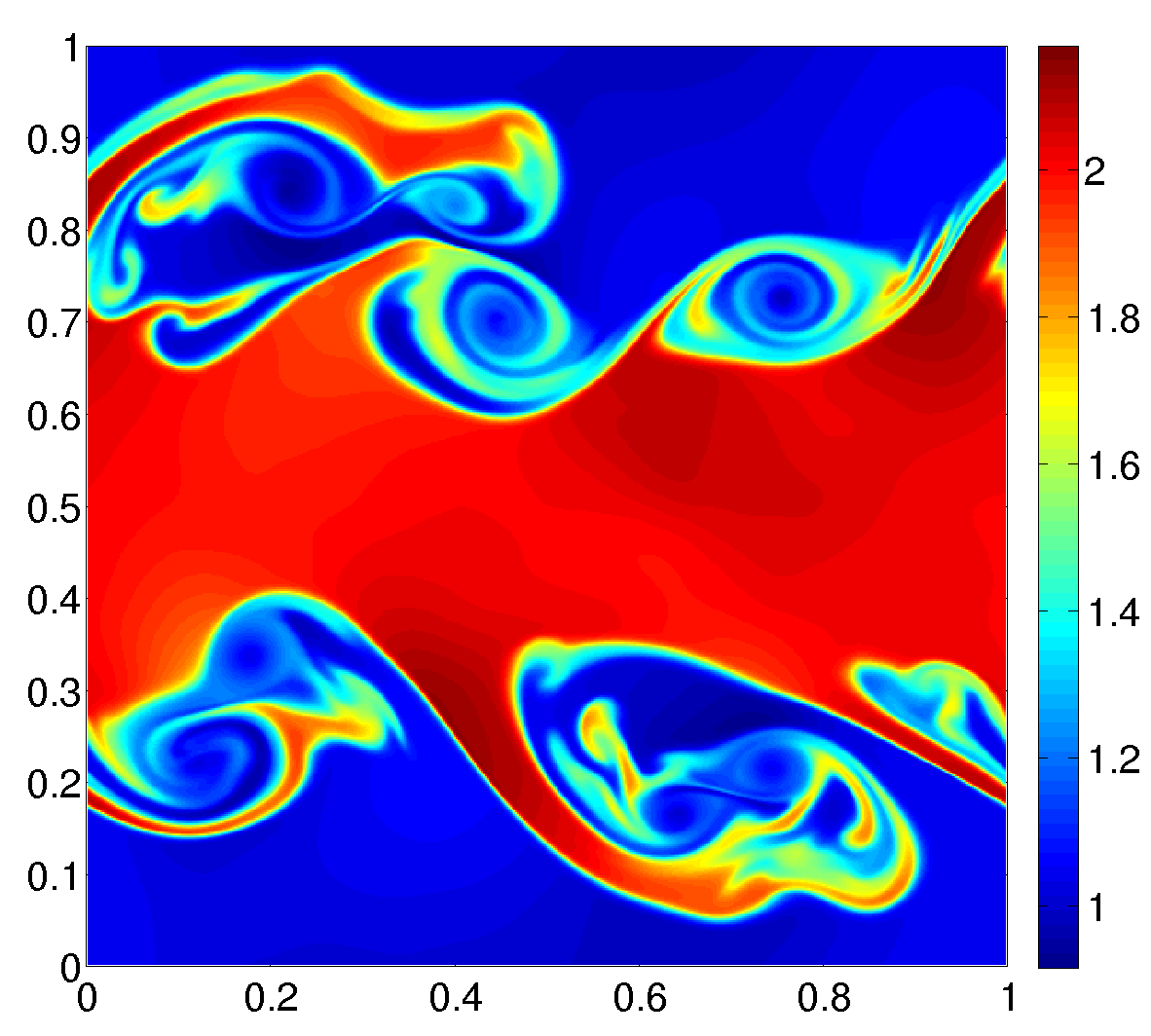}}
\subfigure[$1024^2$]{\includegraphics[width=0.45\linewidth]{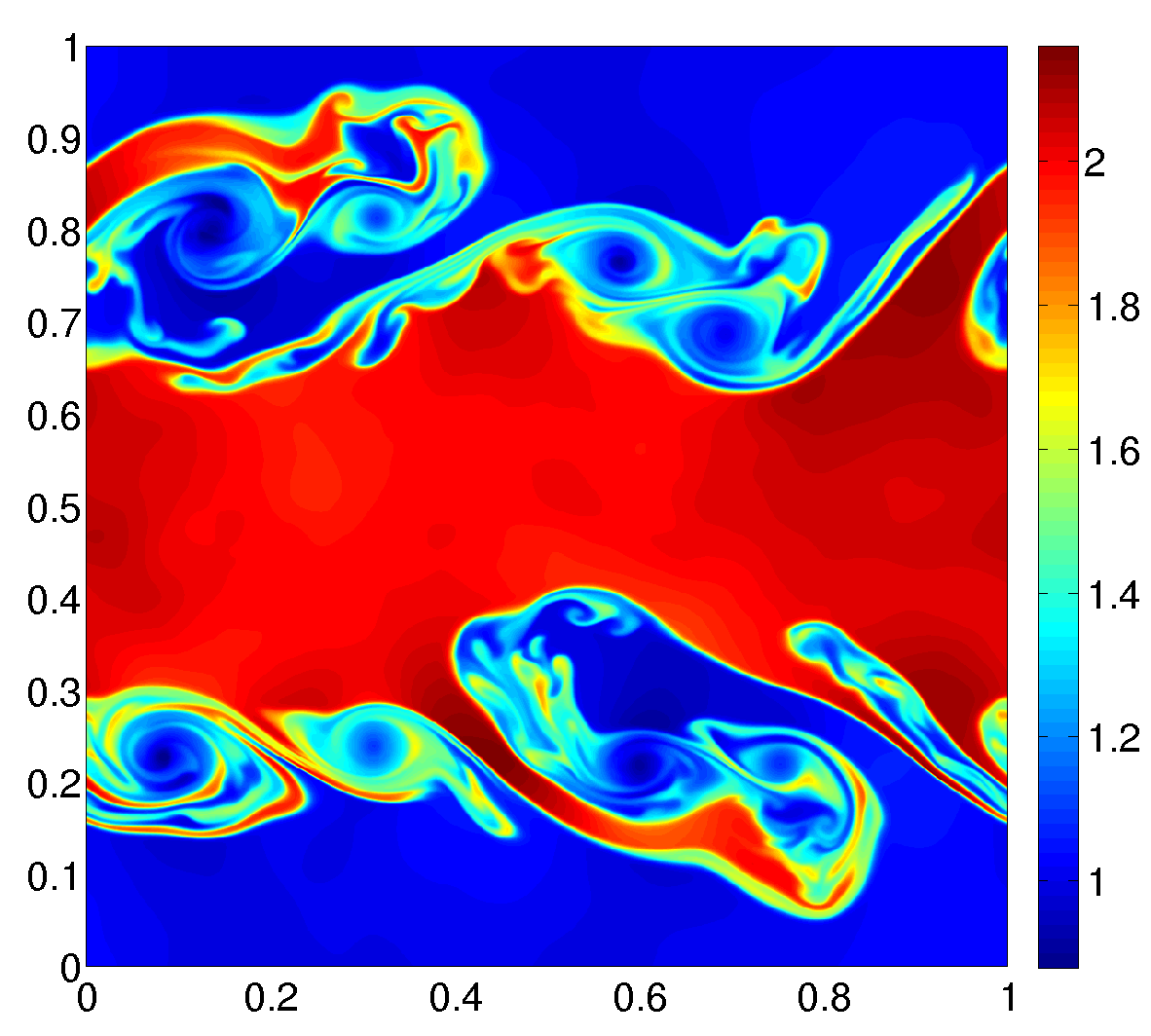}}
\caption{Approximate density for the Euler equations \eqref{eq:2deuler} with initial data \eqref{eq:kh}, $\amp = 0.01$ and for a fixed $\omega$ (single sample), computed with the second-order TeCNO2 scheme of \cite{FTSID2}, at time $t=2$ at different mesh resolutions.}
\label{fig:3}
\end{figure}
\begin{figure}
\centering
\includegraphics[width=6cm]{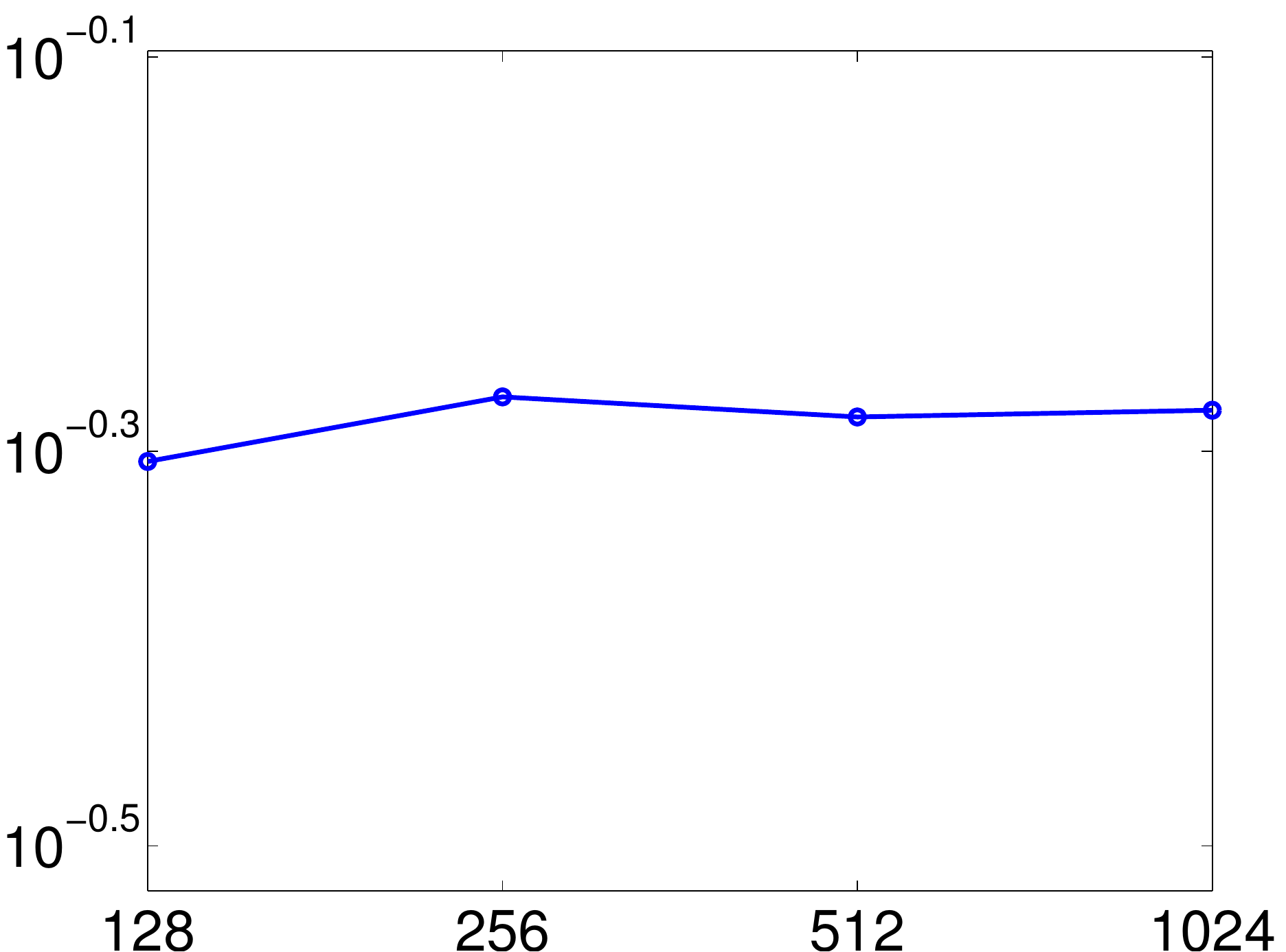}
\caption{The Cauchy rates \eqref{eq:chyrates} at $t=2$ for the density ($y$-axis) for a single sample of the Kelvin-Helmholtz problem, vs.\ different mesh resolutions ($x$-axis)}
\protect \label{fig:4}
\end{figure}

\subsubsection{Convergence of the mean and variance}
The lack of convergence of the numerical schemes for single samples is not unexpected, given the results already mentioned in the introduction. Next, we will compute statistical quantities of the interest for this problem. First, we compute the Monte-Carlo approximation of the mean \eqref{eq:mc1}, denoted by $\bar{u}^{\Dx}(x,t)$, at every point $(x,t)$ in the computational domain. This sample mean of the density, computed with $M=400$ samples and the second-order TeCNO2 scheme is presented in Figure \ref{fig:5} for a set of successively refined grid resolutions. The figure clearly suggests that the sample mean converges as the mesh is refined. This stands in stark contrast with the lack of convergence, at the level of single samples, as shown in Figure \ref{fig:kh-density1} and Figure \ref{fig:3}. Furthermore, Figure \ref{fig:5} also reveals that small scale structures, present in single sample (realization) computations, are indeed smeared or averaged out in the mean. This convergence of the mean is further quantified by computing the $L^1$ difference of the mean,
\begin{equation}
\label{eq:cr2}
\|\bar{u}^{\Dx} - \bar{u}^{\Dx/2}\|_{L^1([0,1]^2)}.
\end{equation}
and plotting the results in Figure \ref{fig:68}(a). As predicted by the theory presented in Theorems \ref{thm:convmv2} and \ref{thm:alphaconv}, these results confirm that the sequence of approximate means form a Cauchy sequence, and hence converge to a limit as the mesh is refined. Similar convergence results were also observed for the means of the other conserved variables, namely momentum and total energy (not shown here). Furthermore, Figure \ref{fig:5} also suggests that the mean is varying in the $y$-direction only. This is completely consistent with the symmetries of the equations, of the initial data and the fact that periodic boundary conditions are employed. This is also in sharp contrast with the situation for single realizations where there is strong variation along both directions (see Figure \ref{fig:3}).
\begin{figure}
\centering
\subfigure[$128^2$]{\includegraphics[width=0.45\linewidth]{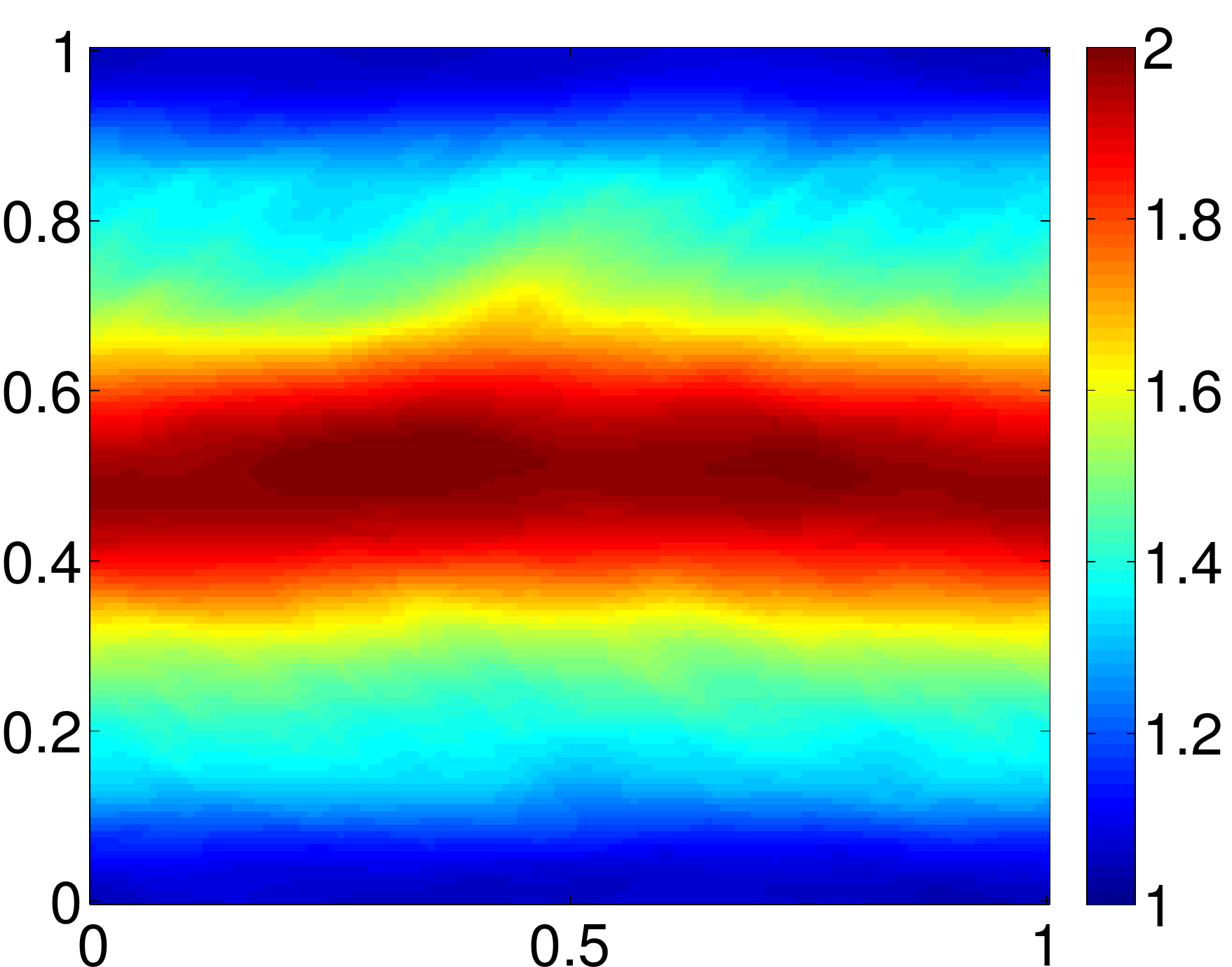}}
\subfigure[$256^2$]{\includegraphics[width=0.45\linewidth]{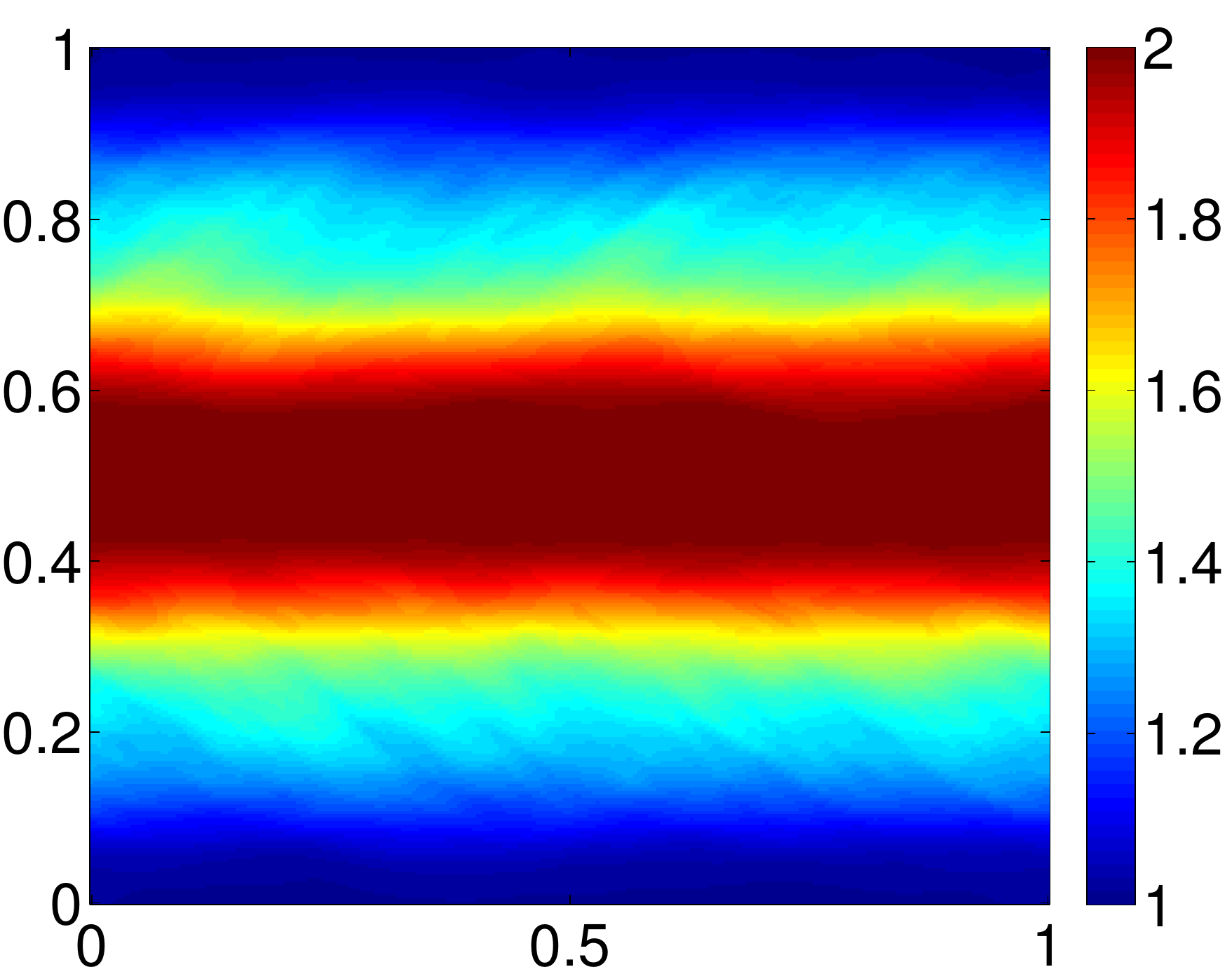}} \\
\subfigure[$512^2$]{\includegraphics[width=0.45\linewidth]{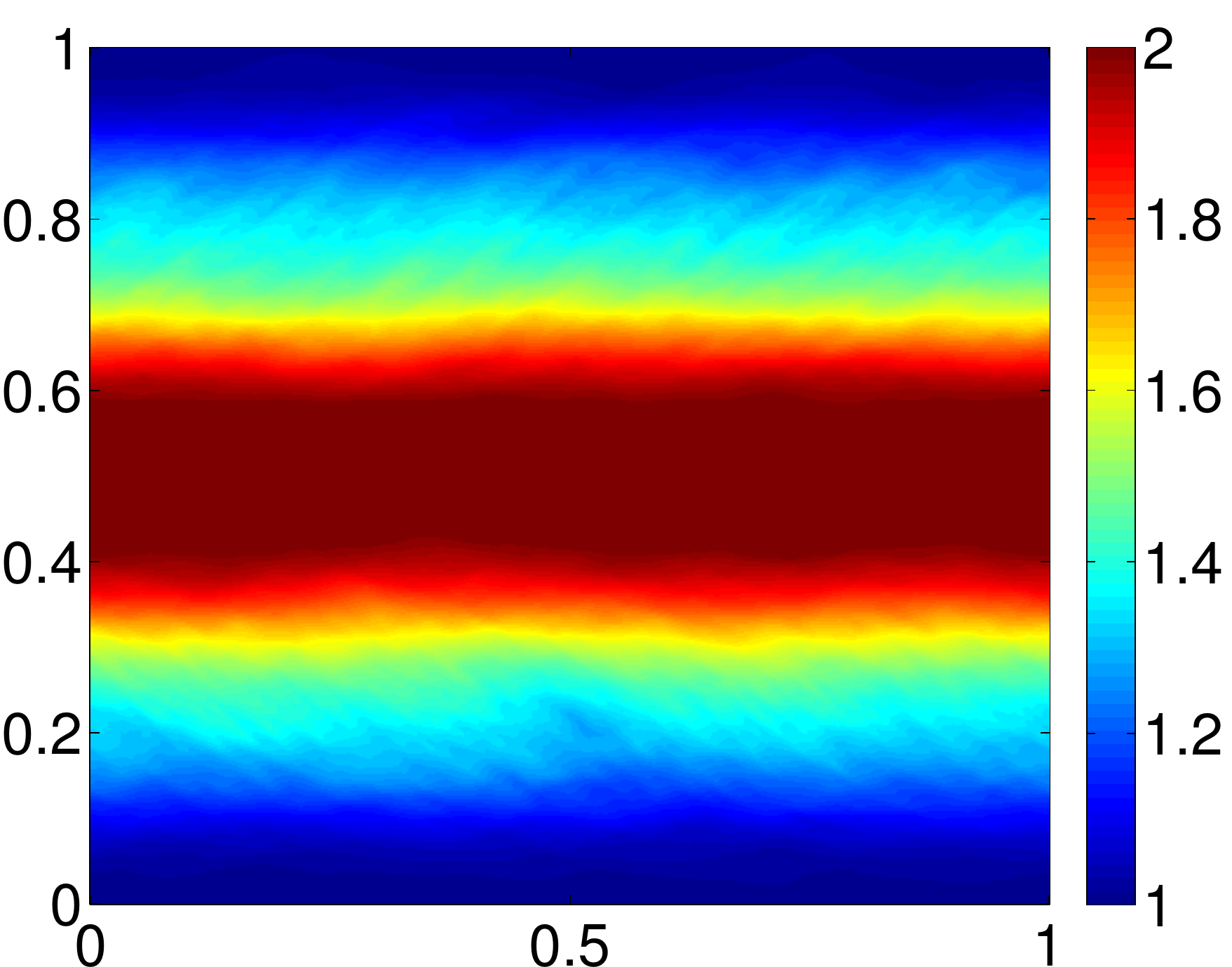}}
\subfigure[$1024^2$]{\includegraphics[width=0.45\linewidth]{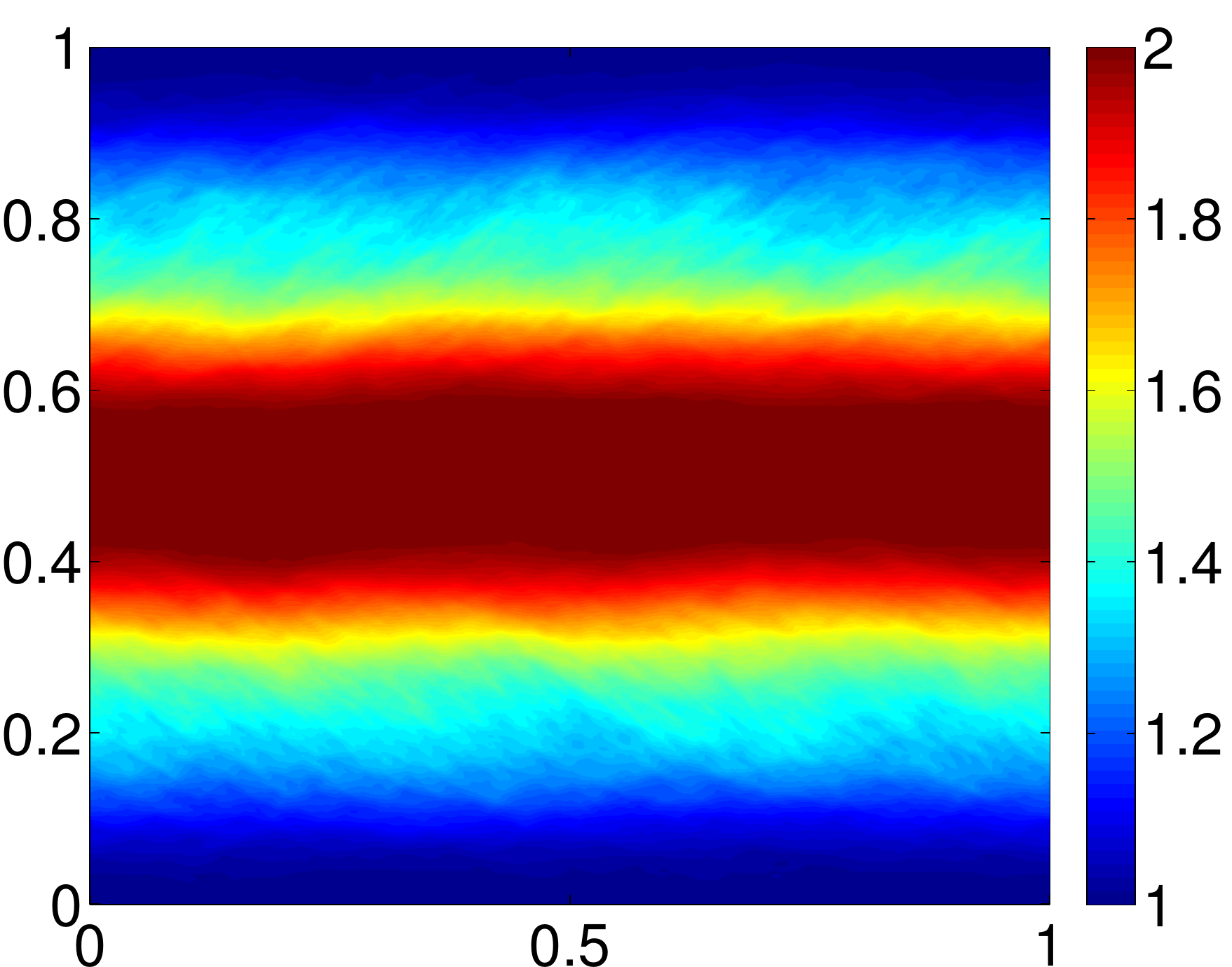}}
\caption{Approximate sample means of the density for the Kelvin-Helmholtz problem \eqref{eq:kh} at time $t=2$ and different mesh resolutions. All results are with $400$ Monte Carlo samples.}
\label{fig:5}
\end{figure}
\begin{figure}
\centering
\subfigure[Mean]{\includegraphics[width=6cm]{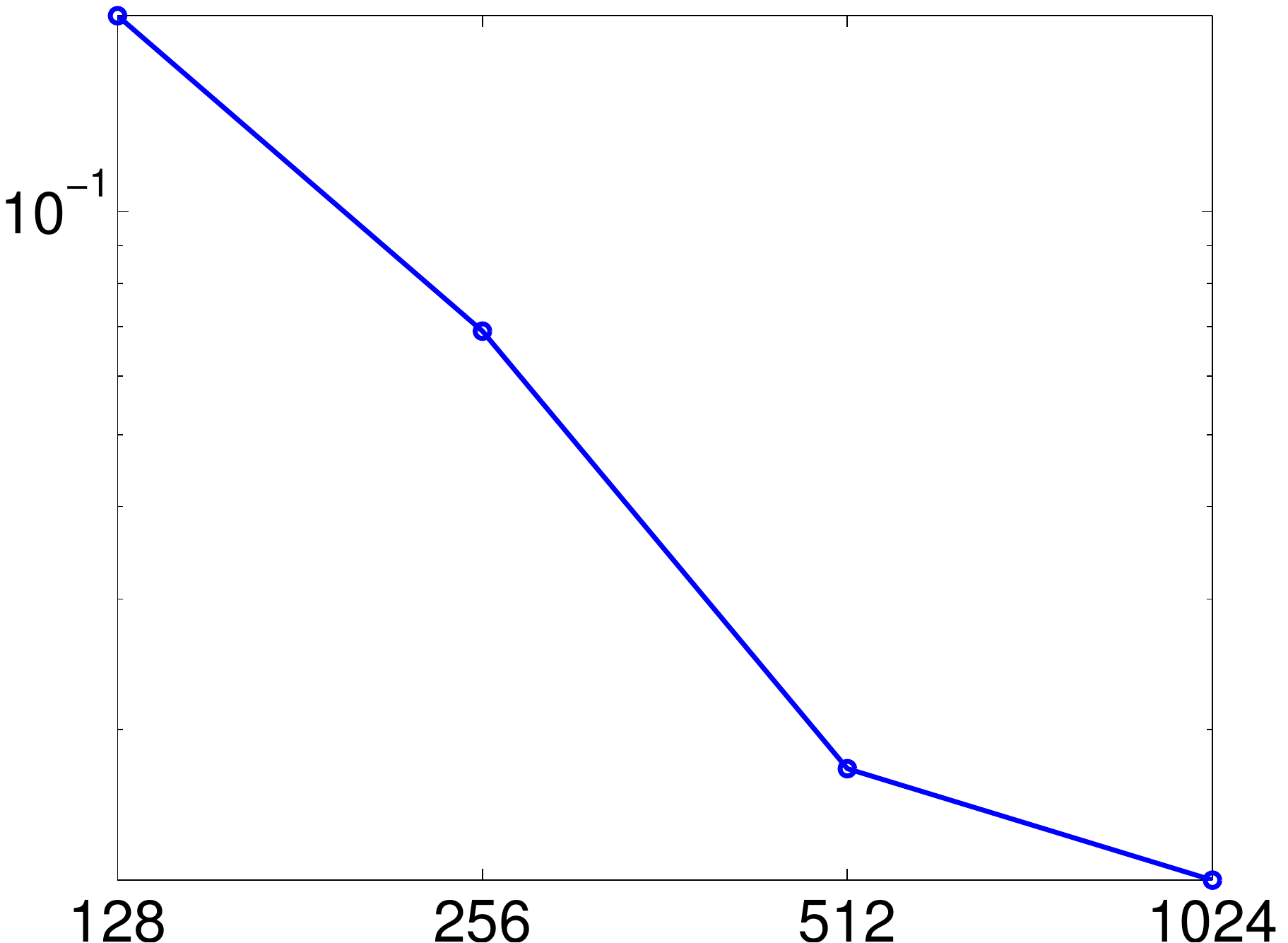}}
\subfigure[Variance]{\includegraphics[width=6cm]{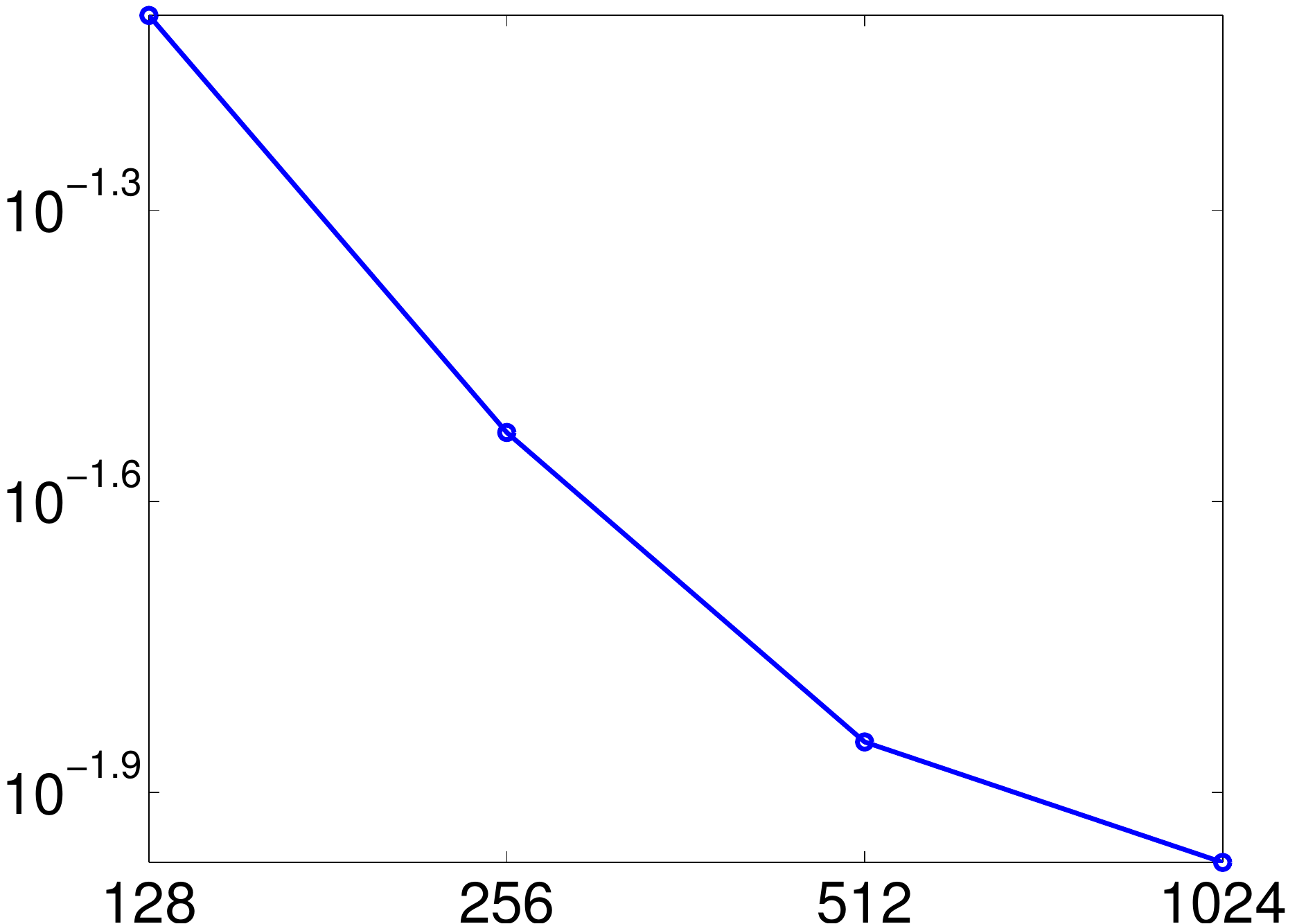}}
\caption{Cauchy rates \eqref{eq:cr2} for the sample mean and variance of the density ($y$-axis) vs.\ mesh resolution ($x$-axis) for the Kelvin-Helmholtz problem \eqref{eq:kh}.}
\protect \label{fig:68}
\end{figure}

Next, we compute the sample variance and show the results in Figure \ref{fig:7}. The results suggest that the variance also converges with grid resolution. This convergence is also demonstrated quantitatively by plotting the $L^1$ differences of the variance at successive levels of resolution, shown in Figure \ref{fig:68}(b). Again, the figure suggests that the sequence forms a Cauchy sequence, and hence is convergent. Furthermore, the variance itself shows no small scale features, even on very fine mesh resolutions (see Figure \ref{fig:7}). This figure also reveals that the variance is higher near the initial mixing layer.
\begin{figure}
\centering
\subfigure[$128^2$]{\includegraphics[width=0.45\linewidth]{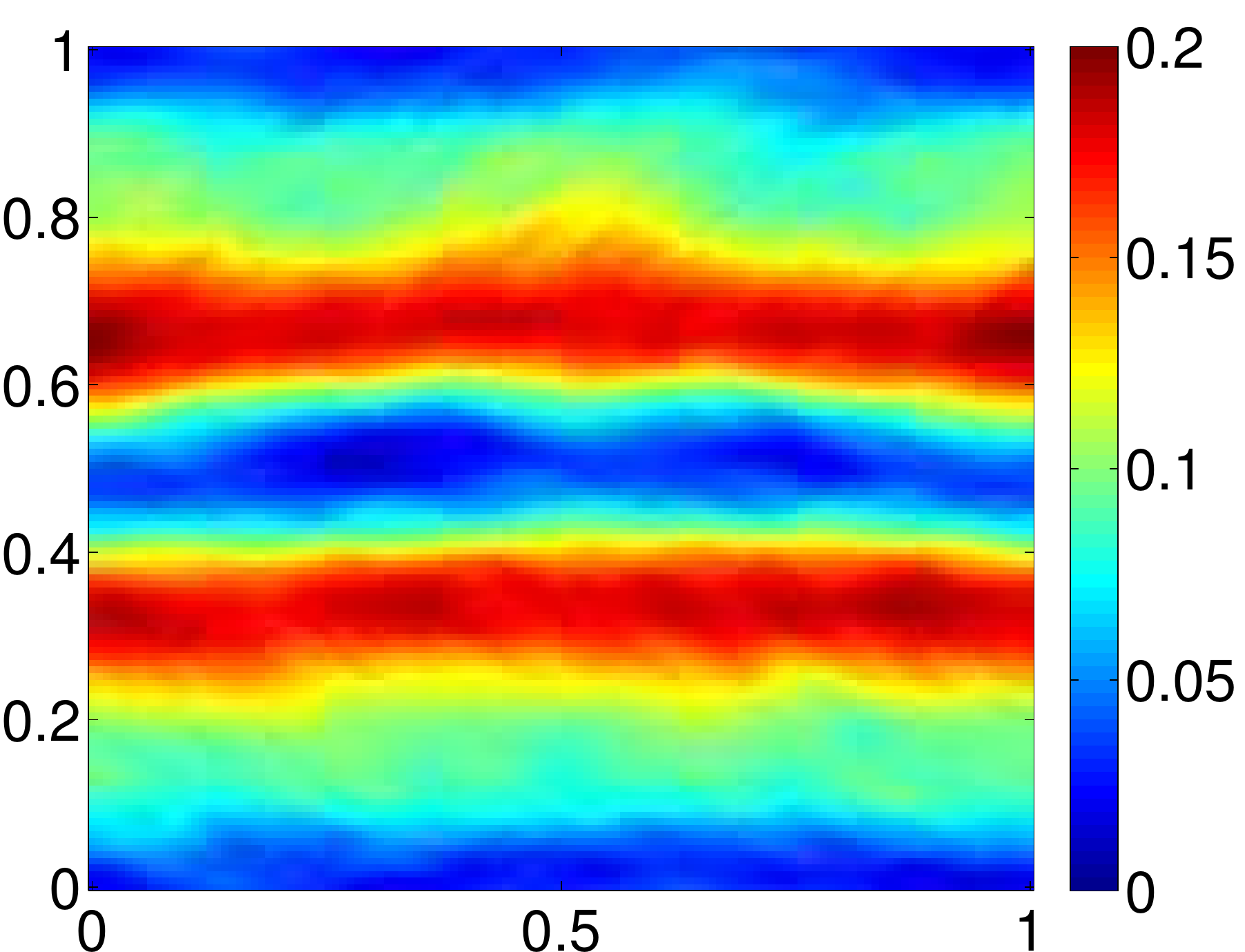}}
\subfigure[$256^2$]{\includegraphics[width=0.45\linewidth]{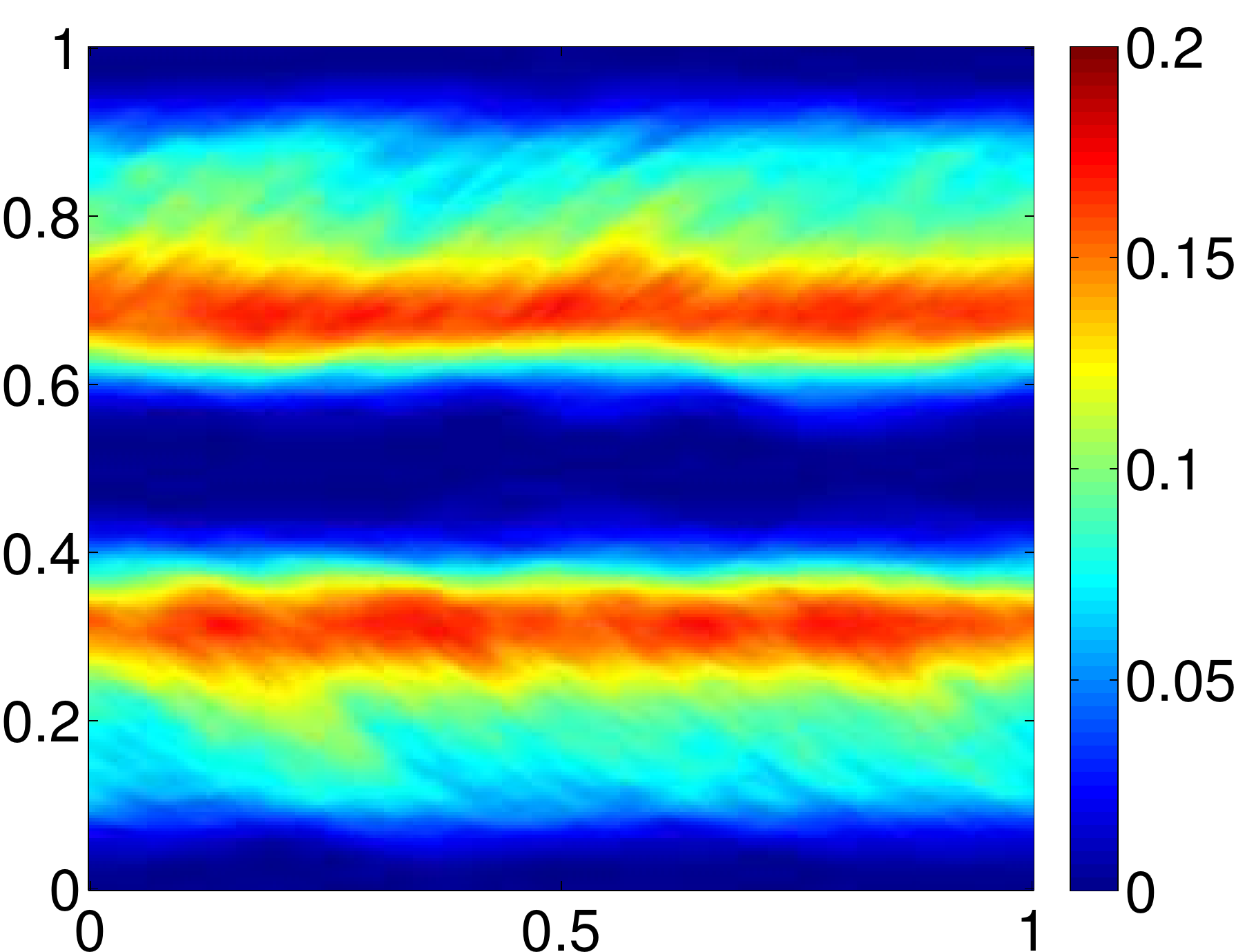}} \\
\subfigure[$512^2$]{\includegraphics[width=0.45\linewidth]{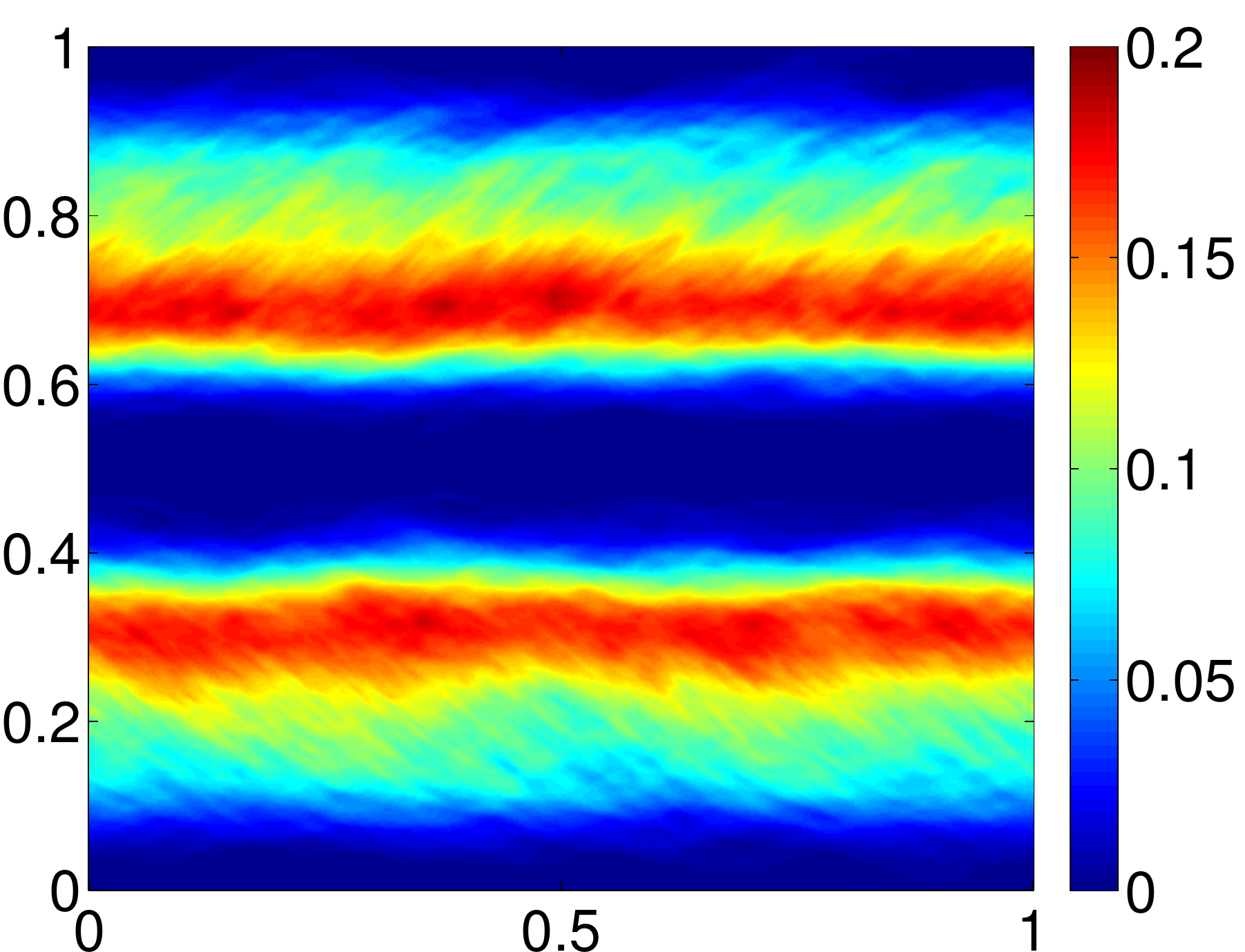}}
\subfigure[$1024^2$]{\includegraphics[width=0.45\linewidth]{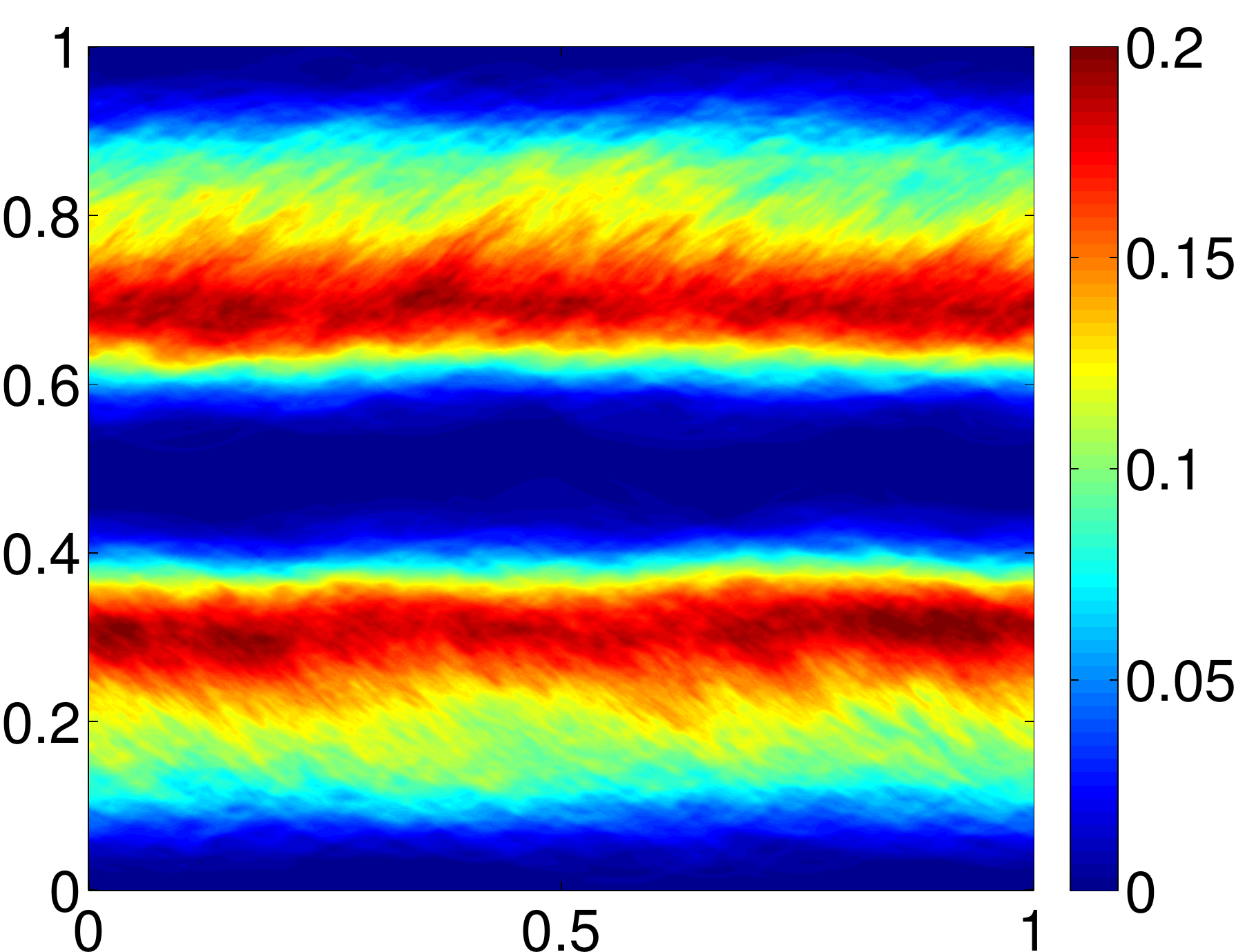}}
\caption{Approximate sample variances of the density for the Kelvin-Helmholtz problem \eqref{eq:kh} at time $t=2$ and different mesh resolutions. All results are with $400$ Monte Carlo samples.}
\label{fig:7}
\end{figure}
\begin{figure}[htbp] 
\centering
\includegraphics[width=6cm]{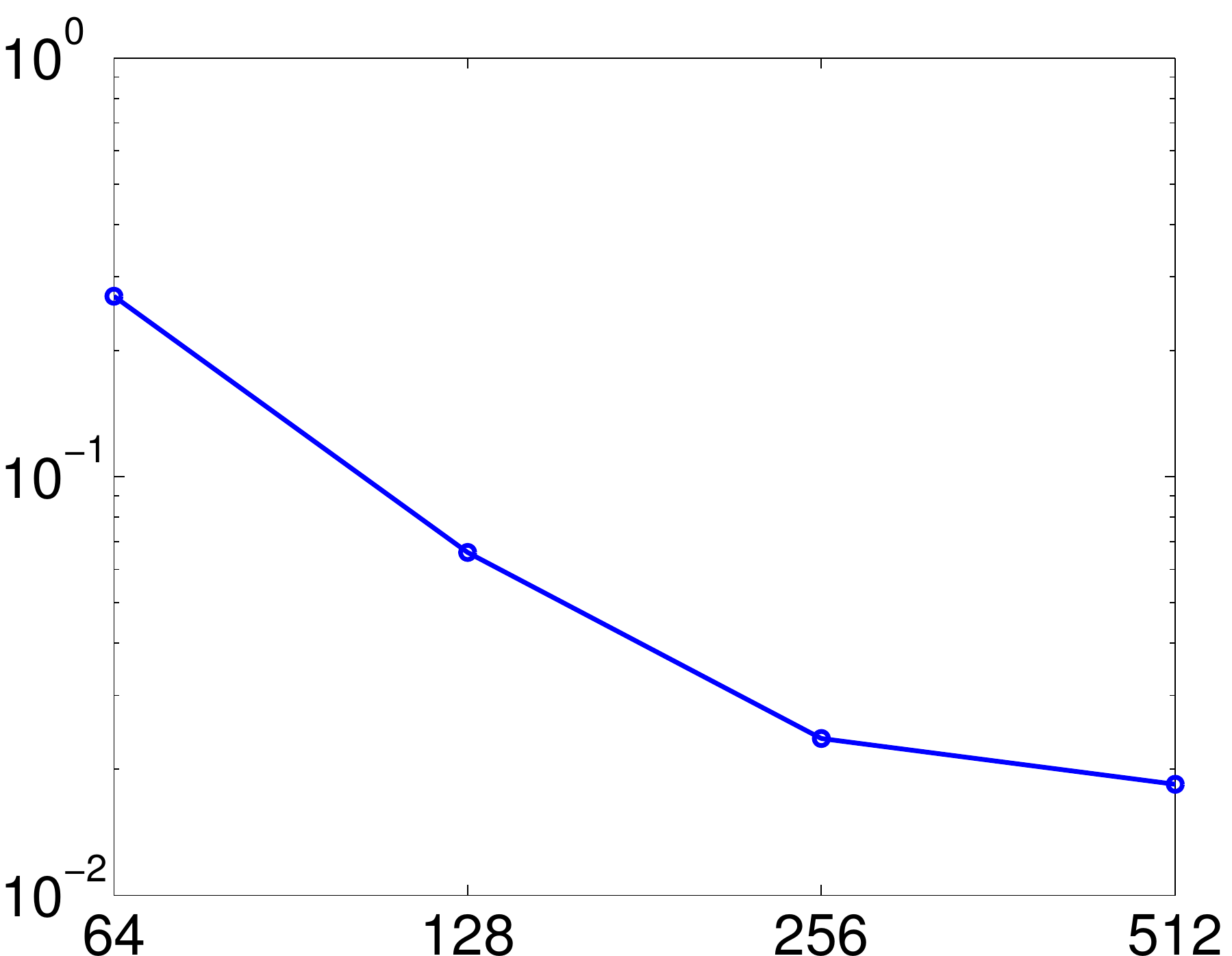}
\caption{Cauchy rates in the Wasserstein distance \eqref{eq:wdist10} at time $t=2$ for the density ($y$-axis) with respect to different mesh resolutions ($x$-axis), for the Kelvin-Helmholtz problem \eqref{eq:kh}.}
\label{fig:10}
\end{figure}
\subsubsection{Strong convergence to an EMV solution}
Convergence of the mean and variance (as well as higher moments (not shown here)) confirm the weak* convergence predicted by (the multi-dimensional version of) Theorems \ref{thm:convmv} and \ref{thm:alphaconv}. Note that the convergence illustrated in Figure \ref{fig:68} is in $L^1$ of the spatial domain. Next, we test strong convergence of the numerical approximations by computing the Wasserstein distance between two successive mesh resolutions:
\begin{equation}
\label{eq:wdist9}
W_1\left(\nu^{\Dx}_{(x,t)},\nu^{\Dx/2}_{(x,t)}\right)
\end{equation}
(see Appendix \ref{app:wasscomp}). In Figure \ref{fig:10} we show the $L^1$-norm of the Wasserstein distance between successive mesh resolutions
\begin{equation}
\label{eq:wdist10}
\left\|W_1\left(\nu^{\Dx}_{(\cdot,t)},\ \nu^{\Dx/2}_{(\cdot,t)}\right)\right\|_{L^1([0,1]^2)}
\end{equation}
at time $t=2$; recall that this is the quantity appearing in \eqref{eq:sconv}. The figure suggests that this difference between successive mesh resolutions converges to zero. Hence, the approximate Young measures converge strongly in both space-time as well as phase space to a limit Young measure.

In Figure \ref{fig:11} we show the pointwise difference in Wasserstein distance \eqref{eq:wdist10} between two successive mesh levels. The figure reveals that this distance decreases as the mesh is refined. Moreover, we see that the Wasserstein distance between approximate Young measures at successive resolutions is concentrated at the interface mixing layers. This is to be expected as the variance is also concentrated along these layers (cf. the variance plots in Figure \ref{fig:7}).
\begin{figure}
\centering
\subfigure[$W_1\left(\nu^{256}_{(x,t)},\nu^{512}_{(x,t)}\right)$]{\includegraphics[width=0.45\linewidth]{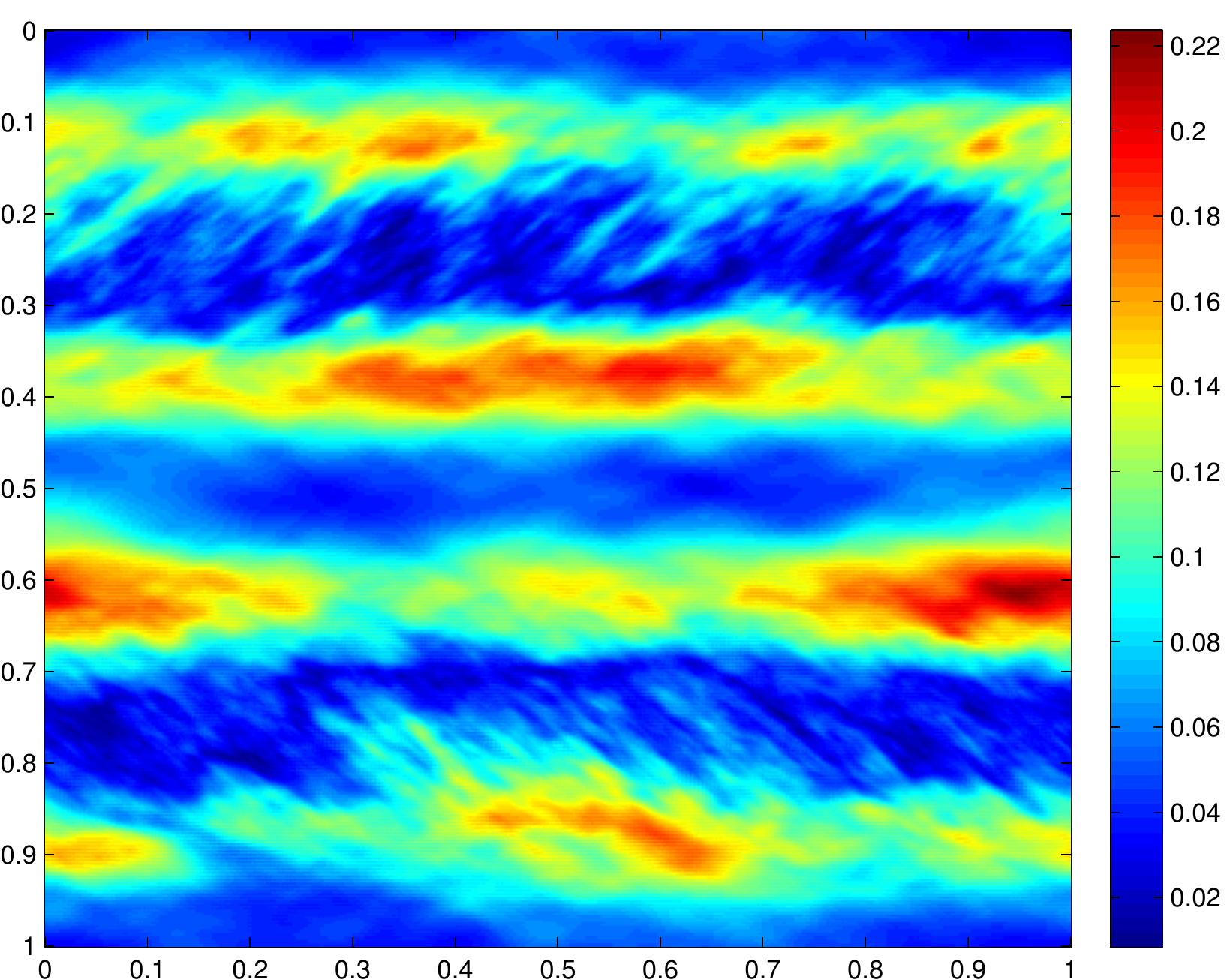}}
\subfigure[$W_1\left(\nu^{512}_{(x,t)},\nu^{1024}_{(x,t)}\right)$]{\includegraphics[width=0.45\linewidth]{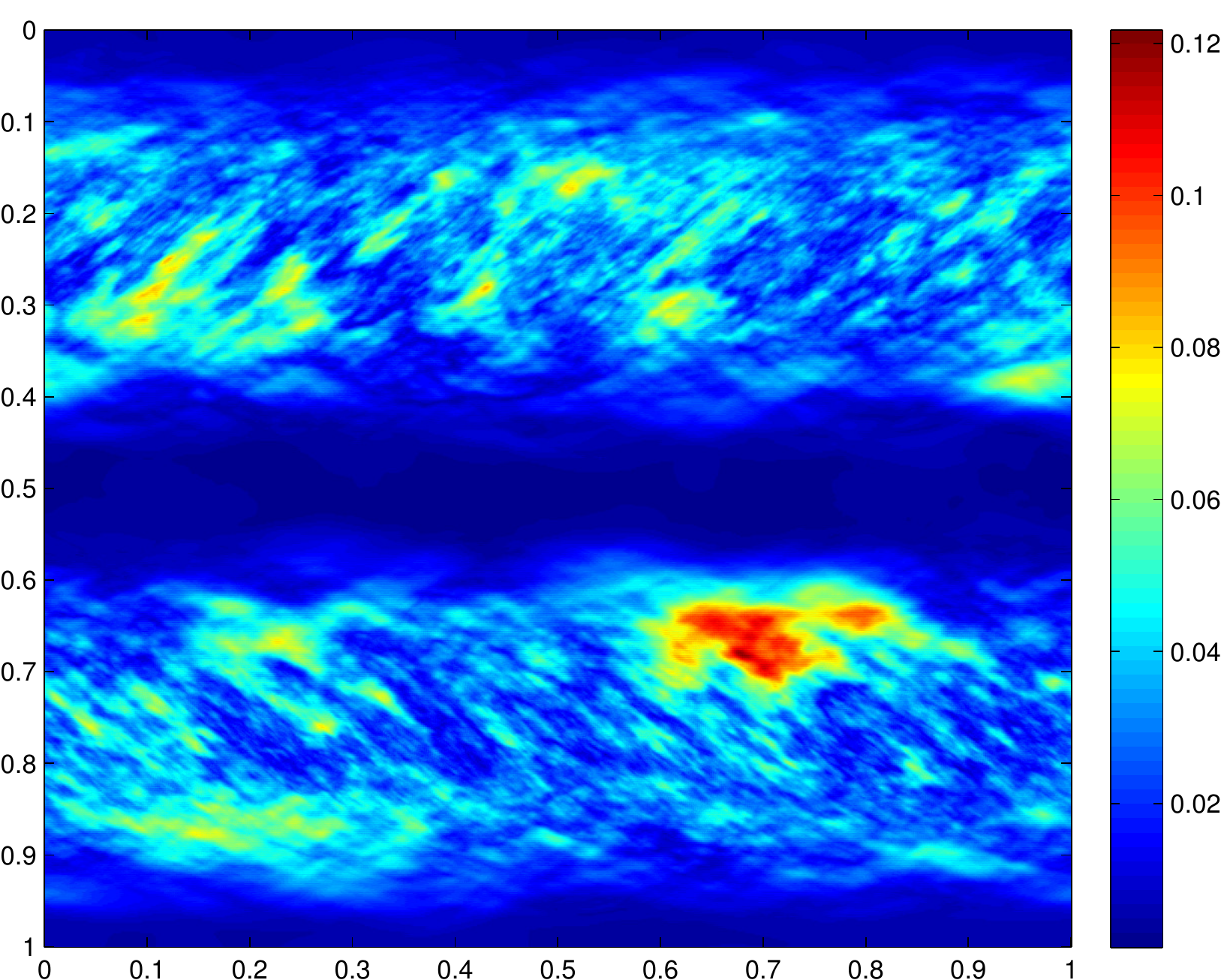}}
\caption{Wasserstein distances between the approximate Young measure (density) \eqref{eq:wdist9} at successive mesh resolutions, at time $t=2$.}
\protect \label{fig:11}
\end{figure}

\subsection{Kelvin-Helmholtz: vanishing variance around atomic initial data ($\amp\downarrow 0$)}\label{sec:khatomic}
Our aim is to compute the entropy measure-valued solutions of the two-dimensional Euler equations  with atomic initial measure, concentrated on the Kelvin-Helmholtz data \eqref{eq:khi}. We utilize Algorithm \ref{alg:atomic} for this purpose and consider the perturbed initial data \eqref{eq:kh}. Observe that this perturbed initial data converges strongly (cf.\ \eqref{eq:sconv}) to the initial data \eqref{eq:khi}  as $\amp \rightarrow 0$. Following Algorithm \ref{alg:atomic}, we wish to study the limit behavior of approximate solutions $\nu^{\Dx,\amp}$ as $\amp \to 0$. To this end, we retain the same set-up as the previous numerical experiment and compute approximate solutions using the TeCNO2 scheme of \cite{FTSID2} at a very fine mesh resolution of $1024^2$ points for different values of $\amp$.

Results \emph{for a single sample} at time $t=2$ and different $\amp$'s are presented in Figure \ref{fig:12}. The figures indicate that there is no convergence as $\amp \rightarrow 0$. The spread of the mixing region seems to remain large even when the perturbation parameter is reduced. This lack of convergence is further quantified in Figure \ref{fig:13}, where we plot the $L^1$ difference of the approximate density for successively reduced values of $\amp$. This difference remains large even when $\amp$ is reduced by an order of magnitude.
\begin{figure}
\centering
\subfigure[$\amp = 2\times 10^{-2}$]{\includegraphics[width=0.45\linewidth]{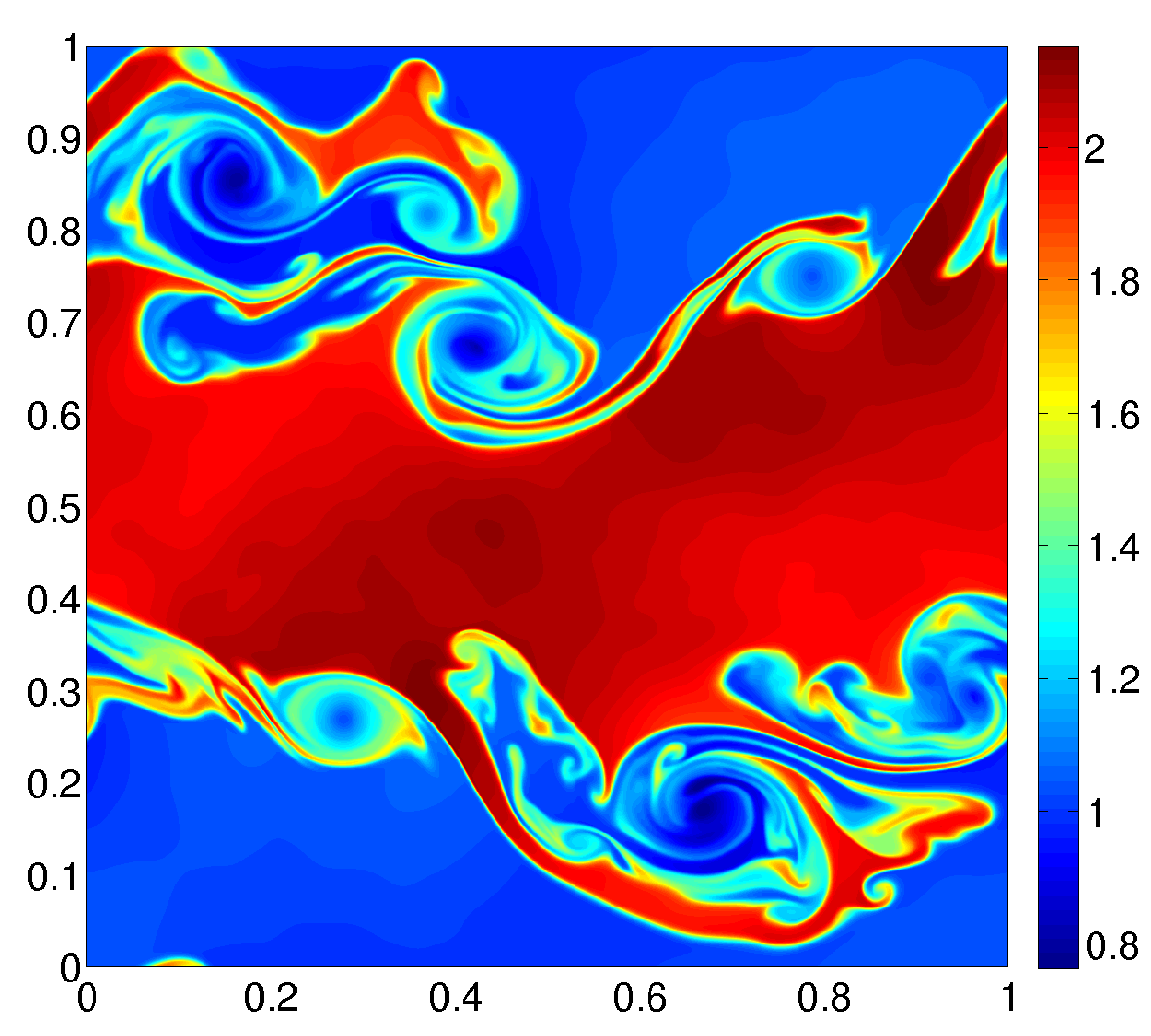}}
\subfigure[$\amp = 10^{-2}$]{\includegraphics[width=0.45\linewidth]{kh-1024-eps=1e-2_video4.png}} \\
\subfigure[$\amp = 5\times 10^{-3}$]{\includegraphics[width=0.45\linewidth]{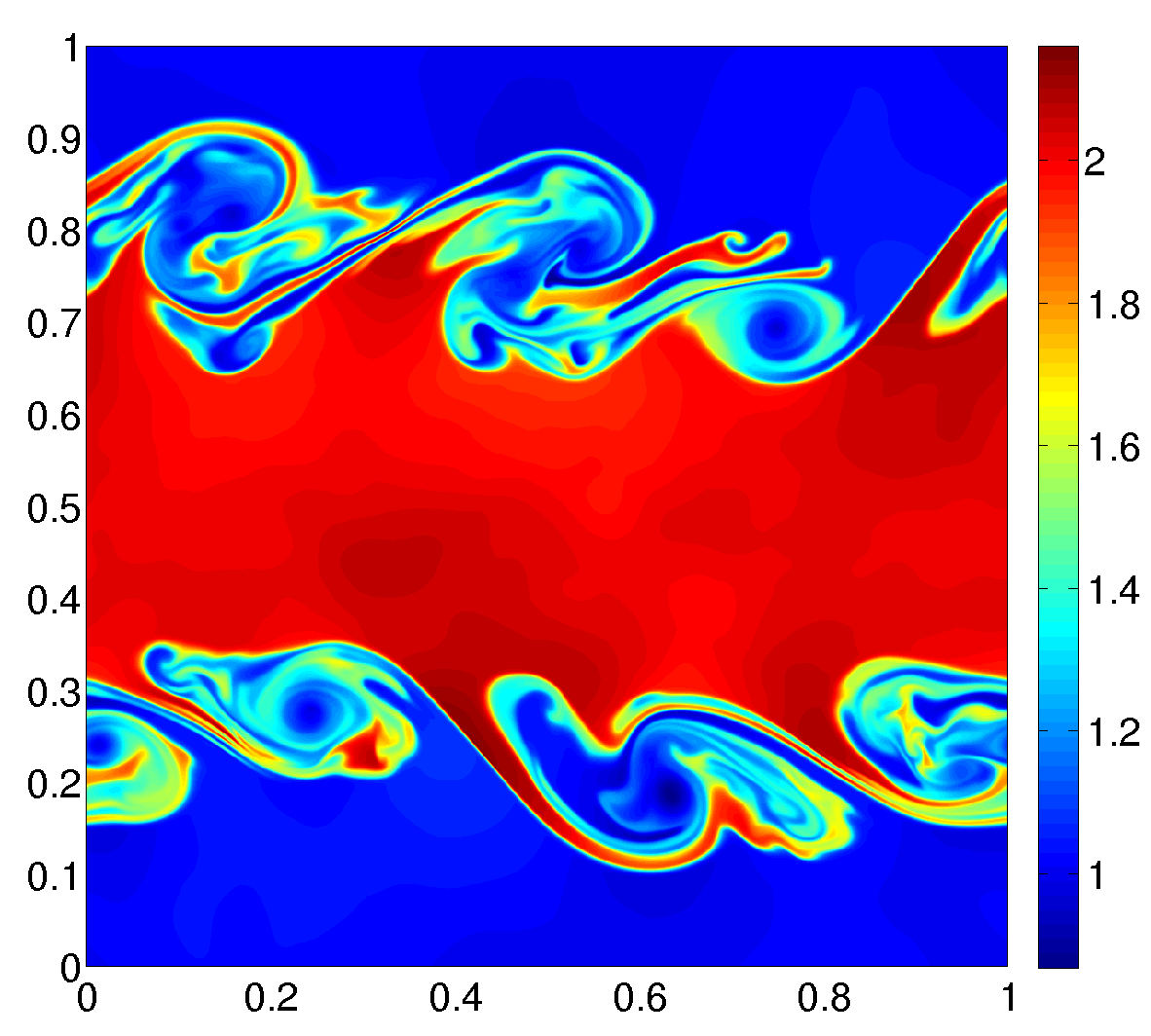}}
\subfigure[$\amp = 2.5\times 10^{-3}$]{\includegraphics[width=0.45\linewidth]{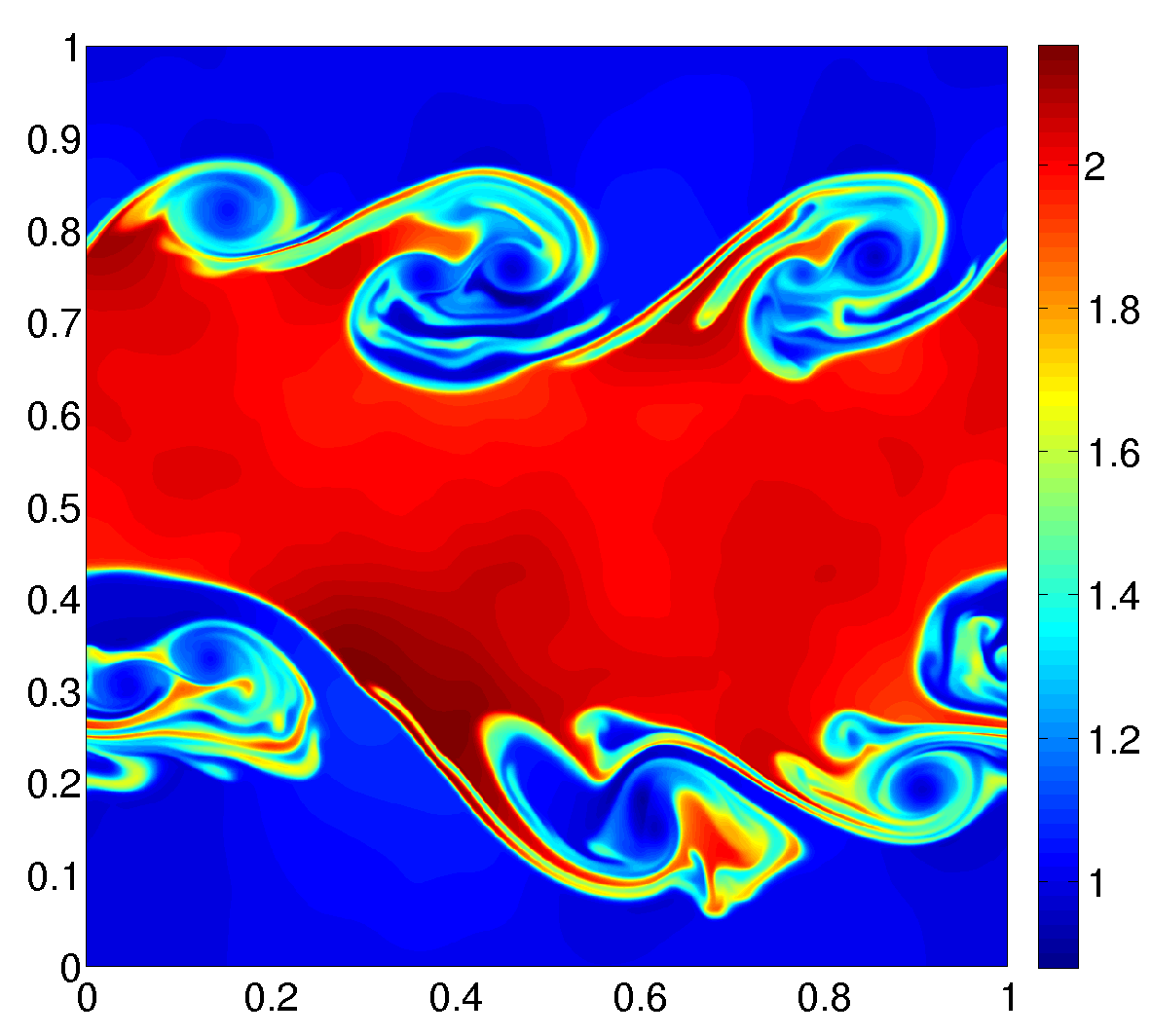}}
\caption{Approximate density, computed with the TeCNO2 scheme for a single sample with initial data \eqref{eq:kh} for different initial perturbation amplitudes $\amp$ on a grid of $1024^2$ points.}
\protect \label{fig:12}
\end{figure}
\begin{figure}
\centering
\includegraphics[width=6cm]{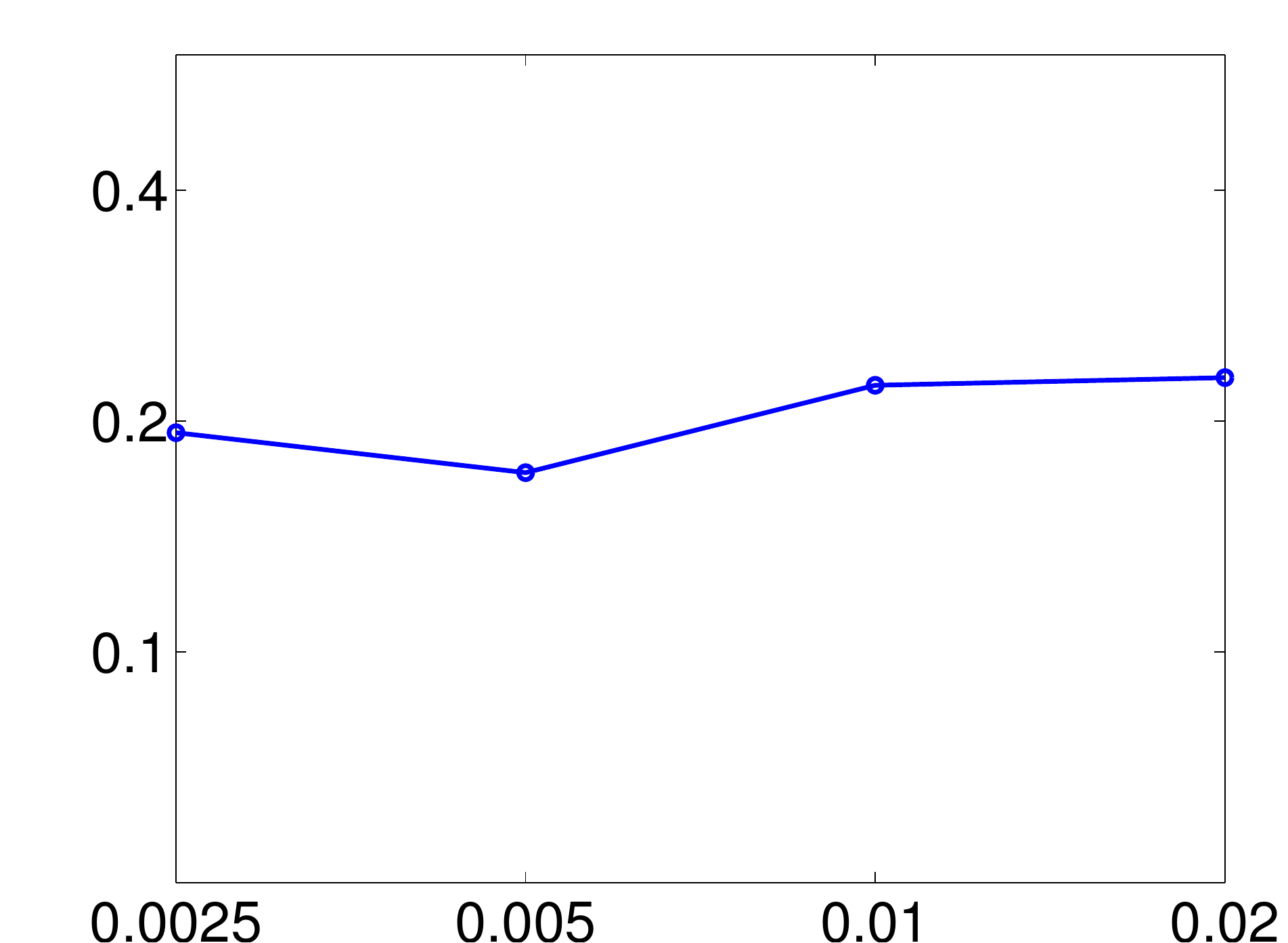}
\caption{The Cauchy rates ($L^1$ difference for successively reduced $\amp$) for the density ($y$-axis) at $t=2$ for a single sample vs. different values of the perturbation parameter $\amp$ ($x$-axis).}
\protect \label{fig:13}
\end{figure}

Next, we compute the mean of the density over $400$ samples at a fixed grid resolution of $1024^2$ points and for different values of the perturbation parameter $\amp$. This sample mean is plotted in Figure \ref{fig:14}. The figure clearly shows pointwise convergence as $\amp \rightarrow 0$, to a limit different from the steady state solution \eqref{eq:khi}. This convergence of the mean with respect to decaying $\amp$ is quantified in Figure \ref{fig:15}(a), where we compute the $L^1$ difference of the mean for successive values of $\amp$. We observe that the mean forms a Cauchy sequence, and hence converges.
\begin{figure}
\centering
\subfigure[$\amp = 2e-2$]{\includegraphics[width=0.45\linewidth]{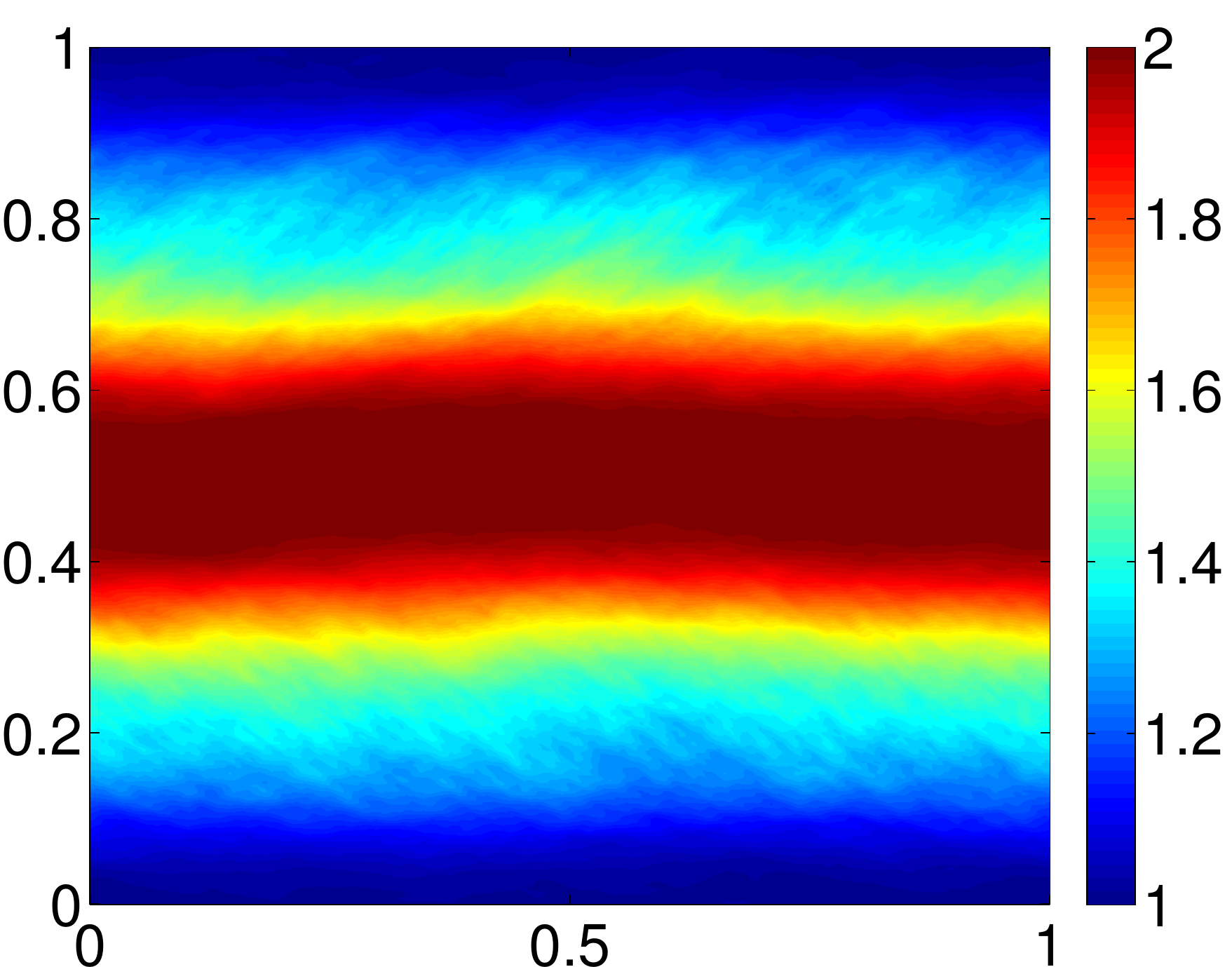}}
\subfigure[$\amp = 1e-2$]{\includegraphics[width=0.45\linewidth]{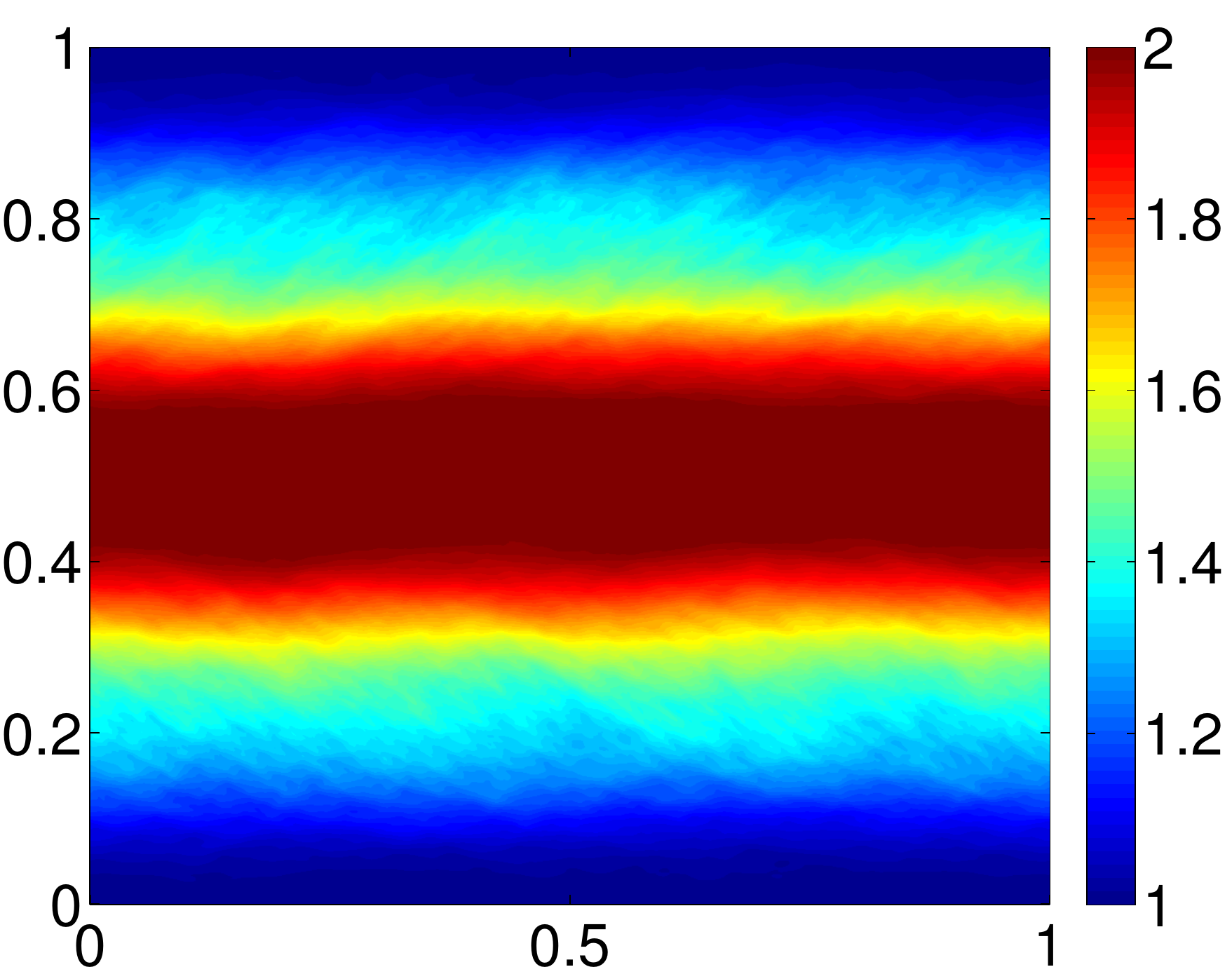}} \\
\subfigure[$\amp = 5e-3$]{\includegraphics[width=0.45\linewidth]{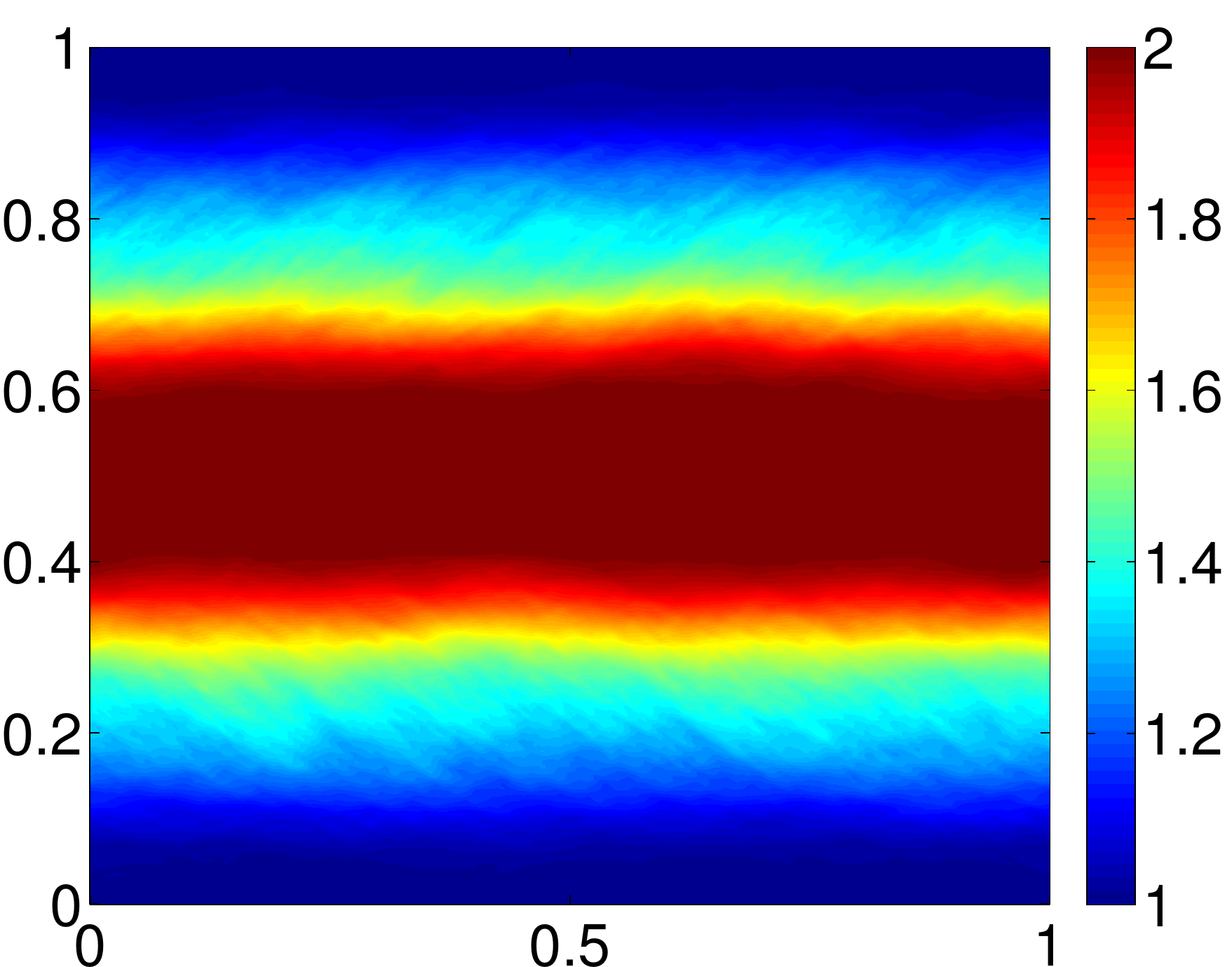}}
\subfigure[$\amp = 2.5e-3$]{\includegraphics[width=0.45\linewidth]{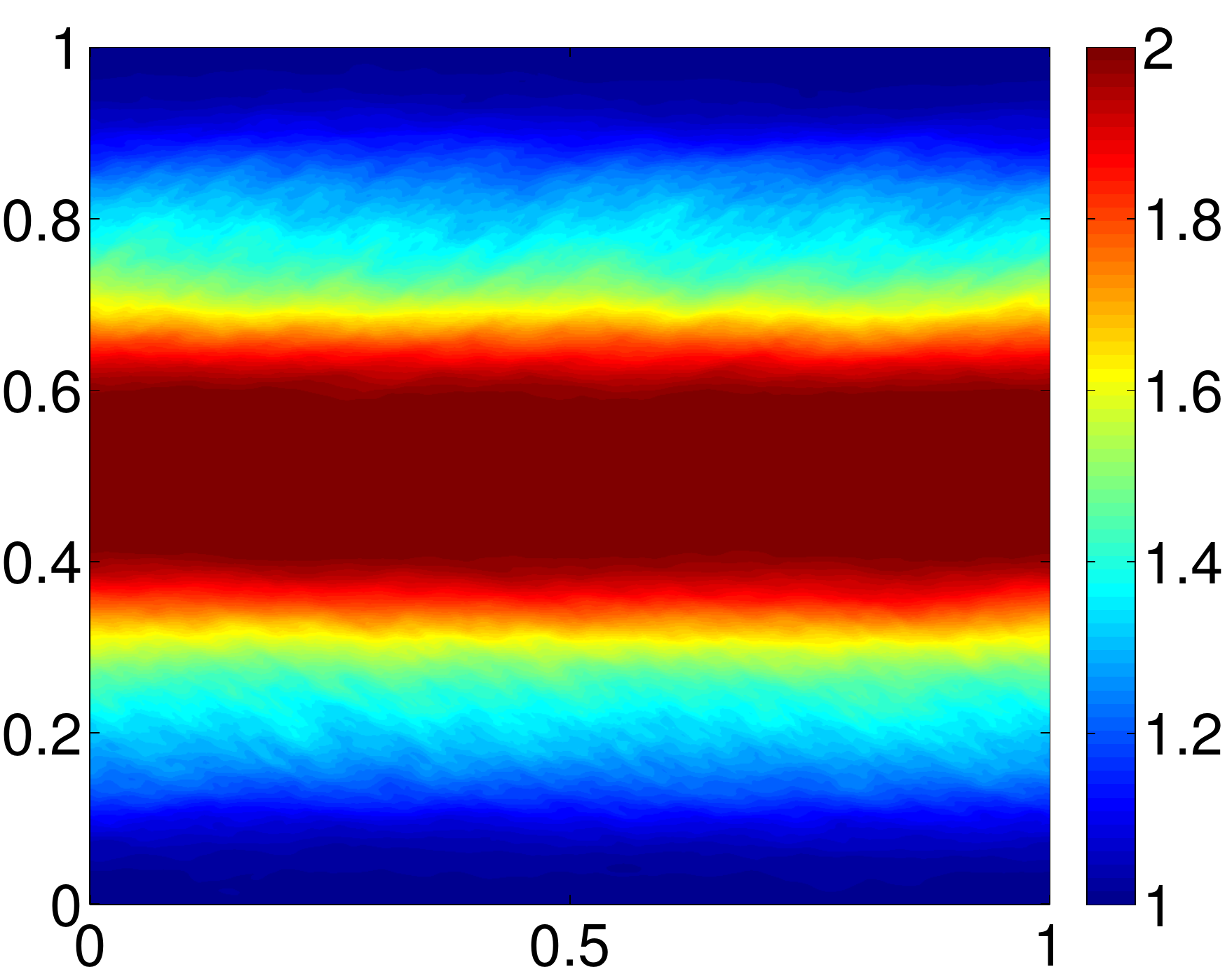}}
\caption{Approximate sample means of the density for the Kelvin-Helmholtz problem \eqref{eq:kh} at time $t=2$ and different values of perturbation parameter $\amp$. All the computations are on a grid of $1024^2$ mesh points and $400$ Monte-Carlo samples.}
\label{fig:14}
\end{figure}
\begin{figure}[htbp] 
\centering
\subfigure[Mean]{\includegraphics[width=0.45\linewidth]{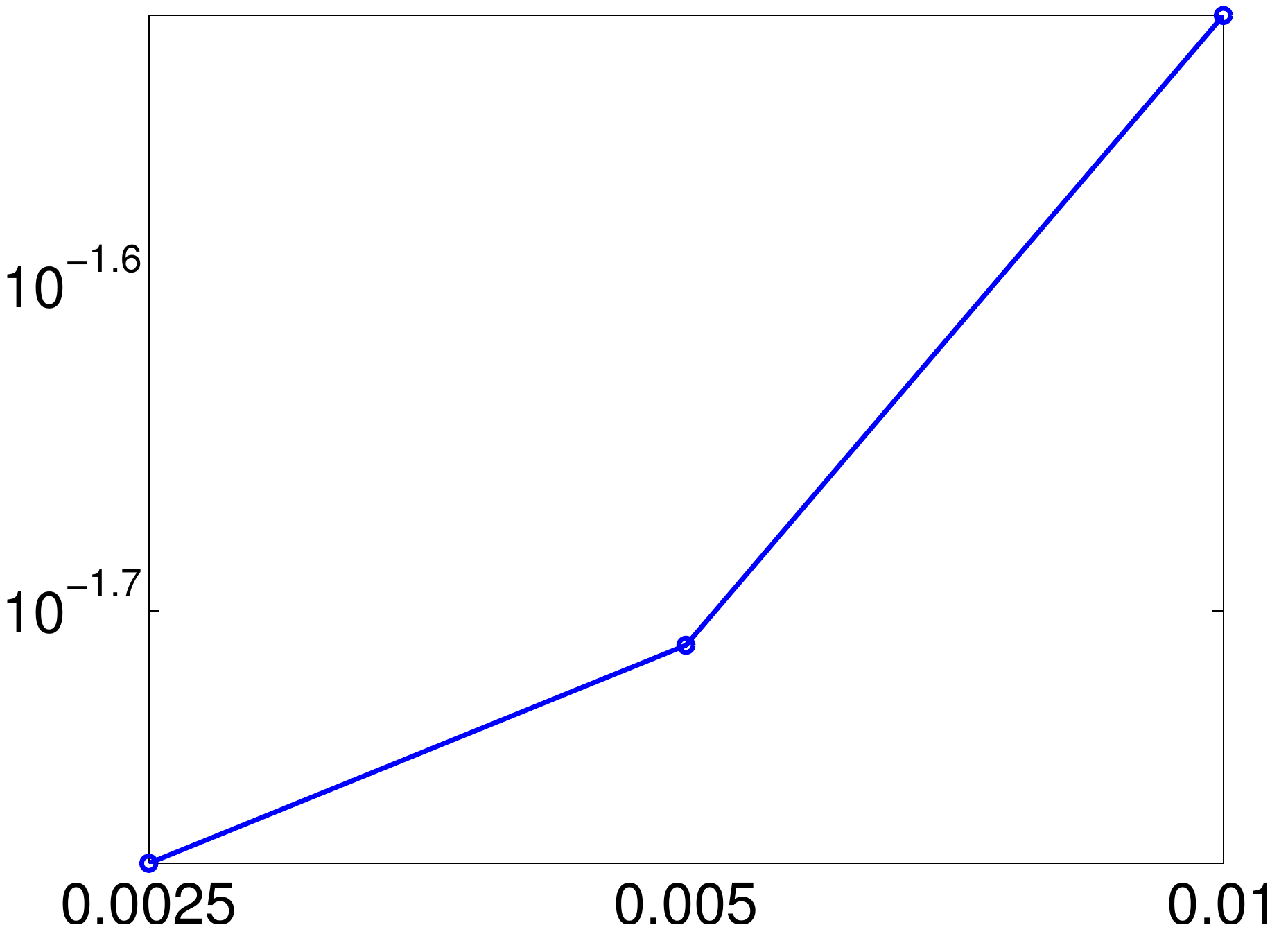} }
\subfigure[Variance]{\includegraphics[width=0.45\linewidth]{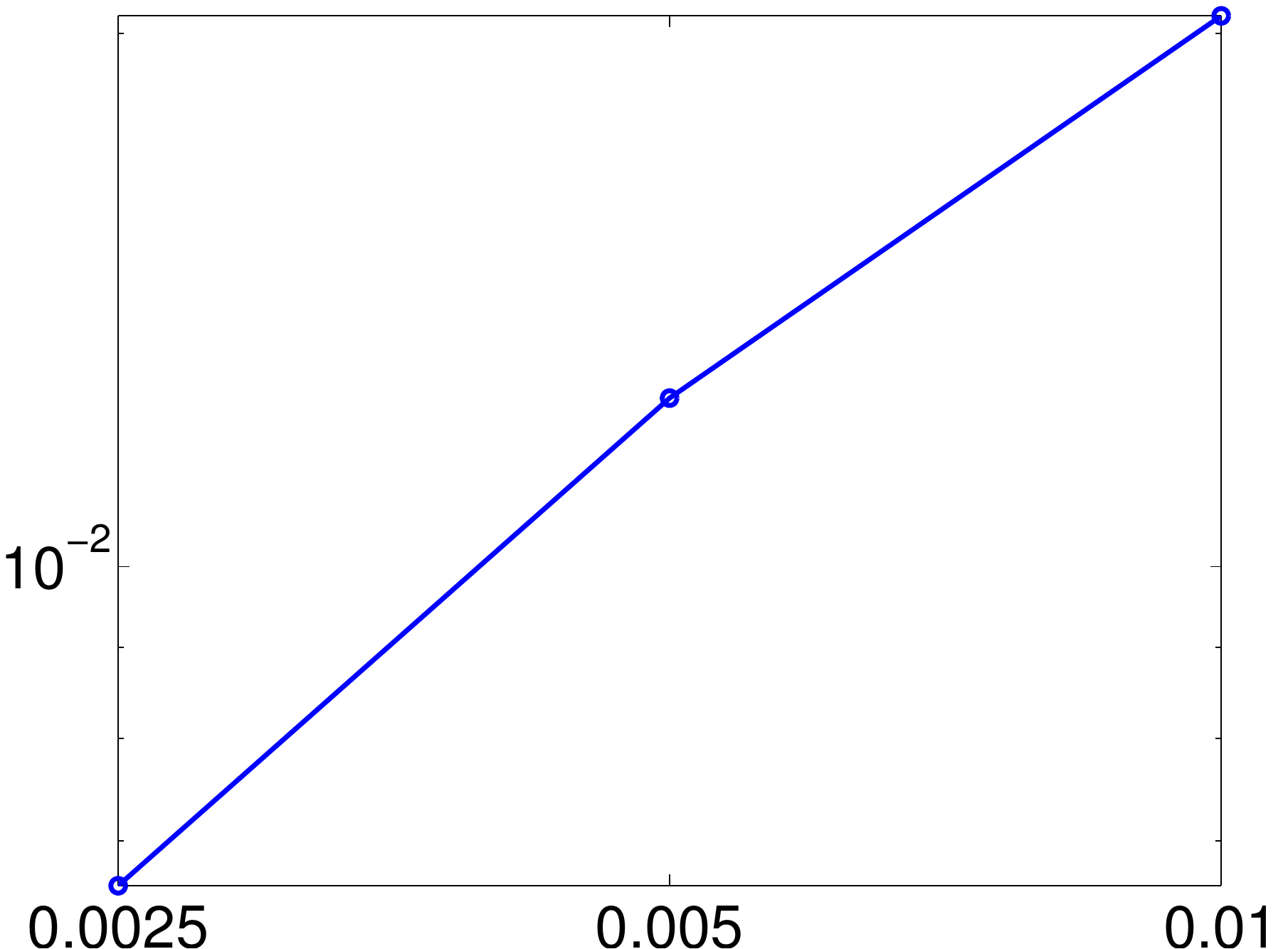} }
\caption{Cauchy rates for the sample mean and the sample variance of the density ($y$-axis) for the Kelvin-Helmholtz problem \eqref{eq:kh} for different values of $\amp$ ($x$-axis). All the computations are on a grid of $1024^2$ mesh points and $400$ Monte-Carlo samples.}
\label{fig:15}
\end{figure}

Similarly the computations of the sample variance for different values of $\amp$ are presented in Figure \ref{fig:16}. Note that this figure, as well as the computations of the difference in variance in $L^1$ for successive reductions of the perturbation parameter $\amp$ (shown in Figure \ref{fig:15}(b)), clearly show convergence of variance as $\amp \rightarrow 0$. Moreover, Figure \ref{fig:16} clearly indicates that in the $\amp \rightarrow  0$ limit, the limit variance is non-zero. Hence, this strongly suggests the fact that EMV solution can be \emph{non-atomic, even for atomic initial data}. These results are consistent with the claims of Theorem \ref{thm:alphaconv}.
\begin{figure}[htbp] 
\centering
\subfigure[$\amp = 2e-2$]{\includegraphics[width=0.45\linewidth]{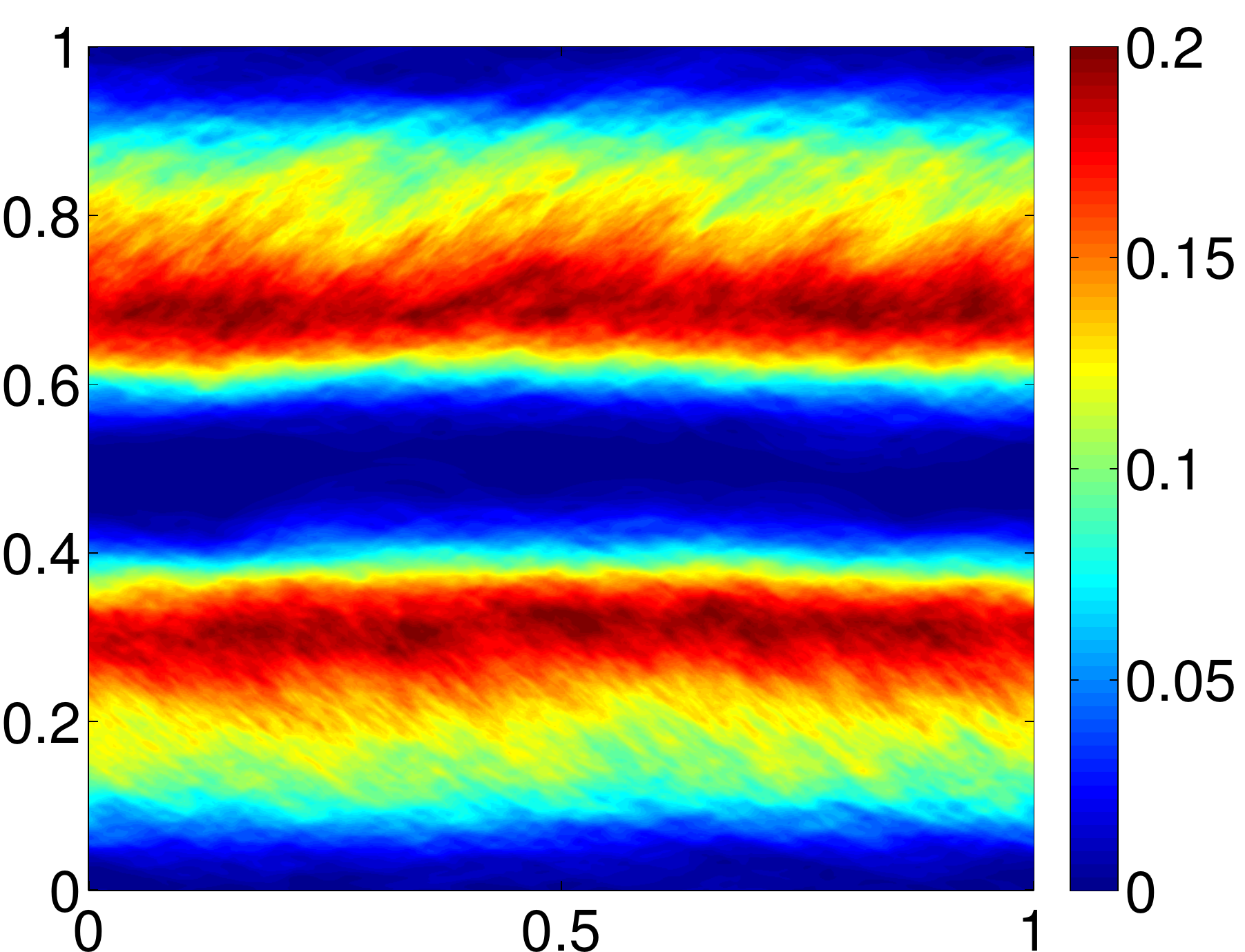}}
\subfigure[$\amp = 1e-2$]{\includegraphics[width=0.45\linewidth]{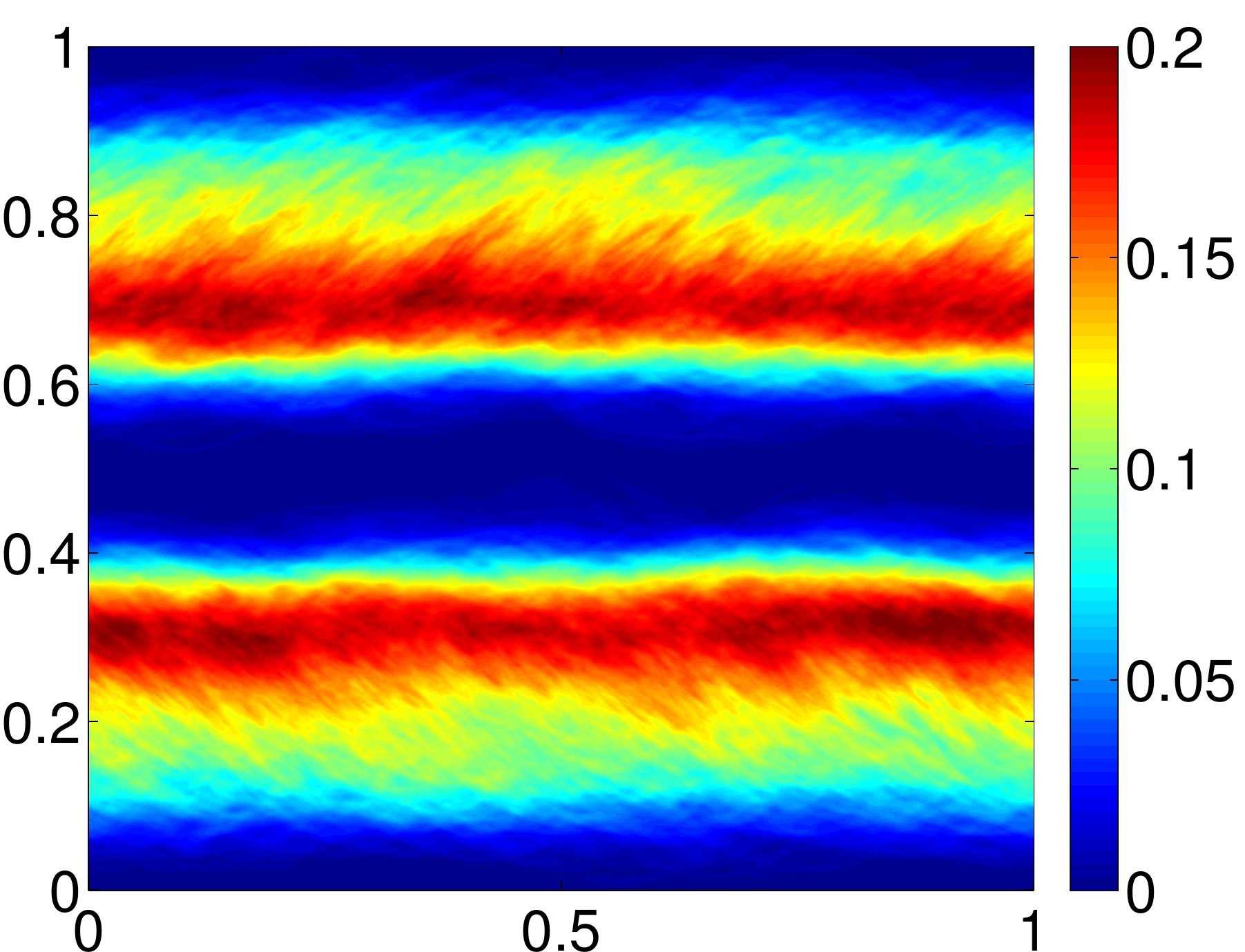}} \\
\subfigure[$\amp = 5e-3$]{\includegraphics[width=0.45\linewidth]{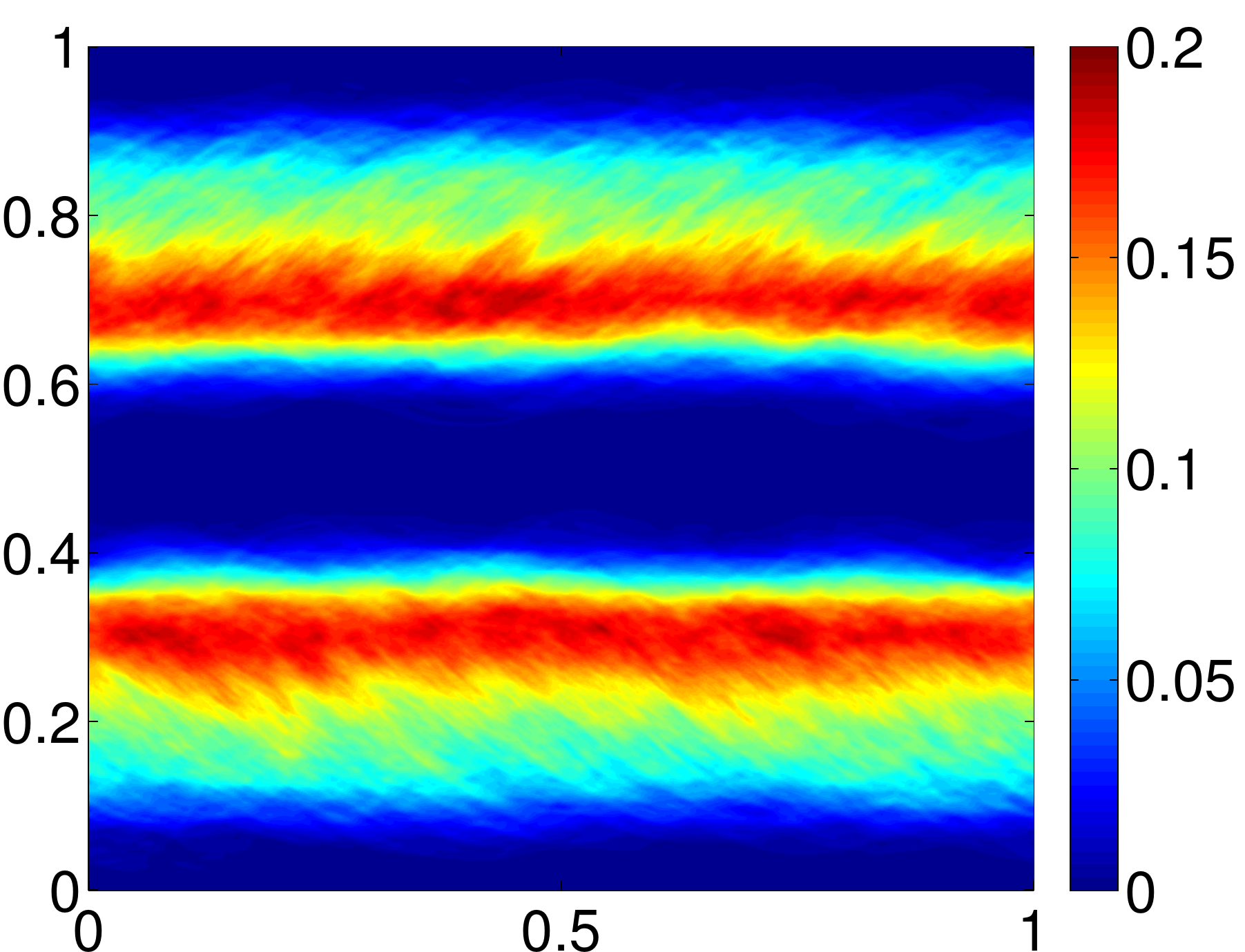}}
\subfigure[$\amp = 2.5e-3$]{\includegraphics[width=0.45\linewidth]{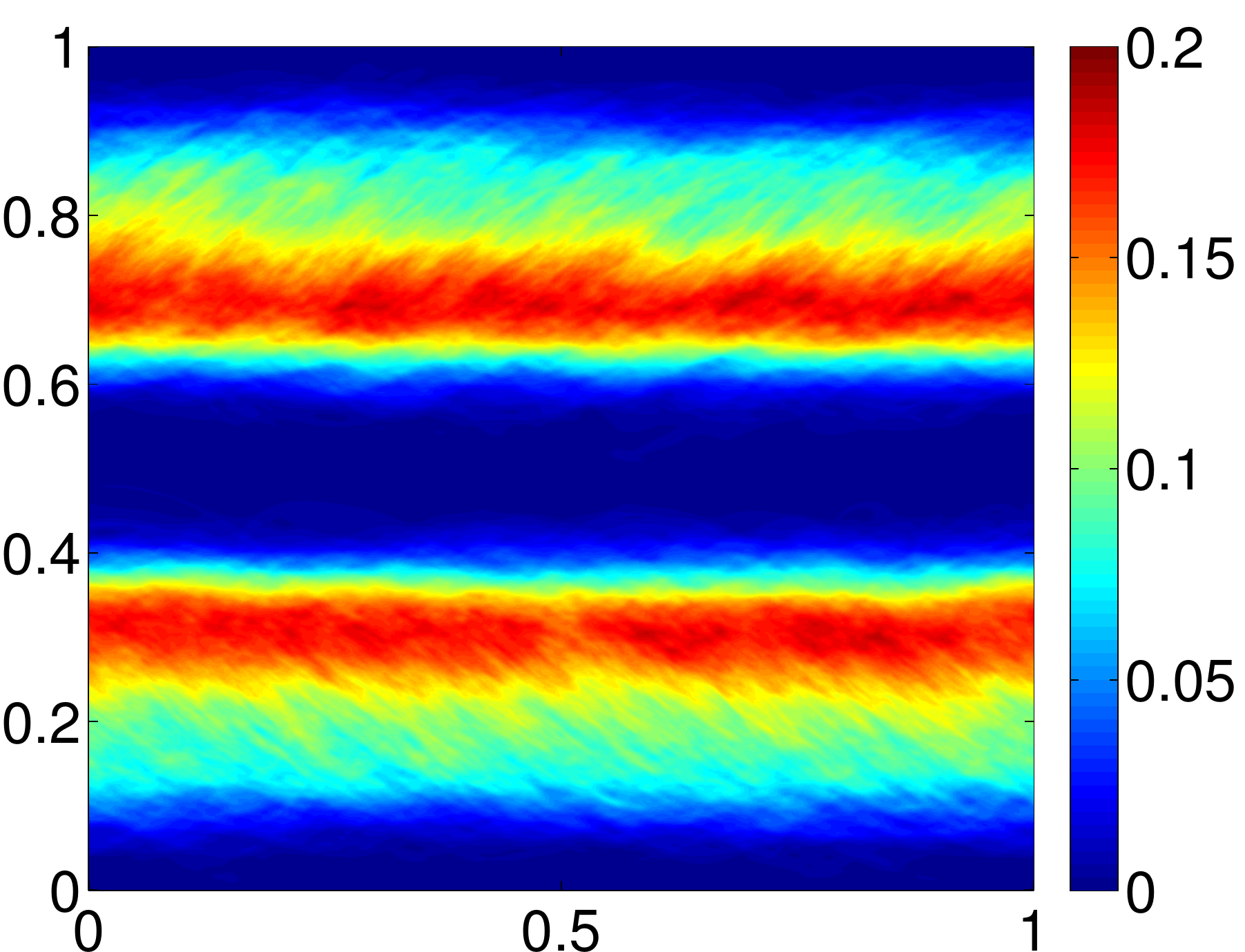}}
\caption{Approximate sample variances of the density for the Kelvin-Helmholtz instability at time $t=2$ and different values of perturbation parameter $\amp$. All the computations are on a grid of $1024^2$ mesh points and $400$ Monte-Carlo samples}
\label{fig:16}
\end{figure}

To further demonstrate the non-atomicity of the resulting measure valued solution, we have plotted the probability density functions (approximated by empirical histograms) for density at the points $x=(0.5,0.7)$ and $x=(0.5,0.8)$ in Figure \ref{fig:pdfTime} for a fixed mesh of size $1024^2$. We see that the initial unit mass centered at $\rho=2$ ($\rho=1$, respectively) at $t=0$ is smeared out over time, and at $t=2$ the mass has spread out over a range of values of $\rho$ between 1 and 2.

Figure \ref{fig:pdfRefine} shows the same quantities, but for a fixed time $t=2$ over a series of meshes. Although a certain amount of noise seems to persist on the finer meshes -- most likely due to the low number of Monte Carlo samples -- it can be seen that the probability density functions seem to converge with mesh refinement.
\begin{figure}
\centering
\subfigure[$t=0$]{\includegraphics[width=0.19\linewidth]{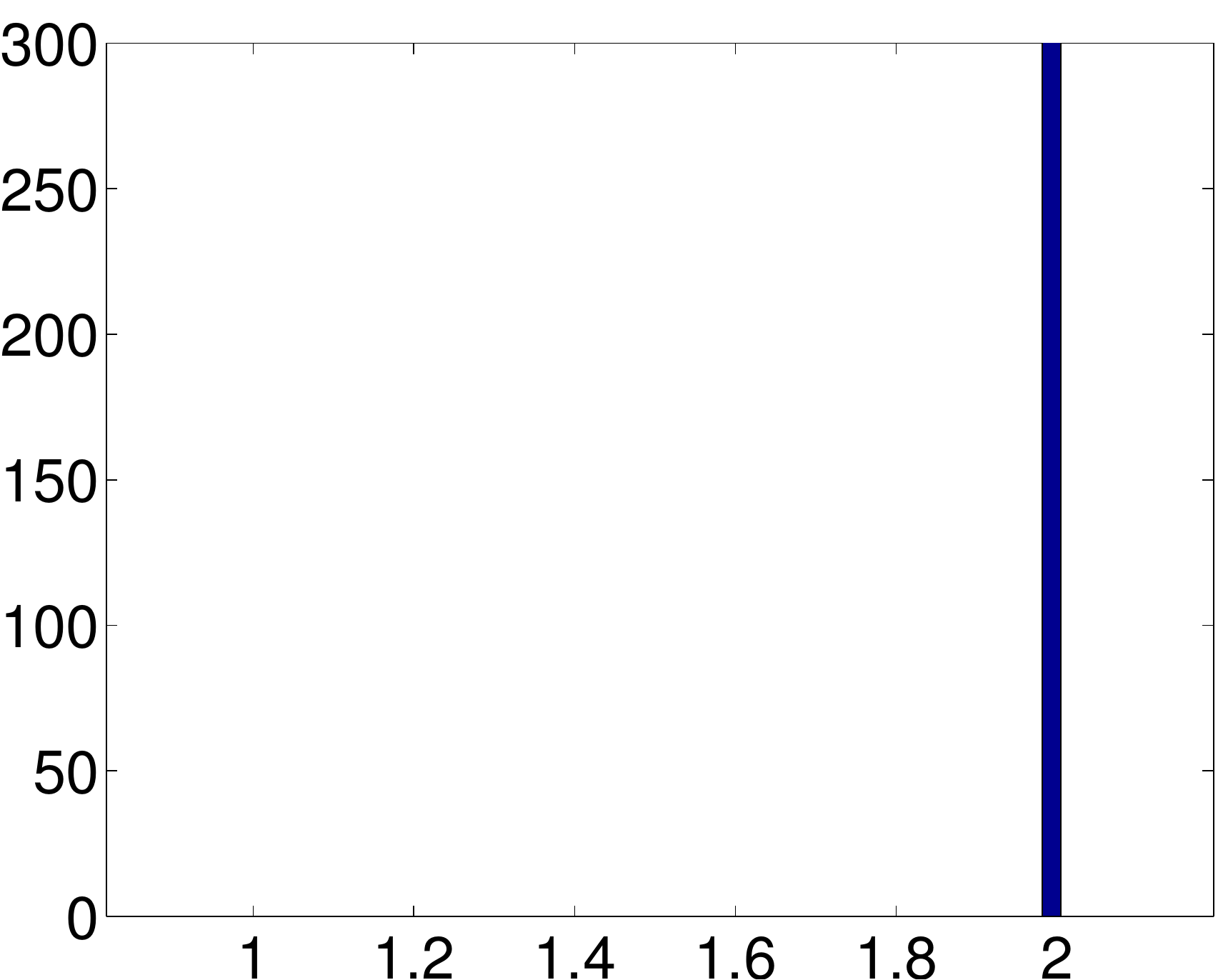}}
\subfigure[$t=0.5$]{\includegraphics[width=0.19\linewidth]{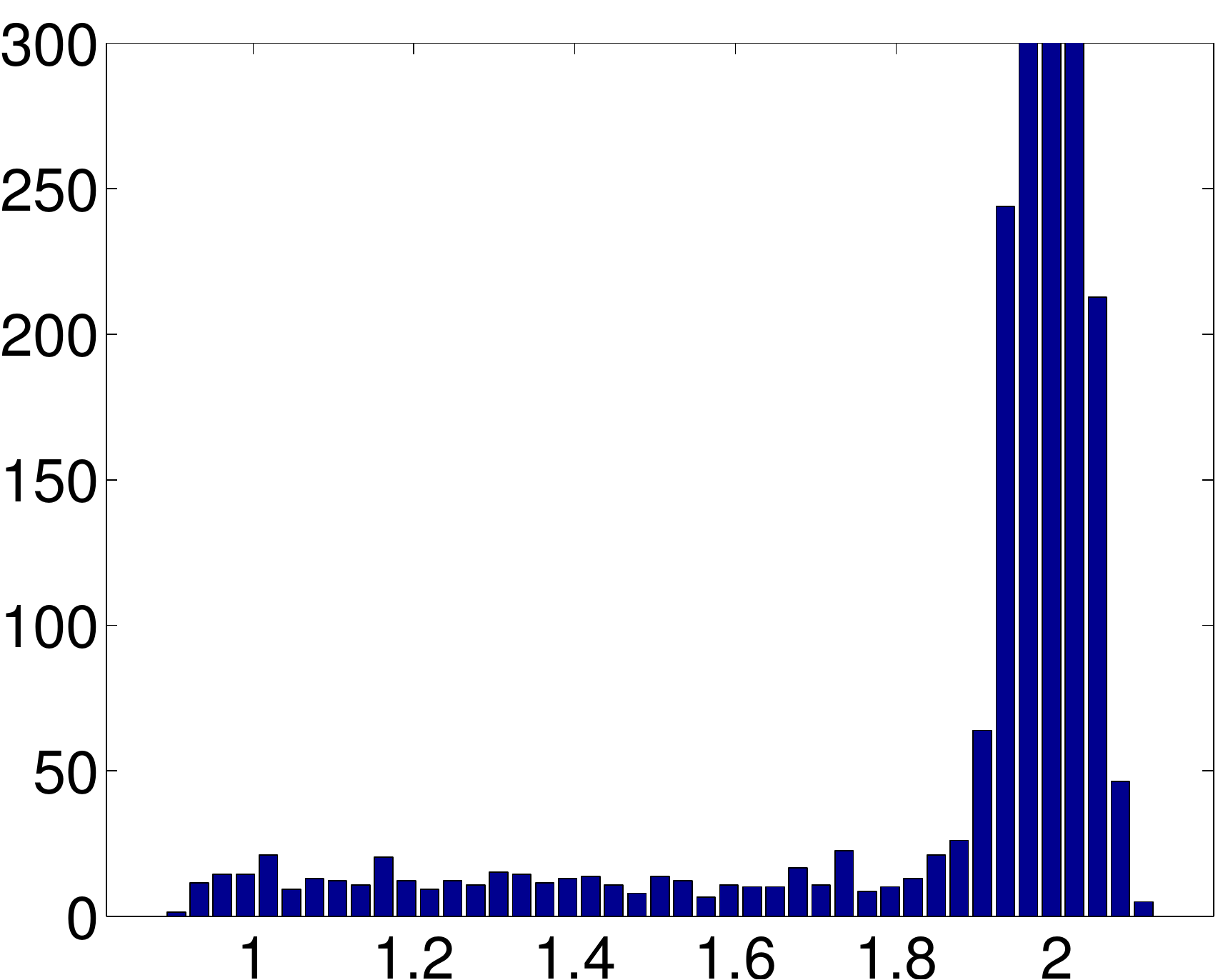}}
\subfigure[$t=1$]{\includegraphics[width=0.19\linewidth]{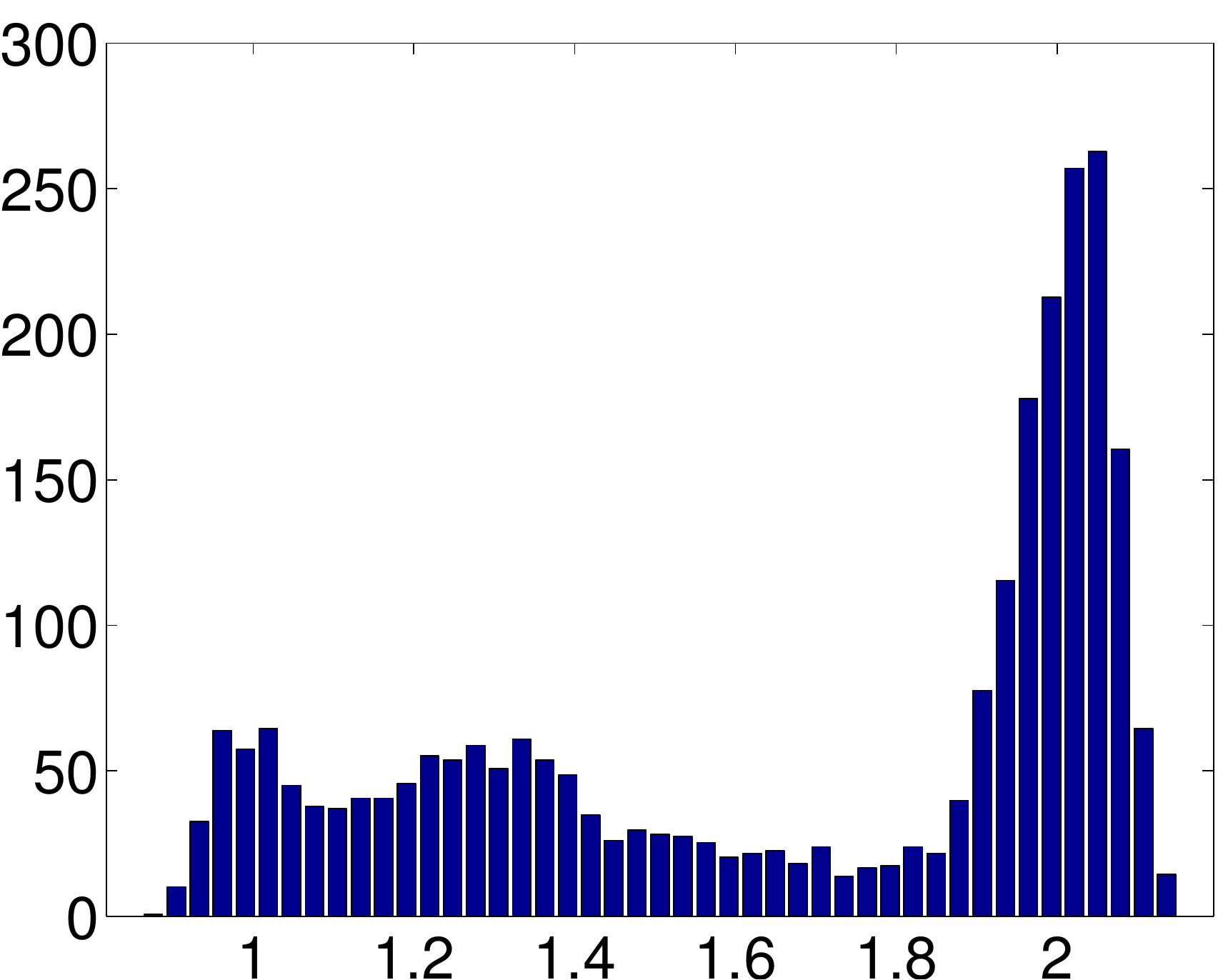}}
\subfigure[$t=1.5$]{\includegraphics[width=0.19\linewidth]{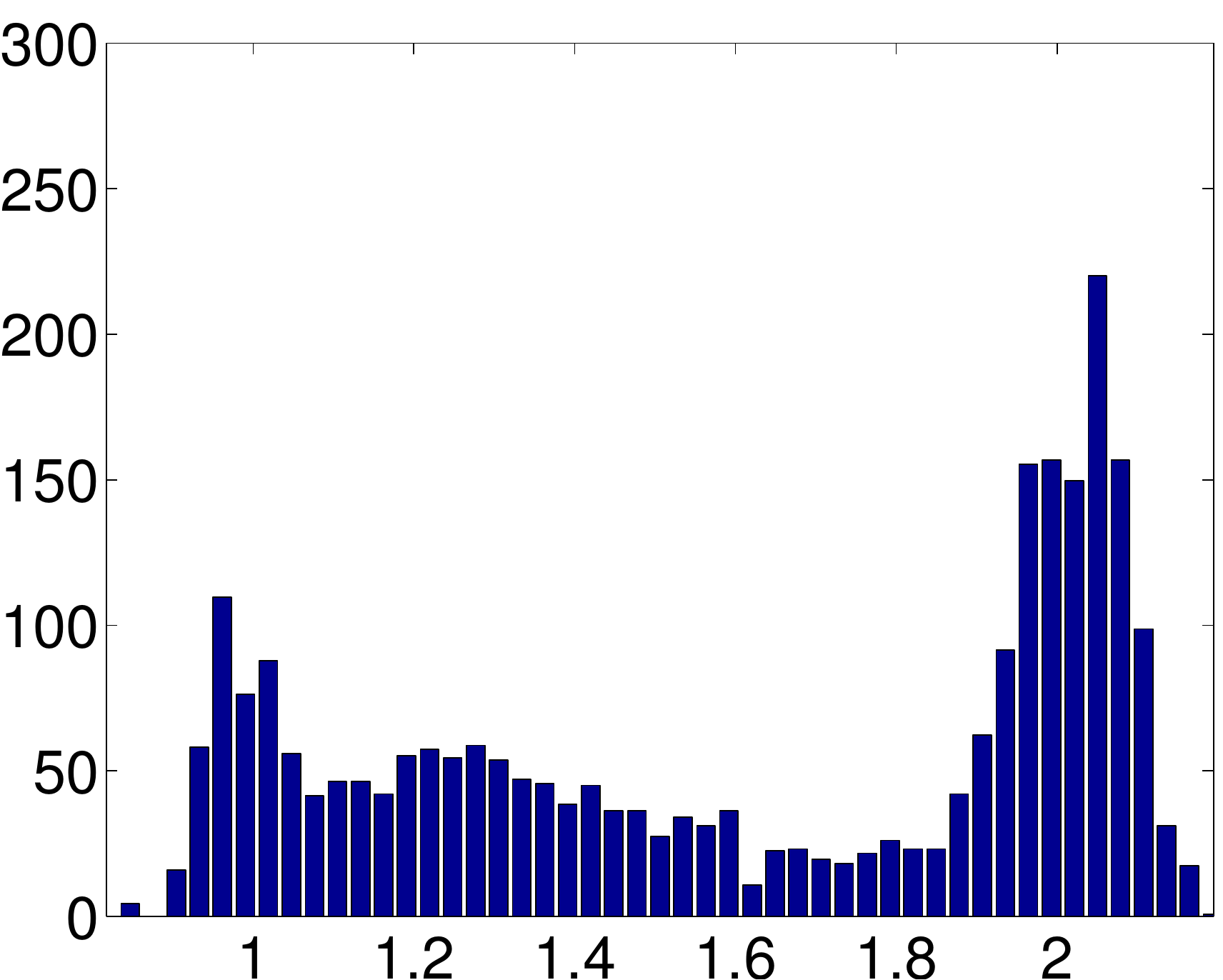}}
\subfigure[$t=2$]{\includegraphics[width=0.19\linewidth]{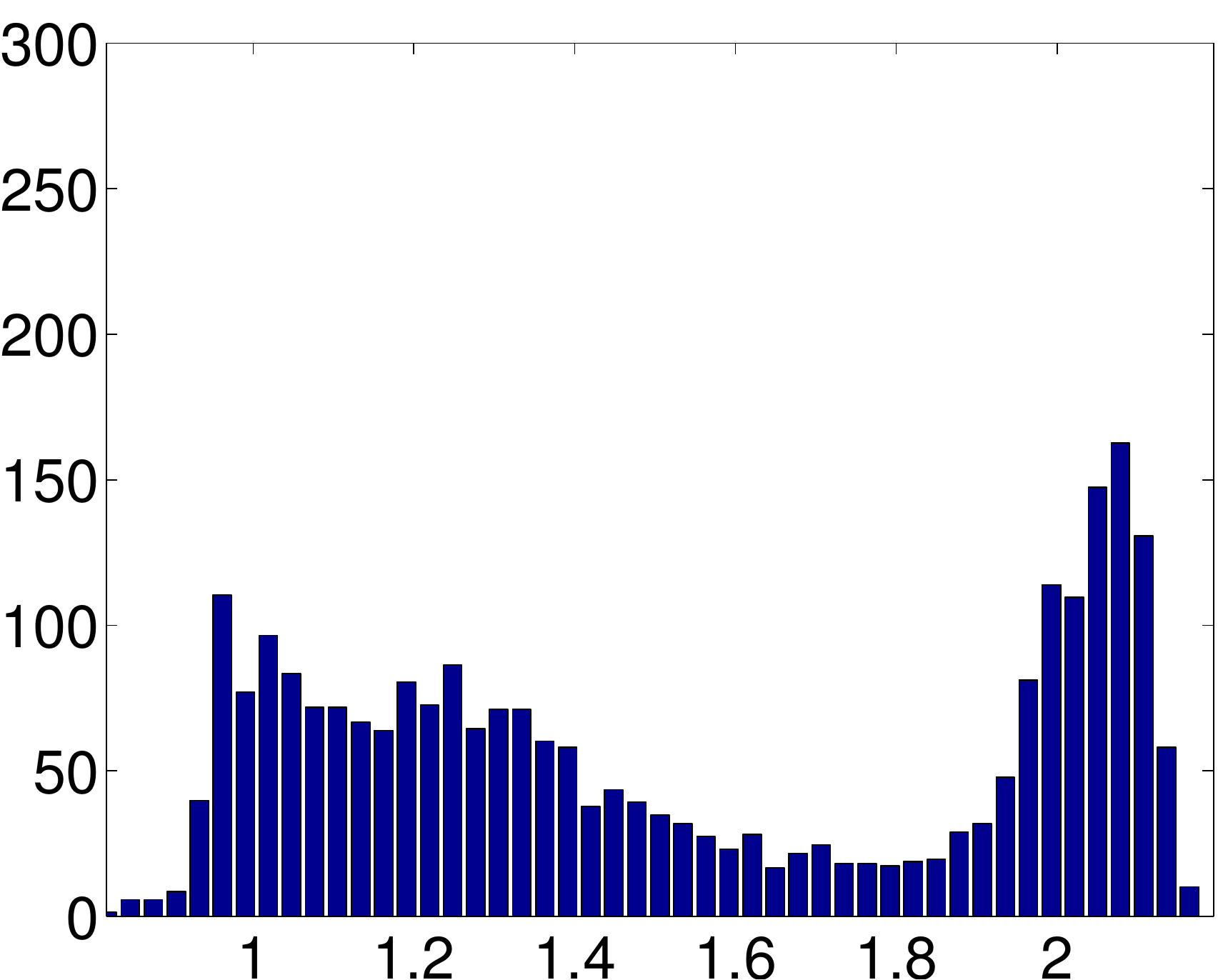}}
\subfigure[$t=0$]{\includegraphics[width=0.19\linewidth]{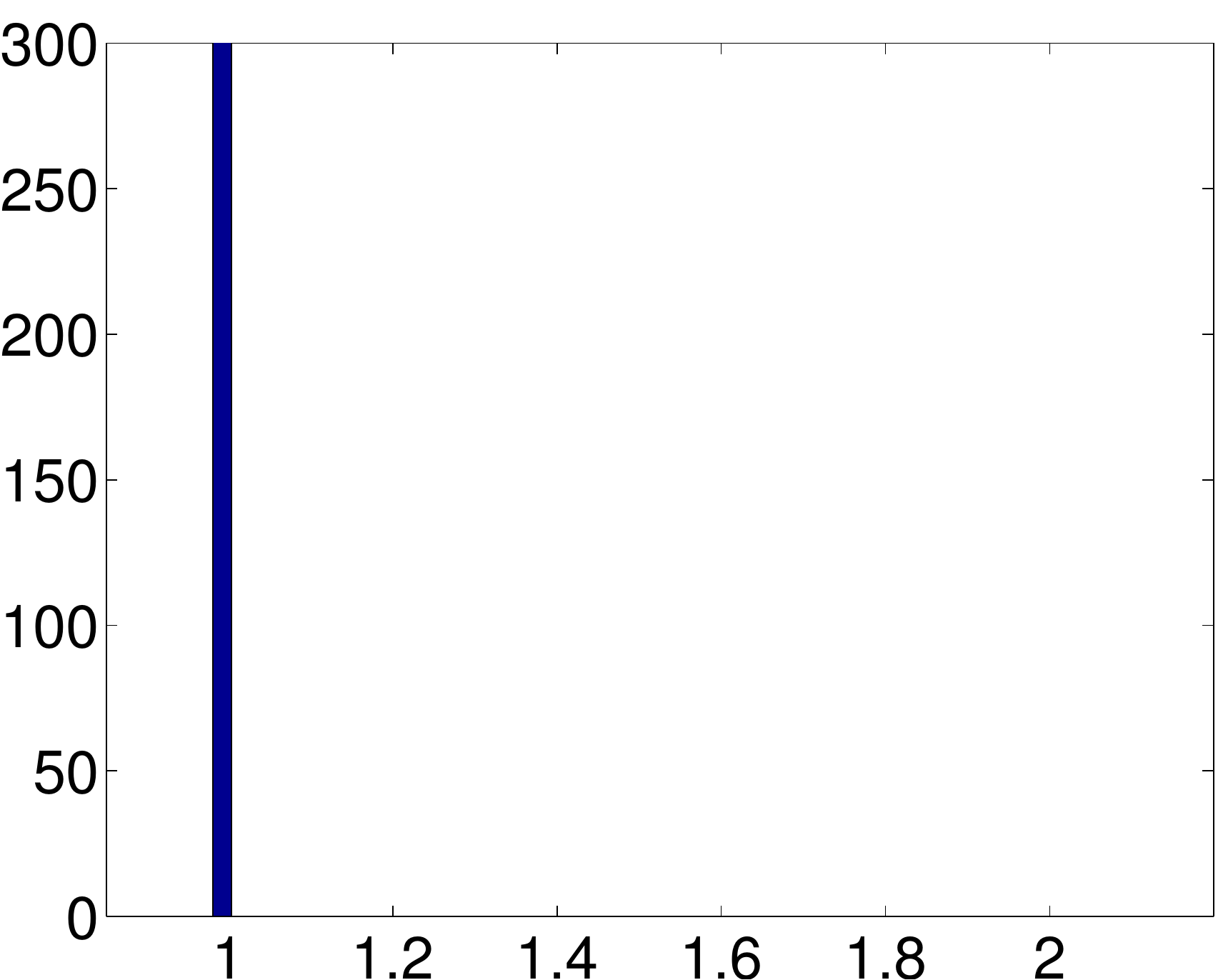}}
\subfigure[$t=0.5$]{\includegraphics[width=0.19\linewidth]{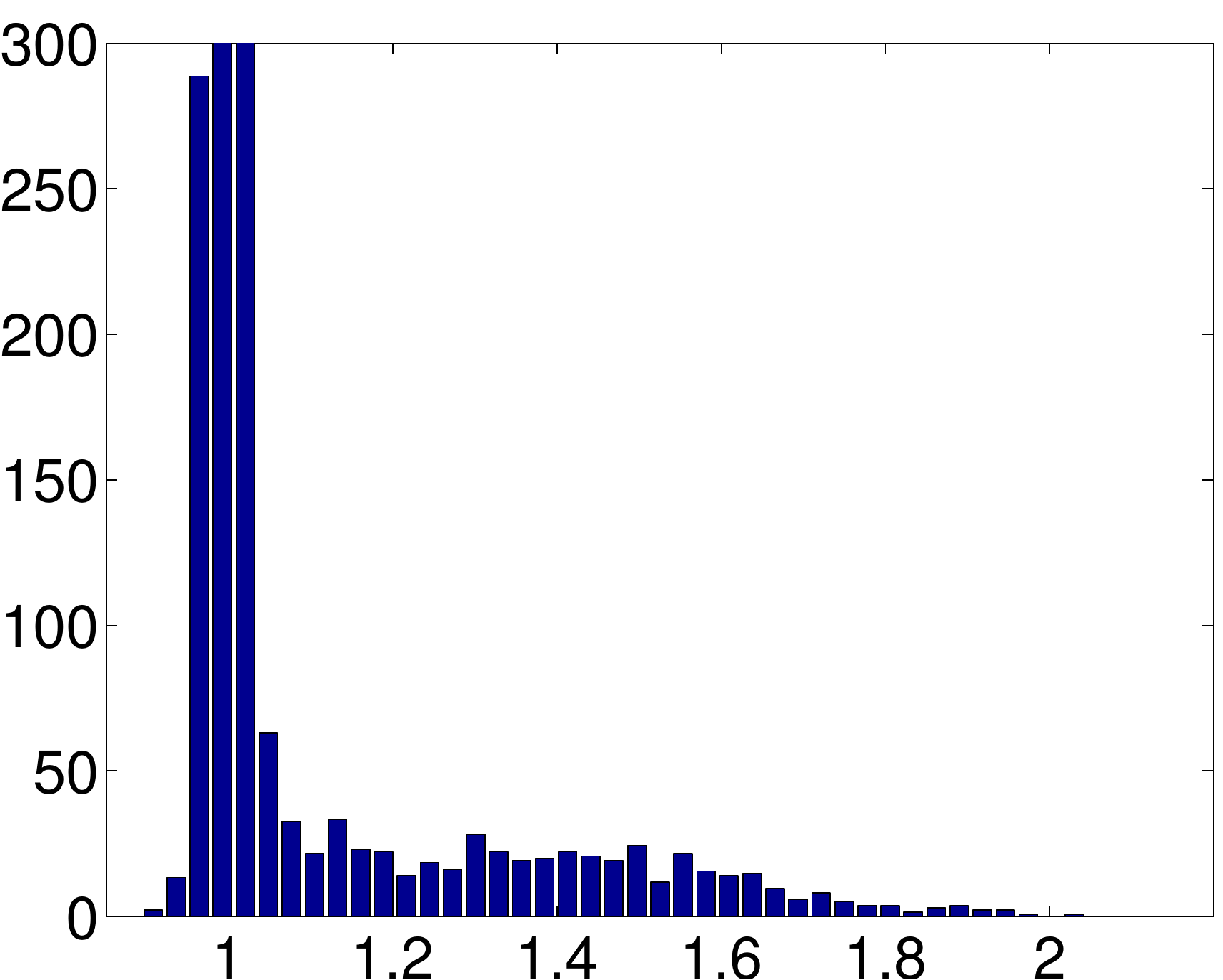}}
\subfigure[$t=1$]{\includegraphics[width=0.19\linewidth]{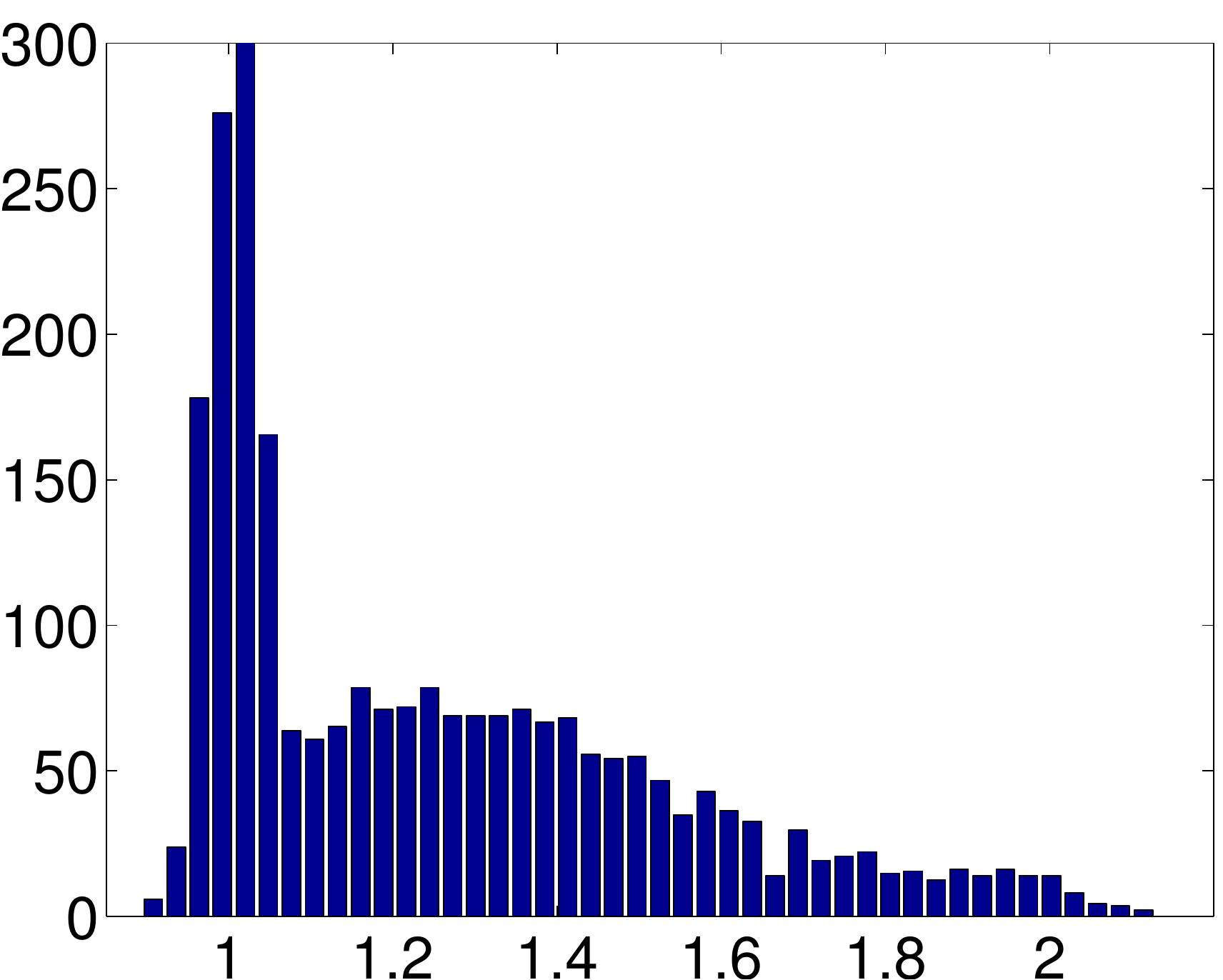}}
\subfigure[$t=1.5$]{\includegraphics[width=0.19\linewidth]{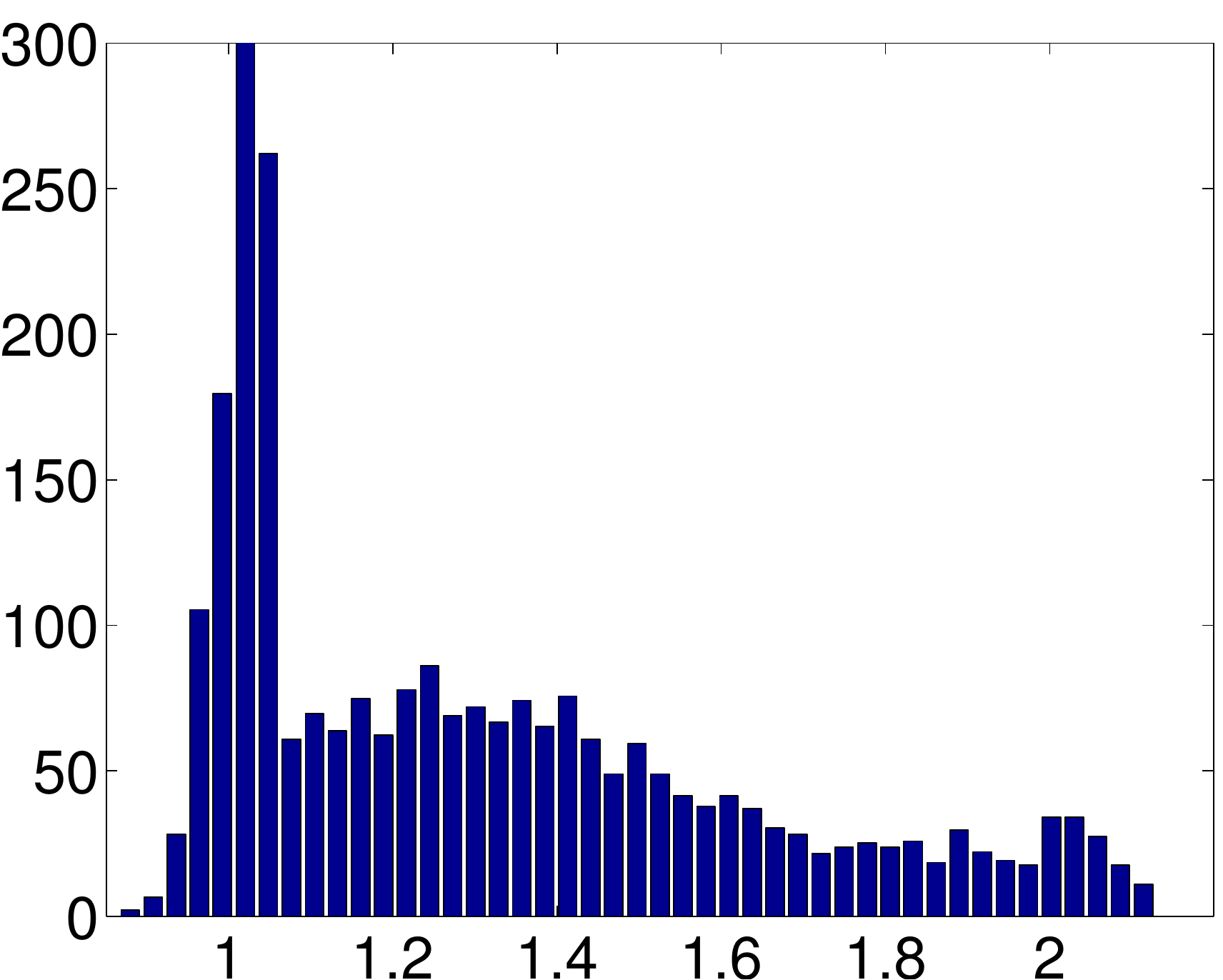}}
\subfigure[$t=2$]{\includegraphics[width=0.19\linewidth]{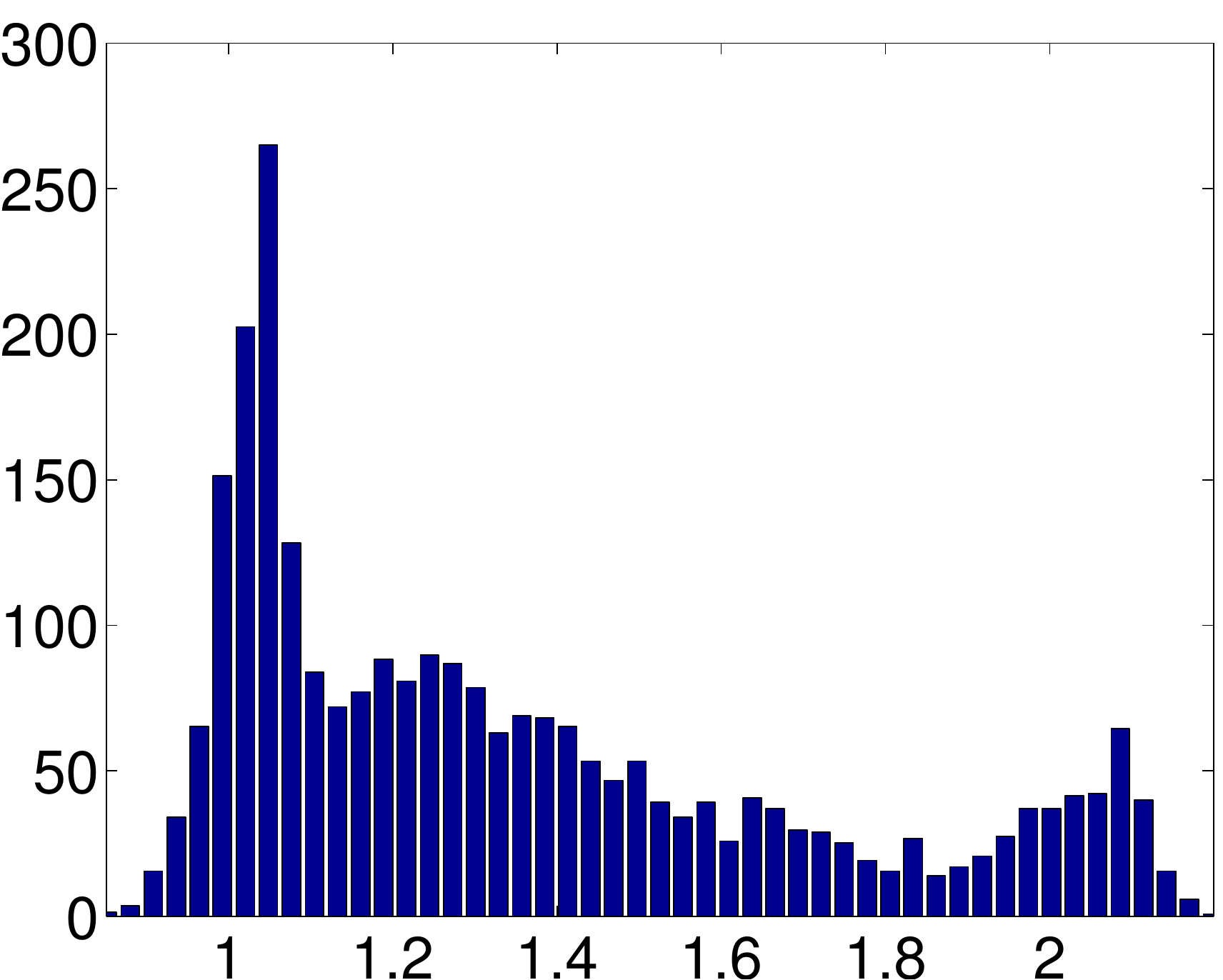}}
\caption{The approximate PDF for density $\rho$ at the points $x=(0.5, 0.7)$ (first row) and $x=(0.5,0.8)$ (second row) on a grid of $1024^2$ mesh points.}
\label{fig:pdfTime}
\end{figure}

\begin{figure}
\centering
\subfigure[nx = 128]{\includegraphics[width=0.24\linewidth]{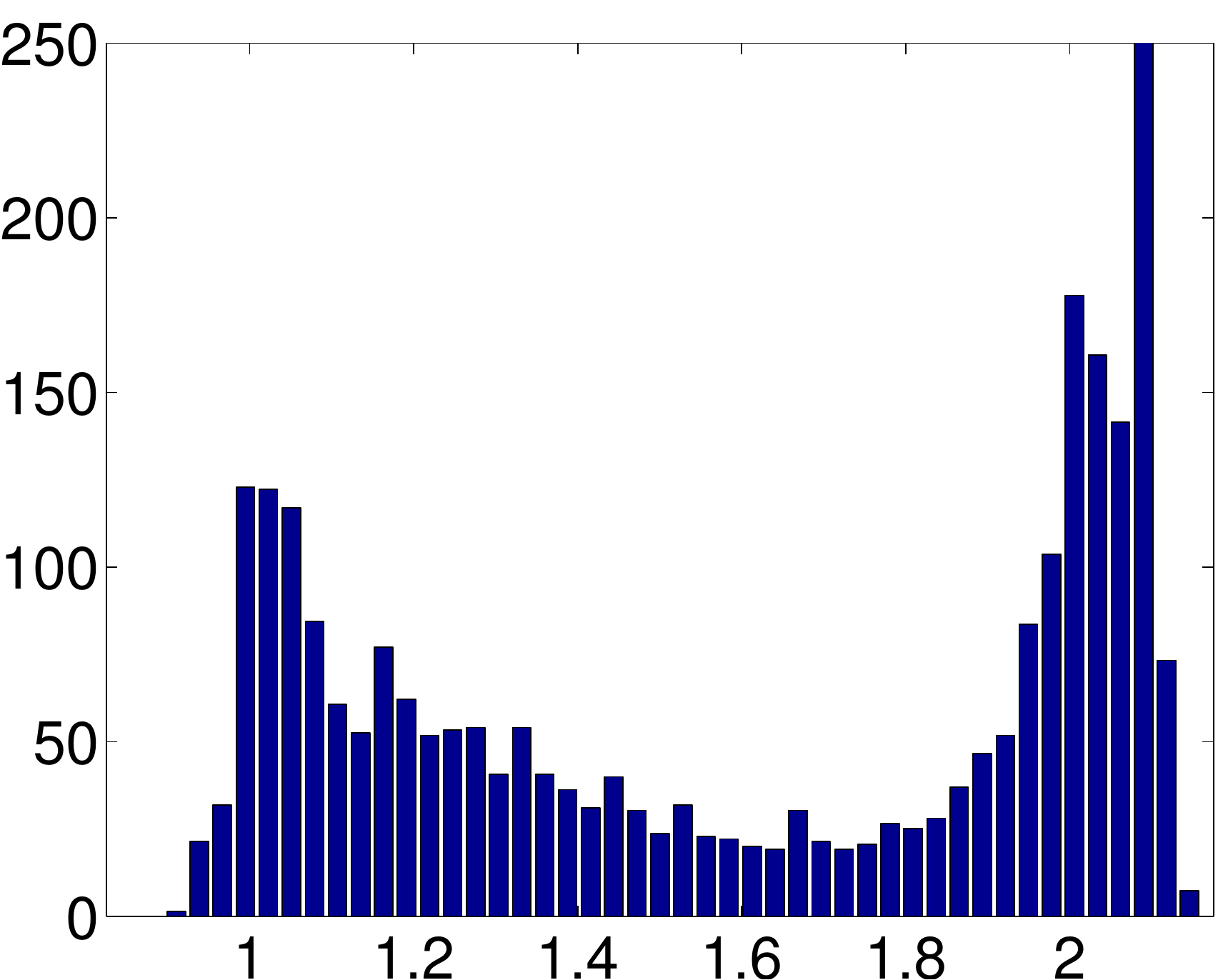}}
\subfigure[nx = 256]{\includegraphics[width=0.24\linewidth]{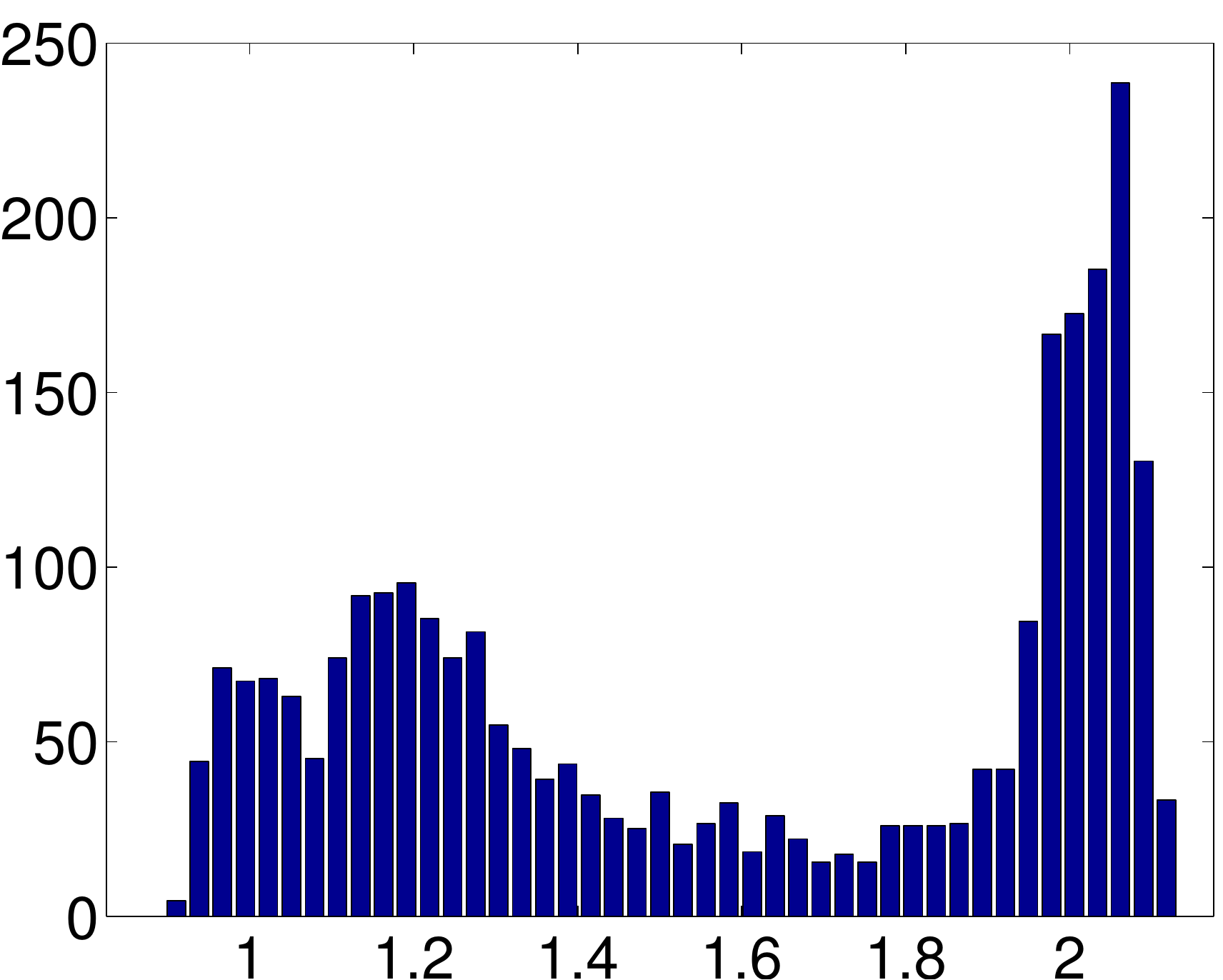}}
\subfigure[nx = 512]{\includegraphics[width=0.24\linewidth]{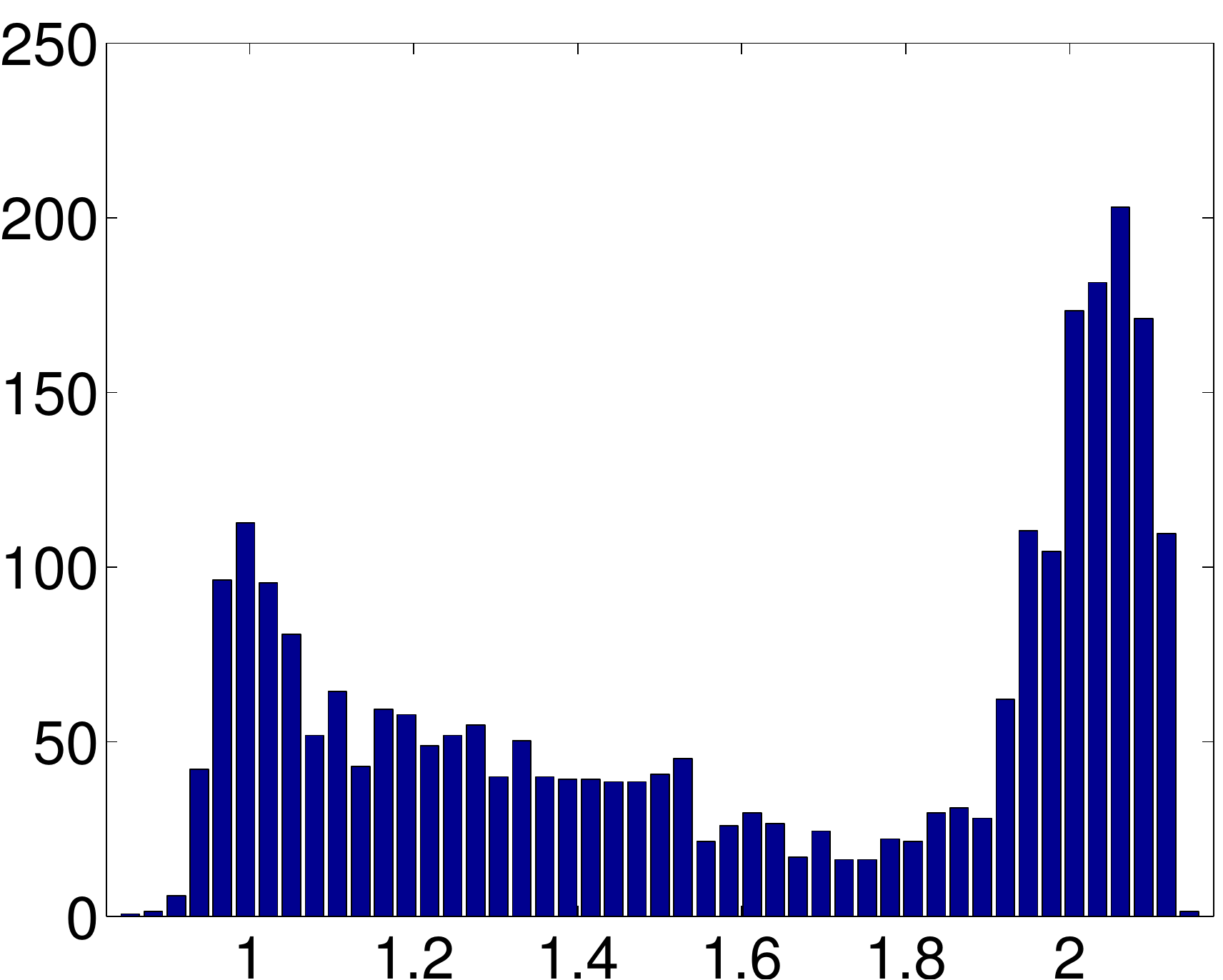}}
\subfigure[nx = 1024]{\includegraphics[width=0.24\linewidth]{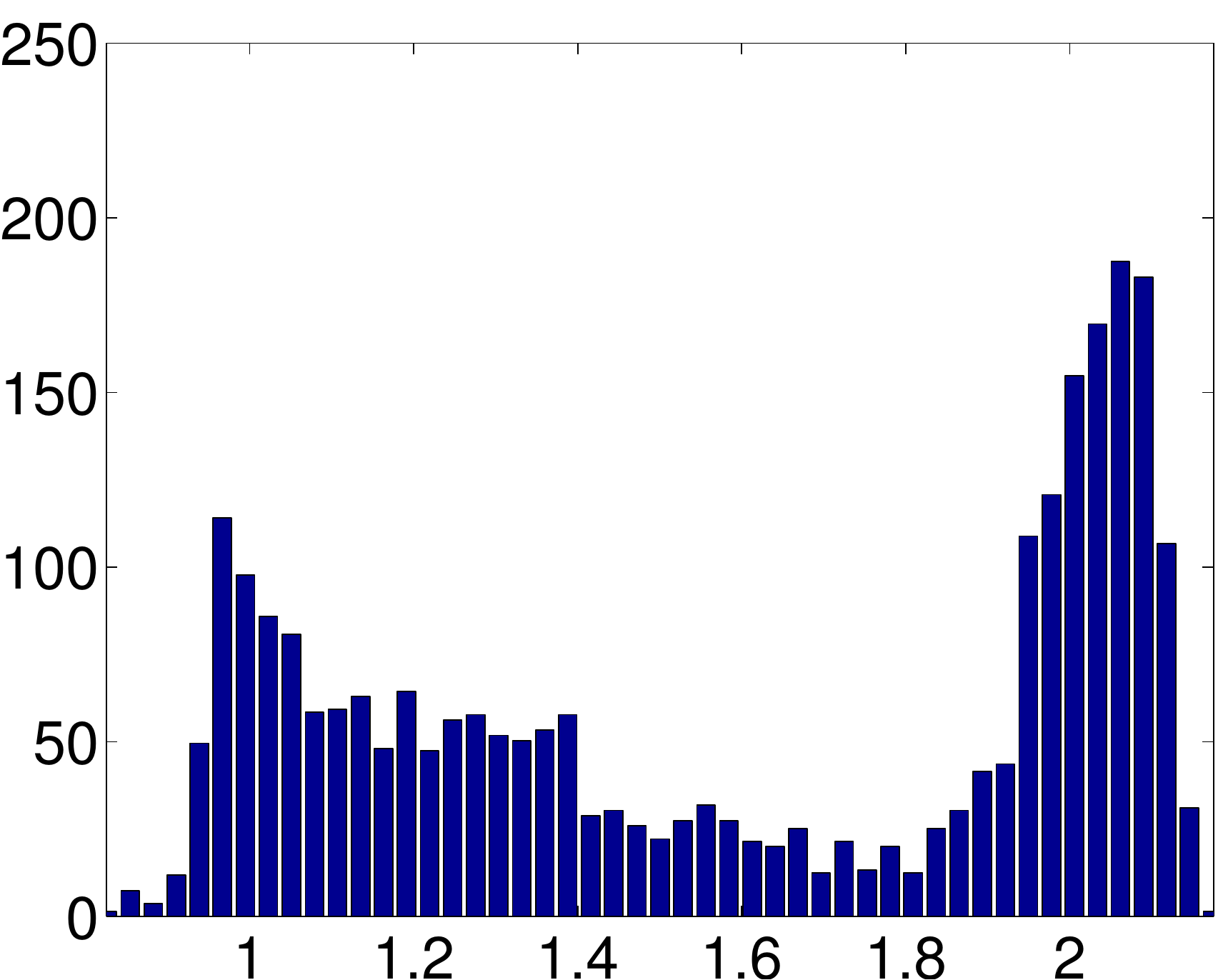}} \\
\subfigure[nx = 128]{\includegraphics[width=0.24\linewidth]{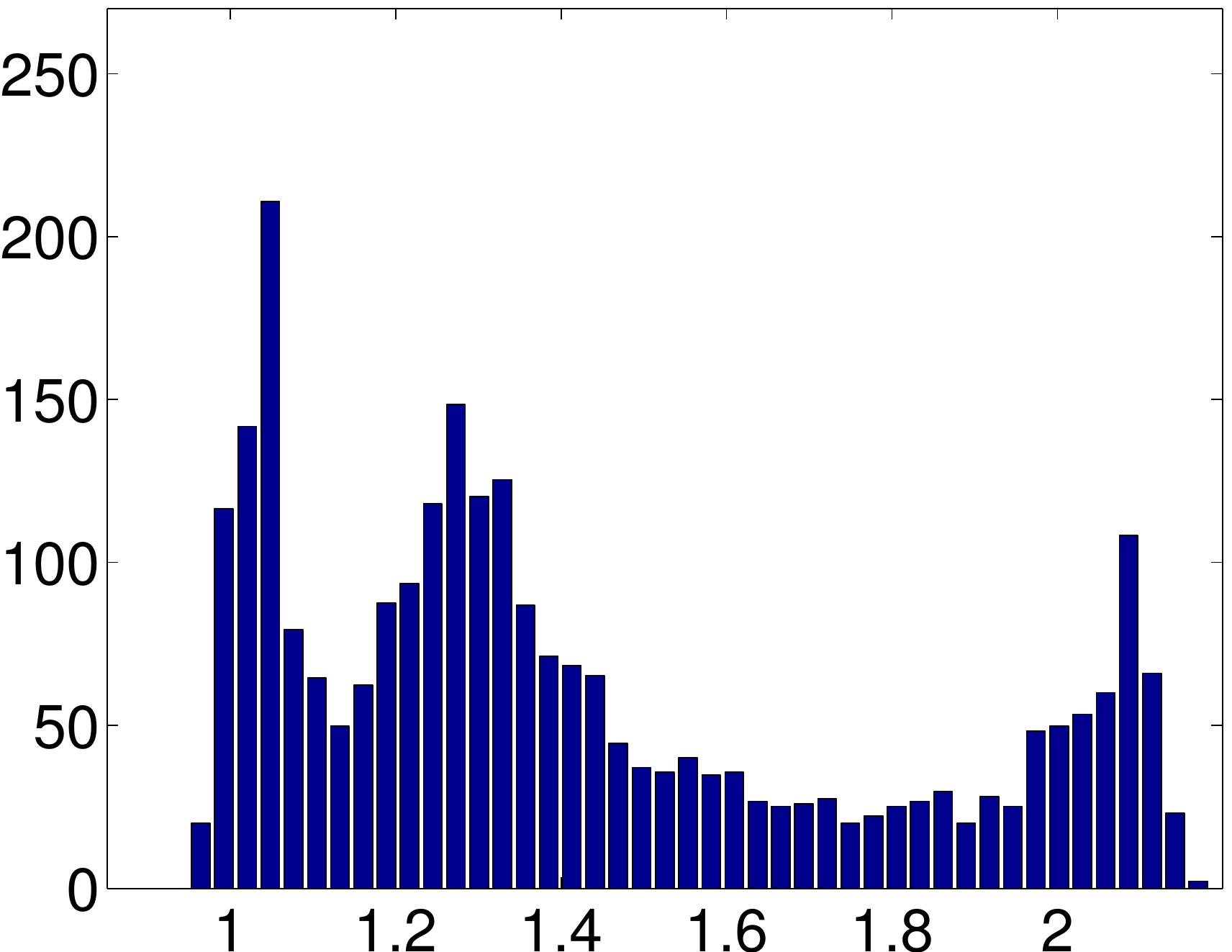}}
\subfigure[nx = 256]{\includegraphics[width=0.24\linewidth]{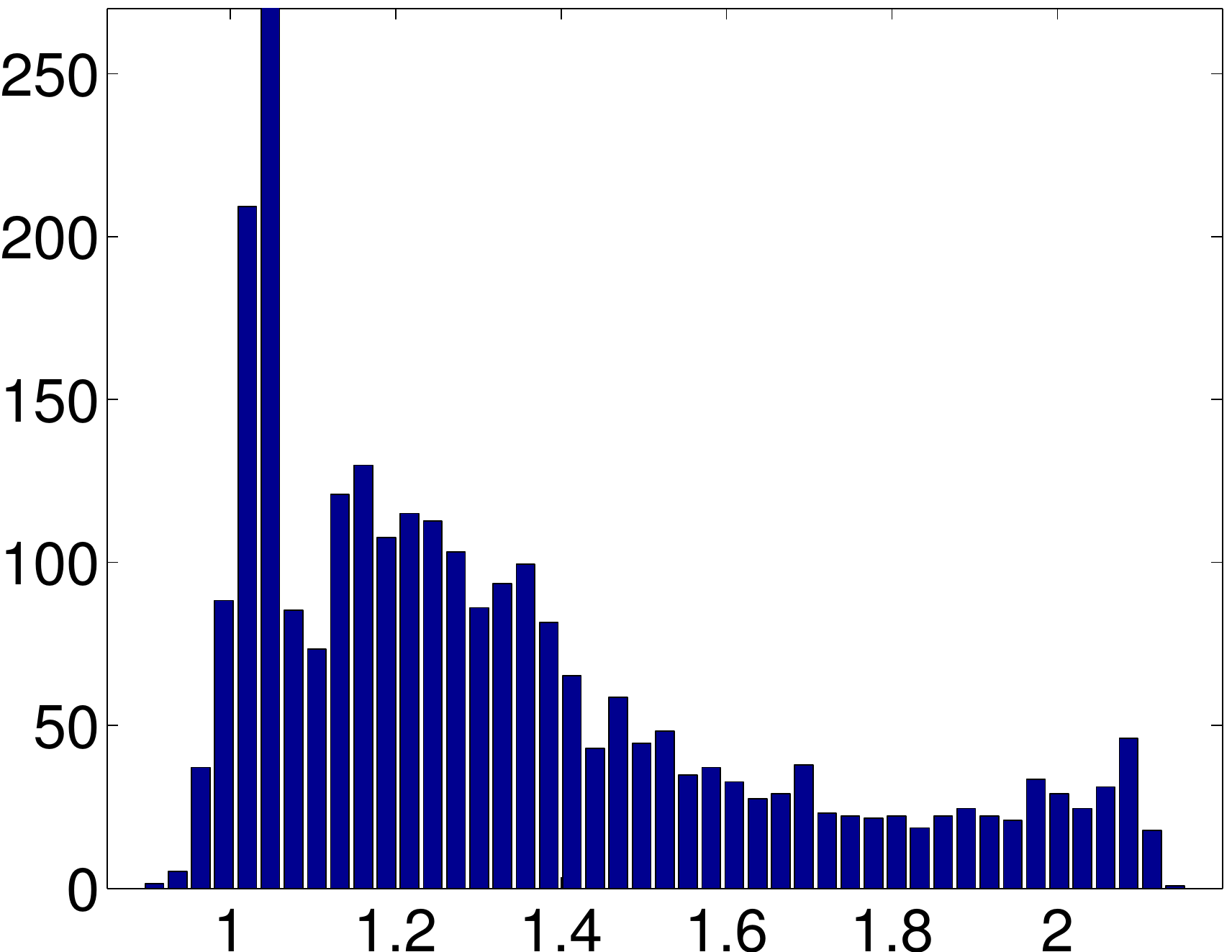}}
\subfigure[nx = 512]{\includegraphics[width=0.24\linewidth]{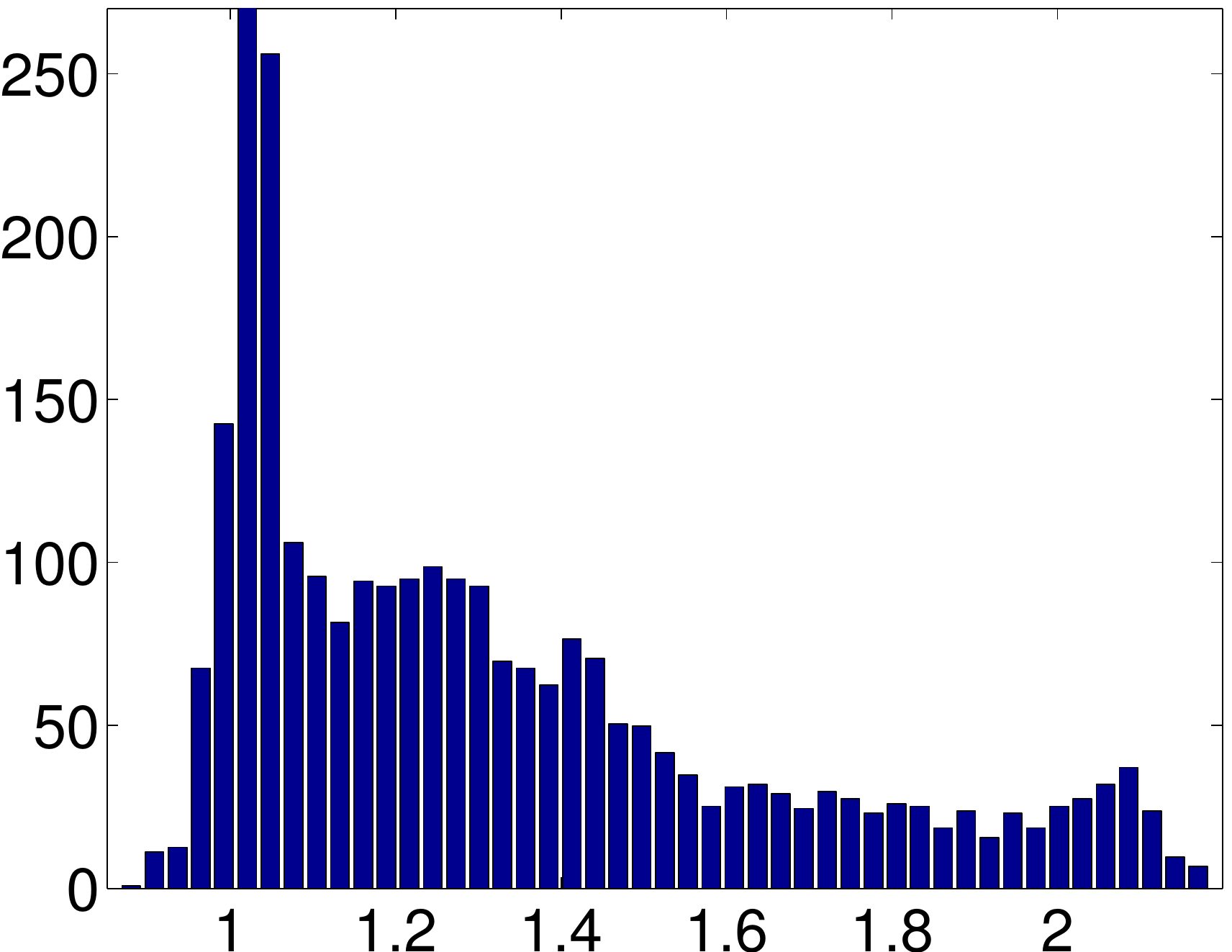}}
\subfigure[nx = 1024]{\includegraphics[width=0.24\linewidth]{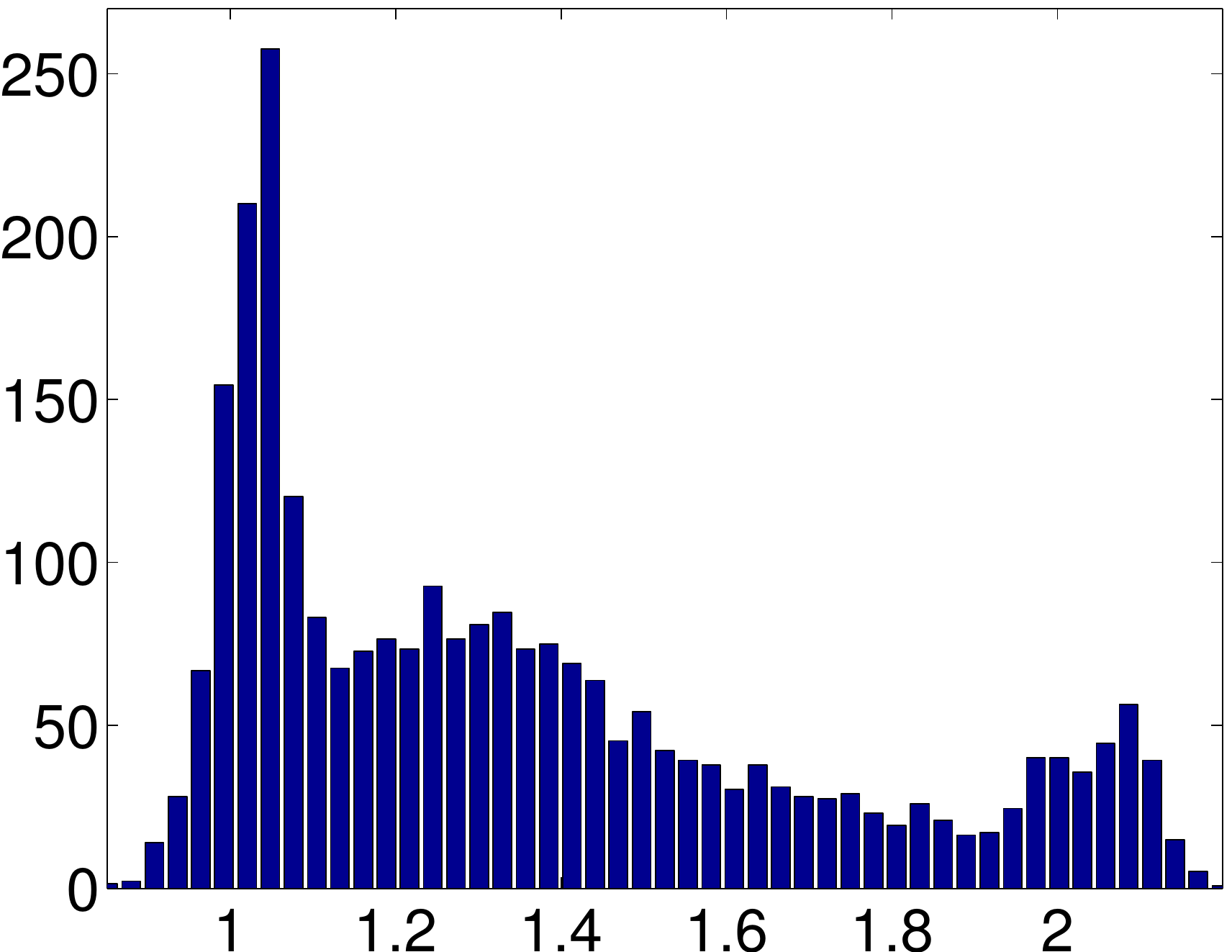}}
\caption{The approximate PDF for density $\rho$ at the points $x=(0.5, 0.7)$ (first row) and $x=(0.5,0.8)$ (second row) on a series of meshes.}
\label{fig:pdfRefine}
\end{figure}


\subsection{Richtmeyer-Meshkov problem}\label{sec:richtmesh}
As a second numerical example, we consider the two-dimensional Euler equations \eqref{eq:2deuler} in the computational domain $x\in[0,1]^2$ with initial data:
\begin{equation}
\label{eq:rminit}
p(x) = \begin{cases}
20 & \text{if $|x-(0.5,0.5)| < 0.1$} \\
1 & \text{otherwise,}
\end{cases}
\qquad
\rho(x) = \begin{cases}
2 & \text{if $|x-(0.5,0.5)| < I(x,\omega)$} \\
1 & \text{otherwise,}
\end{cases}
\qquad \velx = \vely =0.
\end{equation}
The radial density interface $I(x,\omega) = 0.25 + \amp Y(\phi(x),\omega)$
is perturbed with
\begin{equation}
Y(\varphi,\omega) = \sum_{n=1}^m a^n(\omega) \cos \left( \varphi + b^n(\omega) \right),
\end{equation}
where $\varphi(x) = \arccos((x_1 - 1/2)/|x - (0.5,0.5)|)$ and $a_n,b_n,k$ are the same as in Section \ref{sec:khnonatomic}.

We repeat that the computational domain is $[0,1]^2$ and we use periodic boundary conditions in both directions.

\subsubsection{Lack of sample convergence}
As in the case of the Kelvin-Helmholtz problem, we test whether numerical approximations for a single sample converge as the mesh is refined. To this end, we compute the approximations of the two-dimensional Euler equations with initial data \eqref{eq:rminit} using a second-order MUSCL type finite volume scheme, based on the HLLC solver, and implemented in the FISH code \cite{Kap1}. The numerical results, presented in Figure \ref{fig:19}, show the effect of grid refinement on the density for a single sample at time $t=4$. Note that by this time, the leading shock wave has exited the domain but has reentered from the corners on account of the periodic boundary conditions. Furthermore, this reentry shock wave interacts and strongly perturbs the interface forming a very complex region of small scale eddy like structures. As seen from Figure \ref{fig:19}, there seems to be no convergence as the mesh is refined. This lack of convergence is quantified in Figure \ref{fig:20}, where we present differences in $L^1$ for successive mesh resolutions \eqref{eq:chyrates} and see that the approximate solutions for a single sample do not form a Cauchy sequence.
\begin{figure}
\centering
\subfigure[$128^2$]{\includegraphics[width=0.45\linewidth]{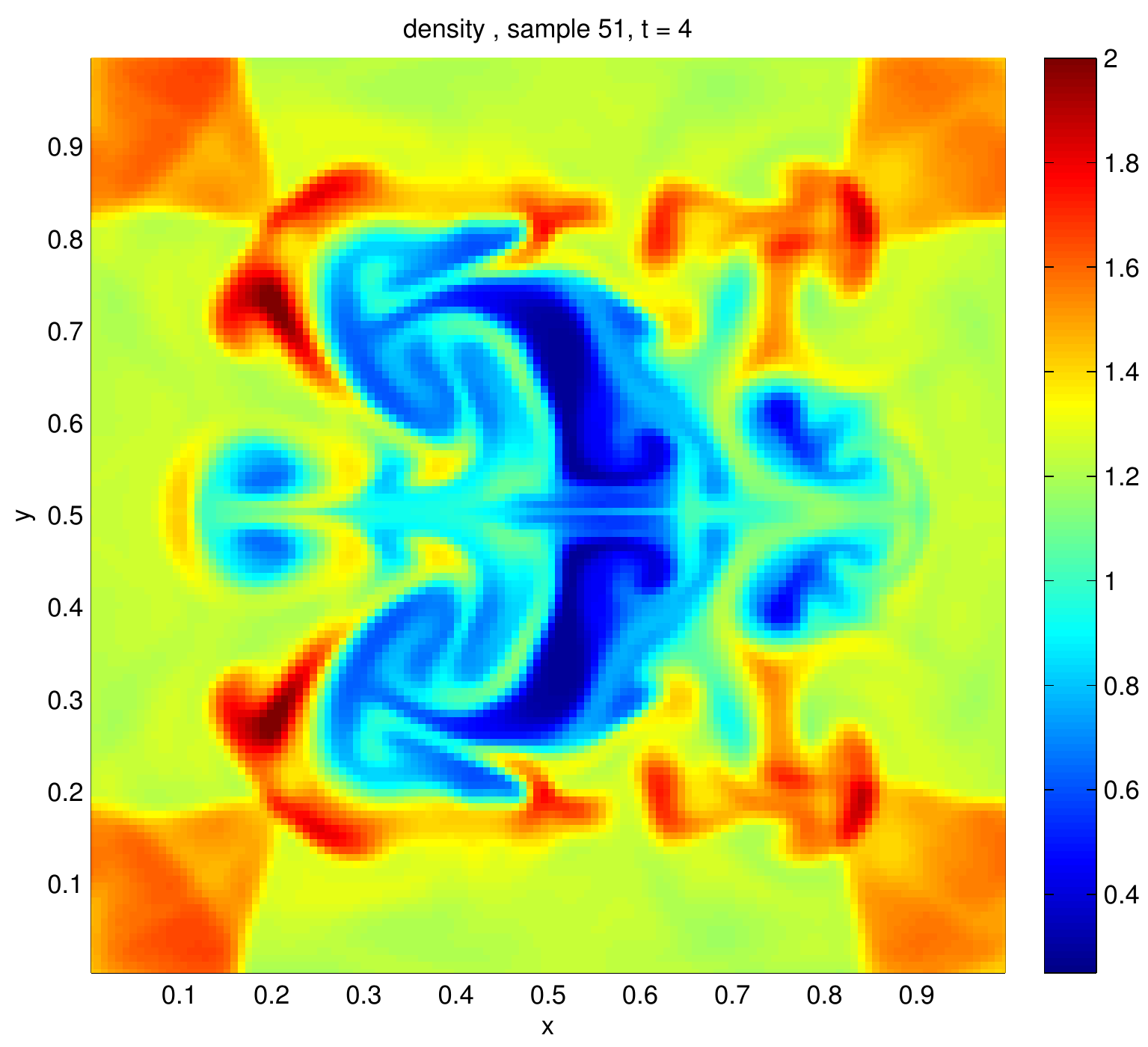}}
\subfigure[$256^2$]{\includegraphics[width=0.45\linewidth]{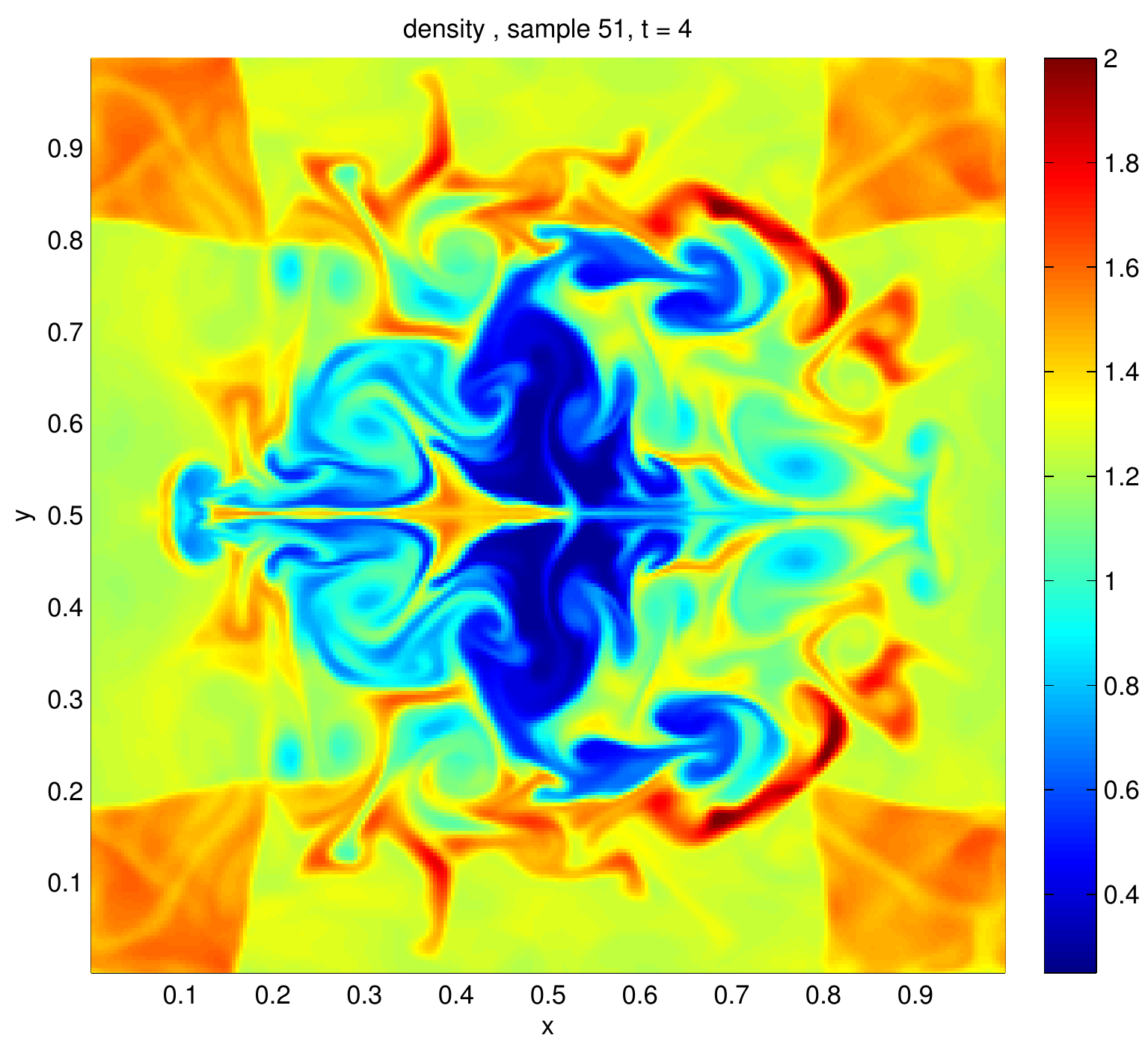}}\\
\subfigure[$512^2$]{\includegraphics[width=0.45\linewidth]{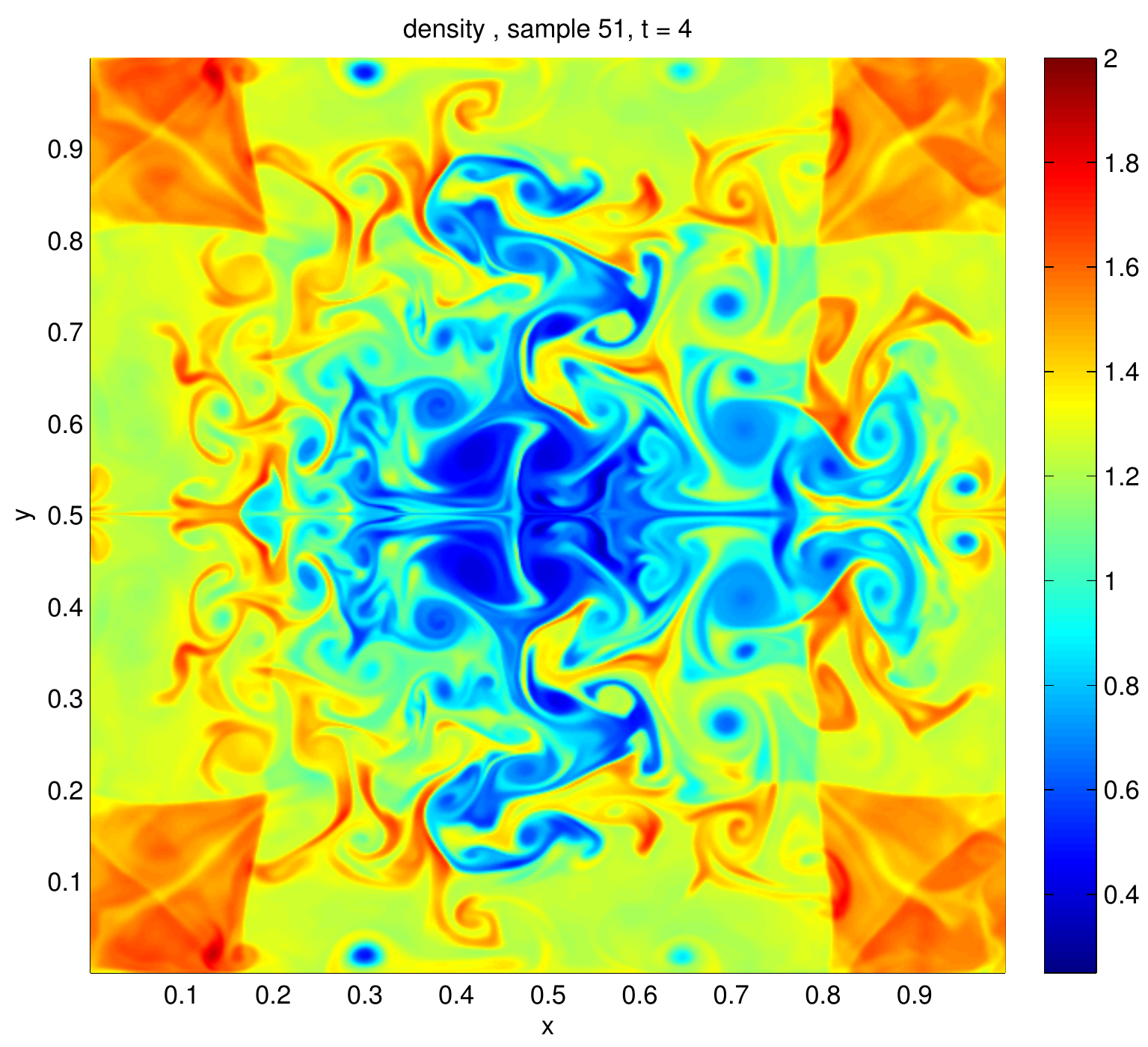}}
\subfigure[$1024^2$]{\includegraphics[width=0.45\linewidth]{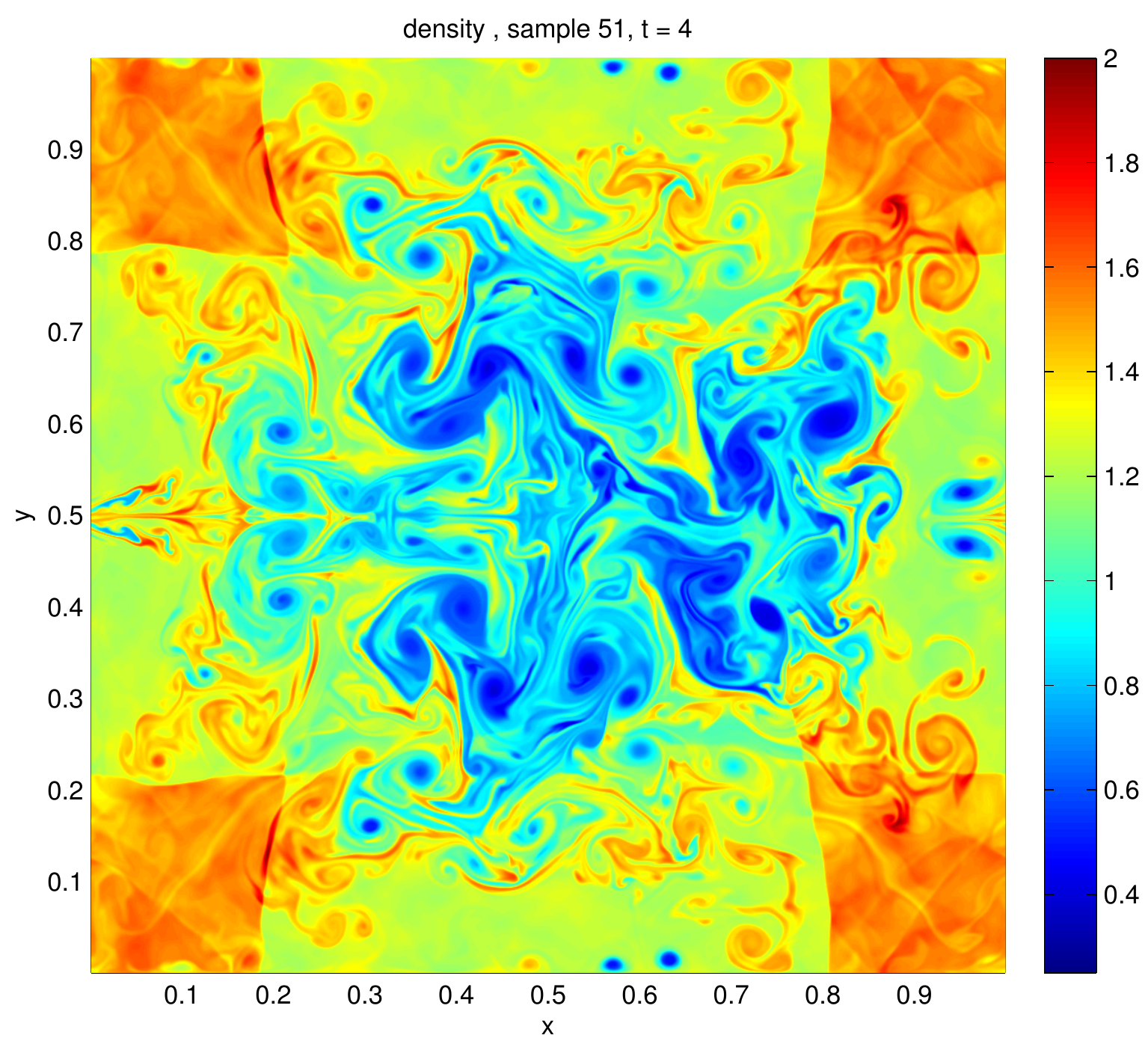}}
\caption{Approximate density for a single sample for the Richtmeyer-Meshkov problem \eqref{eq:rminit} for different grid resolutions at time $t=4$.}
\label{fig:19}
\end{figure}
\begin{figure}
\centering
\includegraphics[width=6cm]{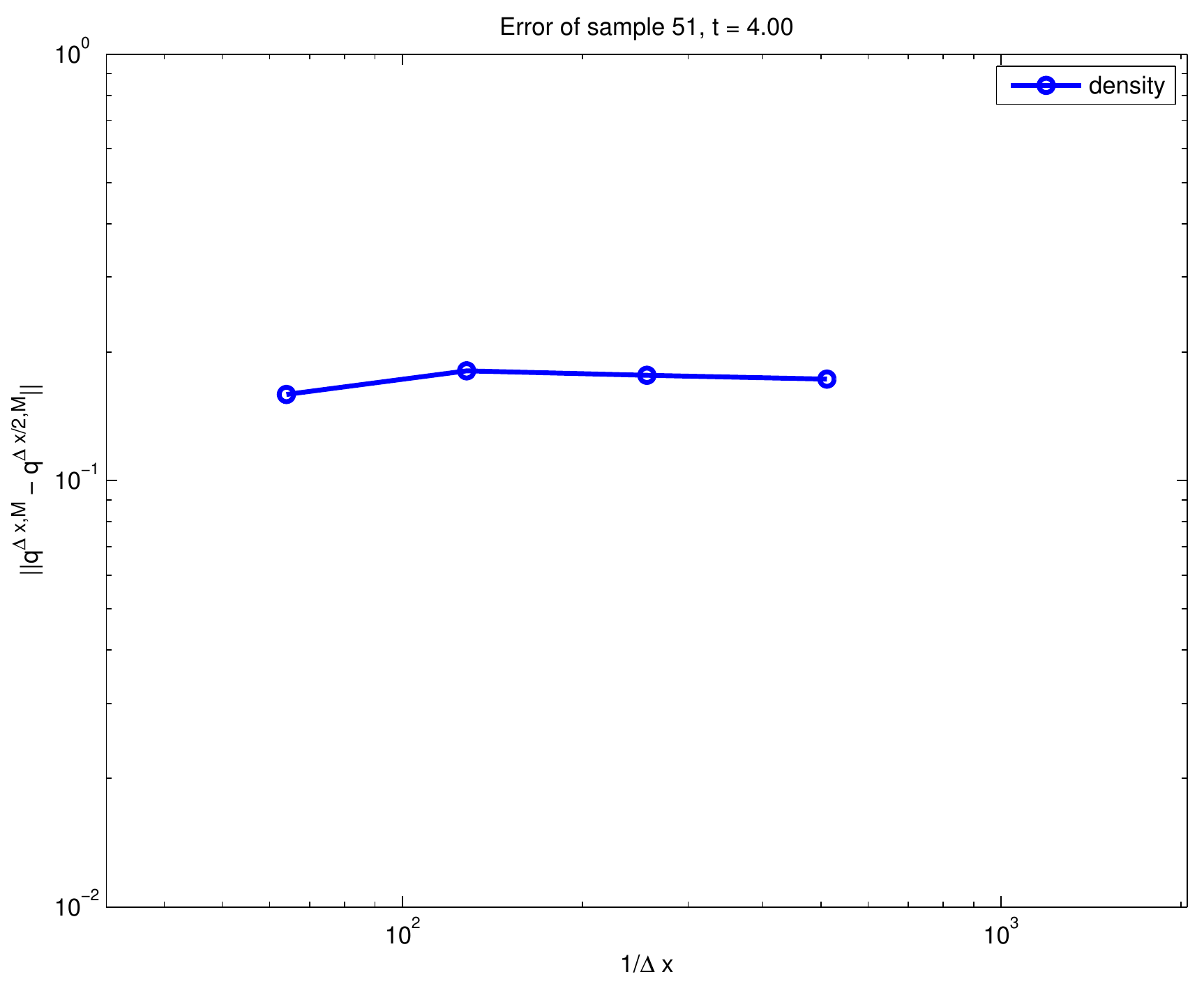}
\caption{Cauchy rates \eqref{eq:chyrates} for the density ($y$-axis) in a single sample of the Richtmeyer-Meshkov problem \eqref{eq:rminit} at time $t=4$, with respect to different grid resolutions ($x$-axis).}
\label{fig:20}
\end{figure}
\subsubsection{Convergence of the mean and the variance}
Next, we test for convergence of statistical quantities of interest as the mesh is refined. First, we check the convergence of the mean through the Monte Carlo approximation \eqref{eq:mc1} with $M=400$ samples. The numerical results for the density at time $t=4$ at different grid resolutions are presented in Figure \ref{fig:21}. The figure clearly shows that the mean converges as the mesh is refined. This convergence is further verified in Figure \ref{fig:2224}(a) where we plot the difference in mean \eqref{eq:cr2} for successive resolutions. This figure proves that the mean of the approximations form a Cauchy sequence and hence, converge. From Figure \ref{fig:21}, we also observe that small scale features are averaged out in the mean and only large scale structures, such as the strong reentrant shocks (mark the periodic boundary conditions) and mixing regions, are retained through the averaging process.
\begin{figure}
\centering
\subfigure[$128^2$]{\includegraphics[width=0.45\linewidth]{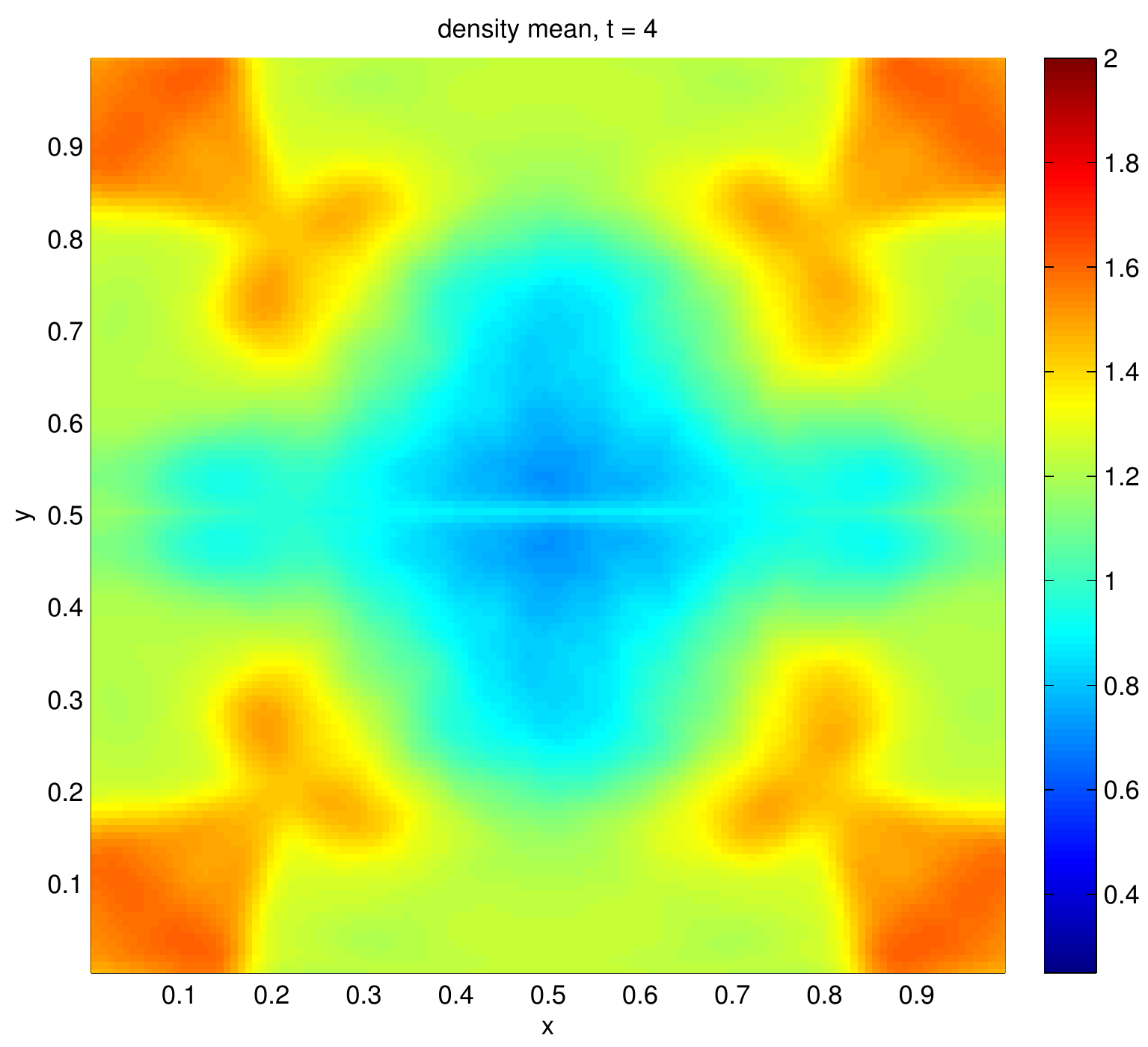}}
\subfigure[$256^2$]{\includegraphics[width=0.45\linewidth]{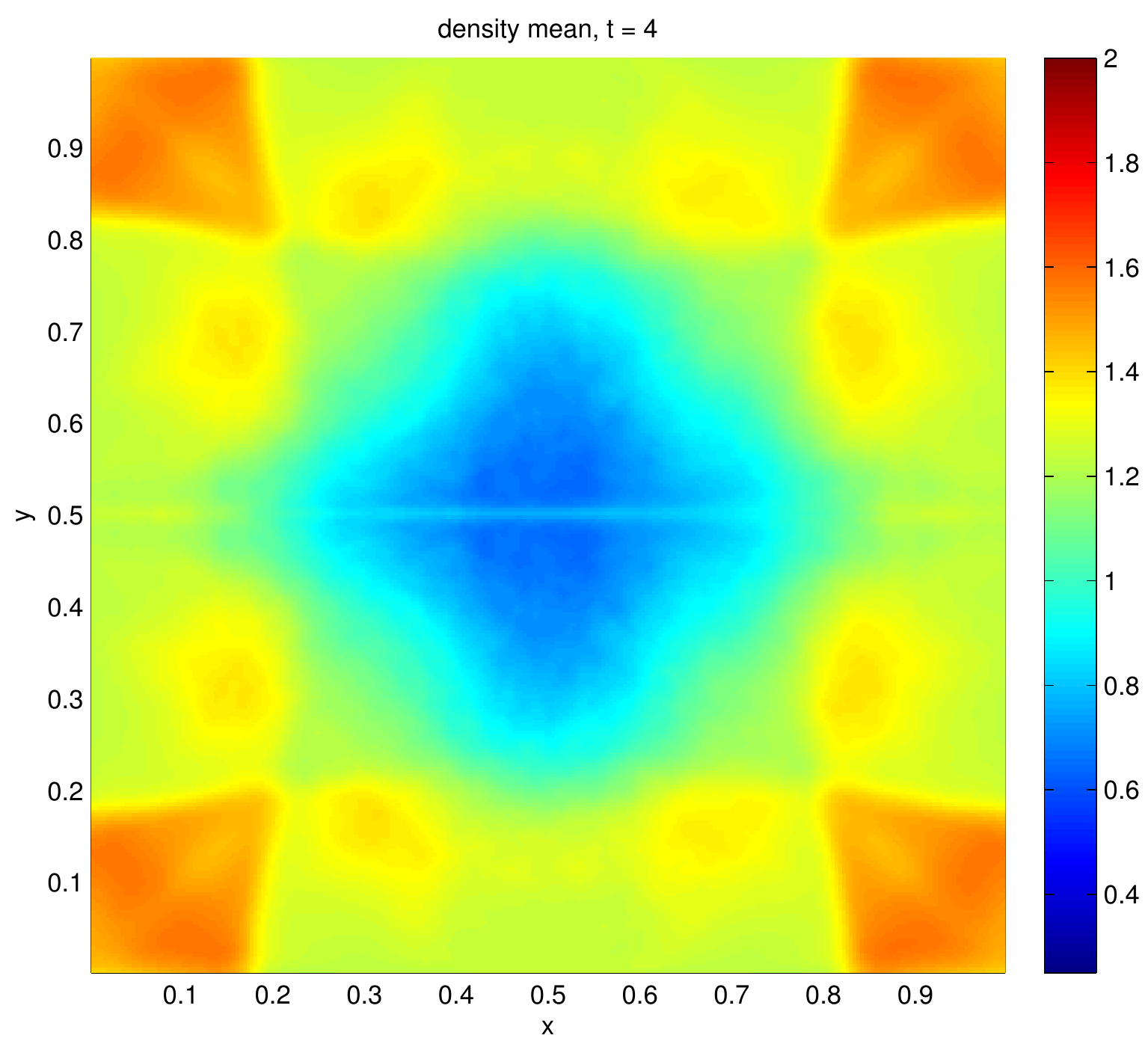}}\\
\subfigure[$512^2$]{\includegraphics[width=0.45\linewidth]{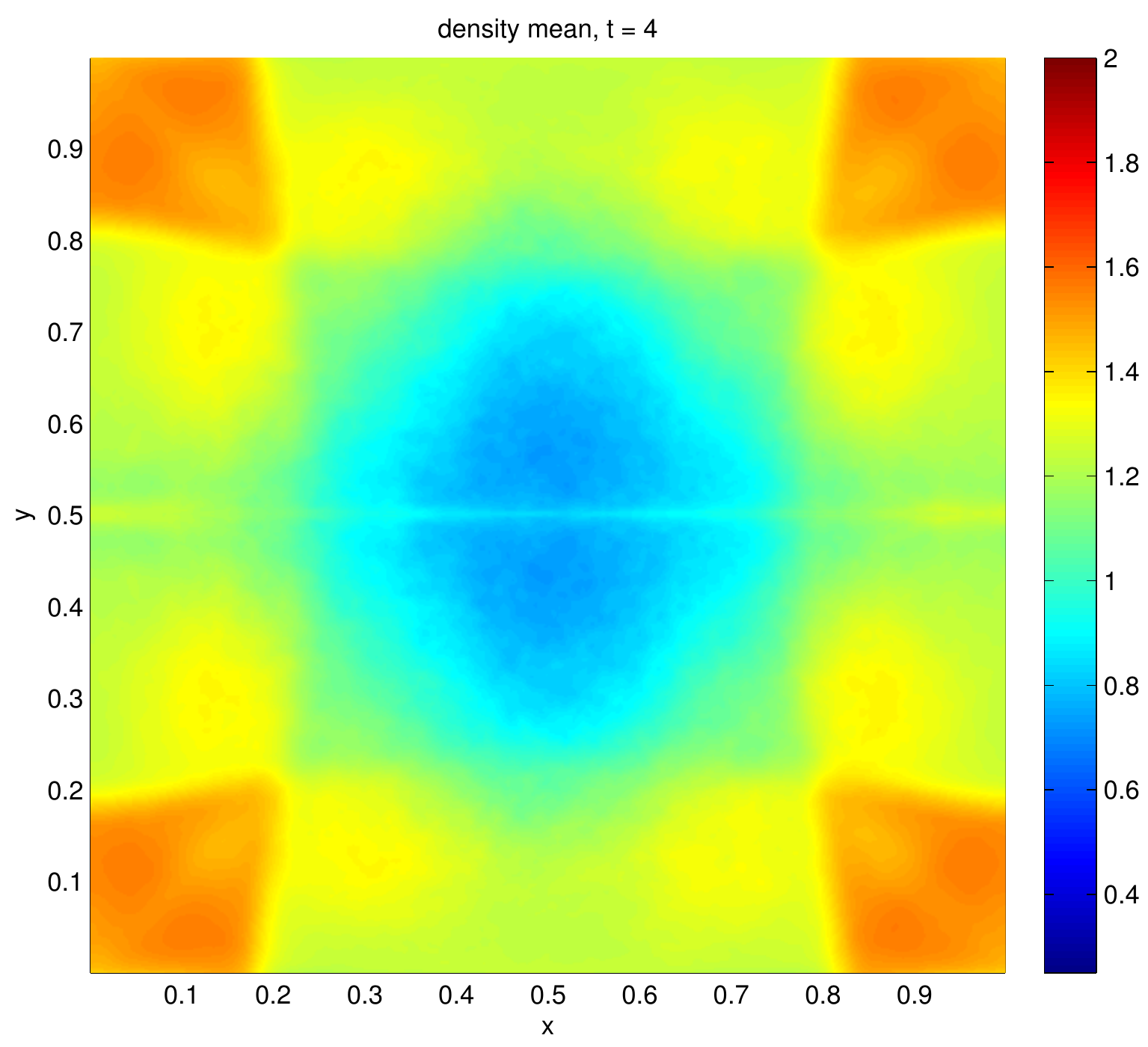}}
\subfigure[$1024^2$]{\includegraphics[width=0.45\linewidth]{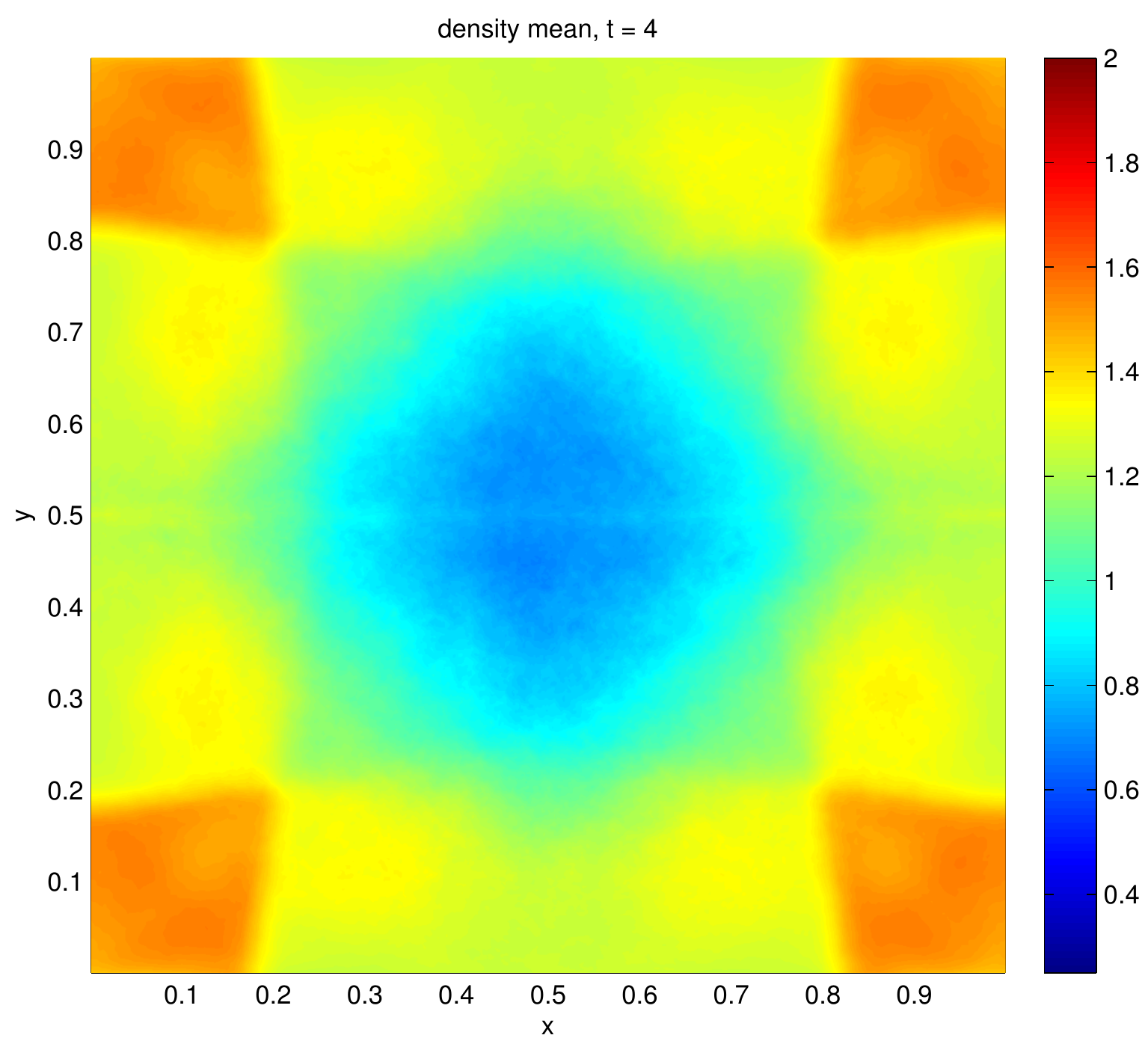}}
\caption{The mean density for the Richtmeyer-Meshkov problem with initial data \eqref{eq:rminit} for different grid resolutions at time $t=4$. All results are obtained with $400$ Monte Carlo samples.}
\label{fig:21}
\end{figure}
\begin{figure}
\centering
\subfigure[Mean]{\includegraphics[width=6cm]{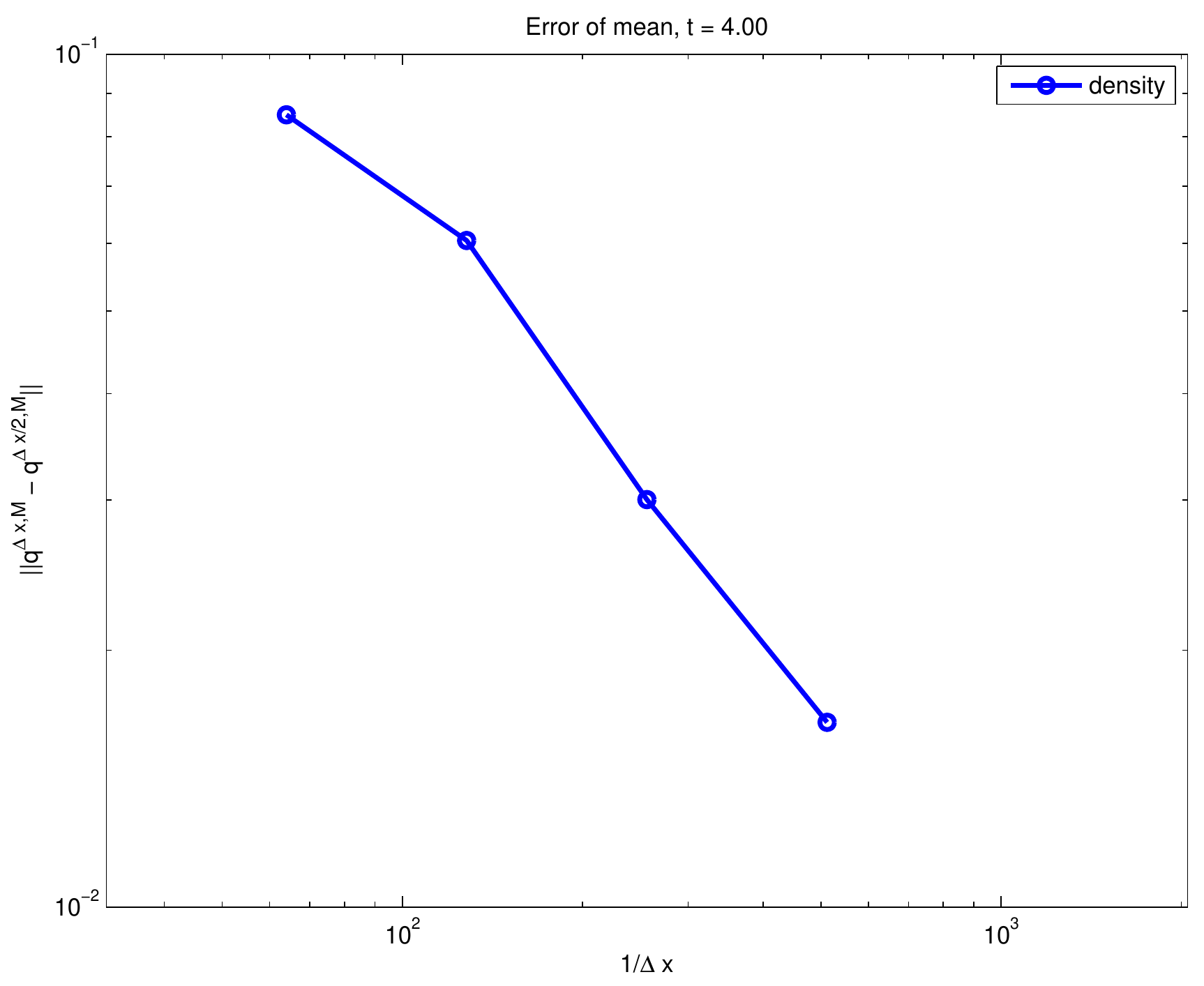} }
\subfigure[Variance]{\includegraphics[width=6cm]{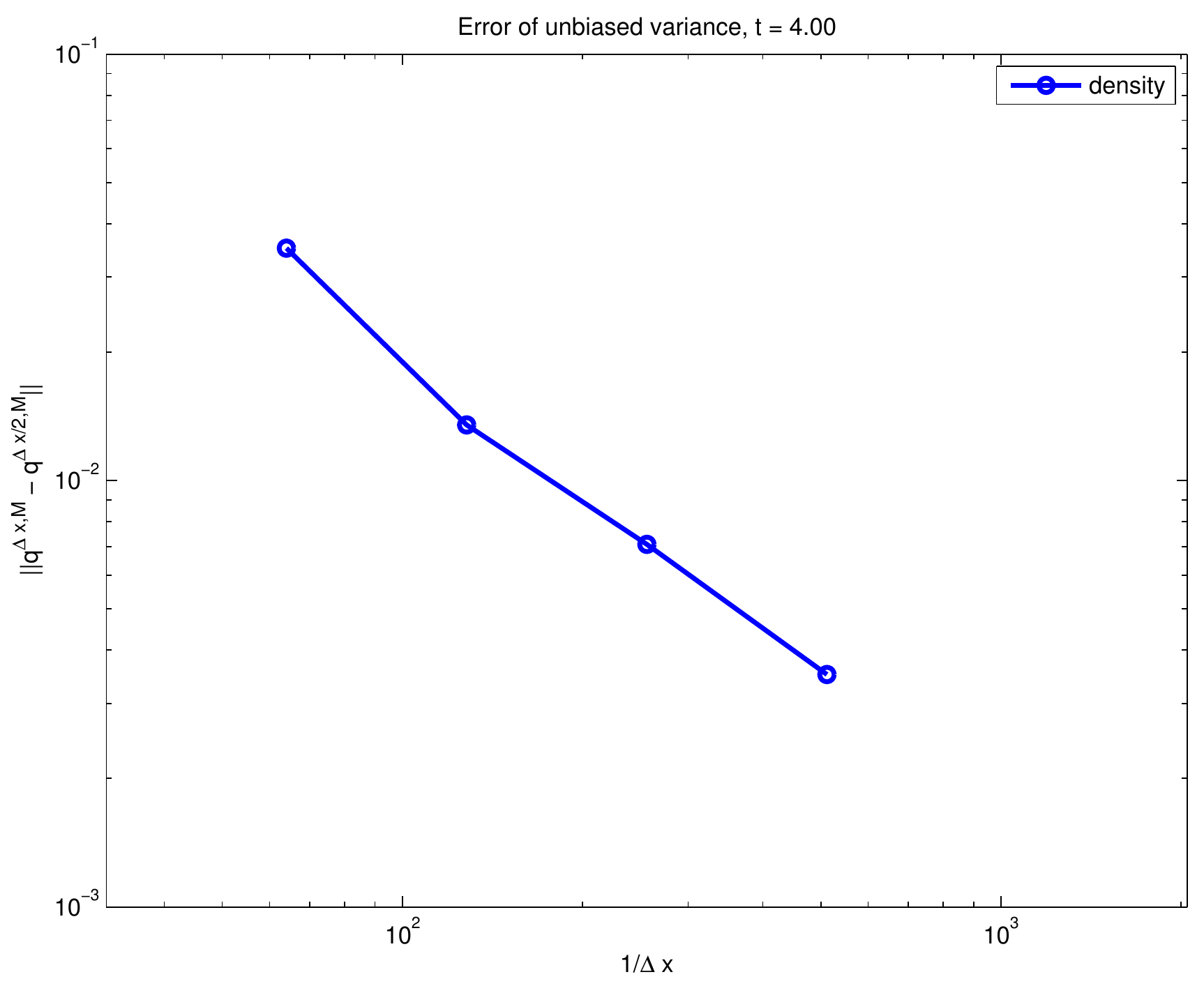}}
\caption{Cauchy rates \eqref{eq:cr2} for the mean and variance ($y$-axis) of the Richtmeyer-Meshkov problem \eqref{eq:rminit} at time $t=4$ and at different grid resolutions ($x$-axis). All results are obtained with $400$ Monte Carlo samples.}
\label{fig:2224}
\end{figure}

Next, we check for the convergence of the variance for the Richtmeyer-Meskhov problem \eqref{eq:rminit}. The results, shown in Figure \ref{fig:23} for time $t=4$, at different mesh resolutions and with $400$ Monte Carlo approximations, clearly indicate that the variance of the approximate Young measures converge as the mesh is refined. This is also verified from Figure \ref{fig:2224}(b) where the difference in $L^1$ of the variances at successive mesh resolutions is plotted and shown to form a Cauchy sequence. Furthermore, Figure \ref{fig:23} also demonstrates that the variance is concentrated at the shocks and even more so in the mixing layer, around the original interface. 
\begin{figure}
\centering
\subfigure[$128^2$]{\includegraphics[width=0.45\linewidth]{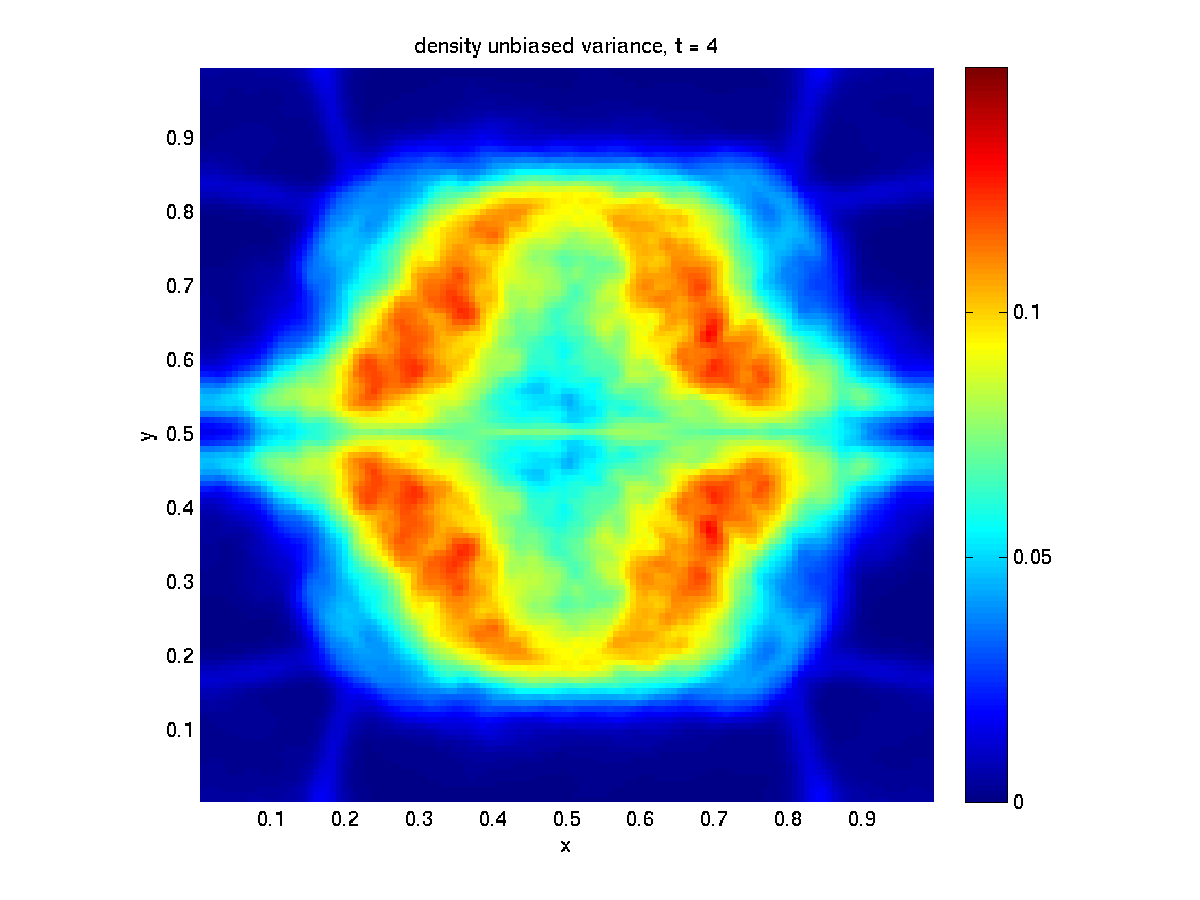}}
\subfigure[$256^2$]{\includegraphics[width=0.45\linewidth]{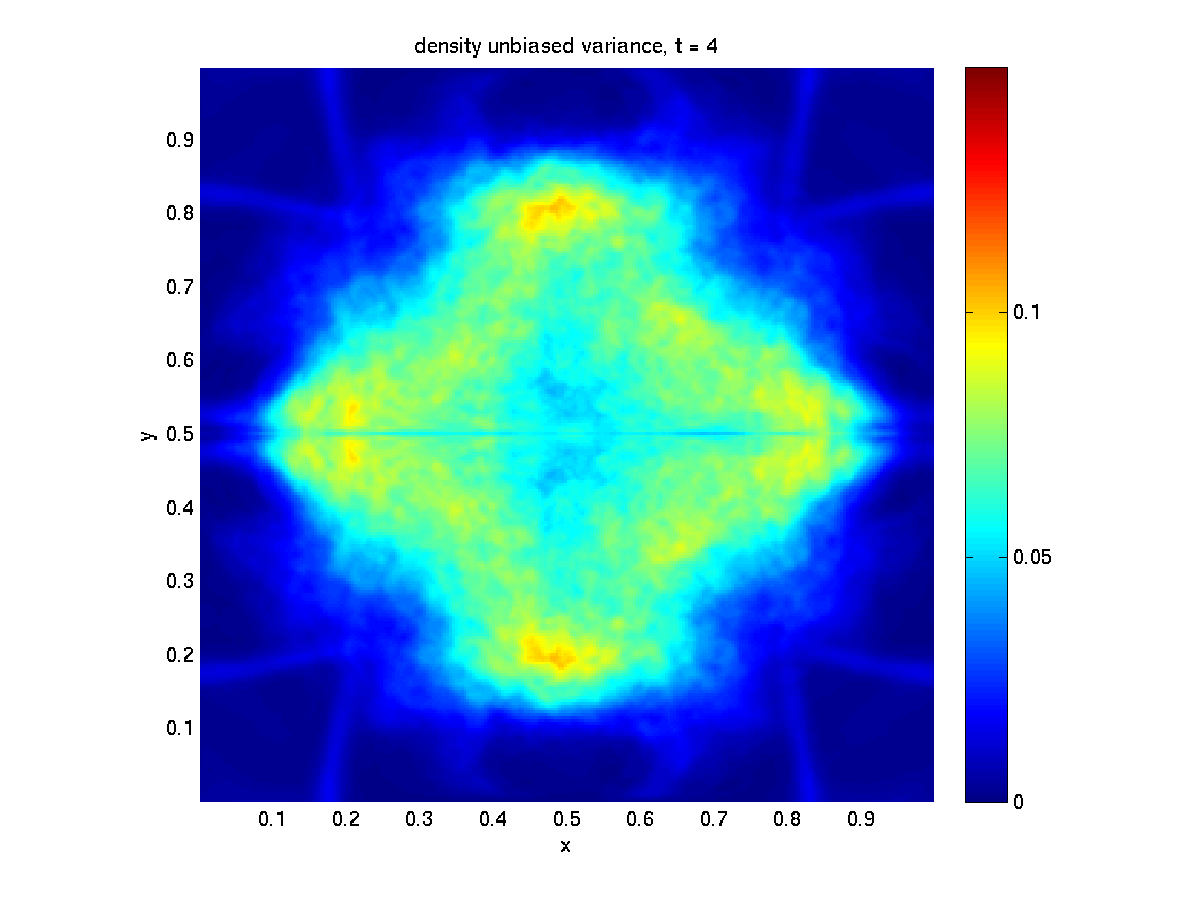}}\\
\subfigure[$512^2$]{\includegraphics[width=0.45\linewidth]{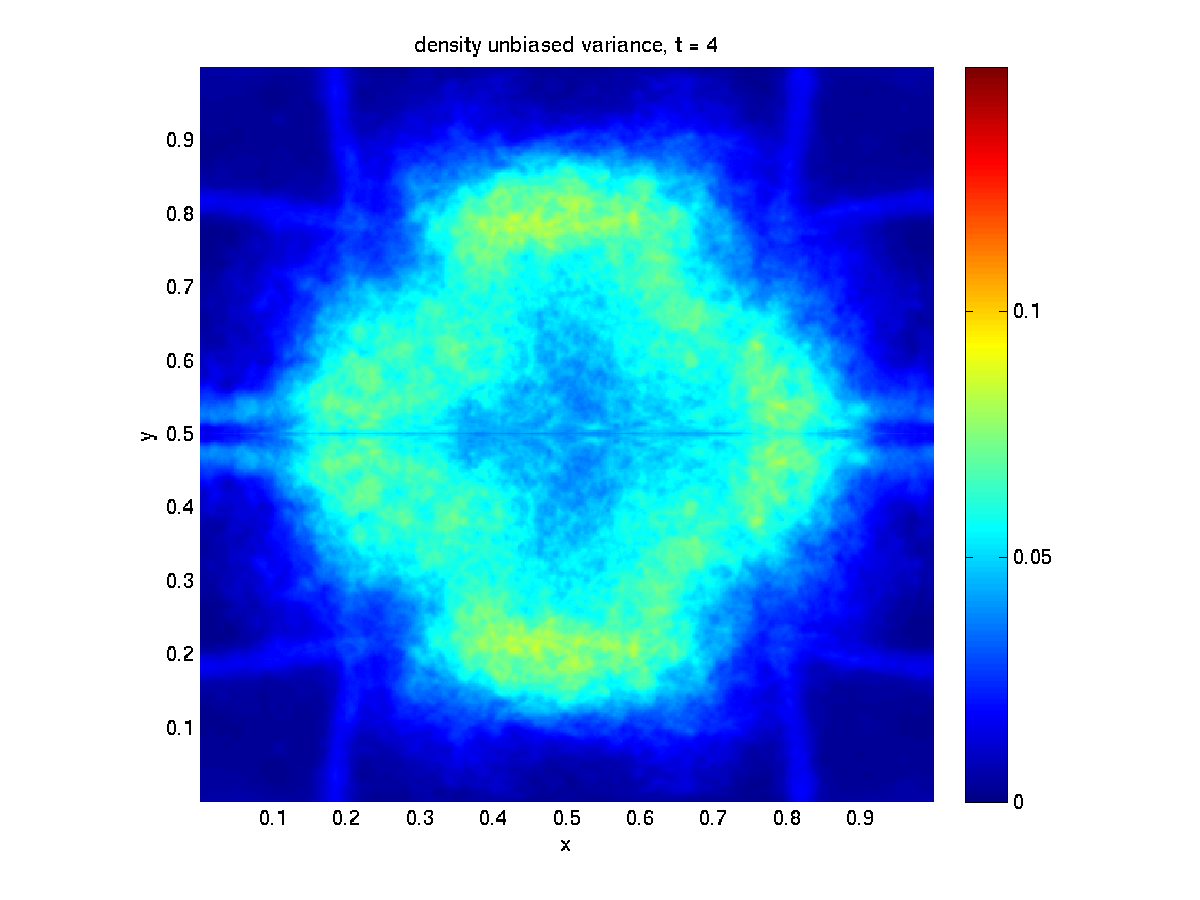}}
\subfigure[$1024^2$]{\includegraphics[width=0.45\linewidth]{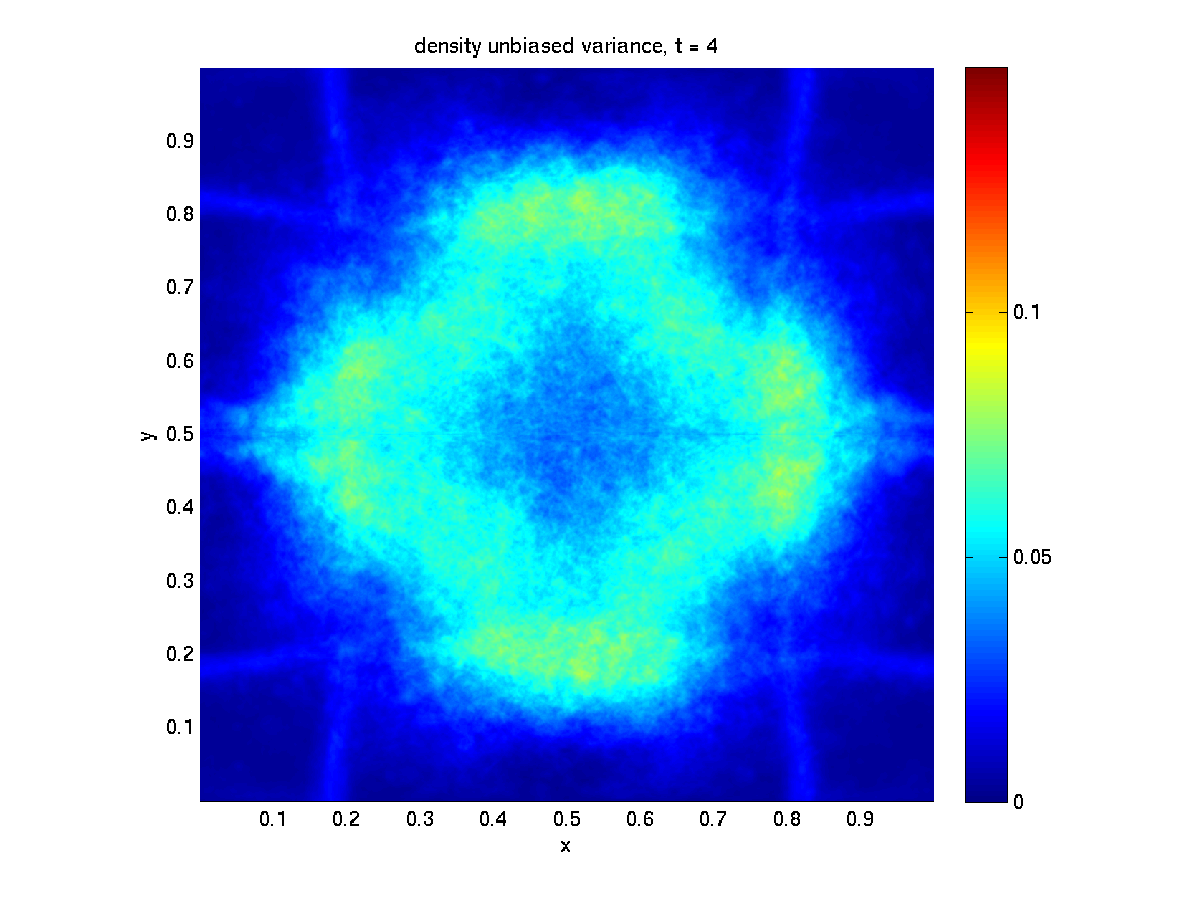}}
\caption{Variance of the density with initial data \eqref{eq:rminit} for different grid resolutions at time $t=4$. All results are obtained with $400$ Monte Carlo samples.}
\protect \label{fig:23}
\end{figure}

\subsection{Measure valued (MV) stability}
\label{sec:mvstab}
The above experiments clearly illustrate that the numerical procedure proposed here does succeed in computing an EMV solution of the underlying systems of conservation laws \eqref{eq:mvcauchy}. Are the computed solutions stable? As argued in Section \ref{sec:3}, uniqueness (stability) of EMV solutions for a general measure valued initial data is not necessarily true, even for scalar conservation laws. Moreover, the scalar case suggests that at most a weaker concept of stability, that of MV stability can be expected for EMV solutions (see Terminology \ref{def:mvstab}). As stated before, MV stability amounts to stability with respect to perturbations of atomic initial data. We examine this weaker notion of stability through numerical experiments.

To this end, we consider the Kelvin-Helmholtz problem as our test bed and investigate stability with respect to the following perturbations:

\subsubsection{Stability with respect to different numerical schemes.}
As a first check of MV stability, we consider the perturbed Kelvin-Helmholtz initial data \eqref{eq:kh} with a fixed perturbation size $\amp = 0.01$ and compute approximate measure valued solutions using Algorithm \ref{alg:approxmv}. Three different schemes are compared:
 \begin{itemize}
 \item [1.] (Formally) second-order TeCNO2 scheme of \cite{FTSID2}.
 \item [2.] Third-order TeCNO3 scheme of \cite{FTSID2}.
 \item [3.] Second-order high-resolution finite volume scheme, based on the HLLC approximate Riemann solver, and implemented in the FISH code \cite{Kap1}.
 \end{itemize}
 We will compare the mean and the variance of the approximate measures, at a resolution of $1024^2$ points and $400$ Monte Carlo samples, at time $t=2$. As the mean and the variance with TeCNO2 scheme have already been depicted in Figures \ref{fig:5}(d) and \ref{fig:7}(d), respectively, we plot the mean and variance with the TeCNO3 and FISH schemes in Figure \ref{fig:mvs1}. These results, together with the results for the TeCNO2 scheme (Figures \ref{fig:5}(d) and \ref{fig:7}(d)) clearly show that mean and variance of the approximate measure valued solution are very similar even though the underlying approximation schemes are different. In particular, comparing the TeCNO2 and TeCNO3 schemes, we remark that although both schemes have the same design philosophy (see \cite{FTSID2} and Section \ref{sec:5}), their formal order of accuracy is different. Hence, the underlying numerical viscosity operators are different. In spite of different numerical regularizations, both schemes seem to be converging to the same measure valued solution -- at least in terms of its first and second moments. This agreement is even more surprising for the FISH scheme of \cite{Kap1}. This scheme utilizes a very different design philosophy based on HLLC approximate Riemann solvers and an MC slope limiter. Furthermore, it is unclear whether this particular scheme satisfies the discrete entropy inequality \eqref{eq:dentrineq2} or the weak BV bound \eqref{eq:tvbound2}. Nevertheless, the measure valued solutions computed by this scheme seem to converge to the same EMV solution as computed by the TeCNO schemes. We have observed similar agreement between different schemes for smaller values of the perturbation parameter $\amp$ as well as in the Richtmeyer-Meshkov problem. Furthermore, all the three schemes agree with respect to higher moments as well. These numerical results at least indicate MV stability with respect to different numerical discretizations.
\begin{figure}
\centering
\subfigure[Mean, TeCNO3]{\includegraphics[width=0.45\linewidth]{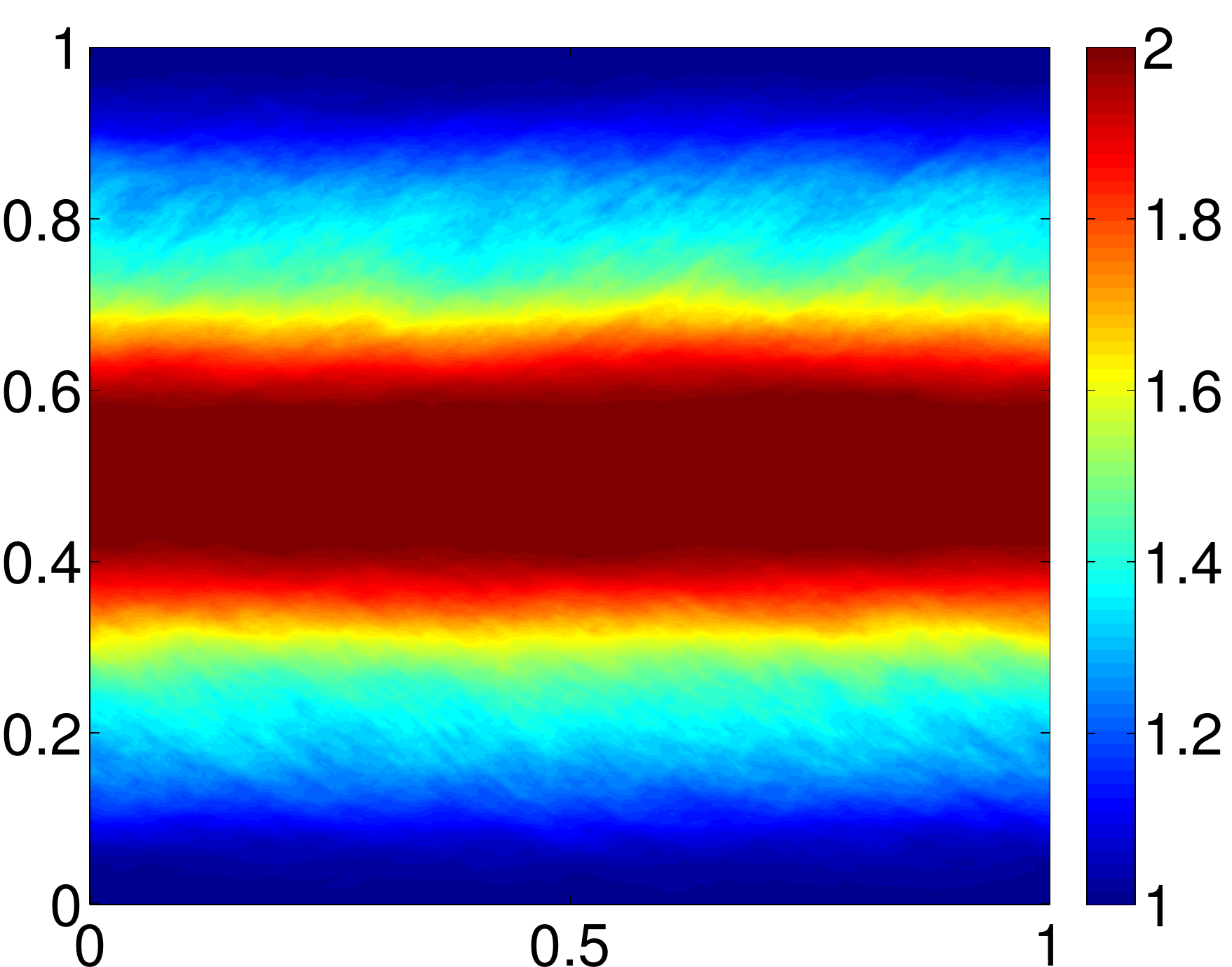}}
\subfigure[Mean, FISH]{\includegraphics[width=0.45\linewidth]{khif_mean_t=2_1024}}\\
\subfigure[Variance, TeCNO3]{\includegraphics[width=0.45\linewidth]{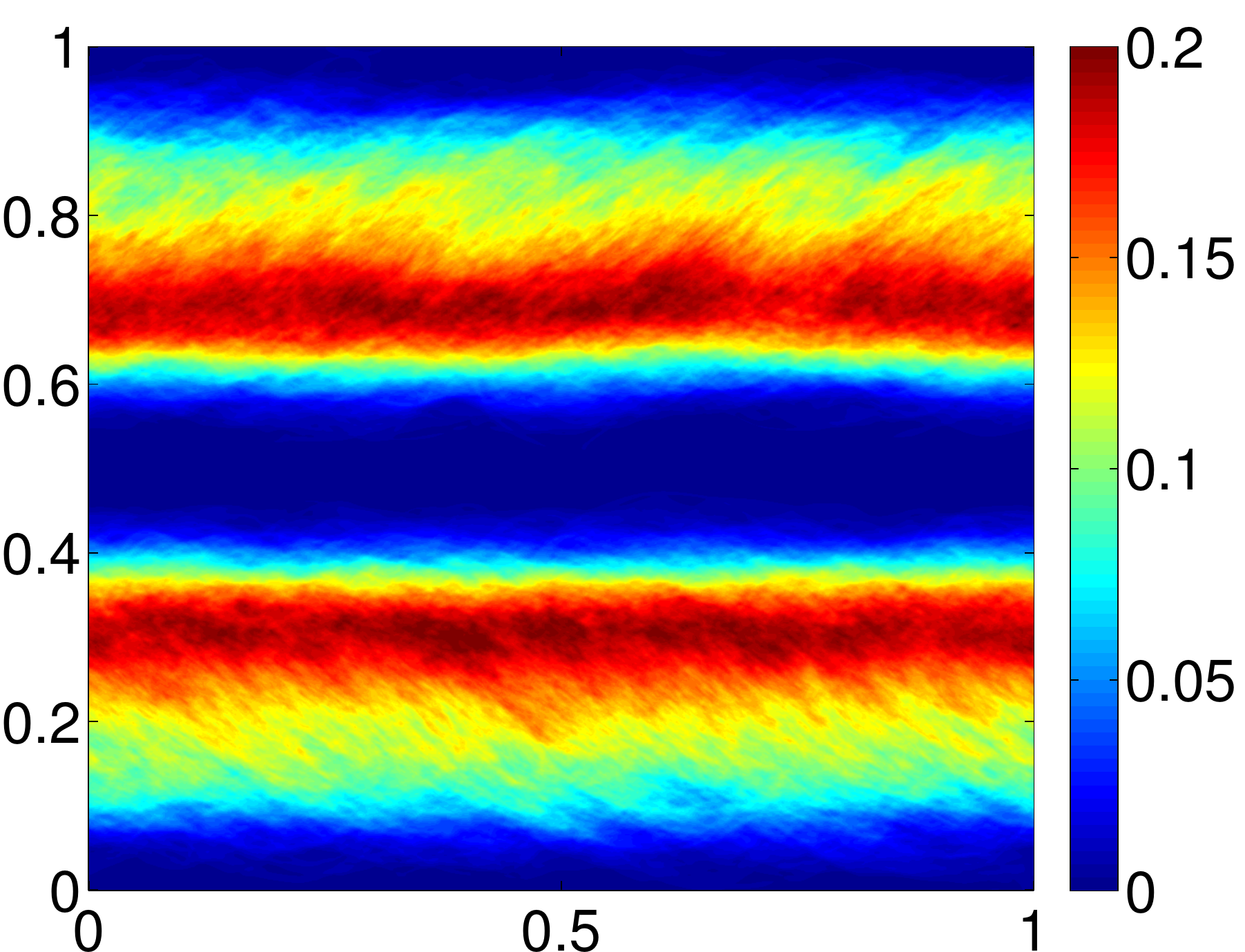}}
\subfigure[Variance, FISH]{\includegraphics[width=0.45\linewidth]{khif_var_t=2_1024}}
\caption{Mean and variance of the density for the Kelvin-Helmholtz problem with initial data \eqref{eq:kh}, at time $t=2$ at a resolution of $1024^2$ points and with $200$ Monte Carlo samples. Different numerical schemes are compared.}
\protect \label{fig:mvs1}
\end{figure}
\subsubsection{MV stability with respect to different perturbations}
A more stringent test of MV stability is with respect to different types of initial perturbations. To be more specific, we consider the Kelvin-Helmholtz problem with the \emph{phase} perturbations of \eqref{eq:kh} and compare them with \emph{amplitude} perturbations \eqref{eq:sodpinit} and \eqref{eq:khi}. Note that for small values of the perturbation parameter $\amp$, both the amplitude and phase perturbations are close to the atomic initial data \eqref{eq:khi} and to one another (for instance in the Wasserstein metric). We test whether the resulting approximate MV solutions are also close. To this end, we compute the approximate measure valued solutions with the phase perturbation and amplitude perturbation, for $\amp = 0.0005$, with the TeCNO3 scheme, at a grid resolution of $1024^2$ points and 400 Monte Carlo samples, and plot the results in Figure \ref{fig:mvs2}. The results show that the mean and variance with different initial perturbations are very similar when the amplitude $\amp$ of the perturbations is small. 

An even more stringent test of stability is provided by the following \emph{phase} perturbation of the Kelvin-Helmholtz problem \eqref{eq:kh}. The same set-up (computational domain of $[0,1]^2$ and periodic boundary conditions) as in the description of \eqref{eq:kh} is used but with an interface perturbation of the form:
\begin{equation}
\label{eq:genper}
I_j = I_j(x_1,\omega) := J_j + \amp Y_j(x_1,\omega).
\end{equation}
As in \eqref{eq:kh}, we set $J_1 = 0.25$ and $J_2 = 0.75$ but with an interface variation of the form:
\begin{equation}
\label{eq:genper1}
Y_j(x_1,\omega) = \sum_{n=1}^k a_j^n \mathbbm{1}_{A_n} \; , \; j=1,2.
\end{equation}
Here, $a_j^n = a_j^n(\omega) \in [-\hf,\hf]$ are randomly chosen numbers from a uniform distribution. As a second variant, the $a_j^n$ are drawn from the standard normal distribution. The $A_n$ are equally spaced intervals, i.e. $A_n = [(n-1)h,nh)$ with $h = \unitfrac{1}{32}$. Thus, the initial interface perturbation is discontinuous, with uncorrelated random variation of the interface inside each interval. Such types of random initial data are motivated from observed or measured data, see \cite{sssid3}. A representative realization of this initial datum is shown in Figure \ref{fig:kh2d_ini} (right). 

The resulting approximate MV solutions, computed with a perturbation of size $\amp = 0.005$, at time $t=2$ and at a resolution of $1024^2$ are shown in Figure \ref{fig:mvs3}. The mean (top) and variance (bottom) are plotted. Results with the coefficients $a_j^n$, chosen from both an uniform distribution (left) as well as a standard normal distribution (right) are shown. As seen from the figure, the computed mean appears identical for the two choices of distributions. The same holds for the variance, where the resulting variances for both sets of distributions are very similar. Furthermore, they are also very similar to the corresponding statistical quantities, computed with the amplitude perturbation  \eqref{eq:sodpinit} and \eqref{eq:khi} as well as the sinusoidal phase perturbation \eqref{eq:kh} (compare with Figure \ref{fig:mvs2}). Thus, we observe that the computed MV solutions are very similar to each other, even for four different sets of initial perturbations. Similar results were also observed for higher moments. This clearly indicates MV stability of the computed MV solution with the Kelvin-Helmholtz initial data.

\begin{figure}
\centering
\subfigure[Mean, phase perturbation]{\includegraphics[width=0.45\linewidth]{khif_mean_1}}
\subfigure[Mean, amplitude perturbation]{\includegraphics[width=0.45\linewidth]{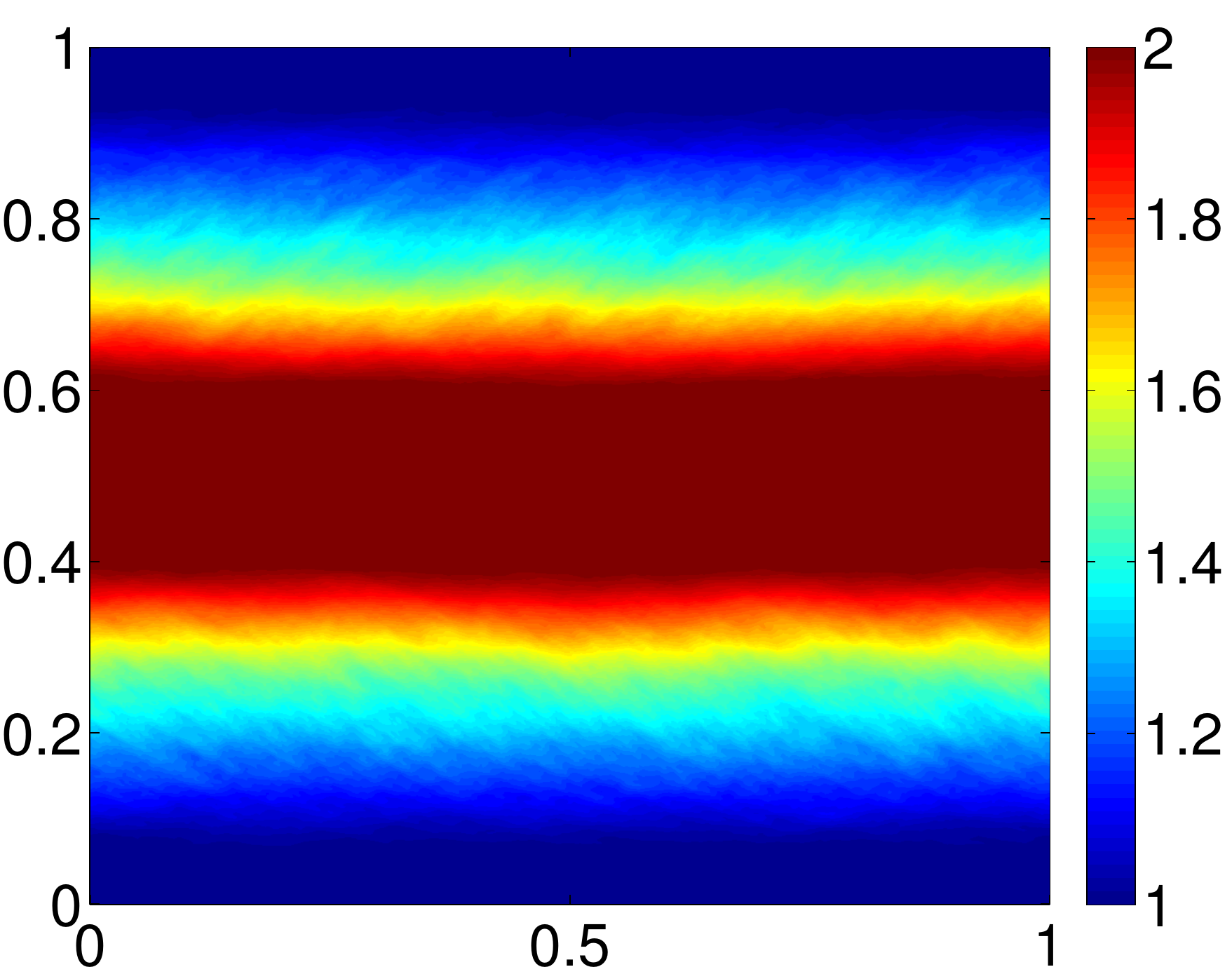}}\\
\subfigure[Variance, phase perturbation]{\includegraphics[width=0.45\linewidth]{khif3_var_t=2_1e-2_1024}}
\subfigure[Variance, amplitude perturbation]{\includegraphics[width=0.45\linewidth]{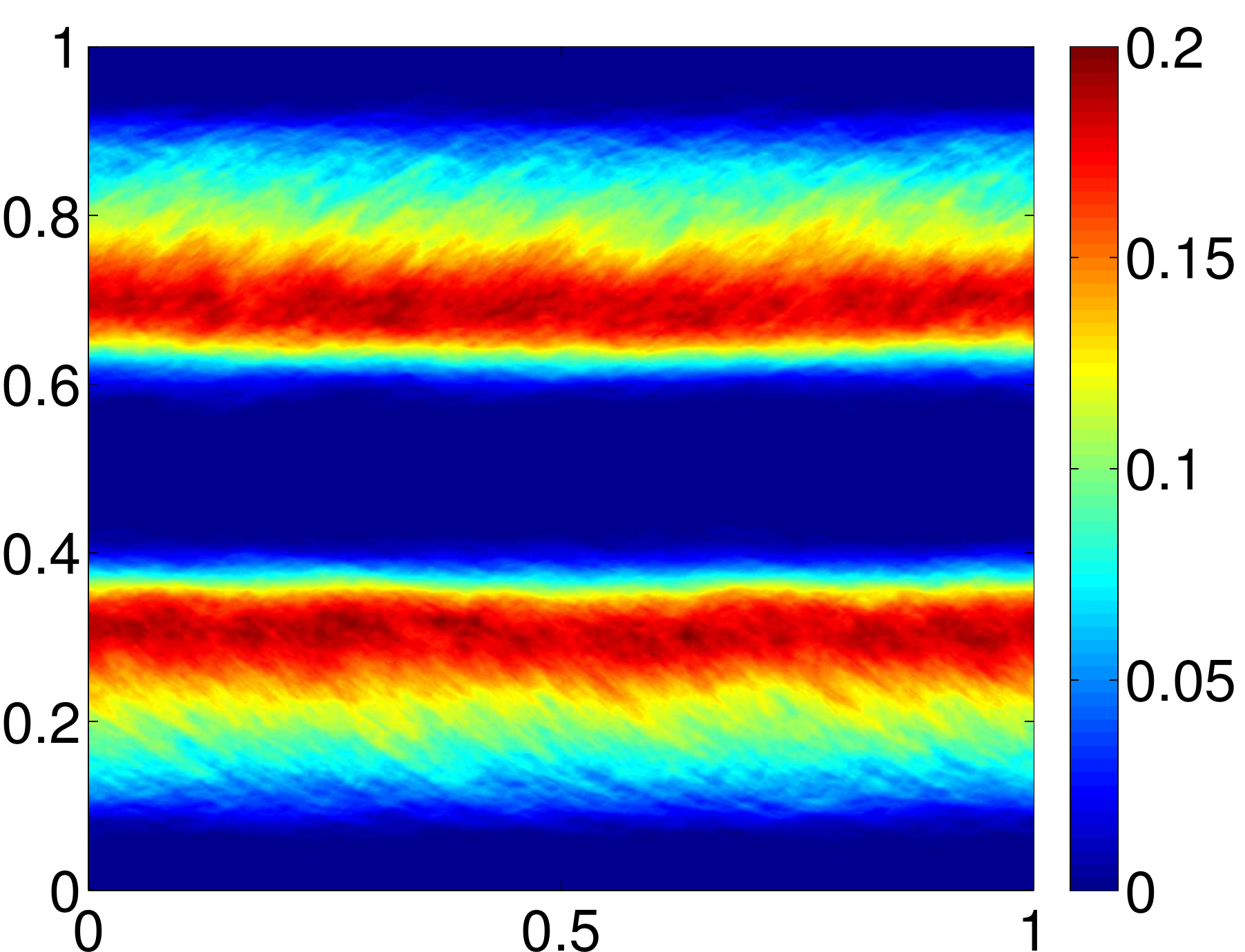}}
\caption{Mean (top) and variance (bottom) of the density for the Kelvin-Helmholtz problem with different initial data: phase perturbations \eqref{eq:kh} (left) and amplitude perturbations \eqref{eq:khi}, \eqref{eq:sodpinit} (right), at time $t=2$ at a resolution of $1024^2$ points and with $400$ Monte Carlo samples. All computations are with the TeCNO3 scheme.}
\protect \label{fig:mvs2}
\end{figure}

\begin{figure}
\centering
\subfigure[Mean, uniform distribution]{\includegraphics[width=0.45\linewidth]{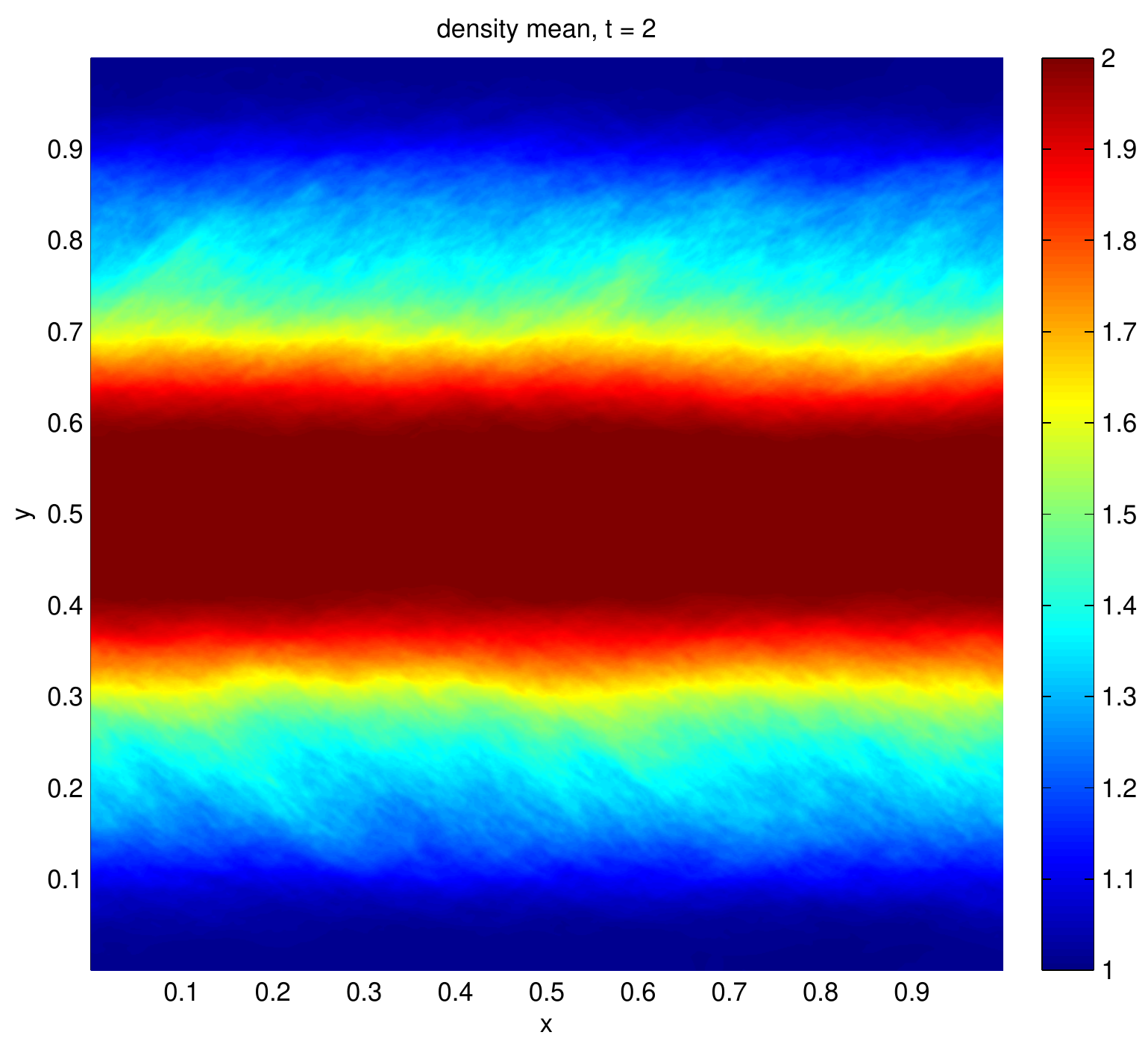}}
\subfigure[Mean, normal distribution]{\includegraphics[width=0.45\linewidth]{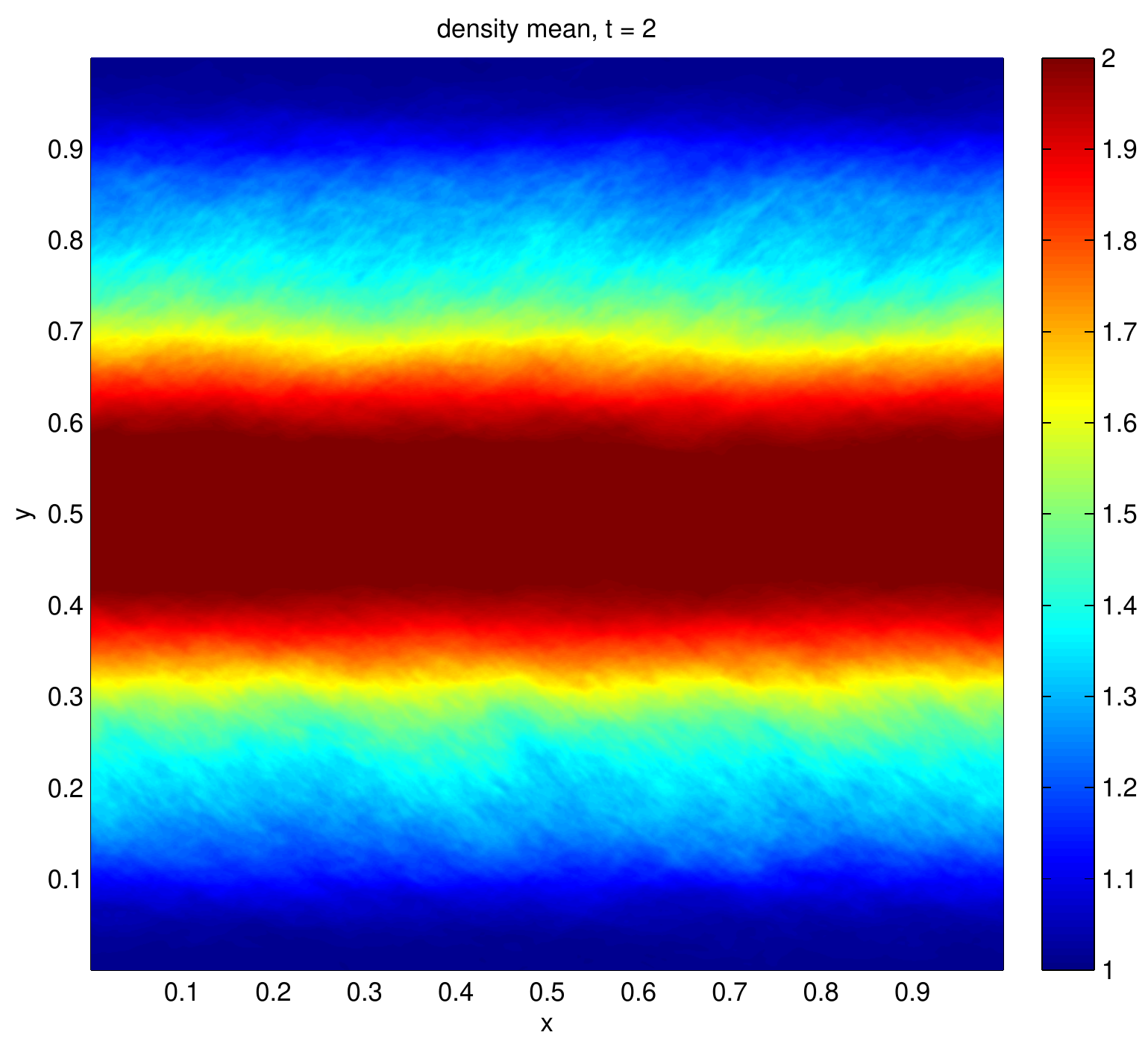}}\\
\subfigure[Variance, uniform distribution]{\includegraphics[width=0.45\linewidth]{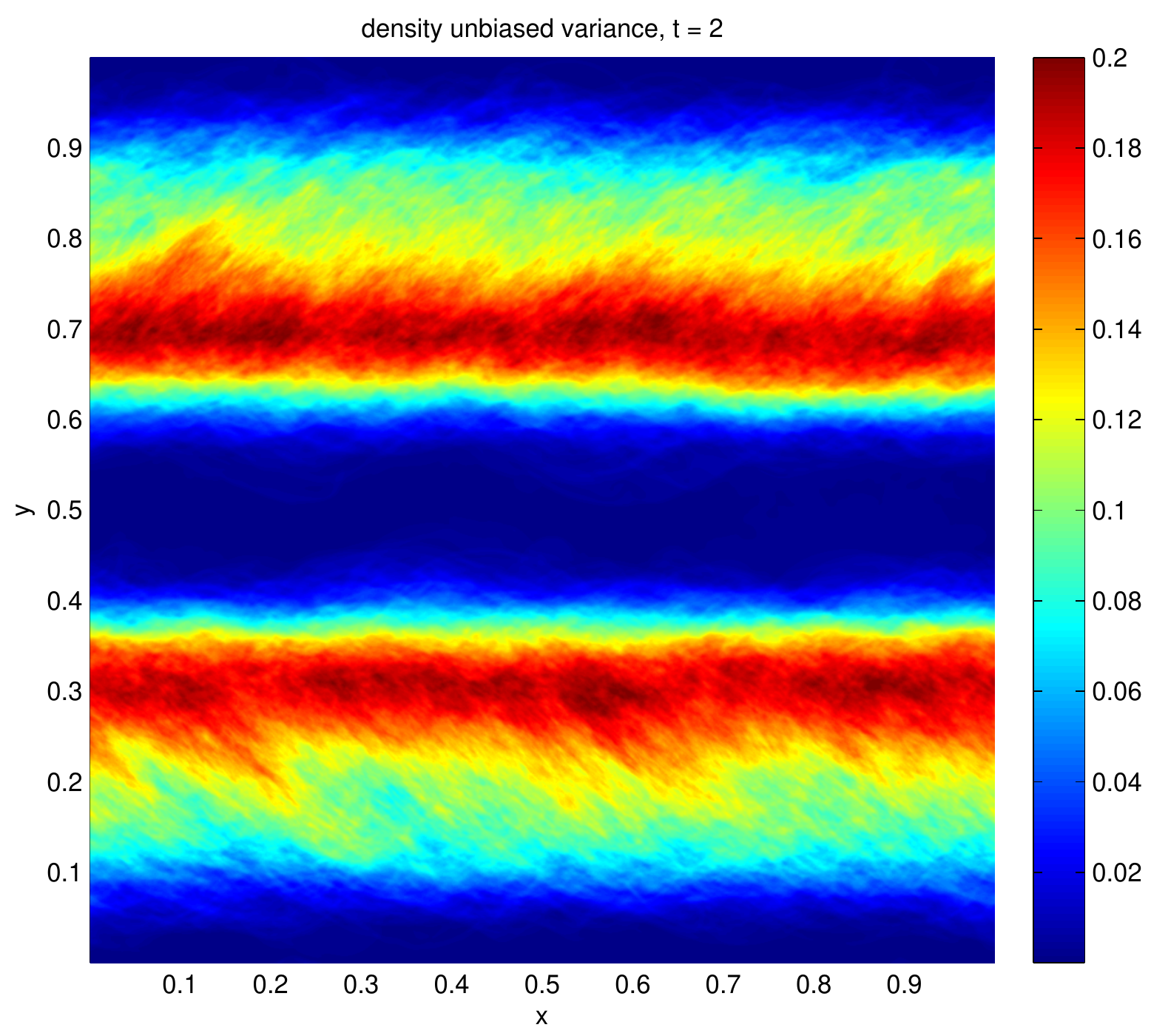}}
\subfigure[Variance, normal distribution]{\includegraphics[width=0.45\linewidth]{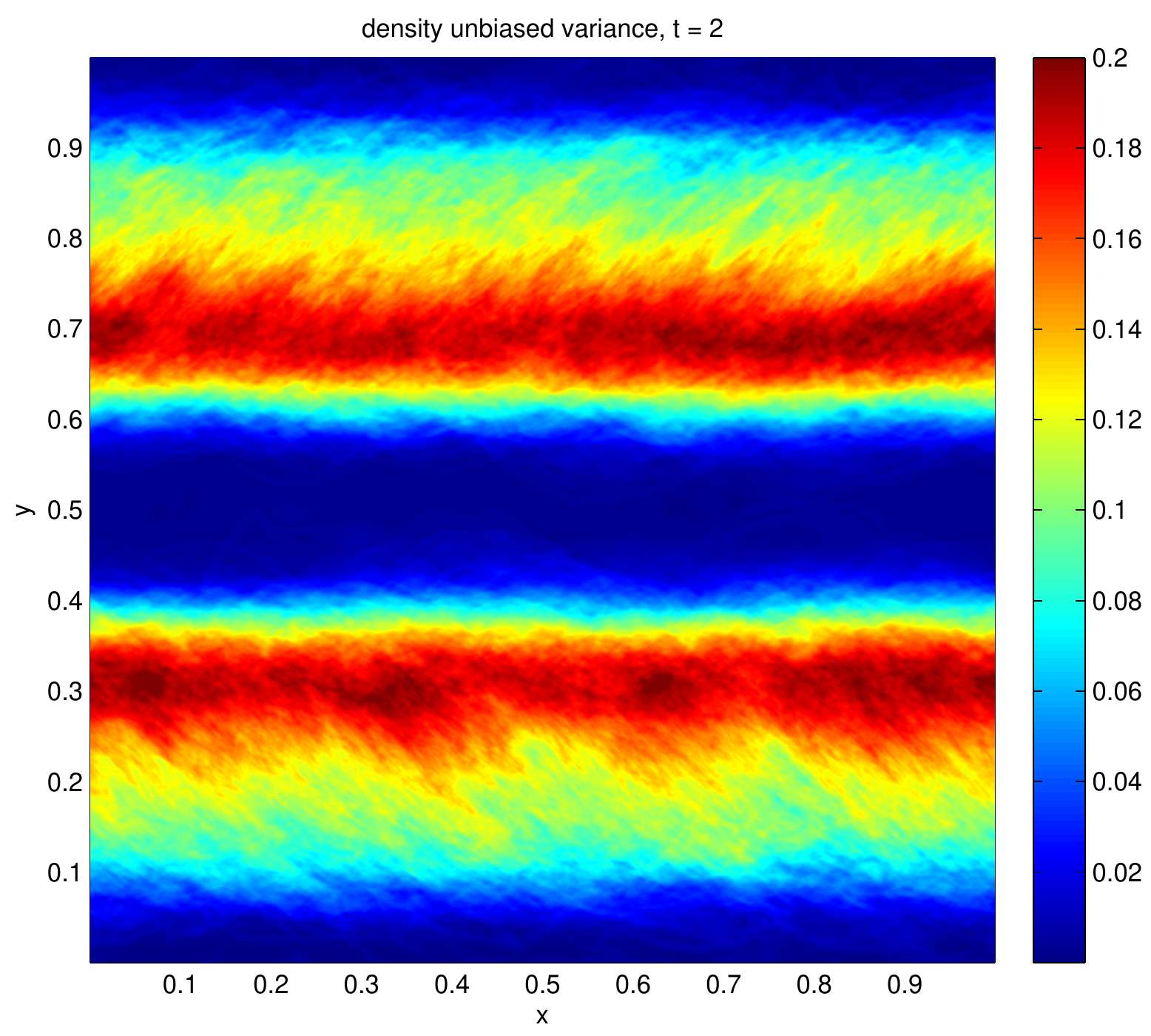}}
\caption{Mean (top) and variance (bottom) of the density for the Kelvin-Helmholtz problem with different initial data: discontinuous phase perturbations \eqref{eq:genper} with uniformly distributed coefficients in \eqref{eq:genper1} (left) and with (standard) normally distributed coefficients in \eqref{eq:genper1} (right), at time $t=2$ at a resolution of $1024^2$ points and with $400$ Monte Carlo samples. All computations are with the second-order FISH scheme.}
\protect \label{fig:mvs3}
\end{figure}

\section{Discussion}\label{sec:7}
We conclude with a brief discussion on the highlights of the current paper which are put in  perspective for future results.
Currently, the notion of entropy solutions is the generic framework for interpreting the notion of solutions for $N\times N$ systems of hyperbolic conservation laws \eqref{eq:cauchy} in $d$-spatial dimensions. Entropy solutions are \emph{bounded functions} which satisfy the equation  \eqref{eq:wsoln} and its associated  entropy inequality(-ies) \eqref{eq:entrcond} in the sense of distributions. Though the existence and uniqueness of entropy solutions has been established for  scalar conservation laws ($N=1$) and for one-dimensional systems ($d=1$), there are no known global existence and uniqueness (stability) results for generic multi-dimensional systems, when $N, d > 1$. In fact, recent papers \cite{CDL1,CDL2,CDL3} provide examples of multi-dimensional systems with infinitely many entropy solutions.

\subsection{What do the numerical experiments tell us}
Despite a wide variety of numerical methods, such as finite volume, finite difference and discontinuous Galerkin methods that have been developed and successfully employed to approximate systems of conservation laws, none of these methods has been shown to converge to an entropy solution for a generic system of conservation laws.
Given this background, we investigate here the issues of convergence of numerical approximations as well as the stability of the underlying entropy solutions. Our numerical experiments demonstrate that  even state of the art numerical methods may not necessarily {converge} as the mesh is refined. As shown in Figures \ref{fig:kh-density1} and \ref{fig:3}, finer and finer structures emerge as the grid is refined. The production  of oscillations at finer and finer scales prevents convergence under mesh refinement. We also present numerical experiments that demonstrate the lack of stability of entropy solutions with respect to perturbations of initial data; see Figures \ref{fig:kh-errors}(b) and \ref{fig:12}.

This lack of convergence to entropy solutions should not be considered as a failure of the numerical methods. Rather, they illustrate the shortcomings of the notion of entropy solutions to multi-dimensional systems of conservation laws. In particular, these experiments support the contention that entropy solutions are inadequate in describing some of the complex flow features that are modeled by systems of conservation laws such as the persistence of structures on finer and finer scales. Together with the recent results on the non-uniqueness of entropy solutions, our numerical evidence 
motivated us to seek a different, more versatile notion of solutions for these equations.

To this end, we focus on the notion of \emph{entropy measure-valued (EMV) solutions}, first introduced by DiPerna in \cite{DIP1}, see also \cite{DM1}. We propose a measure-valued Cauchy problem \eqref{eq:mvcauchy} and seek solutions that are Young measures (parametrized probability measures). These \emph{entropy measure valued solutions} are sought to be consistent with the underlying equations in the sense of distributions \eqref{eq:mvsoln} and satisfy a suitable form of the entropy inequality \eqref{eq:mventrineq}. \emph{The main aim} of the current paper was then  to design numerical procedure that can be rigorously shown to converge to an EMV solution. We work with an equivalent representation of measures as probability laws of random fields. The resulting initial random field is then evolved using a ``reliable'' entropy stable numerical scheme. The law of the resulting (random) weak* convergent approximations provides an approximation to the measure valued solution. For a numerical scheme to be weak* convergent, it is required to satisfy a set of minimal criteria outlined in Theorem  \ref{thm:convmv}:
\begin{itemize}
\item {\bf Uniform boundedness} of the approximations in $L^{\infty}$;
\item {\bf Discrete entropy inequality};
\item {\bf Space-time weak BV bound}.
\end{itemize}
 The TeCNO schemes of \cite{FTSID2} and the space-time DG schemes of \cite{HSID1} are examples of (formally) high-order schemes satisfying the discrete entropy inequality and weak BV bounds. The uniform $L^{\infty}$ bound is  a technical assumption that will be relaxed in a forthcoming paper \cite{FSID5}. 
Thus, we provide sufficient conditions that can guide the {design} of such ``reliable'' numerical methods for systems of conservation laws, with particular attention to multi-dimensional systems. Note that for systems of conservation laws, the above conditions play a role similar to that played by the well-known criteria of  discrete maximum principle(s), entropy inequalities and the TVD property in the numerical analysis of scalar conservation laws.

The convergence of numerical approximations to an EMV solution of \eqref{eq:mvcauchy} is interpreted in the \emph{weak*} sense, namely, that \emph{statistics} of space-time averages of the unknowns converge as the mesh is refined. 
A Monte Carlo method is used to approximate the EMV solution and we also show convergence of the resulting numerical procedure. To our knowledge,  this \emph{provides the first  set of rigorous convergence results for numerical approximations of generic multi-dimensional systems of conservation laws}.
These convergence results are illustrated by a large number of numerical experiments, and we make the following key observations:
\begin{itemize}
\item In general, there is no observed convergence of numerical approximations (neither in $L^1$ or in weaker norms) for single realizations (samples), with respect to increasing mesh resolutions. This has been demonstrated with two examples for the two-dimensional Euler equations.
\item However, as predicted by the theory, statistical quantities of interest such as the mean and the variance (or even higher moments) of an ensemble of solutions do converge as the mesh is refined.
\item In fact, a stronger convergence is observed. The approximate Young measures seem to converge in the strong sense \eqref{eq:sconv} to an EMV solution. 
\end{itemize}
The numerical approximation procedure, presented in Algorithm \ref{alg:atomic}, was also employed to compute EMV solutions with respect to atomic initial data. In general, the computed measure valued solution is \emph{not necessarily atomic}; see Figures \ref{fig:16}, \ref{fig:pdfTime} and \ref{fig:pdfRefine}. This is particularly striking in the specific case of the Kelvin-Helmholtz problem, where an entropy solution (the steady state data \eqref{eq:khi}), interpreted as an atomic entropy measure valued solution, exists but is unstable.  

This key observation implies that the solution operator can act to spread the support of the initial atomic measure. This bursting out of the initial atomic measure is, in our opinion, very significant. Just as the formation of shock waves precludes the existence of global classical solutions, leading to the replacement of point values with local averages as the appropriate solution concept, this observed bursting out of an initial atomic measure into a non-atomic measure implies that we have to look beyond integrable functions in order to obtain existence of solutions to systems of conservation laws. The concept of entropy measure-valued solutions, based on one-point statistics, appears to be a natural extension. In particular, given the proposed Algorithm \ref{alg:atomic}, we are also able to address Lax's question raised in Section \ref{sec:what}: what we are computing are the \emph{statistics}  --- ensemble average, variance etc.\ --- of an entropy measure valued solution.

\subsubsection*{Stability}
As the results of this paper and the forthcoming paper \cite{FSID5} show, the convergence of numerical approximations also provides a (constructive) proof of existence for EMV solutions of \eqref{eq:mvcauchy}. The questions of uniqueness and stability are much more delicate. From Remark \ref{rem:initialdata}, Example \ref{ex:nonunique} and the results of \cite{Sch89} and references therein, we know that EMV solutions may not be unique if the initial measure is non-atomic, even for scalar conservation laws. We propose a weaker stability concept, that of measure-valued stability. This concept implies possible stability for the statistics of space-time averages in problems where the initial measure is close to atomic. Numerical experiments examining this weaker concept of stability were presented in Section \ref{sec:mvstab}. From these experiments, we observed that 
\begin{itemize}
\item Different numerical schemes appear to converge to the same EMV solution as the mesh is refined.
\item Different types of perturbations of atomic initial data were considered and the resulting approximate EMV solutions seemed to converge to the same EMV solution, corresponding to atomic initial data.
\end{itemize}
These experiments indicate that our approximation procedure is indeed stable. Furthermore, they also suggest that the weaker notion of MV stability might be an appropriate framework to discuss the question of stability of EMV solutions.


\subsection{Issues for future investigation}
Our results raise several  issues which are left open. We conclude this section with a few comments suggesting possible paths for  future investigation.\newline
{\bf On the notion of stability}. The only rigorous results available are of the measure valued-strong uniqueness type, see Theorem \ref{thm:classicalsoln} and \cite{BDS1,DMT12}. Here, the stability is ensured when a classical solution (an atomic measure concentrated on a Lipschitz solution) is present. This also implies local (in time) uniqueness of EMV solutions for atomic initial data, concentrated on smooth functions. Given the paucity of rigorous stability results, there is a considerable open territory  for future  theoretical investigation of weaker concepts of stability, such as measure valued  stability for systems of conservation laws. Moreover, additional admissibility criteria such as entropy rate criteria of \cite{DAF1} or other variants might be necessary to ensure even MV stability of the EMV solution. This issue is dealt extensively in a forthcoming paper \cite{FLM15} where the concept of measure valued solutions is further augmented with additional admissibility criteria, in the form of conditions on multi-point correlations, that increase the chance of singling out a unique solution. 

\smallskip\noindent
{\bf Weak* convergent schemes}. As mentioned before, we  provided here a numerical procedure, as well as sufficient conditions on numerical schemes, such that the approximations converge to an EMV solution. Some examples of schemes satisfying these criteria were presented. These results will hopefully encourage the development of other kinds of numerical schemes, such as of the WENO, RKDG and spectral viscosity type, that satisfy the abstract criteria of this paper, and hence converge to measure valued solutions of systems of conservation laws, even in several space dimensions. 

\smallskip\noindent
{\bf Computing the measure valued solutions} requires evaluation of phase space integrals. Our proposal in this paper was to employ Monte Carlo sampling. This procedure can be very expensive computationally, on account of the slow convergence with respect to the number of samples. We foresee the design of more computationally efficient methods by adapting schemes such as Multi-level Monte Carlo \cite{schwabsid1,sssid1,sssid3}, stochastic collocation finite volume methods \cite{SRTSID1} and gPC based stochastic Galerkin methods \cite{DPL1}, which have recently been developed to deal with uncertainty quantification for systems of conservation laws. Such extensions are the subject of ongoing research.

\appendix
\section{Young measures}\label{app:young}
We provide here a very short introduction to Young measures. The reader may wish to consult \cite{Fol99,Bil08} for the theory of Radon measures and probability measures, and \cite{Bal89,Bal95} on the theory of Young measures.

\subsection{Probability measures}\label{app:probmeas}
\begin{enumerate}[label=\textbf{\ref{app:probmeas}.\arabic*},leftmargin=0cm,itemindent=1.1cm,labelsep=0.4cm,align=left]
\item
We denote by $\M(\R^N)$ the set of finite Radon measures on $\R^N$, which are inner regular Borel measures $\mu$ with finite total variation $|\mu|(\R^N)$. Let $C_0(\R^N)$ be the space of continuous real-valued functions on $\R^N$ which vanish at infinity, equipped with the supremum norm. Then it can be shown (see e.g.\ \cite[Section 7.3]{Fol99}) that $\M(\R^N)$ can be identified with the dual space of $C_0(\R^N)$ through the pairing $\ip{\mu}{g} = \int_{\R^N} g(\xi)\ d\mu(\xi)$. We do not distinguish between these two equivalent definitions of $\M$. By a slight abuse of notation, we shall sometimes write $\ip{\mu}{g(\xi)} = \int_{\R^N}g(\xi)\ d\mu(\xi).$ We will be particularly interested in the pairing $\ip{\mu}{\id} = \int_{\R^N}\xi\ d\mu(\xi)$ between $\mu$ and the identity function $\id(\xi) = \xi$.

\item
The duality between $C_0(\R^N)$ and $\M(\R^N)$ induces a weak* topology on $\M(\R^N)$, that of \emph{weak* convergence}. A sequence $\mu^n\in\M(\R^N)$ converges \emph{weak*} to $\mu\in\M(\R^N)$ provided $\ip{\mu^n}{g} \to \ip{\mu}{g}$ for all $g\in C_0(\R^N)$. (This is also called weak or vague convergence, see \cite{Bil08,Fol99}.)

\item\label{app:wasserstein}
The set of probability measures on $\R^N$ is the subset
\[
\Prob(\R^N) := \left\{\mu\in\M(\R^N)\ :\ \mu \geq 0,\ \mu(\R^N) = 1\right\}.
\]
Let $\Prob^p(\R^N) \subset \Prob(\R^N)$ for $p\in[1,\infty)$ denote the set of probability measures $\mu$ such that $\ip{\mu}{|\xi|^p} < \infty$. For $\mu,\rho\in\Prob^p(\R^N)$ the \emph{Wasserstein metric} $W_p$ is defined as
\[
W_p(\mu,\rho) := \inf\left\{\int_{\R^N\times\R^N}|\xi-\zeta|^p\ d\pi(\xi,\zeta)\ :\ \pi\in\Pi(\mu,\rho)\right\}^{1/p},
\]
where $\Pi(\mu,\rho)$ is the set of probability measures on $\R^N\times\R^N$ with marginals $\mu$ and $\rho$:
\[
\Pi(\mu,\rho) := \biggl\{\pi \in \Prob(\R^N\times\R^N)\ :\ \pi(A\times\R^N) = \mu(A),\ \pi(\R^N\times A) = \rho(A)\ \forall \text{ Borel } A\subset\R^N\biggr\}.
\]
It can be shown that $W_p$ for any $p$ metrizes the topology of weak convergence on $\Prob^p(\R^N)$ (see \cite[Proposition 7.1.5]{AGS05} or \cite[Chapter 7]{Vil}).

\item\label{app:wasscomp}
Let $\mu, \rho \in \Prob(\R)$, and let $F, G : \R\to[0,1]$ be their distribution functions,
\[
F(x) := \mu((-\infty,x]), \qquad G(y) := \rho((-\infty,y]).
\]
Then it can be shown that
\[
W_p(\mu, \rho) = \left(\int_0^1 \left|F^{-1}(s) - G^{-1}(s)\right|^p\ ds\right)^{1/p},
\]
see \cite[p.\ 75]{Vil}. This gives rise to an efficient algorithm for computing the Wasserstein distance between discrete probability distributions. Let $x_1, \dots, x_n$ and $y_1, \dots, y_n$ be random numbers drawn from the probability distributions $\mu$ and $\rho$, respectively, and define the discrete distributions $\mu_n := (\delta_{x_1} + \dots + \delta_{x_n})/n$ and $\rho_n := (\delta_{y_1} + \dots + \delta_{y_n})/n$. By the law of large numbers, we have $\mu_n \to \mu$ and $\rho_n \to \rho$ weak* as $n\to\infty$, almost surely. Moreover, their distribution functions are
\[
F_n(x) = \frac{\#\{x_j\ :\ x_j \leq x\}}{n}, \qquad G_n(y) = \frac{\#\{y_j\ :\ y_j \leq y\}}{n}.
\]
Hence, \emph{if the sequences $x_j$ and $y_j$ are sorted in increasing order}, then
\[
W_p(\mu_n, \rho_n)^p = \int_0^1 \left|F_n^{-1}(s) - G_n^{-1}(s)\right|^p\ ds = \frac{1}{n}\sum_{j=1}^n |x_j-y_j|^p.
\]
The latter expression is very easy to implement on a computer.

The analogous problem when $\mu,\rho\in\Prob(\R^N)$ is more complex, but can be solved in $O(n^3)$ time using the so-called Hungarian algorithm; see \cite{Mun57}.
\end{enumerate}

\subsection{Young measures}\label{app:youngmeas}
\begin{enumerate}[label=\textbf{\ref{app:youngmeas}.\arabic*},leftmargin=0cm,itemindent=1.1cm,labelsep=0.4cm,align=left]
\item
A \emph{Young measure} from $D\subset\R^k$ to $\R^N$ is a function which maps $\xt\in D$ to a probability measure on $\R^N$. More precisely, a Young measure is a weak* measurable map $\nu : D \to \Prob(\R^N)$, that is, the mapping $\xt \mapsto \ip{\nu(\xt)}{g}$ is Borel measurable for every $g\in C_0(\R^N)$. We denote the image of $\xt\in D$ under $\nu$ by $\nu_\xt := \nu(\xt) \in \Prob(\R^N)$. The set of all Young measures from $D$ into $\R^N$ is denoted by $\Young(D,\R^N)$. When $N=1$ we write $\Young(D) := \Young(D,\R)$.

\item\label{app:youngunifbounded}
A Young measure $\nu\in\Young(D,\R^N)$ is \emph{uniformly bounded} if there is a compact set $K\subset\R^N$ such that $\supp \nu_\xt \subset K$ for all $\xt\in D$. Note that if $\nu$ is atomic, $\nu = \delta_u$, then $\nu$ is uniformly bounded if and only if $\|u\|_{L^\infty(D)} < \infty$.

\item
If $u:\R^k\to\R^N$ is any measurable function then $\nu_{\xt} := \delta_{u(\xt)}$ defines a Young measure, and we have $u(\xt) = \ip{\nu_\xt}{\id}$ for every $\xt$. Conversely, we will say that a given Young measure $\nu$ is \emph{atomic} if it can be written as $\nu = \delta_{u}$ for a measurable function $u$.

\item
Two topologies on $\Young(D,\R^N)$ arise naturally in the study of Young measures: those of weak* and strong convergence. A sequence $\nu^n \in \Young(D,\R^N)$ converges \emph{weak*} to $\nu \in \Young(D,\R^N)$ if $\ip{\nu^n}{g} \wsto \ip{\nu}{g}$ in $L^\infty(D)$ for all $g\in C_0(\R^N)$, that is,
\[
\int_D \phi(z) \ip{\nu^n_z}{g}\ dz \to \int_D \phi(z) \ip{\nu_z}{g}\ dz \qquad \forall\ \phi\in L^1(D).
\]
We say that $\nu^n\in\Young(D,\R^N)$ converges \emph{strongly} to $\nu\in\Young(D,\R^N)$ if
\[
\bigl\|W_p(\nu^n,\nu)\bigr\|_{L^p(D)} \to 0
\]
for some $p\in[1,\infty)$. If $\nu$ is atomic, $\nu = \delta_u$ for some $u:D\to\R^N$, then $\nu^n \to \nu$ strongly if and only if
\[
\int_D \int_{\R^N} |\xi - u(z)|^p\ d\nu^n_z(\xi) dz \to 0.
\]

\item
The \emph{fundamental theorem of Young measures} was first introduced by Tartar for $L^\infty$-bounded sequences \cite{T1} and then generalized by Schonbek \cite{Sch82} and Ball \cite{Bal89} for  sequences of measurable functions. We provide a further generalization: every sequence $\nu^n\in\Young(D,\R^N)$ which does not ``leak mass at infinity'' (condition \eqref{eq:ballbound}) has a weak* convergent subsequence:
\begin{theorem}\label{thm:young}
Let $\nu^n \in \Young(D,\R^N)$ for $n\in\N$ be a sequence of Young measures. Then there exists a subsequence $\nu^m$  which converges weak* to a nonnegative measure-valued function $\nu:D\to\M_+(\R^N)$ in the sense that
\begin{itemize}
\item[(i)] $\ip{\nu^m_\xt}{g} \wsto \ip{\nu}{g}$ in $L^\infty(D)$ for all \ $g\in C_0(\R^N)$,
\end{itemize}
and moreover satisfies
\begin{itemize}
\item[(ii)] $\|\nu_\xt\|_{\M(\R^N)} \leq 1$ for a.e.\ $\xt\in D$;
\item[(iii)] If $K\subset\R^N$ is closed and $\supp \nu^n_\xt \subset K$ for a.e.\ $\xt\in D$ and $n$ large,  then $\supp \nu_\xt \subset K$ for a.e.\ $\xt\in D$.
\end{itemize}
Suppose further that for every bounded, measurable $E \subset D$, there is a nonnegative $\ballfunc\in C(\R^N)$ with $\lim_{|\xi|\to\infty}\ballfunc(\xi)=\infty$ such that
\begin{equation}\label{eq:ballbound}
\sup_n \int_E \ip{\nu^n_z}{\ballfunc}\ dz < \infty.
\end{equation}
Then
\begin{itemize}
\item[(iv)] $\|\nu_\xt\|_{\M(\R^N)}=1$ for a.e.\ $\xt\in D$,
\end{itemize}
whence $\nu \in \Young(D,\R^N)$.
\end{theorem}
\begin{proof}
The proof is a generalization of Ball \cite{Bal89}.

Denote by $L_{w}^\infty(D;\M(\R^N))$ the set of weak* measurable functions $\mu:D\to\M(\R^N)$, equipped with the norm
$$
\|\mu\|_{\infty,\M}:= \esssup_{\xt\in D}\|\mu_\xt\|_\M.
$$
From the fact that $C_0(\R^N)$ is separable it can be shown (see \cite[Theorem 8.18.2]{Edw}) that $L_{w}^\infty(D;\M(\R^N))$ is isometrically isomorphic to the dual of $L^1(D;C_0(\R^N))$. The sequence $\mu^n$ is bounded in $L_{w}^\infty(D;\M(\R^N))$ since $\|\mu^n\|_{\infty,\M}\equiv 1$, and hence there is a $\mu\in L_{w}^\infty(D;\M(\R^N))$ and a weak* convergent subsequence $\mu^m$ of $\mu^n$ such that $\ip{\mu^m}{\Psi}_{\infty,\M}\to\ip{\mu}{\Psi}_{\infty,\M}$, or equivalently,
$$
\int_{D} \ip{\mu^m_\xt}{\Psi(\xt,\cdot)}~ d\xt \to \int_{D} \ip{\mu_\xt}{\Psi(\xt,\cdot)}~ d\xt \qquad \text{as }m\to\infty
$$
for all $\Psi\in L^1(D;C_0(\R^N))$. In particular, letting $\Psi(\xt,\xi)=\phi(\xt)g(\xi)$ for $\phi\in L^1(D)$ and $g\in C_0(\R^N)$, we obtain \textit{(i)}. We claim that $\mu_\xt \geq 0$ for a.e.\ $\xt\in D$. If not, then there would be a nonnegative $\Psi\in L^1(D;C_0(\R^N))$ such that $\int_D\ip{\mu_\xt}{\Psi(\xt,\cdot)}\ d\xt < 0$. But then
\[
0 > \int_{D} \ip{\mu_\xt}{\Psi(\xt,\cdot)}~ d\xt = \lim_{m\to\infty} \int_D \ip{\mu^m_\xt}{\Psi(\xt,\cdot)}~ d\xt \geq 0
\]
(since $\mu^m_\xt \geq 0$ for all $\xt$), a contradiction.

\textit{(ii)} follows from the weak* lower semicontinuity of the norm $\|\cdot\|_{\infty,\M}$. To see that \textit{(iii)} holds, let $g\in C_0(\R^N)$ be such that $g\bigr|_{K} = 0$. Since $\mu^m \to K$ in measure, it follows that $\ip{\mu^m}{g} \to 0$ in measure (that is, $|\{\xt\in D:|\ip{\mu^m_\xt}{g}|>\delta \}| \to 0$ for all $\delta>0$). Hence,
$$
\int_D \phi(\xt)\ip{\mu_\xt}{g}~d\xt = \lim_m \int_D \phi(\xt)\ip{\mu^m_\xt}{g}~d\xt = 0
$$
for all $\phi\in L^1(D)$, 
and therefore $\ip{\mu_\xt}{g} = 0$ for a.e.~ $\xt\in D$. This is precisely \textit{(ii)}.

Assume now that \eqref{eq:ballbound} holds. Fix a set $E\subset D$ of finite, nonzero Lebesgue measure $|E|$, and denote the average integral over $E$ as $\intavg_E = \frac{1}{|E|}\int_E$. For every $R>0$ we define
$$
\theta_R(\xi) = \begin{cases}
1 & \ballfunc(\xi) \leq R \\
1+R-\ballfunc(\xi) & R<\ballfunc(\xi)\leq R+1 \\
0 & R+1 < \ballfunc(\xi).
\end{cases}
$$
Then $\theta_R \in C_0(\R^N)$, so
$$
\lim_m\intavg_E \ip{\mu^m_\xt}{\theta_R}~d\xt = \intavg_E \ip{\mu_\xt}{\theta_R}~d\xt \leq \intavg_E \|\mu_\xt\|_\R~d\xt \leq 1,
$$
the last inequality following from the fact that $\|\mu_\xt\|_\R \leq 1$ for all $\xt$. Conversely,
\[
0 \leq \intavg_E \left(1-\ip{\mu^m_\xt}{\theta_R}\right)\ d\xt = \intavg_E \ip{\mu^m_\xt}{1-\theta_R}\ d\xt \leq
\frac{1}{R}\intavg_E \ip{\mu^m_z}{\ballfunc}\ d\xt,
\]
so \eqref{eq:ballbound} gives
\begin{align*}
1 &\leq \lim_{R\to\infty}\lim_m\intavg_E \ip{\mu^m_\xt}{\theta_R}~d\xt + \lim_{R\to\infty} \sup_m \frac{1}{R}\intavg_E \ip{\mu^m_z}{\ballfunc}\ d\xt \\
&= \lim_{R\to\infty}\intavg_E \ip{\mu_\xt}{\theta_R}~d\xt \\
&\leq \intavg_E \|\mu_\xt\|_{\M(\R^N)}~d\xt \leq 1,
\end{align*}
whence $\intavg_E \|\mu_\xt\|_{\M(\R^N)}\ d\xt = 1$. Since $E\subset D$ is arbitrary, \textit{(iv)} follows.

\end{proof}

\item\label{app:younglpbound}
An important special case of \eqref{eq:ballbound} is when $\ballfunc(\xi) = |\xi|^p$ for $1\leq p < \infty$, which translates to the $L^p$ bound
\[
\sup_n\int_D \ip{\mu^n}{|\xi|^p}\ dz < \infty.
\]
The case $p=\infty$ translates to the support of $\nu^n_z$ lying in a compact set $K\subset \R^N$ for a.e.\ $z$ and all $n$. Part \textit{(iii)} of Theorem \ref{thm:young} then holds for all $g\in C(\R^N)$, and condition \eqref{eq:ballbound} is automatically satisfied for any such $\ballfunc$. The latter is the original form of the theorem given by Tartar \cite{T1}.
\end{enumerate}

\subsection{Random fields and Young measures}\label{app:randyoung}
\begin{enumerate}[label=\textbf{\ref{app:randyoung}.\arabic*},leftmargin=0cm,itemindent=1.1cm,labelsep=0.4cm,align=left]
\item\label{app:younglaw}
If $(\Omega,\Sigmaalg,P)$ is a probability space, $D\subset\R^k$ is a Borel set and $u : \Omega\times D \to \R^N$ is a random field (i.e., a jointly measurable function), then we can define its law by
\begin{subequations}\label{eq:ulaw}
\begin{equation}
\nu_{\xt}(F) := P\left(u(\xt) \in F\right) = P\left(\left\{\omega\ :\ u(\omega,\xt) \in F\right\}\right)
\end{equation}
for Borel subsets $F\subset\R^N$ of phase space, or equivalently,
\begin{equation}
\ip{\nu_\xt}{g} := \int_\Omega g(u(\omega,z))\ dP(\omega)
\end{equation}
\end{subequations}
for $g\in C_0(\R^N)$. This defines a Young measure:
\begin{proposition}\label{prop:lawwelldef}
If $u:\Omega\times D \to \R^N$ is jointly measurable then \eqref{eq:ulaw} defines a Young measure from $D$ to $\R^N$.
\end{proposition}
\begin{proof}
\newcommand{\ua}{\hat{u}}
\newcommand{\ub}{\tilde{u}}
\newcommand{\Ell}{\EuScript{L}}
First of all, for fixed $\xt\in D$ the set $\bigl\{\omega : u(\omega,\xt) \in U\bigr\}$ is $P$-measurable for Borel sets $U$. Indeed, if $w(\omega) := u(\omega,\xt)$ denotes the $\xt$-section of the measurable function $(\omega,y) \mapsto u(\omega,y)$, then $\bigl\{\omega : u(\omega,\xt) \in U\bigr\} = w^{-1}(U)$ is measurable.

We need to show that the definition of $\nu$ is independent of the choice of mapping in the equivalence classes of mappings from $\Omega\times D \to \R^N$. Let $\ua, \ub : \Omega\times D \to \R^N$ be two mappings such that $\ua(\omega,\xt) = \ub(\omega,\xt)$ for $P\times\lambda$-a.e.\ $(\omega,\xt)$. We apply Tonelli's theorem to find that
\begin{gather*}
0 = \int_{\Omega\times D} \ind_{\{\ua \neq \ub\}}(\omega,\xt)\ d(P\times\lambda)(\omega,\xt)
= \int_D P(\{\ua(\xt) \neq \ub(\xt) \})\ d\xt.
\end{gather*}
Hence, $P(\ua(\xt) \neq \ub(\xt)) = 0$ for a.e.\ $\xt\in D$, so for every Borel set $U\subset\R^N$,
\[
P\left(\ua(\xt) \in U\right) = P\left(\ub(\xt) \in U\right)
\]
for a.e.\ $\xt\in D$.

Finally, $\nu$ is weak* measurable since
\[
\ip{\nu_\xt}{g} = \int_{\R^N}g(\xi)\ d\nu_\xt(\xi) = \int_{\Omega}g(u(\omega,\xt))\ dP(\omega),
\]
which is measurable in $\xt$ for any $g\in C_0(\R^N)$.

\let\ua\undefined
\let\ub\undefined
\let\Ell\undefined
\end{proof}

\item It is well known that every measure on $\R^N$ can be realized as the law of a random variable. Here we show  that for  every Young measure $\nu$, there is always a random field with law $\nu$. 

\begin{proposition}\label{prop:lawexists}
For every Young measure $\nu\in \Young(D,\R^N)$ there exists a probability space $(\Omega, \Sigmaalg, P)$ and a Borel measurable function $u : \Omega\times D \to \R^N$ such that $u$ has law $\nu$, i.e. for all Borel sets $E$,
\[
\nu_\xt(E) = P(u(\omega,\xt) \in E).
\]
In particular, we can choose $(\Omega, \Sigmaalg, P)$ to be the Borel $\sigma$-algebra on $\Omega=[0,1)$ with Lebesgue measure.
\end{proposition}
\begin{proof}
The method of proof is standard; see e.g.\ \cite[Theorem 5.3]{Bil95}.

We assume that $N=1$. The generalization to $N>1$ is straightforward but tedious. For $n\in\N$ and $j\in\Z$, we set
\[
F_n^j := \begin{cases}
(-\infty, -2^{n}) & \text{if }j = -2^{2n} \\
\bigl[2^{-n}(j-1), 2^{-n}j\bigr) & \text{if } j = -2^{2n}+1, \dots, 2^{2n} \\
[2^n, \infty) & \text{if } j = 2^{2n}+1\\
\emptyset & \text{otherwise.}
\end{cases}
\]
Let $p_n^j(\xt) := \sum_{l\leq j} \nu_\xt(F_n^l)$. Note that $p_n^j : \R \to [0,1]$ is measurable for all $n,j$, and that $0\leq p_n^{-j}\leq \dots \leq p_n^j=1$ for $j$ large enough. Choose any $\xi_n^j \in F_n^j$, and for $\omega\in\Omega :=[0,1)$, define
\[
u_n(\omega,\xt) := \xi_n^j \qquad \text{for $j$ such that } p_n^{j-1}(\xt) \leq \omega < p_n^{j}.
\]
We claim that $u_n$ is measurable on the product $\sigma$-algebra between $\Sigmaalg$ and the Borel $\sigma$-algebra on $D$. Each function $u_n$ takes only finitely many values $\xi_n^j$, so it suffices to show that $u_n^{-1}(\{\xi_n^j\})$ is measurable for every $\xi_n^j$. Indeed,
\begin{align*}
u_n^{-1}(\{\xi_n^j\}) &= \Bigl\{(\omega,\xt)\in\Omega\times D\ :\ p_n^j(\xt) \leq \omega < p_n^{j+1}(\xt) \Bigr\} \\
&= \Bigl(\Omega\times D\Bigr) \cap \Bigl\{(\omega,\xt)\in\R\times D\ :\ p_n^j(\xt) \leq \omega\Bigr\} \cap \Bigl\{(\omega,\xt)\in\R\times D\ :\ \omega < p_n^{j+1}(\xt) \Bigr\},
\end{align*}
the intersection between the epigraph of $p_n^j$ and the hypograph of $p_n^{j+1}$, which are measurable by the measurability of the functions $p_n^j$ and $p_n^{j+1}$. 

Because the partition $\{F_m^j\}_{j\in\Z}$ is a refinement of $\{F_n^j\}_{j\in\Z}$ whenever $m>n$, it follows that $|u_n(\omega,\xt) - u_m(\omega,\xt)| < {\rm diam}(F_n^j) = 2^{-n}$ for any $(\omega,\xt)$ whenever $m,n$ are large enough. Hence, $u_n$ converges pointwise to some function $u : \Omega\times D \to \R$, which is measurable by the measurability of each $u_n$.

Finally, for every $g\in C_0(\R)$ and almost every $\xt\in D$, we have by Lebesgue's dominated convergence theorem
\begin{align*}
\int_\Omega g(u(\omega,\xt))\ dP(\omega) = \lim_n \int_\Omega g(u_n(\omega,\xt))\ dP(\omega) = \lim_n \sum_j \nu_\xt(F_n^j)g(\xi_n^j) = \int_\R g(\xi)\ d\nu_\xt(\xi).
\end{align*}
Hence, $u(\cdot,\xt)$ has law $\nu_\xt$.
\end{proof}
\end{enumerate}

\section{Proof of Theorem \ref{thm:mcconv}}
\label{app:MCproof}
\begin{proof}
For any random field $\zeta: \Omega \to L^1(\R^d \times \R_+) \cap L^{\infty}(\R^d \times \R_+)$ on $(\Omega,\Sigmaalg,P)$, we denote the expectation with respect to the probability measure $P$ as
$$
\E(\zeta) := \int\limits_{\Omega} \zeta(\omega) dP(\omega).
$$
For $1 \leq k \leq M$, denote
\begin{equation}\label{eq:mc5}
\begin{aligned}
G(\omega) &= \int_{\R_+}\int_{\R^d} \psi(x,t) g(u^{\Dx}(\omega;x,t)) dx dt,  \\
G_k(\omega) &= \int_{\R_+}\int_{\R^d} \psi(x,t) g(u^{\Dx,k}(\omega;x,t)) dx dt.
\end{aligned}
\end{equation}
Henceforth we suppress the $\omega$-dependence of $G$ and $G_k$ for notational convenience. The $L^2(P)$ error in the approximation can be written as
\begin{align*}
\E\left(\left(\E(G) - \frac{1}{M}\sum\limits_{k=1}^M G_k \right)^2\right) &= \E \left(\frac{1}{M^2}\left(\sum\limits_{k=1}^M (\E(G) - G_k) \right)^2 \right), \\
&= \E \left( \frac{1}{M^2} \left(\sum\limits_{k=1}^M \bigl(\E(G) - G_k\bigr)^2 + 2\sum_{k=1}^{M}\sum_{l \neq k} \bigl(\E(G) - G_k\bigr)\bigl(\E(G) - G_l\bigr)\right)\right) \\
&= \underbrace{\frac{1}{M^2}\sum\limits_{k=1}^M \E\left(\bigl(\E(G) - G_k\bigr)^2\right)}_{=:\ T_1} + \frac{2}{M^2} \sum_{k=1}^{M}\sum_{l \neq k}\underbrace{\E \Bigl( \bigl(\E(G) - G_k\bigr) \bigl(\E(G) - G_l\bigr)\Bigr)}_{=:\ T^{kl}_2}.
\end{align*}

As $u^{\Dx,1}, \ldots, u^{\Dx,M}$ are independent and identically distributed, it follows from the definition of $G_k$ that $G_1,\dots,G_M$ are independent and identically distributed random variables. Hence, $\E(G_k) = \E(G)$ and $\E(G_kG_l) = \E(G_k)\E(G_l)$ for all $k,l$. Consequently, a simple calculation shows that $T^{kl}_2=0$ for all $1 \leq k,l \leq M$ and $k \neq l$.

The fact that $G_1,\dots,G_M$ are independent and identically distributed yields
$$
T_1 = \frac{1}{M} \left(\E(G^2) - \E(G)^2\right).
$$
Hence,
\begin{align*}
\E\left(\left(\E(G) - \frac{1}{M}\sum\limits_{k=1}^M G_k \right)^2\right) &=  \frac{1}{M} \left(\E(G^2) - (\E(G))^2\right) \\
&\leq \frac{1}{M}\|g(u^{\Dx})\|^2_{L^{\infty}(\Omega \times \R^d \times \R_+)} \|\psi\|^2_{L^1(\R^d \times \R_+)} &&\text{(by definition \eqref{eq:mc5})} \\
&\leq \frac{C}{M} &&\text{(by assumption \eqref{eq:linf})}.
\end{align*}
In conclusion, the sample mean
\[
\frac{1}{M} \sum_{k=1}^{M} \int_{\R_+}\int_{\R^d} \psi(x,t) g(u^{\Dx,k}(x,t))\ dxdt
\]
converges to the corresponding ensemble average
\[
\int_{\R_+}\int_{\R^d} \psi(x,t) \ip{\nu^{\Dx}_{x,t}}{g}\ dxdt
\]
in $L^2(\Omega;P)$ with a convergence rate of $\frac{1}{\sqrt{M}}$. Taking a subsequence $M'\to\infty$, the convergence also holds $P$-almost surely.
\end{proof}

\section{Time continuity of approximations}\label{app:timecont}
From the time integration procedure \eqref{eq:exactint} we can show that the approximate MV solutions are time continuous. Consequently, the initial data is attained in a certain sense, and moreover, it is meaningful to evaluate the MV solution at a specific time $t$.

We state the theorem without proof, since the results are straightforward generalizations of ``deterministic'' counterparts.

\begin{theorem}
Let $\psi \in C_c^1(\R)$ and assume that \eqref{eq:linf} and \eqref{eq:tvbound} are satisfied. Let $\nu^\Dx$ be generated by Algorithm \ref{alg:approxmv}. Then the functions
\[
\Psi^\Dx(t) := \int_\R \psi(x)\ip{\nu^\Dx_{(x,t)}}{\id}\ dx
\]
and
\[
\Psi(t) := \int_\R \psi(x) \ip{\nu_{(x,t)}}{\id}\ dx
\]
are H\"older continuous with exponent $\holder := \frac{r-1}{r}$ and with constant independent of $\Dx$, and $\Psi^\Dx(t) \to \Psi(t)$ as $\Dx\to 0$ for a.e.\ $t\in[0,T]$. Moreover,
\[
\Psi(0) = \lim_{t\to 0} \Psi(t) = \int_\R \psi(x) \ip{\sigma_{x}}{\id}\ dx.
\]
\end{theorem}

\end{document}